\documentclass[10pt,leqno]{article}
\usepackage{amsmath,amsfonts,amsthm,amscd,amssymb}
\numberwithin{equation}{section}
\theoremstyle{plain}
\newtheorem{theo}{Theorem}[section]
\newtheorem{prop}[theo]{Proposition}
\newtheorem{coro}[theo]{Corollary} 
\newtheorem{lemm}[theo]{Lemma}
\theoremstyle{definition}
\newtheorem{defi}[theo]{Definition}
\newtheorem{rema}[theo]{Remark}
\newtheorem{theo-defi}[theo]{Theorem-Definition}
\newtheorem{prop-defi}[theo]{Proposition-Definition}
\newtheorem{rema-defi}[theo]{Remark-Definition}
\newtheorem{exem-defi} [theo]{Example-Definition}

\newtheorem{exem}[theo]{Example}
\newtheorem{conj}[theo]{Conjecture}
\newtheorem{prob}[theo]{Problem}

\def \al{\alpha}

\def \bul{\bullet}
\def \col{\colon}
\def \Del{\Delta}

\def \del{\delta}
\def \eps{\epsilon}
\def \Gam{\Gamma}
\def \gam{\gamma}
\def \inf{\infty}

\def \kap{\kappa}
\def \Lam{\Lambda}
\def \lam{\lambda}

\def \Lo{\Longrightarrow}
\def \lo{\longrightarrow}
\def \lom{\longmapsto}
\def \mab{\mathbb}

\def \nat{\natural} 

\def \Om{\Omega}
\def \om{\omega}
\def \ol{\overline}
\def \os{\overset}
\def \parno{\par\noindent}

\def \sig{\sigma}
\def \sq{\square}
\def \sus{\subset}

\def \ul{\underline}
\def \us{\underset}

\def \vp{\varpi}

\def \wt{\widetilde}
\bigskip

\newcommand{\getsfrom}{\ensuremath{
\longleftarrow\kern-.50em\lower.0ex\hbox%
{$\shortmid\,$}}}

\begin{document}

\title{The zariskian $p$-adic bifiltered El Zein-Steenbrink-Zucker complex
of a proper SNCL scheme with a relative SNCD}
%and a $p$-adic relative monodromy filtration}
\author{Yukiyoshi Nakkajima 
 \date{}\thanks{
 \endgraf}}
\maketitle

%\title[$p$-adic El Zein-Steenbrink-Zucker complex]{The zariskian $p$-adic bifiltered El Zein-Steenbrink-Zucker complex 
%of a proper SNCL scheme with a relative SNCD} 
%\author{Yukiyoshi Nakkajima}  
%\thanks{
%2020 Mathematics subject classification number: 14F30. Keywords: 
%log crystalline cohomology, weight filtration, relative monodromy filtration.  
%The authors is supported by JSPS
%KAKENHI Grant Numbers JP24K06652.} 
%\maketitle

\bigskip
\parno
{\bf Abstract.---}
This paper aims to formulate the log $p$-adic 
relative monodromy-weight conjecture. 
This conjecture 
is a generalization of the famous $p$-adic monodromy-weight conjecture 
by A.~Mokrane (and the author) (\cite{msemi}, \cite{nb}). 
We prove that, if the log $p$-adic monodromy-weight conjecture by 
the author (\cite{nb}) is true, then the log $p$-adic relative monodromy-weight conjecture 
is true. By using this result, we prove that 
the log $p$-adic relative monodromy-weight conjecture is true in certain cases.

$${\bf Contents}$$
\parno
\S\ref{sec:intro}. Introduction
\parno 
\S\ref{sec:snclv}. SNCL schemes 
\parno 
\S\ref{sec:snrdlv}. SNCL schemes with relative SNCD's
\parno 
\S\ref{sec:pwtr}. Preweight filtrations
\parno 
\S\ref{sec:psc}.  Zariskian $p$-adic bifiltered El Zein-Steenbrink-Zucker complexes
\parno 
\S\ref{sec:fcuc}.  Contravariant functorialities of zariskian 
$p$-adic bifiltered El Zein-Steenbrink-Zucker complexes
\parno
\S\ref{sec:mod}.  Monodromy operators
\parno 
\S\ref{sec:bckf}. Bifiltered base change theorem 
\parno 
\S\ref{sec:infhdi}. Infinitesimal deformation invariance 
\parno 
\S\ref{sec:filbo}. The $E_2$-degeneration of the $p$-adic weight spectral sequence
\parno 
\S\ref{sec:e2}. Log convergences of the weight filtrations 
\parno 
\S\ref{sec:st}.  Strict compatibility
\parno 
\S\ref{sec:mn}. Log $p$-adic relative monodromy-weight conjecture
\medskip 
\parno 
{\bf Appendix}
\medskip 
\parno 
\S\ref{sec:lbddmlif}. Edge morphisms between 
the $E_1$-terms of $p$-adic weight spectral sequences

\smallskip
\parno
References

\section{Introduction}\label{sec:intro} 
In \cite{nh2} we have constructed the weight filtration on 
the log crystalline cohomological sheaf 
of a proper smooth scheme with a relative SNCD(=simple normal crossing divisor) 
in characteristic $p>0$; 
in \cite{nb} and \cite{nhir} we have constructed the weight filtration on 
the log crystalline cohomological sheaf of 
a proper SNCL(=simple normal crossing log) scheme in characteristic $p>0$. 
This paper is a continuation of my previous papers \cite{nh2} and \cite{nb}. 
In this paper we construct the weight filtration on the log crystalline cohomological sheaf 
of a proper SNCL scheme with a relative SNCD in characteristic $p>0$. 
This paper simultaneously generalizes parts of \cite{nh2} and \cite{nb}. 
\par  
In \cite{stz} Steenbrink and Zucker have 
constructed a bifiltered complex for a 
proper strict semistable family with a relative  SNCD over the unit disk over ${\mab C}$. 
In \cite{ezth} El Zein has also constructed the same bifiltered complex. 
They have proved that there exists a relative monodromy filtration 
with respect to the weight filtration arising from the relative SNCD, provided  
that the semistable family is projective. (However one must use M.~Saito's famous result 
in \cite{sm}.) 
Inspired by their work, in \cite{nlf}, 
we have constructed an analogous $l$-adic bifiltered complex for 
the $l$-adic Kummer log \'{e}tale 
cohomological sheaf of a proper SNCL scheme with a relative SNCD in any characteristic. 
By using this bifiltered complex, 
we have proposed a conjectural 
$l$-adic relative monodromy weight filtration on 
the $l$-adic Kummer log \'{e}tale cohomology of it 
and we have proved that this filtration exists in certain cases. 
Because the relative monodromy filtration is expected to be motivic, 
in this paper
we construct the $p$-adic analogue of 
the conjectural $l$-adic relative monodromy filtration
for the log crystalline cohomological sheaf of 
a proper SNCL scheme with a relative SNCD in characteristic $p>0$. 
\par 
Before stating our main result in this paper, we briefly recall a result in \cite{nb}. 
\par 
For a log (formal) scheme $Y$, denote by $\os{\circ}{Y}$ and 
$M_Y=(M_Y,\al_Y\col M_Y\lo {\mathcal O}_Y)$ 
the underlying (formal) scheme of $Y$ and the log structure of $Y$, respectively. 
Let $S$ be a $p$-adic formal family of log points defined in \cite{nb}; 
locally on $S$, $S$ is isomorphic to a log $p$-adic formal scheme 
$(\os{\circ}{S}, {\mab N}\oplus {\mathcal O}_S^*\lo {\mathcal O}_S)$, 
where the morphism ${\mab N}\oplus {\mathcal O}_S^*\lo {\mathcal O}_S$ 
is defined by the morphism
$(n,a)\lom 0^na$ $(n\in {\mab N}, a\in {\mathcal O}_S^*)$, where $0^n=0\in {\mathcal O}_S$ 
for $n\not =0$ and $0^0:=1\in {\mathcal O}_S$. 
Let $(S,{\mathcal I},\gam)$ be a $p$-adic formal PD-family of log points 
($S$ is a $p$-adic formal family of log points and 
${\mathcal I}$ is a quasi-coherent $p$-adic PD-ideal sheaf of ${\mathcal O}_S$ 
with PD-structure $\gam$). 
Let $S_0$ be an exact closed log subscheme of $S$ defined by ${\mathcal I}$. 
Let $X/S_0$ be a proper SNCL scheme with structural morphism 
$f\col X\lo S_0\os{\sus}{\lo} S$. 
(In \S\ref{sec:snclv} below we recall the definition of the SNCL scheme briefly.)
Let $\{\os{\circ}{X}_{\lam}\}_{\lam \in \Lam}$ be 
the set of smooth components of $\os{\circ}{X}/\os{\circ}{S}_0$ defined in \cite{nb}.  
(When $\os{\circ}{S}_0$ is the spectrum of a field of characteristic $p>0$, 
$\{\os{\circ}{X}_{\lam}\}_{\lam \in \Lam}$ can be taken as 
the set of the irreducible components of $\os{\circ}{X}$.) 
For a nonnegative integer $k$, let 
\begin{align*} 
\os{\circ}{X}{}^{(k)}:=
\coprod_{\{\{\lam_0,\ldots, \lam_k\}~\vert \lam_i\in \Lam, \lam_i\not=\lam_j (i\not=j)\}} 
\os{\circ}{X}_{\lam_0}\cap \cdots \cap \os{\circ}{X}_{\lam_k}
\tag{1.0.1}\label{ali:hkis}
\end{align*} 
be a scheme over $\os{\circ}{S}_0$ well-defined in \cite{nb}. 
%For example, $\os{\circ}{X}{}^{(0)}=$. 
Let $a^{(k)} \col \os{\circ}{X}{}^{(k)}\lo \os{\circ}{X}$ be the natural morphism. 
Let $F_{\os{\circ}{S}_0}\col \os{\circ}{S}_0\lo \os{\circ}{S}_0$ be the absolute Frobenius 
endomorphism of $\os{\circ}{S}_0$ and 
set $S_0^{[p]}:=S_0\times_{\os{\circ}{S}_0,F_{\os{\circ}{S}_0}}\os{\circ}{S}_0$.  
Let 
$F_{S_0/\os{\circ}{S}_0}\col  S_0\lo S_0^{[p]}$ 
be the relative Frobenius morphism of $S_0$ over $\os{\circ}{S}_0$. 
Let $S_0^{[p]}(S)$ be a log formal scheme whose underlying formal scheme 
is $\os{\circ}{S}$ and whose log structure $M_{S_0^{[p]}(S)}$ 
is a unique sub-log structure of $S$ 
such that the isomorphism 
$M_S/{\mathcal O}_S^*\os{\sim}{\lo} M_{S_0}/{\mathcal O}_{S_0}^*$ induces 
the following isomorphism 
\begin{align*} 
M_{S_0^{[p]}(S)}/{\mathcal O}_S^*\os{\sim}{\lo} 
{\rm Im}(F_{S_0/\os{\circ}{S}_0}^*\col 
F_{S_0/\os{\circ}{S}_0}^*(M_{S_0^{[p]}})
\lo M_{S_0})/{\mathcal O}_{S_0}^*. \tag{1.0.2}\label{ali:rus0avp}
\end{align*} 
(The structural morphism of $M_{S_0^{[p]}(S)}$ is the composite morphism 
$M_{S_0^{[p]}(S)}\os{\sus}{\lo} M_S\lo {\mathcal O}_S$.) 
We have an obvious morphism $(S,{\mathcal I},\gam)
\lo (S_0^{[p]}(S),{\mathcal I},\gam)$ of log PD-formal schemes. 
For a fine log scheme $Y$ over $S_0$ with structural morphism $g\col Y\lo S_0\os{\sus}{\lo} S$, 
let $(Y/S)_{\rm crys}$ be the log crystalline topos of $Y/(S,{\mathcal I},\gam)$ 
defined in \cite{klog1} and let ${\mathcal O}_{Y/S}$ be the structure sheaf of 
$(Y/S)_{\rm crys}$. 
Let $\os{\circ}{Y}_{\rm zar}$ be the Zariski topos  of $\os{\circ}{Y}$. 
Let $u_{Y/S} \col (Y/S)_{\rm crys}\lo \os{\circ}{Y}_{\rm zar}$ be the canonical projection. 
Set $g_{Y/S}:=g\circ u_{Y/S}$. 
Let ${\rm D}^+{\rm F}(g^{-1}({\mathcal O}_S))$ be the derived category of 
bounded below filtered complexes of $g^{-1}({\mathcal O}_S)$-modules 
and let $D^+(g^{-1}({\mathcal O}_S))$ be the derived category of 
bounded below complexes of $g^{-1}({\mathcal O}_S)$-modules.  
(See \cite{nh2} for the definition of the filtered derived category 
${\rm D}^+{\rm F}(g^{-1}({\mathcal O}_S))$ (cf.~\cite{dh2}, \cite{ilc}).) 
In \cite{nb} we have proved the following, which is a generalization of 
a result in \cite{msemi}:

\begin{theo}[{\bf \cite[Existence of the zarisikian $p$-adic 
filtered Steenbrink complex]{nb}}]\label{theo:wdpc}
Let $\vp_{\rm crys}^{(m)}(\os{\circ}{X}/\os{\circ}{S})$ $(m\in {\mab N})$ 
be the crystalline orientation sheaf associated to 
the set $\{\os{\circ}{X}_{\lam}\}_{\lam \in \Lam}$ for $m$. That is, 
$\vp_{\rm crys}^{(m)}(\os{\circ}{X}/\os{\circ}{S})$ is 
the extension to $(\os{\circ}{X}{}^{(m)}/\os{\circ}{S})_{\rm crys}$ 
of the direct sum of  
$\os{m+1}{\bigwedge}{\mab Z}^E_{\os{\circ}{X}_{\lam_0}
\cap \cdots \cap \os{\circ}{X}_{\lam_m}}$'s  
in the Zariski topos $\os{\circ}{X}{}^{(m)}_{\rm zar}$ of $\os{\circ}{X}{}^{(m)}$ 
for the subsets $E=\{\os{\circ}{X}_{\lam_0}, \ldots, \os{\circ}{X}_{\lam_m}\}$'s 
of $\{\os{\circ}{X}_{\lam}\}_{\lam \in \Lam}$
with $\# E=m+1$.  
Then there exists a filtered complex 
\begin{align*} 
(A_{\rm zar}(X/S),P) 
\in {\rm D}^+{\rm F}(f^{-1}({\mathcal O}_S))
\tag{1.1.1}\label{ali:raxs}
\end{align*} 
with a canonical isomorphism 
\begin{align*} 
\theta \wedge \col Ru_{X/S*}({\mathcal O}_{X/S})\os{\sim}{\lo} A_{\rm zar}(X/S)
\tag{1.1.2}\label{ali:ixuaoa} 
\end{align*} 
in $D^+(f^{-1}({\mathcal O}_S))$ 
such that 
\begin{align*} 
{\rm gr}^P_kA_{\rm zar}(X/S)\os{\sim}{\lo} \bigoplus_{j\geq \max \{-k,0\}} 
&a^{(2j+k)}_* 
(Ru_{\os{\circ}{X}{}^{(2j+k)}/\os{\circ}{S}*}
({\mathcal O}_{{\os{\circ}{X}{}^{(2j+k)}/\os{\circ}{S}}}) \tag{1.1.3}\label{ali:ruoovp}
\\
&\otimes_{\mab Z}\vp_{\rm crys}^{(2j+k)}(\os{\circ}{X}/\os{\circ}{S}))(-j-k)[-2j-k]
\end{align*} 
in $D^+(f^{-1}({\mathcal O}_S))$.  
Here the Tate twist $(-j-k)$ means the Tate twist 
with respect to 
the morphism $X\lo X\times_{\os{\circ}{S}_0,F_{\os{\circ}{S}_0}}\os{\circ}{S}_0$ 
over $(S,{\mathcal I},\gam) \lo (S_0^{[p]}(S),{\mathcal I},\gam)$
induced by the absolute Frobenius endomorphism $F_X\col X\lo X$ of $X$. 
\end{theo}
As a corollary of this theorem, we obtain 
the weight filtration $P$ on $R^qf_{X/S*}({\mathcal O}_{X/S})$ $(q\in {\mab N})$: 
\begin{align*} 
P_{k+q}R^qf_{X/S*}({\mathcal O}_{X/S})&:={\rm Im}
(R^qf_{*}(P_kA_{\rm zar}(X/S))\lo R^qf_{*}(A_{\rm zar}(X/S))) 
\tag{1.1.4}\label{ali:wamcp}\\
&\simeq {\rm Im}(R^qf_{*}(P_kA_{\rm zar}(X/S))\lo R^qf_{X/S*}({\mathcal O}_{X/S})) 
\end{align*} 
and the following spectral sequence 
\begin{align*} 
E_1^{-k,q+k}:=\bigoplus_{j\geq \max \{-k,0\}} &
R^{q-2j-k}f_{\os{\circ}{X}{}^{(2j+k)/\os{\circ}{S}}}
({\mathcal O}_{\os{\circ}{X}{}^{(2j+k)}/\os{\circ}{S}}
\otimes_{\mab Z}\vp^{(2j+k)}_{\rm crys}
(\os{\circ}{X}/\os{\circ}{S}))(-j-k) \tag{1.1.5}\label{eqn:endsp}\\
&\Lo R^qf_{X/S}({\mathcal O}_{X/S}) \quad (q\in {\mab Z}). 
\end{align*} 
%See \cite{msemi} and \cite{ndw} for another approach for the construction of 
%$P$ on $R^qf_{X/S*}({\mathcal O}_{X/S})$ 
%by the use of log de Rham-Witt complexes in the case where $S$ is the canonical lift of 
%the log point of a perfect field $\kap$ of characteristic $p>0$ over 
%the Witt ring ${\mathcal W}$ of $\kap$. In this case, in \cite{nb} 
%we have proved that $(A_{\rm zar}(X/S),P)$ is 
%canonically isomorphic to the weight-filtered complex 
%$({\mathcal W}A^{\bul}_X,P)$ constructed in 
%\cite{msemi} and \cite{ndw}. 
Let ${\mathcal V}$ be a complete discrete valuation ring of mixed characteristics 
$(0,p)$ with perfect residue field. 
In \cite{nb} we have proved that,  
if $\os{\circ}{S}$ is a $p$-adic formal ${\mathcal V}$-scheme, 
then (\ref{eqn:endsp}) degenerates at $E_2$. 
\par
For a log smooth scheme $Y$ over $S_0$, 
let  
\begin{equation*}
N \col Ru_{Y/S*}({\mathcal O}_{Y/S}) 
\lo 
Ru_{Y/S*}({\mathcal O}_{Y/S})(-1)
\tag{1.1.6}\label{eqn:cgm}
\end{equation*} 
be the monodromy operator defined in \cite{hk} and \cite{nb}. 
It is not difficult to prove that the induced morphism 
\begin{equation*}
N \col R^qf_{X/S*}({\mathcal O}_{X/S})
\lo 
R^qf_{X/S*}({\mathcal O}_{X/S})(-1) \quad (q\in {\mab N})
\tag{1.1.7}\label{eqn:cgnm}
\end{equation*} 
by the morphism (\ref{eqn:cgm}) 
for the case $Y/S=X/S$ 
induces the morphism 
\begin{equation*}
N \col P_kR^qf_{X/S*}({\mathcal O}_{X/S})
\lo 
P_{k-2}R^qf_{X/S*}({\mathcal O}_{X/S})(-1)\quad (k\in {\mab Z}). 
\tag{1.1.8}\label{eqn:cpgnm}
\end{equation*} 
In \cite{nb} we have conjectured the following, 
which is a generalization of the conjecture in \cite{msemi} and a generalized 
$p$-adic analogue of Kato's conjecture in \cite{kln}:

\begin{conj}[{\bf Log $p$--adic monodromy-weight conjecture}]\label{conj:rmnc} 
Assume 
%that $\os{\circ}{S}$ is a $p$-adic formal scheme and 
that $\os{\circ}{X} \lo \os{\circ}{S}$ is projective. 
Let $q$ be nonnegative integer. 
Then the induced morphism 
\begin{equation*} 
N^e \col {\rm gr}^P_{q+e}R^qf_{X/S*}({\mathcal O}_{X/S})
\lo 
{\rm gr}^P_{q-e}R^qf_{X/S*}({\mathcal O}_{X/S})(-e)
\quad (q,e\in {\mab N})
\tag{1.2.1}\label{eqn:mj}
\end{equation*}
by the monodromy operator (\ref{eqn:cpgnm}) 
is an isomorphism modulo torsion.  
\end{conj}
This conjecture has been solved for several special cases in 
e.~g., \cite{msemi}, \cite{ndw}, \cite{ndeg}, \cite{itpad}, \cite{ds}, 
\cite{lp}, \cite{nb} and \cite{bkv}. 
In \cite[(5.5.3)]{nb} we have proved the following: 

\begin{theo}[{\rm {\bf \cite[(5.5.3)]{nb}}}]\label{theo:snbme} 
Let ${\mathcal V}$ be a complete discrete valuation ring of mixed characteristics 
$(0,p)$ with perfect residue field. 
If $\os{\circ}{S}$ is a $p$-adic formal ${\mathcal V}$-scheme 
and if, for each connected component $S'_0$ of $S_0$, 
there exists an exact closed point $s\in S'_0$  
such that the fiber $X_s/s$ of $X/S_0$ at $s$ is the log special fiber 
of a proper strict semistable family over a complete discrete valuation ring 
of equal characteristic, then {\rm (\ref{conj:rmnc})} is true. 
\end{theo}

\par
This paper aims to formulate the log $p$-adic relative monodromy-weight conjecture. 
This conjecture is a generalization of (\ref{conj:rmnc}). 
In the rest of this introduction, we explain our main results. 
%which includes generalizations of (\ref{theo:wdpc}) and (\ref{conj:rmnc}). 
\par 
%Let $S_0 \os{\sus}{\lo} S$ be a closed immersion defined by 
%a PD-ideal ${\mathcal I}$ with PD-structure $\gam$. 
Let $(X,D)/S_0$ be an SNCL scheme with a relative SNCD 
with structural morphism $f\col (X,D)\lo S_0\os{\sus}{\lo} S$. 
(We recall the definition of an SNCL scheme with a relative SNCD 
in \S\ref{sec:snrdlv} below.) 
Let ${\rm D}^+{\rm F}^2(f^{-1}({\mathcal O}_S))$ be the 
derived category of  bifiltered  complexes of $f^{-1}({\mathcal O}_S)$-modules. 
(See \cite{dh3} and \cite{nlf} for the definition of 
${\rm D}^+{\rm F}^2(f^{-1}({\mathcal O}_S))$.) 
Let $\{\os{\circ}{D}_{\mu}\}_{\mu \in M}$ be 
the set of smooth components of $\os{\circ}{D}/\os{\circ}{S}_0$.  
For a positive integer $l$, let 
\begin{align*} 
\os{\circ}{D}{}^{(l)}:=
\coprod_{\{\{\mu_1,\ldots, \mu_l\}~\vert \mu_i\in M, \mu_i\not=\mu_j (i\not=j)\}} 
\os{\circ}{D}_{\mu_1}\cap \cdots \cap \os{\circ}{D}_{\mu_l}
\tag{1.3.1}\label{ali:hkbis}
\end{align*} 
be an analogous scheme over $\os{\circ}{S}_0$ to (\ref{ali:hkis}). 
Endow $\os{\circ}{D}{}^{(l)}$ with the inverse image of the log structure of $X$ 
and let $D^{(l)}$ be the resulting log scheme over $S_0$. 
Set $D^{(0)}:=X$.  
%For example, $\os{\circ}{X}{}^{(0)}=$. 
Let $a^{(m,l)} \col \os{\circ}{X}{}^{(m)} \cap 
\os{\circ}{D}{}^{(l)} \lo \os{\circ}{X}$ be the natural morphism. 
Let 
$\vp^{(m,l)}_{\rm crys}
((\os{\circ}{X},\os{\circ}{D})/\os{\circ}{S})$ 
be the crystalline orientation sheaf defined in \S\ref{sec:pwtr} below.  
In this paper we prove the following: 

\begin{theo}\label{theo:bex}
There exists a bifiltered complex 
\begin{align*} 
(A_{\rm zar}((X,D)/S),P^D,P)
\in {\rm D}^+{\rm F}^2(f^{-1}({\mathcal O}_S))
\tag{1.4.1}\label{ali:rabxs}
\end{align*} 
%The complex $(A_{\rm zar}((X,D)/S),P^{\os{\circ}{D}},P)$ 
%is represented by a single bifiltered complex of 
%a bifiltered double complex 
%$(A_{\rm zar}((X,D)/S)^{\bul \bul},P^D,P)$. 
with a canonical isomorphism 
\begin{align*} 
\theta \wedge \col Ru_{(X,D)/S*}({\mathcal O}_{(X,D)/S})\os{\sim}{\lo} A_{\rm zar}((X,D)/S)
\tag{1.4.2}\label{ali:ixubaoa} 
\end{align*} 
in $D^+(f^{-1}({\mathcal O}_S))$ 
such that 
\begin{align*}  
{\rm gr}^{P^D}_kA_{\rm zar}((X,D)/S)=a^{(0,k)}_*(A_{\rm zar}(D^{(k)}/S))(-k)[-k] 
\tag{1.4.3}\label{ali:rudrbp}
\end{align*} 
in $D^+(f^{-1}({\mathcal O}_S))$
and
\begin{align*} 
{\rm gr}^P_kA_{\rm zar}((X,D)/S)&
\os{\sim}{\lo} 
\bigoplus_{k'\leq k}\bigoplus_{j\geq \max \{-k',0\}} 
a^{(2j+k',k-k')}_{*}
Ru_{\os{\circ}{X}{}^{(2j+k')}\cap \os{\circ}{D}{}^{(k-k')}/\os{\circ}{S}*}
\tag{1.4.4}\label{ali:ruobp}\\
&({\mathcal O}_{\os{\circ}{X}{}^{(2j+k')}\cap 
\os{\circ}{D}{}^{(k-k')}/\os{\circ}{S}}\otimes_{\mab Z}
\vp^{(2j+k',k-k')}_{\rm crys}
((\os{\circ}{X},\os{\circ}{D})/\os{\circ}{S}))(-j-k)[-2j-k] 
\end{align*} 
in $D^+(f^{-1}({\mathcal O}_S))$. 
\end{theo}
\parno
We call $(A_{\rm zar}((X,D)/S),P^D,P)$ 
the {\it zariskian $p$-adic bifiltered El Zein-Steenbrink-Zucker complex} of $(X,D)/S$. 
As a corollary of this theorem, we obtain 
the following filtration $P^D$ on $R^qf_{(X,D)/S*}({\mathcal O}_{(X,D)/S})$ $(q\in {\mab N})$: 
\begin{align*} 
&P^D_kR^qf_{(X,D)/S*}({\mathcal O}_{(X,D)/S})\tag{1.4.5}\label{ali:wacp}\\
&:={\rm Im}
(R^qf_*(P^D_kA_{\rm zar}((X,D)/S))\lo R^qf_{*}(A_{\rm zar}((X,D)/S))) \\
&\simeq {\rm Im}(R^qf_*(P^D_kA_{\rm zar}((X,D)/S))\lo R^qf_{(X,D)/S*}({\mathcal O}_{(X,D)/S}))
\end{align*} 
and the  weight filtration $P$ on $R^qf_{(X,D)/S*}({\mathcal O}_{(X,D)/S})$ $(q\in {\mab N})$: 
\begin{align*} 
&P_{k+q}R^qf_{(X,D)/S*}({\mathcal O}_{(X,D)/S})\tag{1.4.6}\label{ali:wpacp}\\
&:={\rm Im}
(R^qf_*(P_kA_{\rm zar}((X,D)/S))\lo R^qf_{*}(A_{\rm zar}((X,D)/S))) \\
&\simeq {\rm Im}(R^qf_*(P_kA_{\rm zar}((X,D)/S))\lo R^qf_{(X,D)/S*}({\mathcal O}_{(X,D)/S})). 
\end{align*} 
We also obtain the following new spectral sequences:  
\begin{align*} 
&E_1^{-k,q+k}:=R^{q-k}f_{D^{(k)}/\os{\circ}{S}}
({\mathcal O}_{D^{(k)}/S}\otimes_{\mab Z}
\eps_{D/S{\rm crys}}^{-1}(\vp^{(0,k)}_{\rm crys}((\os{\circ}{X},\os{\circ}{D})/\os{\circ}{S})))(-k) 
\tag{1.4.7}\label{eqn:eddsp}\\
&\Lo 
R^qf_{(X,D)/S}({\mathcal O}_{(X,D)/S}) 
\quad (q\in {\mab Z}) 
\end{align*} 
and 
\begin{align*} 
E_1^{-k,q+k}&:=\bigoplus_{k'\leq k}\bigoplus_{j\geq \max \{-k',0\}} 
R^{q-2j-k}f_{\os{\circ}{X}{}^{(2j+k')}\cap \os{\circ}{D}{}^{(k-k')}
/\os{\circ}{S}}(
{\mathcal O}_{\os{\circ}{X}{}^{(2j+k')}\cap \os{\circ}{D}{}^{(k-k')}
/\os{\circ}{S}}\otimes_{\mab Z}
\tag{1.4.8}\label{eqn:sppdsp}\\
&\vp^{(2j+k',k-k')}_{\rm crys}
((\os{\circ}{X},\os{\circ}{D})/\os{\circ}{S}))(-j-k) \Lo 
R^qf_{(X,D)/S}({\mathcal O}_{(X,D)/S}) 
\quad (q\in {\mab Z}).
\end{align*} 
Here $\eps_{D/S}\col D\lo \os{\circ}{D}$ is 
%the inverse image of the crystalline orientation sheaf 
%$\vp^{(k)}_{\rm crys}(\os{\circ}{D}/\os{\circ}{S})$ of $\os{\circ}{D}/\os{\circ}{S}$ 
%by 
the natural morphism forgetting the log structure of $D$ over $S\lo \os{\circ}{S}$. 
\par 
It is not difficult to prove that 
the monodromy operator 
$N\col R^qf_{(X,D)/S}({\mathcal O}_{(X,D)/S})\lo 
R^qf_{(X,D)/S}({\mathcal O}_{(X,D)/S})(-1)$ induces the following morphisms 
\begin{align*} 
N\col P^D_kR^qf_{(X,D)/S}({\mathcal O}_{(X,D)/S})\lo 
P^D_kR^qf_{(X,D)/S}({\mathcal O}_{(X,D)/S})(-1)\quad (k\in  {\mab Z})
\tag{1.4.9}\label{ali:ondxd}
\end{align*} 
and 
\begin{align*} 
N\col P_kR^qf_{(X,D)/S}({\mathcal O}_{(X,D)/S})\lo 
P_{k-2}R^qf_{(X,D)/S}({\mathcal O}_{(X,D)/S})(-1)\quad (k\in  {\mab Z}). 
\tag{1.4.10}\label{ali:onxd}
\end{align*} 
%by using $(A_{\rm zar}((X,D)/S)),P^D,P)$.
In this paper we give the following conjecture: 

\begin{conj}[{\bf Log $p$--adic relative monodromy weight conjecture}]\label{conj:repmc} 
Assume that $\os{\circ}{X}$ is projective over $\os{\circ}{S}$. 
Then the following morphism 
\begin{align*} 
N^e \col 
{\rm gr}_{q+k+e}^P{\rm gr}_k^{P^{D}}
R^qf_{(X,D)/S*}({\mathcal O}_{(X,D)/S}) 
\lo {\rm gr}_{q+k-e}^P{\rm gr}_k^{P^{D}}
R^qf_{(X,D)/S*}({\mathcal O}_{(X,D)/S})(-e)
\tag{1.5.1}
\end{align*}
induced by 
{\rm (\ref{ali:ondxd})} and {\rm (\ref{ali:onxd})}
%by the monodromy operator 
%$N\col R^qf_{(X,D)/S}({\mathcal O}_{(X,D)/S})\lo R^qf_{(X,D)/S}({\mathcal O}_{(X,D)/S})(-1)$ 
for $e,k\in {\mab N}$ is an isomorphism modulo torsion.  
\end{conj}
\parno
This conjecture is equivalent to the existence of the relative monodromy filtration $M$ on 
$R^qf_{(X,D)/S}({\mathcal O}_{(X,D)/S})\otimes_{\mab Z}{\mab Q}$ 
with respect to the filtration $P^D$ on 
$R^qf_{(X,D)/S}({\mathcal O}_{(X,D)/S})\otimes_{\mab Z}{\mab Q}$  
and it is equal to $P$ on $R^qf_{(X,D)/S}({\mathcal O}_{(X,D)/S})\otimes_{\mab Z}{\mab Q}$. 
We prove the following: 

\begin{theo}\label{theo:tmd}
Let ${\mathcal V}$ be as in {\rm (\ref{theo:snbme})}. 
If $\os{\circ}{S}$ is a $p$-adic formal ${\mathcal V}$-scheme and 
if {\rm (\ref{conj:rmnc})} is true for $D^{(k)}/S$ for any $k\in {\mab N}$, 
then {\rm (\ref{conj:repmc})} is true. 
\end{theo} 

Because it is known that the conjecture  (\ref{conj:rmnc}) is true if 
the relative dimension of $\os{\circ}{X}$ over $\os{\circ}{S}$ is less than or equal $2$, 
we obtain the following corollary: 

\begin{coro}\label{theo:coj2}
Let ${\mathcal V}$ and $\os{\circ}{S}$ be as in {\rm (\ref{theo:tmd})}. 
If the relative dimension of $\os{\circ}{X}$ over $\os{\circ}{S}$ is less than or equal $2$, 
then {\rm (\ref{conj:repmc})} is true. 
\end{coro}

As a corollary of (\ref{theo:tmd}) and (\ref{theo:snbme}), 
we also obtain the following: 

\begin{coro}\label{theo:cojr}
%The conjecture {\rm (\ref{conj:repmc})} 
%is true in the following cases$:$ 
%\par 
%$(1)$ The relative dimension of $\os{\circ}{X}$ over $\os{\circ}{S}$ is less than or equal $2$. 
%\par 
%$(2)$ 
Let the notations be as in {\rm (\ref{theo:snbme})}. 
Assume that, 
for each connected component $S'_0$ of $S_0$, there exists an exact closed point $s\in S'_0$ 
such that the fiber $D^{(k)}_s/s$ of $D^{(k)}/S_0$ at $s$ is the log special fiber 
of a proper strict semistable family over a complete discrete valuation ring 
of equal characteristic. 
Then {\rm (\ref{conj:repmc})} is true. 
\end{coro}

\par 
This paper is organized as follows. 
\par 
In \S\ref{sec:snclv}  we recall the definition of an SNCL scheme.  
\par
In \S\ref{sec:snrdlv} we recall the definition of an SNCL scheme with a relative SNCD. 
\par 
In \S\ref{sec:pwtr} we define the preweight filtration on the log crystalline complex of 
an SNCL scheme with a relative SNCD in characteristic $p>0$ 
and investigate its fundamental properties. 
\par 
In \S\ref{sec:psc} we define the zariskian $p$-adic bifiltered El Zein-Steenbrink-Zucker complex. 
\par  
In \S\ref{sec:fcuc} we prove the contravariant functoriality of the zariskian 
$p$-adic bifiltered El Zein-Steenbrink-Zucker complex, 
which plays key roles in several results 
in later sections. 
\par
In \S\ref{sec:mod} we recall the monodromy operator defined in  
\cite{hk} and \cite{nb} and we show that it is identified with an endomorphism of 
the zariskian $p$-adic bifiltered El Zein-Steenbrink-Zucker complex. 
\par 
In \S\ref{sec:bckf} we prove the bifiltered base change theorem of 
the zariskian $p$-adic bifiltered El Zein-Steenbrink-Zucker complex.  
\par 
In \S\ref{sec:infhdi} we prove the infinitesimal deformation invariance of 
the zariskian $p$-adic bifiltered El Zein-Steenbrink-Zucker complex modulo torsion. 
\par 
In \S\ref{sec:filbo} we prove 
the $E_2$-degeneration of the $p$-adic weight spectral sequence
(\ref{eqn:sppdsp}) modulo torsion. 
\par 
In \S\ref{sec:e2} we prove the log convergence of the weight filtration on 
the log isocrystalline cohomological sheaf of 
a proper SNCL scheme with a relative SNCD in characteristic $p>0$. 
\par 
In \S\ref{sec:st} we prove the strict compatibility of the pull-back of  a morphism 
of proper SNCL schemes with relative SNCD's in characteristic $p>0$ with respect to 
weight filtrations. 
\par 
In \S\ref{sec:mn} we give (\ref{conj:repmc}) and we prove (\ref{theo:tmd}).  
\par
In \S\ref{sec:lbddmlif} we give the explicit descriptions of the edge morphisms between 
the $E_1$-terms of (\ref{eqn:eddsp}) and (\ref{eqn:sppdsp}). 
%\par
%\bigskip
%\parno
%{\bf Acknowledgment.}
% I would like to express my thanks to K.~Fujiwara and C.~Nakayama  
%for telling me a technique in the proof of (\ref{theo:sprme}) and 
%giving me a suggestion (\ref{rema:gfs}), respectively.

\bigskip
\parno
{\bf Notations.}
(1) For a log scheme $X$, $\os{\circ}{X}$ denotes the underlying scheme of $X$. 
For a morphism 
$\varphi \col X \lo Y$, $\os{\circ}{\varphi}$ denotes 
the underlying morphism 
$\os{\circ}{X} \lo \os{\circ}{Y}$ of $\varphi$. 
\par
(2) SNC(L)=simple normal crossing (log), 
SNCD=simple normal crossing divisor.  
\par
(3) For a complex $(E^{\bul}, d^{\bul})$ 
of objects in an exact additive category ${\mathcal A}$, 
we often denote  $(E^{\bul}, d^{\bul})$ only by 
$E^{\bul}$ as usual. 
%\par 
%(4) For a complex $(E^{\bul}, d^{\bul})$ in (3) and 
%for an integer $n$, $(E^{\bul}\{n\}, d^{\bul}\{n\})$ 
%denotes the following complex:
%\begin{equation*}
%\cdots \lo \us{q-1}{E^{q-1+n}}
%\os{d^{q-1+n}}{\lo} 
%\us{q}{E^{q+n}} \os{d^{q+n}}{\lo} 
%\us{q+1}{E^{q+1+n}} \os{d^{q+1+n}}{\lo} 
%\cdots.
%\end{equation*}
%Here the numbers under the objects 
%above in ${\mathcal A}$ mean the degrees.
\par 
(4)  For a morphism 
$f \col  (E^{\bul},d^{\bul}_E) \lo (F^{\bul},d^{\bul}_F)$ 
of complexes of objects in an exact additive category ${\mathcal A}$,  
let ${\rm MF}(f)$ (resp.~${\rm MC}(f)$) 
be the mapping fiber (resp.~the mapping cone) 
of $f$: ${\rm MF}(f):=E^{\bul} \oplus F^{\bul}[-1]$ 
with boundary morphism ``$(x,y)\lom (d_E(x),-d_F(y)+f(x))$"
(resp.~${\rm MC}(f):=E^{\bul}[1] \oplus F^{\bul}$ 
with boundary morphism ``$(x,y)\lom (-d_E(x),d_F(y)+f(x))$").  
\par
(5) Let $({\mathcal T},{\mathcal A})$ be a ringed topos. 
\medskip
\par
%(a) $C({\mathcal T},{\mathcal A})$ (resp.~${\rm C}^{\pm}({\mathcal T},{\mathcal A})$, 
%$C^{\rm b}({\mathcal T},{\mathcal A})$): the category of 
%(resp.~bounded below, bounded above, 
%bounded) complexes of ${\mathcal A}$-modules, 
\par
(a) $K({\mathcal T},{\mathcal A})$ (resp.~$K^{\pm}({\mathcal T},{\mathcal A})$, 
$K^{\rm b}({\mathcal T},{\mathcal A})$): the category of 
(resp.~bounded below, bounded above, 
bounded) complexes of ${\mathcal A}$-modules modulo homotopy, 
\par
(b) $D({\mathcal T},{\mathcal A})$ (resp.~$D^{\pm}({\mathcal T},{\mathcal A})$, 
$D^{\rm b}({\mathcal T},{\mathcal A})$): the derived category of
$K({\mathcal T},{\mathcal A})$ (resp.~$K^{\pm}({\mathcal T},{\mathcal A})$, 
$K^{\rm b}({\mathcal T},{\mathcal A})$). 
For an object $E^{\bul}$ of 
$C({\mathcal T},{\mathcal A})$ (resp.~$C^{\pm}({\mathcal T},{\mathcal A})$, 
$C^{\rm b}({\mathcal T},{\mathcal A})$), we denote 
simply by $E^{\bul}$
the corresponding object to $E^{\bul}$ in 
$D({\mathcal T},{\mathcal A})$ (resp.~$D^{\pm}({\mathcal T},{\mathcal A})$, 
$D^{\rm b}({\mathcal T},{\mathcal A})$). 
\par
%(d) The additional notation ${\rm F}$ 
%to the categories above means ``the filtered ''. 
%Here the filtration is an increasing filtration indexed by ${\mab Z}$.   
%For example,  ${\rm K}^+{\rm F}({\mathcal T},{\mathcal A})$ is 
%the category of bounded below filtered complexes modulo filtered homotopy.
\par
(c) ${\rm DF}^2({\mathcal T},{\mathcal A})$ 
(resp.~${\rm D}^{\pm}{\rm F}^2({\mathcal T},{\mathcal A})$, 
${\rm D}^{\rm b}{\rm F}^2({\mathcal T},{\mathcal A})$): 
the derived category of (resp.~bounded below, bounded  above, bounded) 
bifiltered complexes of ${\mathcal A}$-modules defined in \cite{nlf}.

\section{SNCL schemes}\label{sec:snclv} 
In this section we quickly recall the definition of an SNCL(=simple normal crossing log) scheme 
defined in \cite{nb} and \cite{nhir}. 
\par 
 Let $S$ be a log (formal) scheme with ideal of definition ${\mathcal J}$ 
(${\mathcal J}$ may be the zero ideal $(0)$) 
 such that there exists an open covering 
$S=\bigcup_{i\in I}S_i$ such that $M_{S_i}$ 
is the association of the log structure 
is ${\mab N}\oplus {\mathcal O}_{S_i}^*$ with a morphism 
${\mab N}\oplus {\mathcal O}_{S_i}^*\owns (n,a) \lom 0^na \in {\mathcal O}_{S_i}$, 
where $0^0:=1\in {\mathcal O}_{S_i}$.    
In \cite{nb} we have called $S$ a (formal) family of log points. 
When $(\os{\circ}{S},{\mathcal I},\gam)$ is 
a formal PD-scheme with quasi-coherent PD-ideal sheaf 
and PD-structure, 
we call $(S,{\mathcal I},\gam)$ a formal PD-family of log points. 
\par 
Let $S=\bigcup_{i\in I}S_i$ be an open covering 
of $S$ such that 
$M_{S_i}/{\mathcal  O}^*_{S_i}\simeq {\mab N}$.   
Take a local section
$t_i\in \Gam(S_i,M_S)$ $(i\in I)$ such that 
the image of $t_i$ in $\Gam(S_i,M_S/{\mathcal  O}^*_S)$ is a generator.   
Set $S_{ij}:=S_i\cap S_j$. 
Then there exists a unique section 
$u_{ji}\in \Gam(S_{ij},{\mathcal  O}^*_S)$ such that 
$t_j\vert_{S_{ij}}=u_{ji}(t_i \vert_{S_{ij}})$ 
in $\Gam(S_{ij},M_S)$.  
Denote $\ul{\rm Spf}_{\os{\circ}{S}_i}({\mathcal  O}_{S_i}\{\tau_i]\})$ by 
${\mab A}^1_{\os{\circ}{S}_i}$, where $\tau_i$ is a variable. 
Endow 
${\mab A}^1_{\os{\circ}{S}_i}$ with the log structure 
$({\mab N}\owns 1 \lom \tau_i\in {\mathcal  O}_{S_i}\{\tau_i\})^a$. 
Denote the resulting log scheme by $\ol{S}_i$. 
Then, by patching $\ol{S}_i$ and $\ol{S}_j$ along 
$\ol{S}_{ij}:=\ol{S}_i\cap \ol{S}_j$ by the equation 
$\tau_j\vert_{\ol{S}_{ij}}=u_{ji}\tau_i\vert_{\ol{S}_{ij}}$,
we have a log formal scheme $\ol{S}=\bigcup_{i\in I}\ol{S}_i$. 
The ideal sheaves $\tau_i{\mathcal  O}_{\ol{S}_i}$'s ($i\in I$) 
patch together and we denote by 
${\mathcal  I}_{\ol{S}}$ 
the resulting ideal sheaf of 
${\mathcal  O}_{\ol{S}}$. 
The isomorphism class of the log scheme $\ol{S}$ 
is independent of the choice of the system of 
generators $\tau_i$'s. 
%By using this easy fact and by considering 
%the refinement of two open coverings of $S$, 
We see that the isomorphism class
of the log scheme $\ol{S}$ and the ideal sheaf 
${\mathcal  I}_{\ol{S}}$ 
are also independent of 
the choice of the open covering $S=\bigcup_{i\in I}S_i$. 
The natural morphism $\ol{S} \lo \os{\circ}{S}$ is formally log smooth 
by the criterion of the log smoothness (\cite[(3.5)]{klog1}). 
For a log scheme $Y$ over $\ol{S}$, 
we denote 
${\mathcal  I}_{\ol{S}}\otimes_{{\mathcal  O}_{\ol{S}}}{\mathcal  O}_Y$
by ${\mathcal  I}_Y$ by abuse of notation. 
\par 
By killing ${\mathcal  I}_{\ol{S}}$, 
we have a natural exact closed immersion 
$S\os{\sus}{\lo} \ol{S}$ over $\os{\circ}{S}$.  
 
%In \cite[(1.1.5), (1.1.6), (1.1.7)]{nb} we have proved 
%the following two easy propositions: 

%\begin{prop}[{\bf A special case of \cite[(1.1.5)]{nb}}]\label{prop:bar}
%Let $u\col S\lo S'$ be a morphism of families of log points. 
%Then $u$ induces a morphism $\ol{u}\col \ol{S}\lo \ol{S}{}'$ 
%fitting into the following commutative diagram 
%\begin{equation*} 
%\begin{CD}
%S@>{\subset}>> \ol{S}\\ 
%@V{u}VV @VV{{\ol{u}}}V \\
%S'@>{\subset}>> \ol{S}{}'. 
%\end{CD}
%\end{equation*} 
%\end{prop}  

Let $B$ be a scheme. 
For two nonnegative integers $a$ and $d$ such that 
$a\leq d$,  
consider the following scheme 
\begin{equation*} 
\os{\circ}{\mab A}_B(a,d):=
\ul{{\rm Spec}}_B
({\mathcal  O}_B[x_0, \ldots, x_d]/(\prod_{i=0}^ax_i)). 
\end{equation*}

\begin{defi}[{\bf \cite[(1.1.9)]{nb}}]\label{defi:zbsnc}
Let $Z$ be a scheme over $B$ 
with structural morphism $g \col Z \lo B$. 
We call $Z$ an 
{\it SNC$($=simple normal crossing$)$ scheme} over 
$B$ if $Z$ is a union of smooth schemes 
$\{Z_{\lam}\}_{\lam \in \Lam}$ over $B$ ($\Lam$ is a set) 
and if, for any point of $z \in Z$, 
there exist an open neighborhood $V$ of 
$z$ and an open neighborhood $W$ of 
$g(z)$ such that  
there exists an \'{e}tale morphism  
$V \lo \os{\circ}{\mab A}_W(a,d)$ 
such that 
\begin{equation*} 
\{Z_{\lam}\vert_{V}\}_{\lam \in \Lam} 
=\{\{x_i=0\}\}_{i=0}^a, 
\tag{2.1.1}\label{eqn:defilsnc}
\end{equation*}   
where $a$ and $d$ are nonnegative integers 
such that $a\leq d$,  
which depend on a 
zariskian local neighborhood in $Z$.  
Here $\{Z_{\lam}\vert_{V}\}_{\lam \in \Lam}$ means the set of 
$Z_{\lam}\vert_{V}$'s such that $Z_{\lam}\vert_{V}\not=\emptyset$ by abuse of notation. 
We call the set $\{Z_{\lam}\}_{\lam \in \Lam}$ 
a {\it decomposition of $Z$ by smooth components of} 
$Z$ over $B$. 
We call $Z_{\lam}$ a 
{\it smooth component} of $Z$ over $B$. 
\end{defi} 

\parno  
Set $\Del:=\{Z_{\lam}\}_{\lam \in \Lam}$.  
For an open subscheme $V$ of $Z$, 
set $\Del_{V}:=\{Z_{\lam}\vert_{V}\}_{\lam \in \Lam}$. 
For a nonnegative integer $m$ and a subset  
$\ul{\lam}=\{\lam_0,\cdots, \lam_m\}$ 
$(\lam_i\in \Lam,\lam_i \not= \lam_j~{\rm if}~i\not= j)$ of $\Lam$,  
set 
\begin{equation}
Z_{\ul{\lam}} 
:=Z_{\lam_0}\cap \cdots \cap Z_{\lam_m}.  
\tag{2.1.2}\label{eqn:parlm}
\end{equation}
Set 
%$Z^{(0)}:=\us{\lam}{\coprod}Z_{\lam}$ and 
\begin{equation}
Z^{(m)} := \us{\# \ul{\lam}=m+1}{\coprod}Z_{\ul{\lam}}  
\tag{2.1.3}\label{eqn:kfntd}
\end{equation} 
for $m\in {\mab N}$ and 
%\begin{equation}
%Z^{(m)} :={\rm cosk}_0^Z(Z^{(0)})_m
%\tag{2.4.2}\label{eqn:kfdintd}
%\end{equation} 
%In this book, for the convenience of notation, 
%for the case where $\ul{\lam}=\phi$, 
%we set
$Z_{\phi}=Z$ ($\phi$ is the empty set) and $Z^{(-1)}=Z$. 
We also set $Z^{(m)}=\emptyset$ for $m\leq -2$.  
Note that, for an element $\lam$ of $\Lam$, 
$Z_{\{\lam\}}=Z_{\lam}$; we have to use both notations 
$Z_{\{\lam\}}$ and $Z_{\lam}$. 
In \cite[(1.1.12)]{nb} we have proved that $Z^{(m)}$ 
is independent of the choice of $\Del$.  
We have the natural morphism 
$Z^{(m)}\lo Z$.

As in \cite[(3.1.4)]{dh2} and \cite[(2.2.18)]{nh2}, 
we have an orientation sheaf $\vp^{(m)}_{\rm zar}(Z/B)$ 
$(m\in {\mab N})$ in $Z^{(m)}_{\rm zar}$ 
associated to the set $\Del$. 
If $B$ is a closed subscheme of $B'$ defined by 
a quasi-coherent nil-ideal sheaf ${\mathcal J}$ 
which has a PD-structure $\del$, 
then $\vp^{(m)}_{\rm zar}(Z/B)$ 
extends to an abelian sheaf 
$\vp^{(m)}_{\rm crys}(Z/B')$ in 
$(Z^{(m)}/(B',{\mathcal J},\del))_{\rm crys}$. 
\par 
Assume that $\os{\circ}{S}$ is a scheme  
and that $M_S$ is the free log structure of rank 1 
for the time being and 
we fix an isomorphism 
\begin{align*} 
(M_S,\al_S)\simeq ({\mab N}\oplus {\mathcal O}_S^*\lo {\mathcal O}_S) 
%\tag{1.1.15.1}\label{ali:nsos} 
\end{align*} 
globally on $S$.  
Let $M_S(a,d)$ be the log structure on 
${\mab A}_{\os{\circ}{S}}(a,d)$ associated to the following morphism 
\begin{equation*} 
{\mab N}^{{\oplus}(a+1)}\owns 
%e_i:=
(0, \ldots,0,\os{i}{1},0,\ldots, 0)\lom x_{i-1}\in 
{\mathcal O}_S[x_0, \ldots, x_d]/(\prod_{i=0}^ax_i).
\tag{2.1.4}  
\end{equation*} 
%Here we set $0^0:=1$. 
Let ${\mab A}_S(a,d)$ be the resulting log scheme over $S$.
%We call $M_S(a,d)$ the 
%{\it standard log structure} on 
%${\mab A}_{\os{\circ}{S}}(a,d)$. 
The diagonal morphism ${\mab N} \lo {\mab N}^{{\oplus}(a+1)}$ 
induces a morphism ${\mab A}_S(a,d) \lo S$ of log schemes.

\begin{defi}\label{defi:lfac}  
Let $S$ be a family of log points 
(we do not assume that $M_S$ is free).  
Let $f \col X(=(\os{\circ}{X},M_X)) \lo S$ 
be a morphism of log schemes 
such that $\os{\circ}{X}$ is 
an SNC scheme over $\os{\circ}{S}$ 
with a decomposition 
$\Del:=\{\os{\circ}{X}_{\lam}\}_{\lam \in \Lam}$ 
of $\os{\circ}{X}/\os{\circ}{S}$ by its smooth components. 
We call $f$ (or $X/S$) an 
{\it SNCL$($=simple normal crossing log$)$ scheme} if, 
for any point of $x \in \os{\circ}{X}$, 
there exist an open neighborhood $\os{\circ}{V}$ of 
$x$ and an open neighborhood $\os{\circ}{W}$ of 
$\os{\circ}{f}(x)$ such that $M_W$ is 
the free log structure of rank $1$ 
and such that  $f\vert_V$ factors through 
a solid and \'{e}tale morphism 
$V {\lo} {\mab A}_W(a,d)$  
such that $\Del_{\os{\circ}{V}}
=\{x_i=0\}_{i=0}^a$ in $\os{\circ}{V}$.  
(Similarly we can give the definition of 
a formal SNCL scheme over $S$ 
in the case where $\os{\circ}{S}$ is a formal scheme.) 
\end{defi}

\par 
Let $X$ be an SNCL scheme over $S$ 
with a decomposition 
$\Del:=\{\os{\circ}{X}_{\lam}\}_{\lam \in \Lam}$ 
of $\os{\circ}{X}/\os{\circ}{S}$ 
by its smooth components. 
%For a subset $\ul{\lam}=\{\lam_0,\cdots, \lam_m\}$ 
%$(\lam_i \not= \lam_j~{\rm if}~i\not= j, \lam_i\in \Lam)$ of $\Lam$,  
%endow $\os{\circ}{X}_{\ul{\lam}}$ with the inverse image of 
%the log structure of $X$ by  the natural morphism $\os{\circ}{X}_{\ul{\lam}}\lo \os{\circ}{X}$ 
%and let $X_{\ul{\lam}}$ be the resulting log scheme. 
We set $X_{\emptyset}:=X$ for convenience of notation. 
%For an integer $m\geq -1$, 
%let $X^{(m)}$ be the log scheme 
%whose underlying scheme 
%is $\os{\circ}{X}{}^{(m)}$ and 
%whose log structure is the inverse image of the log structure of $X$ 
%by the natural morphism $\os{\circ}{X}{}^{(m)}\lo \os{\circ}{X}$.
%For a negative integer $m\leq -2$, 
%set $X^{(m)}=\emptyset$.  
Let 
%$a_{\ul{\lam}}\col \os{\circ}{X}_{\ul{\lam}} \os{\sus}{\lo} X$ 
%and 
$\os{\circ}{X}{}^{(m)} \lo \os{\circ}{X}$ 
be 
%the natural exact closed immersion and 
the natural morphism. 
%\par 
%Because the structural morphism $X\lo S$ is integral, 
%$X_T:=X\times_ST$ is also integral for a morphism 
%$T\lo S$ of fine log schemes. Note that $(X_T)^{\circ}=
%\os{\circ}{X}\times_{\os{\circ}{S}}\os{\circ}{T}$. 
%We denote $(X_T)^{\circ}$ by $\os{\circ}{X}_T$. 

\par 
Assume that $M_S\simeq ({\mab N}\owns 1 \lom 0 \in {\mathcal O}_S)^a$.
Then $\ol{S}=({\mab A}^1_{\os{\circ}{S}}, 
({\mab N}\owns 1 \lom t \in {\mathcal O}_S[t])^a)$.  
%In this case, we say that $M_{\ol{S}}$ is free of rank $1$. 
Set 
\begin{equation*} 
\os{\circ}{\mab A}_{\ol{S}}(a,d):=
\ul{\rm Spec}_{\os{\circ}{S}}
({\mathcal O}_S[x_0, \ldots, x_{d},t]/(x_0\cdots x_a-t)). 
\end{equation*} 
Then we have a natural structural morphism 
$\os{\circ}{{\mab A}}_{\ol{S}}(a,d) \lo \os{\circ}{\ol{S}}$. 
Let $\ol{M}_{\ol{S}}(a,d)$ 
be the log structure associated to 
a morphism 
$${\mab N}^{a+1} \owns e_i=
(0, \ldots,0,\os{i}{1},0,\ldots, 0) \lom x_{i-1} \in 
{\mathcal O}_S[x_0, \ldots, x_{d},t]/(x_0\cdots x_a-t).$$ 
Set 
$${\mab A}_{\ol{S}}(a,d)
:=(\ul{\rm Spec}_{\os{\circ}{S}}
({\mathcal O}_S[x_0, \ldots, x_{d},t]
/(x_0\cdots x_a-t)),\ol{M}_{\ol{S}}(a,d)).$$  
Then we have the following natural morphism 
\begin{equation*} 
{\mab A}_{\ol{S}}(a,d) \lo \ol{S}.  
\tag{2.2.1}
\end{equation*}

\par 
By killing ``$t$'', we have the following natural exact closed immersion 
\begin{equation*} 
{\mab A}_S(a,d) \os{\sus}{\lo} 
{\mab A}_{\ol{S}}(a,d) 
\tag{2.2.2}
\end{equation*}
of fs(=fine and saturated) log schemes 
over $S\os{\sus}{\lo} \ol{S}$ 
if the log structure of $S$ is 
the free hollow log structure of rank 1.  
%we have the following commutative diagram of log schemes 
%\begin{equation*} 
%\begin{CD} 
%{\mab A}_S(a,d) @>{\sus}>> {\mab A}_{\ol{S}}(a,d) \\ 
%@VVV @VVV \\
%S @>{\sus}>> \ol{S} \\ 
%@VVV @VVV\\
%\os{\circ}{S} @= \os{\circ}{S}.  
%\end{CD} 
%\tag{2.8.2}\label{cd:aabss} 
%\end{equation*} 

\begin{defi}[{\bf A special case of \cite[(1.1.16)]{nb}}]\label{defi:lfeac}  
Let $\ol{f} \col \ol{X}\lo \ol{S}$ 
be a morphism of log schemes (on the Zariski sites). 
Set $X:=\ol{X}\times_{\ol{S}}S$.  
We call $\ol{f}$ (or $\ol{X}/\ol{S}$) 
a {\it strict semistable log scheme} over $\ol{S}$ 
if $\os{\circ}{\ol{X}}$ is a smooth scheme over $\os{\circ}{S}$, 
if $\os{\circ}{X}$ is a relative SNCD on 
$\os{\circ}{\ol{X}}/\os{\circ}{\ol{S}}$ 
(with some decomposition 
$\Del:=\{\os{\circ}{X}_{\lam}\}_{\lam \in \Lam}$ of 
$\os{\circ}{X}$ by smooth components of 
$\os{\circ}{X}$ over $\os{\circ}{S}$) 
and if, for any point of $x \in \os{\circ}{\ol{X}}$, 
there exist an open neighborhood $\os{\circ}{\ol{V}}$ of 
$x$ and an open neighborhood 
$\os{\circ}{\ol{W}}\simeq \ul{\rm Spec}_W({\mathcal O}_W[t])$ 
(where $W:=\ol{W}\times_{\ol{S}}S$) of 
$\os{\circ}{f}(x)$ such that 
$M_{\ol{W}}\simeq  
({\mab N}\owns 1 \lom t\in {\mathcal O}_W[t])^a$ and such that  
$\ol{f}\vert_{\ol{V}}$ factors through 
a solid and \'{e}tale morphism 
$\ol{V} {\lo} {\mab A}_{\ol{W}}(a,d)$ 
%(here we assume that $S$ is free of rank $1$) 
such that  
$\Del_{\os{\circ}{\ol{V}}}
=\{x_i=0\}_{i=0}^a$ in $\os{\circ}{\ol{V}}$.  
(Similarly we can give the definition of 
a strict semistable log {\it formal} scheme over $\ol{S}$.) 
\end{defi}

\section{SNCL schemes with relative SNCD's}\label{sec:snrdlv} 
In this section we recall the definition of an SNCL scheme with a relative SNCD
in \cite{ny} and we give elementary results which are necessary in later sections. 
First let us recall the following definition. 

\begin{defi}[{\bf \cite[(6.1)]{ny}}]\label{defi:hdd} 
(1) Let $S$ be a family of log points and let $X/S$ be an SNCL scheme. 
Let ${\rm Div}(\os{\circ}{X}/\os{\circ}{S})_{\geq 0}$  
be the set of effective Cartier divisors on $\os{\circ}{X}/\os{\circ}{S}$. 
Let $\os{\circ}{D}$ be an effective Cartier divisor on $\os{\circ}{X}/\os{\circ}{S}$. 
Endow $\os{\circ}{D}$ with the inverse image of the log structure of $X$ and 
let $D$ be the resulting log scheme. 
We call $D$ a {\it relative simple normal crossing divisor 
$(=:$relative SNCD$)$} on $X/S$ if there exists a family 
$\Del:=\{\os{\circ}{D}_{\mu}\}_{\mu \in M}$ of 
non-zero effective Cartier divisors on $\os{\circ}{X}/\os{\circ}{S}$ 
of locally finite intersection which are 
SNC(=simple normal crossing) schemes over $\os{\circ}{S}$ such that 
\begin{equation*}  
\os{\circ}{D} = \sum_{\mu \in M}\os{\circ}{D}_{\mu}
\quad \text{in} \quad {\rm Div}(\os{\circ}{X}/\os{\circ}{S})_{\geq 0} 
\tag{3.1.1}\label{eqn:dcsncd}
\end{equation*} 
and,  
for any point $z$ of $\os{\circ}{D}$, there exist 
a Zariski open neighborhood $\os{\circ}{V}$ of $z$ in $\os{\circ}{X}$ and 
the following cartesian diagram
\begin{equation*}
\begin{CD}
D\vert_V @>>> (y_1\cdots y_b=0) \\ 
@V{\bigcap}VV  @VVV \\
V @>{g}>> {\mab A}_{S'}(a,d-e)\times_{S'}{\mab A}^e_{S'} \\ 
@VVV  @VVV \\
S'@=S'
\end{CD}
\tag{3.1.2}\label{cd:1b}
\end{equation*}
for some nonnegative integers $a$, $b$, $d$ and $e$ 
such that $a\leq d-e$ and $b\leq e\leq d$. 
Here $S'$ is an open log subscheme of $S$ 
whose log structure is associated to 
the morphism ${\mab N}\owns 1\lom 0\in {\mathcal O}_{S'}$, 
${\mab A}_{S'}(a,d-e)\times_{S'}{\mab A}^{e}_{S'}={\mab A}_{S'}(a,d)$ 
$(0\leq a\leq d-e, e\leq d)$ is a log scheme whose 
underlying scheme is $\ul{\rm Spec}_{\os{\circ}{S}{}'}
({\mathcal O}_{S'}[x_0,\ldots, x_{d-e},y_1, \ldots, y_{e}]/(x_0\cdots x_a))$ and whose  
log structure is the association 
of the following morphism 
$${\mab N}^{\oplus a+1}\owns e_i\lom x_{i-1}\in {\mathcal O}_{S'}
[x_0,\ldots, x_{d-e},y_1, \ldots, y_e]/(x_0\cdots x_a),$$ 
%and ${\mab A}_{S'}(a,d-e)\times_{S'}{\mab A}^e_{S'}$ is obtained by 
%the diagonal embedding ${\mab N} \os{\subset}{\lo} {\mab N}^{\oplus a+1}$. 
$(y_1\cdots y_b=0)$ is an exact closed log subscheme of 
${\mab A}_{S'}(a,d-e)\times_{T}{\mab A}^e_{S'}$ 
defined by an ideal sheaf $(y_1\cdots y_b)$ and  
$g$ is solidly log \'{e}tale. 
Endow $\os{\circ}{D}_{\mu}$ with the inverse image of 
the log structure of $X$ and let $D_{\mu}$ be the resulting log scheme. 
We call $D_{\mu}$ an {\it SNCL component} of $D$ and 
the equality (\ref{eqn:dcsncd}) a {\it decomposition} of $D$ 
by SNCL components of $D$. 
\end{defi} 
We set 
$${\mab A}_{S'}(a,b,d,e):=
({\mab A}_{S'}(a,d-e)\times_S{\mab A}^e_{S'},(y_1\cdots y_b=0))$$
for $a\leq d-e$ and $b\leq e\leq d$. 
That is, ${\mab A}_{S'}(a,b,d,e)$ is a log scheme 
whose underlying scheme is 
$$\ul{\rm Spec}_{S'}({\mathcal O}_{S'}[x_0,\ldots, x_{d-e},y_1, \ldots, y_e]/(x_0\cdots x_a))$$ 
and whose log structure is the association 
of the following morphism 
$${\mab N}^{\oplus a+1}\oplus{\mab N}^{\oplus b}\lo {\mathcal O}_{S}
[x_0,\ldots, x_{d-e},y_1, \ldots, y_e]/(x_0\cdots x_a)$$ 
defined by 
$\owns e_i\lom x_{i-1}$ for $1\leq i\leq a+1$ and $\owns e_i\lom y_{i-(a+1)}$ for $a+1< i\leq a+b+1$.

\par
Before \cite[(6.2)]{ny} we have constructed a log structure $M(D)$ 
in the zariski topos $\os{\circ}{X}_{\rm zar}$ as in \cite[p.~61]{nh2} 
and in \cite[(6.2)]{ny} we have proved the following:

\begin{prop}\label{prop:logst}
Let the notations be as above. 
Let $z$ be a point of $D$ and let 
$V$ be an open neighborhood of $z$ in $X$ 
in the diagram {\rm (\ref{cd:1b})}. 
Assume that $z \in \bigcap_{i=1}^b\{y_i=0\}$. 
If $V$ is small, then the log structure 
$M(D) \vert_V \lo {\mathcal O}_V$ is 
isomorphic to ${\mathcal O}_V^*y_1^{{\mab N}} 
\cdots y_b^{{\mab N}} 
\os{\subset}{\lo} {\mathcal O}_V$. 
Consequently $M(D) \vert_V$ is associated to 
the homomorphism 
${\mab N}^b_V \owns e_i \lom  y_i \in M(D) \vert_V$ 
$(1 \leq i \leq b)$ of sheaves of monoids on $V$, 
where $\{e_i\}_{i=1}^b$ is the canonical 
basis of ${\mab N}^b$.  
In particular, $M(D)$ is fs{\rm (}=fine and saturated{\rm )}. 
\end{prop}

\parno 
Set 
\begin{align*} 
(X,D):=(X,M_X\oplus_{{\mathcal O}^*_X}M(D)\lo {\mathcal O}_X). 
\end{align*} 
Then $(X,D)/S$ is log smooth, integral and saturated. 
This is nothing but the fiber product of $(X,M_X)$ and 
$(\os{\circ}{X},M(D))$ over $\os{\circ}{X}$. 
As in the classical case (e.~g., \cite{dh2}), 
we can consider the log de Rham complex 
$\Om^{\bul}_{X/S}(\log D)$. 
%with logarithmic poles along $D$.  
It is clear that 
the complex $\Om^{\bul}_{X/S}(\log D)$ is equal to 
the log de Rham complex $\Om^{\bul}_{(X,D)/S}$.

\par
For $\ul{\mu}:=\{\mu_1, \mu_2,\ldots \mu_k\}$ $(\mu_i \not= \mu_j~{\rm if}~i\not= j)$,  
set 
\begin{equation*}
D_{\ul{\mu}} :=D_{\mu_1}\cap D_{\mu_2} \cap \cdots \cap D_{\mu_k} 
%%%\tag{9.13.1}
%\label{eqn:parlm}
\end{equation*}
for a positive integer $k$ and set
\begin{equation*}
D^{(k)} = 
\begin{cases} 
\quad X & (k=0), \\
\us{\# \ul{\mu}=k}{\coprod}D_{\ul{\mu}} & (k\geq 1)
\end{cases}
%%%\tag{9.13.2}
\label{eqn:kfdintd}
\end{equation*}
for a nonnegative integer $k$.
%Set 
%\begin{equation}
%D_{\emptyset}:=X  
%%%\tag{9.13.3}
%\label{eqn:dphix}
%\end{equation}
%for later convenience.
We see that 
$D^{(k)}$ is independent of the choice of 
the decomposition of $D$ by SNCL components of $D$.
\par

\begin{defi}\label{defi:simpr}
Let $S$ be a family of log points and let $\ol{X}/\ol{S}$ be a strictly semistable  scheme. 
Let $\os{\circ}{\ol{D}}$ be an effective Cartier divisor on $\os{\circ}{\ol{X}}/\os{\circ}{\ol{S}}$. 
Endow $\os{\circ}{\ol{D}}$ with the inverse image of the log structure of $\ol{X}$ and 
let $\ol{D}$ be the resulting log scheme. 
We call $\ol{D}$ a {\it relative simple normal crossing divisor 
$(=:$relative SNCD$)$} on $\ol{X}/\ol{S}$ if there exists a family 
$\ol{\Del}:=\{\os{\circ}{\ol{D}}_{\lam}\}_{\lam \in \Lam}$ of 
non-zero effective Cartier divisors on $\os{\circ}{\ol{X}}/\os{\circ}{\ol{S}}$ 
of locally finite intersection which are 
strictly semistable schemes over $\os{\circ}{\ol{S}}$ such that 
\begin{equation*}  
\os{\circ}{\ol{D}} = \sum_{\lam \in \Lam}\os{\circ}{\ol{D}}_{\lam}
\quad \text{in} \quad {\rm Div}(\os{\circ}{\ol{X}}/\os{\circ}{\ol{S}})_{\geq 0} 
\tag{3.3.1}\label{eqn:dcsdvd}
\end{equation*} 
and,  
for any point $z$ of $\os{\circ}{D}$, there exist 
a Zariski open neighborhood $\os{\circ}{V}$ of $z$ in $\os{\circ}{X}$ and 
the following cartesian diagram
\begin{equation*}
\begin{CD}
\ol{D}\vert_V @>>> (y_1\cdots y_b=0) \\ 
@V{\bigcap}VV  @VVV \\
V @>{g}>> {\mab A}_{\ol{S}{}'}(a,d-e)\times_{T}{\mab A}^e_{S'} \\ 
@VVV  @VVV \\
\ol{S}{}'@=\ol{S}{}'
\end{CD}
\tag{3.3.2}\label{cd:1vb}
\end{equation*}
for some nonnegative integers $a$, $b$, $d$ and $e$ 
such that $0\leq a\leq d-e$ and $b\leq e\leq d$. 
Here $\ol{S}{}'$ is an open log subscheme of $\ol{S}$ 
whose log structure is associated to 
the morphism ${\mab N}\owns 1\lom t\in {\mathcal O}_{\ol{S}{}'}$, 
$(y_1\cdots y_b=0)$ is an exact closed log subscheme of 
${\mab A}_{\ol{S}{}'}(a,d-e)\times_{T}{\mab A}^e_{S'}$ defined by 
an ideal sheaf $(y_1\cdots y_b)$,  
$g$ is solidly log \'{e}tale. 
%and ${\mab A}_{S'}(a,d-e)\times_{S'}{\mab A}^e_{S'}$ is obtained by 
%the diagonal embedding ${\mab N} \os{\subset}{\lo} {\mab N}^{\oplus a+1}$. 
Endow $\os{\circ}{\ol{D}}_{\lam}$ with the inverse image of 
the log structure of $X$ and let $\ol{D}_{\lam}$ be the resulting log scheme. 
We call $\ol{D}_{\lam}$ a {\it strictly semistable component} of $\ol{D}$ and 
the equality (\ref{eqn:dcsdvd}) a {\it decomposition} of $\ol{D}$ 
by strictly semistable components of $\ol{D}$. 
\end{defi} 

%Assume that the log structure of $\ol{S}{}'$ is associated to 
%the morphism ${\mab N}\owns 1\lom t\in {\mathcal O}_{\ol{S}{}'}$.  
We set 
$${\mab A}_{\ol{S}'}(a,b,d,e):=
({\mab A}_{\ol{S'}}(a,d-e)\times_{\ol{S}}{\mab A}^e_{\ol{S'}},(y_1\cdots y_b=0))$$
for $a\leq d-e$ and $b\leq e\leq d$. That is, ${\mab A}_{\ol{S'}}(a,b,d,e)$ is 
a log scheme whose underlying scheme is 
$$\ul{\rm Spec}_{\ol{S'}}({\mathcal O}_{S'}[t][x_0,\ldots, x_{d-e},y_1, \ldots, y_e]/(x_0\cdots x_a-t))$$ 
and whose log structure is the association 
of the following morphism 
$${\mab N}^{\oplus a+b+1}
={\mab N}^{\oplus a+1}\oplus{\mab N}^{\oplus b}\lo {\mathcal O}_{S}
[x_0,\ldots, x_{d-e},y_1, \ldots, y_e]/(x_0\cdots x_a-t)$$ 
defined by 
$\owns e_i\lom x_{i-1}$ for $1\leq i\leq a+1$ and $\owns e_i\lom y_{i-(a+1)}$ for $a+1< i\leq a+b+1$.

\begin{lemm}
[{\bf A special case of \cite[(1.1.6)]{nb}}]
\label{lemm:etl}
Let $S$ be a family of log points. Then the following hold$:$ 
\par 
$(1)$ Let $Y\lo S$ be a log smooth scheme 
which has a global chart ${\mab N}\lo P$. 
Then, Zariski locally on $Y$,  
there exists a log smooth scheme 
$\ol{Y}$ over $\ol{S}$ fitting into 
the following cartesian diagram 
\begin{equation*}
\begin{CD} 
Y @>{\sus}>> \ol{Y} \\
@VVV @VVV \\
S\times_{{\rm Spec}^{\log}({\mab Z}[{\mab N}])}
{\rm Spec}^{\log}({\mab Z}[P])
@>{\sus}>> \ol{S}\times_{{\rm Spec}^{\log}({\mab Z}[{\mab N}])}
{\rm Spec}^{\log}({\mab Z}[P]) \\ 
@VVV @VVV \\
S@>{\sus}>> \ol{S},  
\end{CD} 
\tag{3.4.1}\label{cd:xwedtx} 
\end{equation*}
where the vertical morphism 
$\ol{Y} \lo \ol{S}
\times_{{\rm Spec}^{\log}({\mab Z}[{\mab N}])}
{\rm Spec}^{\log}({\mab Z}[P])$ 
is solid and \'{e}tale.  
\par 
$(2)$ Let $S$ be a family of log points.  
Let $X$ be an SNCL scheme over $S$ with a relative SNCD 
$D$ on $X/S$. 
Zariski locally on $X$, there exists a strictly semistable log scheme 
$\ol{X}$ over $\ol{S}$ with a relative SNCD $\ol{D}$ on $\ol{X}/\ol{S}$ fitting into 
the following cartesian diagram for $0\leq a\leq d-e$ and $b\leq e\leq d:$ 
\begin{equation*}
\begin{CD} 
(X,D) @>{\sus}>> (\ol{X},\ol{D}) \\
@VVV @VVV \\
{\mab A}_{S}(a,b,d,e)
@>{\sus}>> {\mab A}_{\ol{S}}(a,b,d,e)\\ 
@VVV @VVV \\
S@>{\sus}>> \ol{S},  
\end{CD} 
\tag{3.4.2}\label{cd:xwbdtx} 
\end{equation*}
where we take a local chart of the log structure of $\ol{S}$ 
and 
the vertical morphism 
$(\ol{X},\ol{D}) \lo {\mab A}_{\ol{S}}(a,b,d,e)$ is solid and \'{e}tale.  
\end{lemm}

We also recall the following (\cite[(2.1.5)]{nh2}) describing 
the local structure of an exact closed immersion, 
which will be used in this section and later sections: 

\begin{prop}[{\bf \cite[(2.1.5)]{nh2}}]\label{prop:adla}
Let $T_0 \os{\sus}{\lo} T$ be a closed immersion of fine log schemes. 
Let $Y$ $($resp.~${\mathcal Q})$ be a log smooth scheme over $T_0$ 
$($resp.~$T)$, which can be considered as a log scheme over $T$.  
Let $\iota \col Y \os{\sus}{\lo} {\mathcal Q}$ 
be an exact closed immersion over $T$. 
Let $y$ be a point of $\os{\circ}{Y}$ and 
assume that there exists a chart $(Q \lo M_T, P \lo M_Y, Q 
\os{\rho}{\lo} P)$ of $Y \lo T_0 \os{\subset}{\lo} T$ 
on a neighborhood of $y$ such that 
$\rho$ is injective, ${\rm Coker}(\rho^{\rm gp})$ is torsion free 
and the natural homomorphism ${\mathcal O}_{Y,y} \otimes_{{\mab Z}} 
(P^{{\rm gp}}/Q^{{\rm gp}}) \lo \Om^1_{Y/T_0,y}$ is an isomorphism. 
Then, on a neighborhood of $y$, 
there exist a nonnegative integer $c$
and the following cartesian diagram 
\begin{equation}
{\small{\begin{CD}
Y @>>> {\mathcal Q}' @>>> {\mathcal Q}\\ 
@VVV  @VVV @VVV\\ 
(T_0\otimes_{{\mab Z}[Q]}{\mab Z}[P],P^a) 
@>{\sus}>> (T\otimes_{{\mab Z}[Q]}{\mab Z}[P],P^a)
@>{\sus}>>  
(T\otimes_{{\mab Z}[Q]}{\mab Z}[P],P^a)
\times_T{\mab A}^{c}_T, 
\end{CD}}}
\tag{3.5.1}\label{eqn:0txda}
\end{equation}
where the vertical morphisms are solid and \'{e}tale and 
the lower second horizontal morphism is the base change of 
the zero section $T \os{\sus}{\lo} {\mab A}^c_T$ 
and ${\mathcal Q}':={\mathcal Q}\times_{{\mab A}^c_T}T$. 
\end{prop}

\par 
Let $S_0 \os{\sus}{\lo} S$ 
be a nil-immersion of families of log points. 
Let $X/S_0$ be an SNCL scheme with a relative SNCD $D$ on $X/S_0$. 
Let 
$(X,D) \os{\sus}{\lo} {\mathcal P}$ be 
an immersion into a log smooth scheme over $S$. 
Let 
$(X,D) \os{\sus}{\lo} \ol{\mathcal P}$ be also 
an immersion into a log smooth scheme over $\ol{S}$. 
(We do not assume that there exists an immersion 
${\mathcal P}\os{\sus}{\lo} \ol{\mathcal P}$).

\begin{prop}\label{prop:fosi} 
Assume that  
$(X,D)\os{\sus}{\lo} {\mathcal P}$ 
$($resp.~$(X,D)\os{\sus}{\lo} \ol{\mathcal P})$
has a global chart 
$P\lo Q$ $($resp.~$\ol{P}\lo Q)$. 
Let $P^{\rm ex}$ $($resp.~$\ol{P}{}^{\rm ex})$ 
be the inverse image of 
$Q$ by the morphism $P^{\rm gp}\lo Q^{\rm gp}$ 
$($resp.~$\ol{P}{}^{\rm gp}\lo Q^{\rm gp})$. 
Set ${\mathcal P}^{\rm prex}
:={\mathcal P}\times_{{\rm Spec}^{\log}({\mab Z}[P])}
{\rm Spec}^{\log}({\mab Z}[P^{\rm ex}])$ 
$($resp.~$\ol{\mathcal P}{}^{\rm prex}
:=\ol{\mathcal P}\times_{{\rm Spec}^{\log}({\mab Z}[\ol{P}])}
{\rm Spec}^{\log}({\mab Z}[\ol{P}{}^{\rm ex}]))$. 
Then, locally on $X$, there exists an open neighborhood 
${\mathcal P}^{\rm prex}{}'$ 
$($resp.~$\ol{\mathcal P}{}^{\rm prex}{}')$ 
of ${\mathcal P}^{\rm prex}$ 
$($resp.~$\ol{\mathcal P}{}^{\rm prex})$ 
fitting into 
the following cartesian diagram 
for some $0\leq a\leq d-e$ and $b\leq e\leq d\leq d':$ 
\begin{equation*}
\begin{CD}
(X,D) @>{\subset}>> {\mathcal P}^{\rm prex}{}'\\ 
@VVV  @VVV \\
{\mab A}_{S_0}(a,b,d,e) @>{\sus}>> {\mab A}_{S}(a,b,d',e)\\ 
\end{CD}
\tag{3.6.1}\label{eqn:xdbxda}
\end{equation*}
$($resp.~
\begin{equation*}
\begin{CD}
X @>{\subset}>> \ol{\mathcal P}{}^{\rm prex}{}'\\ 
@VVV  @VVV \\
{\mab A}_{S_0}(a,b,d,e) @>{\sus}>> {\mab A}_{\ol{S}}(a,b,d',e),
\end{CD}
\tag{3.6.2}\label{eqn:xdpelba}
\end{equation*}
where the vertical morphisms are solid and \'{e}tale.  
\end{prop}
\begin{proof} 
The proof is the same as that of \cite[(1.1.40)]{nb} by using (\ref{prop:adla}). 
\end{proof}

\begin{prop}[{\bf cf.~\cite[(1.1.41)]{nb}}]\label{prop:neoxeo}
$(1)$ Let ${\mathcal P}^{\rm ex}$ be the exactification of 
the immersion $(X,D)\os{\sus}{\lo}{\mathcal P}$.  
Then ${\mathcal P}^{\rm ex}$ is a formal SNCL scheme over $S$ with 
a unique relative SNCD ${\mathcal D}$ on ${\mathcal P}^{\rm ex}/S$ such that 
${\mathcal D}\times_{{\mathcal P}^{\rm ex}}X=D$.  
\par 
$(2)$ Let $\ol{\mathcal P}{}^{\rm ex}$ 
be the exactification of 
the immersion $X\os{\sus}{\lo}\ol{\mathcal P}$. 
Then $\ol{\mathcal P}{}^{\rm ex}$ 
is a formal strict semistable family over $\ol{S}$ with a unique relative SNCD 
$\ol{\mathcal D}$ on $\ol{\mathcal P}{}^{\rm ex}/\ol{S}$ such that 
$\ol{\mathcal D}\times_{\ol{\mathcal P}{}^{\rm ex}}X=D$. 
\end{prop}
\begin{proof} 
Because the proof of (1) is the same as that of that of (2), we give the proof of (2). 
Since the immersion $X\os{\sus}{\lo}\ol{\mathcal P}{}^{\rm ex}$ is exact, 
the natural morphism 
$(M_{\ol{\mathcal P}{}^{\rm ex}}/{\mathcal O}^*_{\ol{\mathcal P}{}^{\rm ex}})_x
\lo (M_{(X,D)}/{\mathcal O}^*_{(X,D)})_x\simeq {\mab N}^{\oplus a+1}\oplus {\mab N}^{\oplus b}$ 
is an isomorphism for a point $x\in \os{\circ}{X}$. 
The local coordinates of $\ol{\mathcal P}{}^{\rm ex}$ corresponding to 
${\mab N}^{\oplus a+1}$ tells us that  $\ol{\mathcal P}{}^{\rm ex}$ is a formal 
strict semistable scheme over $\ol{S}$;  
the local coordinates of $\ol{\mathcal P}{}^{\rm ex}$ corresponding to 
${\mab N}^{\oplus b}$ tells us that $\ol{\mathcal P}{}^{\rm ex}$ has a relative SNCD 
$\ol{\mathcal D}$ on $\ol{\mathcal P}{}^{\rm ex}/\ol{S}$ such that 
$\ol{\mathcal D}\times_{\ol{\mathcal P}{}^{\rm ex}}X=D$. 
The uniqueness of $\ol{\mathcal D}$ is obvious since 
the underlying topological space of 
$\os{\circ}{\ol{\mathcal P}}{}^{\rm ex}$ is equal to that of $\os{\circ}{X}$
and since the immersion $X\os{\sus}{\lo}\ol{\mathcal P}{}^{\rm ex}$ is exact. 
\end{proof}

\par 
Let $(X,D)$ be an SNCL scheme over $S$ with a relative SNCD. 
Let $S$ be a family of log points. Let $M_S=(M_S,\al_S)$ be the log structure of $S$. 
In \cite{nb} we have defined a log PD-enlargement 
$((T,{\mathcal J},\del),z)$ of $S$ as follows (cf.~\cite{oc}): 
$(T,{\mathcal J},\del)$ is a fine log PD-scheme such that ${\mathcal J}$ is quasi-coherent and 
$z\col T_0\lo S$ is a morphism of fine log schemes, 
where $T_0:=T \mod {\mathcal J}$. 
When we are given a morphism $S\lo S'$ of families of log points, 
we can define a morphism of  log PD-enlargements over 
the morphism $S\lo S'$ in an obvious way. 
Endow $\os{\circ}{T}_0$ with the inverse image of the log structure of $S$. 
We denote the resulting log scheme by $S_{\os{\circ}{T}_0}$. 
It is easy to see that the natural morphism $z^*(M_S)\lo M_{T_0}$ is injective
(\cite[(1.1.4)]{nb}). Hence we can consider $z^*(M_S)$ is a sub log structure of $M_{T_0}$. 
Let $M$ be the sub log structure of the log structure $(M_T,\al_T)$ of $T$ such that 
the natural morphism $M_T\lo M_{T_0}$ induces an isomorphism 
$M/{\mathcal O}_T^*\os{\sim}{\lo} z^*(M_S)/{\mathcal O}_{T_0}^*$.  
Let $S(T)$ be the log scheme $(\os{\circ}{T},(M,\al_T\vert_M))$. 
Since $M/{\mathcal O}_T^*$ is constant, we can consider the hollowing out  
$S(T)^{\nat}$ of $S(T)$ (\cite[Remark 7]{oc}); 
the log scheme $S(T)^{\nat}$ is a family of log points. 
Set $X_{\os{\circ}{T}_0}:=X_{}\times_SS_{\os{\circ}{T}_0}$. 
%=
%X\times_{\os{\circ}{S}}\os{\circ}{T}_0$
%(we can consider $X$ as a fine log scheme over $\os{\circ}{S}$).    
%By abuse of notation, we denote by the same symbol $f$ 
Let  
$f\col (X_{\os{\circ}{T}_0},D_{\os{\circ}{T}_0}) \lo S_{\os{\circ}{T}_0}$ 
be the structural morphism. 
Let 
$a^{(l,m)}\col 
\os{\circ}{X}{}^{(l)}_{\os{\circ}{T}_0}\cap \os{\circ}{D}{}^{(m)}_{\os{\circ}{T}_0} \lo 
\os{\circ}{X}_{\os{\circ}{T}_0}$ $(l,m\in {\mab N})$ 
be the base change morphisms of the natural morphism 
$\os{\circ}{X}{}^{(l)}\cap \os{\circ}{D}{}^{(m)} 
\lo \os{\circ}{X}$ with respect to $\os{\circ}{T}_0\lo \os{\circ}{S}$. 
\par 
Let $(X_{\os{\circ}{T}_0},D_{\os{\circ}{T}_0})'=
\coprod_{i\in I}(X_{\os{\circ}{T}_0},D_{\os{\circ}{T}_0})_i$ be the disjoint union 
of an affine open covering of $(X_{\os{\circ}{T}_0},D_{\os{\circ}{T}_0})$ over $S_{\os{\circ}{T}_0}$ 
($(X_{\os{\circ}{T}_0},D_{\os{\circ}{T}_0})_i$ is a log open subscheme of 
$(X_{\os{\circ}{T}_0},D_{\os{\circ}{T}_0})$). 
Assume that $f((X_{\os{\circ}{T}_0},D_{\os{\circ}{T}_0})_i)$ is contained in 
an affine open subscheme of 
$\os{\circ}{T}_0=(S_{\os{\circ}{T}_0})^{\circ}$ 
such that the restriction of $M_{S_{\os{\circ}{T}_0}}$ 
to this open subscheme is free of rank $1$. 
Assume also that there exists a solid and log \'{e}tale morphism 
$(X_{\os{\circ}{T}_0},D_{\os{\circ}{T}_0})_i\lo {\mab A}_{S_{\os{\circ}{T}_0}}(a,b,d,e)$. 
Then, replacing $(X_{\os{\circ}{T}_0},D_{\os{\circ}{T}_0})_i$ by 
a small log open subscheme of 
$(X_{\os{\circ}{T}_0},D_{\os{\circ}{T}_0})$, 
we can assume that there exists a log smooth scheme 
$\ol{\mathcal P}{}'_{i}/\ol{S(T)^{\nat}}$ fitting into the following commutative diagram 
\begin{equation*}
\begin{CD} 
(X_{\os{\circ}{T}_0},D_{\os{\circ}{T}_0})_{i}@>{\subset}>> \ol{\mathcal P}{}'_{i}\\
@VVV @VVV \\ 
{\mab A}_{S_{\os{\circ}{T}_0}}(a,b,d,e) @>{\subset}>>
{\mab A}_{\ol{S(T)^{\nat}}}(a,b,d',e) \\
@VVV @VVV \\ 
S_{\os{\circ}{T}_0}@>{\subset}>> \ol{S(T)^{\nat}},   
\end{CD}
\tag{3.7.1}\label{cd:xnip} 
\end{equation*}
where $d\leq d'$ and the morphism 
$\ol{\mathcal P}{}'_{i}\lo {\mab A}_{\ol{S(T)^{\nat}}}(a,b,d,e)$ 
is solid and \'{e}tale ((\ref{lemm:etl})). 
%By the axiom of choice, there exists a map 
%$\tau \col I \lo J$ such that 
%$f(\os{\circ}{X}_{0i}) \subset \os{\circ}{S}_{\tau(i)}$ 
%for all $i$'s.  
Set $\ol{\mathcal P}{}':=\coprod_{i\in I}\ol{\mathcal P}{}'_{i}$. 
Set also 
\begin{equation*} 
(X_{\os{\circ}{T}_0,n},D_{\os{\circ}{T}_0,n})
:={\rm cosk}_0^{(X_{\os{\circ}{T}_0},D_{\os{\circ}{T}_0})}((X_{\os{\circ}{T}_0},D_{\os{\circ}{T}_0})')_n 
\quad (n\in {\mab N})
\tag{3.7.2}\label{eqn:iincl}
\end{equation*}   
and 
\begin{equation*} 
\ol{\mathcal P}_{n}:={\rm cosk}_0^{\ol{S(T)^{\nat}}}(\ol{\mathcal P}{}')_n\quad (n\in {\mab N}). 
\tag{3.7.3}\label{eqn:iigncl}
\end{equation*}  
Then we have a simplicial SNCL scheme $(X_{\os{\circ}{T}_0\bul},D_{\os{\circ}{T}_0\bul})$ 
with a relative SNCD 
and an immersion 
\begin{equation*}  
(X_{\os{\circ}{T}_0\bul},D_{\os{\circ}{T}_0\bul}) \os{\sus}{\lo} \ol{\mathcal P}_{\bul} 
\tag{3.7.4}\label{eqn:eixd} 
\end{equation*} 
into a log smooth simplicial log scheme over $\ol{S(T)^{\nat}}$.  
Thus we have obtained the following:

\par 

\begin{prop}\label{prop:xbn}  
The following hold$:$ 
\par 
$(1)$ There exist a \v{C}ech diagram 
$(X_{\os{\circ}{T}_0\bul},D_{\os{\circ}{T}_0\bul})$ of $(X_{\os{\circ}{T}_0},D_{\os{\circ}{T}_0})$ 
and an immersion 
\begin{equation*}  
\begin{CD} 
(X_{\os{\circ}{T}_0\bul},D_{\os{\circ}{T}_0\bul})
@>{\sus}>> \ol{\mathcal P}_{\bul} \\
@VVV @VVV \\
S_{\os{\circ}{T}_0} @>{\subset}>> \ol{S(T)^{\nat}}
\end{CD} 
\tag{3.8.1}\label{eqn:eipxd} 
\end{equation*} 
into a log smooth simplicial log scheme over $\ol{S(T)^{\nat}}$. 
\par 
$(2)$ There exists a \v{C}ech diagram 
$(X_{\os{\circ}{T}_0\bul},D_{\os{\circ}{T}_0\bul})$ of 
$(X_{\os{\circ}{T}_0},D_{\os{\circ}{T}_0})$ 
and an immersion 
\begin{equation*}  
\begin{CD} 
(X_{\os{\circ}{T}_0\bul},D_{\os{\circ}{T}_0\bul})
@>{\sus}>> {\mathcal P}_{\bul} \\
@VVV @VVV \\
S_{\os{\circ}{T}_0} @>{\subset}>> S(T)^{\nat}
\end{CD} 
\tag{3.8.2}\label{eqn:eolnpxd} 
\end{equation*} 
into a log smooth simplicial log scheme over $S(T)^{\nat}$. 
\end{prop} 

%\begin{lemm}\label{lemm:xdst}
%Let $Y\os{\sus} {\mathcal Q}_1$ and $Y\os{\sus} {\mathcal Q}_2$ be an immersion. 
%Assume that $\os{\circ}{\mathcal Q}_1=\os{\circ}{\mathcal Q}_2$. 
%If $\os{\circ}{\mathcal Q}{}^{\rm ex}_1=\os{\circ}{\mathcal Q}{}^{\rm ex}_2$, 
%then the immersion $Y \os{\sus}{\lo} {\mathcal Q}_1^{\rm ex}\otimes
%\end{lemm}
%\begin{proof} 
%Obvious. 
%\end{proof}

\begin{coro}\label{coro:exem} 
Let the notations be as above. 
Then there exist  an exact immersion 
\begin{equation*}  
\begin{CD} 
(X_{\os{\circ}{T}_0\bul},D_{\os{\circ}{T}_0\bul}) @>{\sus}>> 
(\ol{\mathcal X}_{\bul},\ol{\mathcal D}_{\bul}) \\
@VVV @VVV \\
S_{\os{\circ}{T}_0} @>{\subset}>> \ol{S(T)^{\nat}}
\end{CD} 
\tag{3.9.1}\label{eqn:eipexxd} 
\end{equation*} 
into a simplicial strict semistable log formal scheme 
with a simplicial relative SNCD over 
$\ol{S(T)^{\nat}}$ and an exact immersion 
\begin{equation*}  
\begin{CD} 
(X_{\os{\circ}{T}_0\bul},D_{\os{\circ}{T}_0\bul}) @>{\sus}>> 
({\mathcal X}_{\bul},{\mathcal D}_{\bul}) \\
@VVV @VVV \\
S_{\os{\circ}{T}_0} @>{\subset}>> S(T)^{\nat}
\end{CD} 
\tag{3.9.2}\label{eqn:eoexpxd} 
\end{equation*} 
into a simplicial formal SNCL scheme over $S(T)^{\nat}$. 
\end{coro}
\begin{proof}
This follows from (\ref{prop:neoxeo}) and (\ref{prop:xbn}). 
\end{proof}

We conclude this section by recalling   
the ``mapping degree function'' defined in \cite{nb}.  
\par  
Let $v \col S \lo S'$ be a morphism of families of log points. 
Let $y$ be a point of $\os{\circ}{S}$. 
Let $h\col {\mab N}= M_{S',v(y)}/{\mathcal O}^*_{S',v(y)}
\lo M_{S,y}/{\mathcal O}^*_{S,y}={\mab N}$ 
be the induced morphism. 
Let $d\in {\mab N}$ be the image of $1\in {\mab N}$ by $h$.

\begin{defi}[{\bf A special case of \cite[(1.1.42)]{nb}}]\label{defi:ddef} 
We call $\deg(v)_y:=d$ 
the {\it $($mapping$)$ degree} of $v$ at $y$. 
We call $\deg(v)\col \os{\circ}{S}\lo {\mab Z}_{\geq 1}$ 
the {\it $($mapping$)$ degree function} of $v$. 
\end{defi} 
 
%\par 
% In the case where $\os{\circ}{S}$ is a formal scheme, 
% we can obtain the analogous results to the results in this section. 

\section{Preweight filtrations}\label{sec:pwtr} 
Let $Y$ be a fine log (formal) scheme 
over a fine log (formal) scheme $U$ 
with structural morphism $g\col Y\lo U$. 
Let $(N_Y,\al\vert_{N_Y})$ be a sub log structure of the log structure of $(M_Y,\al)$. 
Set $Y_{N_Y}:=(\os{\circ}{Y},(N_Y,\al\vert_{N_Y}))$. 
%Assume that 
%$${\rm Im}(g^*(M_U)\lo M_Y)\subset N_Y.$$ 
We define the {\it pre-weight filtration}
$P^{M_Y\setminus N_Y}$ on the sheaf ${\Om}^i_{Y/U}$ $(i\in {\mab N})$ 
of log differential forms on $\os{\circ}{Y}_{\rm zar}$ with respect to 
$M_Y\setminus N_Y$ as follows: 
\begin{equation*} 
P^{M_Y\setminus N_Y}_k{\Om}^i_{Y/U} =
\begin{cases} 
0 & (k<0), \\
{\rm Im}({\Om}^k_{Y/U}{\otimes}_{{\mathcal O}_Y}
\Om^{i-k}_{Y_{N_Y}/U}
\lo {\Om}^i_{Y/U}) & (0\leq k\leq i), \\
{\Om}^i_{Y/U} & (k > i).
\end{cases}
\tag{4.0.1}\label{eqn:pkdefpw}
\end{equation*}  
Let $g\col Y\lo Z$ be a morphism of fine log (formal) schemes over $U$. 
Let $N_Y$ and $N_Z$ be sub log structures of $Y$ and $Z$, respectively. 
Assume that $g$ induces a morphism 
$Y_{N_Y}:=(\os{\circ}{Y},N_Y)\lo Z_{N_Z}:=(\os{\circ}{Z},N_Z)$ over $U$. 
For a flat ${\mathcal O}_Y$-module ${\mathcal E}$ 
and a flat ${\mathcal O}_Z$-module ${\mathcal F}$ with a morphism 
$h \col {\mathcal F}\lo g_*({\mathcal E})$ of ${\mathcal O}_Z$-modules, 
we have the following morphism of filtered complexes: 
\begin{equation*} 
h \col ({\mathcal F}\otimes_{{\mathcal O}_Z}\Om^{\bul}_{Z/\os{\circ}{U}},P^{M_Z\setminus N_Z})
\lo 
g_*(({\mathcal E}\otimes_{{\mathcal O}_Y}{\Om}^{\bul}_{Y/\os{\circ}{U}},P^{M_Y\setminus N_Y})).  
\tag{4.0.2}\label{eqn:lyytp}
\end{equation*}  
%For a morphism $T'\lo T$ of fine log (formal) schemes, 
%set $Y_{T'}:=Y\times_TT'$ and 
%let $q\col Y_{T'}\lo Y$ be the first projection. 
%For a flat ${\mathcal O}_{Y_{T'}}$-module ${\mathcal F}$, 
%let $P$ be the induced filtration on 
%${\mathcal F}\otimes_{{\mathcal O}_{Y_{T'}}}\Om^i_{Y_{T'}/\os{\circ}{T}{}'}$ 
%by the filtration $P$ on $\Om^i_{Y_{T'}/\os{\circ}{T}{}'}$: 
%\begin{align*} 
%P_k({\mathcal F}\otimes_{{\mathcal O}_{Y_{T'}}}\Om^i_{Y_{T'}/\os{\circ}{T}{}'})
%:={\mathcal F}\otimes_{{\mathcal O}_{Y_{T'}}}P_k\Om^i_{Y_{T'}/\os{\circ}{T}{}'}. 
%\tag{4.1.0.4}\label{eqn:lyep}
%\end{align*}
%We also set 
%\begin{equation*} 
%P_k{\Om}^i_{{\mathcal P}^{{\rm ex},(m)}/\os{\circ}{T}} =
%\begin{cases} 
%0 & (k<0), \\
%{\rm Im}(\Om^k_{{\mathcal P}^{{\rm ex},(m)}
%/\os{\circ}{T}}{\otimes}_{{\mathcal O}_{{\mathcal P}^{{\rm ex},(m)}}}
%\Om^{i-k}_{\os{\circ}{\mathcal P}{}^{{\rm ex},(m)}/\os{\circ}{T}}
%\lo {\Om}^i_{{\mathcal P}^{{\rm ex},(m)}/\os{\circ}{T}}) & 
%(0\leq k\leq i), \\
%{\Om}^i_{{\mathcal P}^{{\rm ex},(m)}/\os{\circ}{T}} & (k >i).  
%\end{cases}
%\tag{3.4.6}\label{eqn:pkdefpw}
%\end{equation*} 

%\par 
%The complex 
%$\Om^{\bul}_{{\mathcal P}^{{\rm ex},(m)}/\os{\circ}{T}}$ 
%has the following pre-weight filtration 

\par 
The following is a slight generalization of \cite[(1.3.4)]{nb}:

\begin{prop}[{\bf cf.~\cite[(1.3.4)]{nb}}]\label{prop:injpf}
Assume that $U$ has a PD-structure $({\mathcal J},\del)$. 
Let $Y\os{\sus}{\lo} {\mathcal Q}$ be an immersion into a log smooth scheme over 
$(U,{\mathcal J},\del)$. 
%Let ${\mathfrak g} \col {\mathcal Q}\lo U$ be the structural morphism. 
Let ${\mathfrak E}$ be the log PD-envelope of the immersion 
$Y\os{\sus}{\lo} {\mathcal Q}$ over $(U,{\mathcal J},\del)$. 
Let $M_{{\mathcal Q}^{\rm ex}}$ be the log structure of ${\mathcal Q}^{\rm ex}$ 
and let $N_{{\mathcal Q}^{\rm ex}}$ be a sub log structure of $M_{{\mathcal Q}^{\rm ex}}$. 
%such that 
%$${\rm Im}({\mathfrak g}^*(M_U)\lo M_{\mathcal Q})\subset N_Y.$$ 
Then the natural morphism 
\begin{equation*} 
{\mathcal O}_{\mathfrak E}
\otimes_{{\mathcal O}_{{\mathcal Q}^{\rm ex}}}
P^{M_{{\mathcal Q}^{\rm ex}}\setminus N_{{\mathcal Q}^{\rm ex}}}_k
\Om^i_{{\mathcal Q}^{\rm ex}/U}
\lo 
{\mathcal O}_{\mathfrak E}
\otimes_{{\mathcal O}_{{\mathcal Q}^{\rm ex}}}{\Om}^i_{{\mathcal Q}^{\rm ex}/U} 
\quad (i,k\in {\mab Z})
\tag{4.1.1}\label{eqn:yxnpd}
\end{equation*}
is injective. 
\end{prop}
\begin{proof} 
(The proof of this propositon is the same 
as that of \cite[(2.2.17) (1)]{nh2}.) 
The question is local on $Y$; 
we may assume the existence of 
the commutative diagram (\ref{eqn:0txda}) for $T:=U$. 
%Let $P^{\rm ex}$ be the inverse image of 
%$Q$ by the morphism $P^{\rm gp} \lo Q^{\rm gp}$. 
%Then the natural morphism $P^{\rm ex}\lo Q$ is surjective. 
%We have a fine log formal ${\mab Z}_p$-scheme 
%${\cal Q}^{{\rm prex}}:=
%{\cal Q}\times_{{\rm Spf}^{\log}({\mab Z}_p\{P\})}
%{\rm Spf}^{\log}({\mab Z}_p\{P^{\rm ex}\})$ 
%over ${\cal Q}$ with a morphism $Y\lo {\cal Q}^{{\rm prex}}$. 
Let $x_1,\ldots, x_c$ be the coordinates of 
${\mab A}^c_T$ in (\ref{eqn:0txda}). 
Set ${\mathcal K}:=(x_1,\ldots, x_c){\mathcal O}_{{\mathcal Q}^{\rm ex}}$ 
and ${\mathcal Q}':=\ul{\rm Spec}^{\log}_{{\mathcal Q}^{\rm ex}}
({\mathcal O}_{{\mathcal Q}^{\rm ex}}/{\mathcal K})$. 
Then ${\mathcal Q}'$ is a log smooth lift of $Y$ over $T$.  
Let $M_{{\mathcal Q}'}$ and $N_{{\mathcal Q}'}$ be the inverse images of 
$M_{{\mathcal Q}^{\rm ex}}$ and $N_{{\mathcal Q}^{\rm ex}}$, respectively. 
Let $e$ be a positive integer. 
Since ${\mathcal Q}'$ is log smooth over $T$, 
there exists a section of the surjection 
${\mathcal O}_{{\mathcal Q}^{\rm ex}}/{\mathcal K}^e \lo 
{\mathcal O}_{{\mathcal Q}^{\rm ex}}/{\mathcal K}={\mathcal O}_{{\mathcal Q}'}$. 
Set ${\mathcal K}_0:=(x_1,\ldots, x_c)$ in 
${\mathcal O}_{{\mathcal Q}'}[x_1,\ldots, x_c]$.  
Then, as in \cite[3.32 Proposition]{bob}, we have a morphism 
${\mathcal O}_{{\mathcal Q}'}[x_1,\ldots, x_c] \lo  {\mathcal O}_{{\mathcal Q}^{\rm ex}}/{\mathcal K}^e$ 
of sheaves of commutative rings on ${\mathcal Q}^{\rm ex}$ 
such that the induced morphism 
${\mathcal O}_{{\mathcal Q}'}[x_1,\ldots, x_c]/{\mathcal K}_0^e
\lo {\mathcal O}_{{\mathcal Q}^{\rm ex}}/{\mathcal K}^e$ is an isomorphism. 
Set ${\mathcal Q}'':={\mab A}^c_T$.  
Because $p$ is locally nilpotent on $T$, we may assume that 
there exists a positive integer $e$ such that 
${\mathcal K}^e{\mathcal O}_{\mathfrak E}=0$. 
By \cite[3.32 Proposition]{bob}, 
${\mathcal O}_{\mathfrak E}$ is isomorphic to 
the PD-polynomial algebra 
${\mathcal O}_{{\mathcal Q}'}\langle x_1,\ldots, x_c\rangle$. 
Hence we have the following isomorphisms  
\begin{equation*} 
{\mathcal O}_{\mathfrak E}
\otimes_{{\mathcal O}_{{\mathcal Q}^{\rm ex}}}
\Om^i_{{\mathcal Q}^{\rm ex}/\os{\circ}{T}}\os{\sim}{\lo} 
\bigoplus_{i'+i''=i}
\Om^{i'}_{{\mathcal Q}'/\os{\circ}{T}}
\otimes_{{\mathcal O}_T}{\mathcal O}_T\langle x_1,\ldots, x_c\rangle 
\otimes_{{\mathcal O}_{{\mathcal Q}''}}
\Om^{i''}_{\os{\circ}{\mathcal Q}{}''/\os{\circ}{T}}  
\tag{4.1.2}\label{eqn:dopx} 
\end{equation*}
and 
\begin{equation*} 
{\mathcal O}_{\mathfrak E}
\otimes_{{\mathcal O}_{{\mathcal Q}^{\rm ex}}}
P^{M_{{\mathcal Q}^{\rm ex}}
\setminus N_{{\mathcal Q}^{\rm ex}}}_k{\Om}^i_{{\mathcal Q}^{\rm ex}/\os{\circ}{T}}
\os{\sim}{\lo} 
\bigoplus_{i'+i''=i}
P^{M_{{\mathcal Q}'}\setminus N_{{\mathcal Q}'}}_k{\Om}^{i'}_{{\mathcal Q}'/\os{\circ}{T}}
\otimes_{{\mathcal O}_T}
{\mathcal O}_T\langle x_1,\ldots, x_c\rangle 
\otimes_{{\mathcal O}_{{\mathcal Q}''}}
\Om^{i''}_{\os{\circ}{\mathcal Q}{}''/\os{\circ}{T}}. 
\tag{4.1.3}\label{eqn:dpopx} 
\end{equation*} 
Since the complex 
${\mathcal O}_T\langle x_1,\ldots, x_c\rangle 
\otimes_{{\mathcal O}_{{\mathcal Q}''}}
\Om^{\bul}_{\os{\circ}{\mathcal Q}{}''/\os{\circ}{T}}$    
consists of free ${\mathcal O}_T$-modules, 
we obtain the desired injectivity. 
\end{proof} 

Set 
\begin{align*} 
& P^{M_{{\mathcal Q}^{\rm ex}}\setminus N_{{\mathcal Q}^{\rm ex}}}_k
({\mathcal O}_{\mathfrak E}
\otimes_{{\mathcal O}_{{\mathcal Q}^{\rm ex}}}
{\Om}^i_{{\mathcal Q}^{\rm ex}/U}):= \tag{4.1.4}\label{eqn:yenpd}
\\
& {\rm Im}
({\mathcal O}_{\mathfrak E}
\otimes_{{\mathcal O}_{{\mathcal Q}^{\rm ex}}}
P^{M_{{\mathcal Q}^{\rm ex}}\setminus N_{{\mathcal Q}^{\rm ex}}}_k{\Om}^i_{{\mathcal Q}^{\rm ex}/U}
\lo 
{\mathcal O}_{\mathfrak E}
\otimes_{{\mathcal O}_{{\mathcal Q}^{\rm ex}}}{\Om}^i_{{\mathcal Q}^{\rm ex}/U})  
\quad (i\in {\mab N},k\in {\mab Z}). 
\end{align*}

Let $S$, $((T,{\mathcal J},\del),z)$ and $T_0$ be as in the previous section. 
Let $\ol{\mathcal Q}$ be a log smooth integral scheme over 
$\ol{S(T)^{\nat}}$. 
Let $\ol{g} \col \ol{\mathcal Q}\lo \ol{S(T)^{\nat}}$ be the structural morphism. 
Set ${\mathcal Q}:=\ol{\mathcal Q}\times_{\ol{S(T)^{\nat}}}S(T)^{\nat}$ and 
${\mathcal I}_{\ol{\mathcal Q}}={\mathcal I}_{\ol{S(T)^{\nat}}}\otimes_{{\mathcal O}_{\ol{S(T)^{\nat}}}}{\mathcal O}_{\ol{\mathcal Q}}$. 
The following is an important proposition in this paper, 
which is quite hard to formulate correctly.  
This is a slight generalization of \cite[(1.3.12)]{nb}. 

\begin{prop}\label{prop:incol}
%Let the notations and the assumptions be as in {\rm (\ref{prop:ed0})}.   
%Then 
%The following hold$:$
%\par 
$(1)$ Set 
$\Om^{\bul}_{{\ol{\mathcal Q}/\os{\circ}{T}}}
(-\os{\circ}{\mathcal Q}):={\mathcal I}_{\ol{\mathcal Q}}
\otimes_{{\mathcal O}_{\ol{\mathcal Q}}}\Om^{\bul}_{\ol{\mathcal Q}/\os{\circ}{T}}$. 
$($Recall the notation ${\mathcal I}_{\ol{\mathcal Q}}$ defined in the beginning of \S{\rm \ref{sec:snclv}}$.)$ 
Then 
$\Om^{\bul}_{{\ol{\mathcal Q}}/\os{\circ}{T}}(-\os{\circ}{\mathcal Q})$ is 
a subcomplex of $\Om^{\bul}_{{\ol{\mathcal Q}}/\os{\circ}{T}}$. 
\par 
$(2)$ 
Let $L_{\ol{\mathcal Q}}$ and $N_{\ol{\mathcal Q}}$ be sub log structures of $M_{\ol{\mathcal Q}}$ 
such that $M_{\ol{\mathcal Q}}=L_{\ol{\mathcal Q}}\oplus_{{\mathcal O}^*_{\ol{\mathcal Q}}}N_{\ol{\mathcal Q}}$. 
Assume that, for any point 
$x\in \os{\circ}{\ol{\mathcal Q}}$, 
there exists a basis $\{\ol{m}_i\}_{i=1}^k$ 
of $(L_{\ol{\mathcal Q}}
/{\mathcal O}_{\ol{\mathcal Q}}^*)_x$ 
and a positive integer $e_i>0$ $(1\leq i\leq k)$
such that $\prod_{i=1}^k\ol{m}{}^{e_i}_i$ is the image of 
the generator 
of $(\ol{g}{}^*(M_{\ol{S(T)^{\nat}}}/{\mathcal O}^*_{\ol{S(T)^{\nat}}}))_{x}$ 
in $(L_{\ol{\mathcal Q}}/{\mathcal O}_{\ol{\mathcal Q}}^*)_x$. 
Then $\Om^{\bul}_{{\ol{\mathcal Q}/\os{\circ}{T}}}
(-\os{\circ}{\mathcal Q})$ 
is a subcomplex of 
$P^{M_{\ol{\mathcal Q}}\setminus N_{\ol{\mathcal Q}}}
_0\Om^{\bul}_{\ol{\mathcal Q}/\os{\circ}{T}}$.  
\par 
$(3)$ Let the assumption be as in $(2)$. 
Let $M_{\mathcal Q}$ and $N_{\mathcal Q}$ be the inverse images of 
$M_{\ol{\mathcal Q}}$ and $N_{\ol{\mathcal Q}}$ by the closed immersion 
${\mathcal Q}\os{\subset}{\lo} \ol{\mathcal Q}$, respectively. 
Then the following formula holds$:$  
\begin{equation*}
P^{M_{\mathcal Q}\setminus N_{\mathcal Q}}_k\Om^{\bul}_{{\mathcal Q}/\os{\circ}{T}}=
P^{M_{\ol{\mathcal Q}}\setminus N_{\ol{\mathcal Q}}}_k\Om^{\bul}_{\ol{\mathcal Q}/\os{\circ}{T}}
/\Om^{\bul}_{\ol{\mathcal Q}/\os{\circ}{T}}
(-\os{\circ}{\mathcal Q}) \quad (k\in {\mab N}).
\tag{4.2.1}\label{eqn:pops}
\end{equation*}   
\par 
$(4)$ Let the assumptions be as in $(2)$ without assuming that 
$\prod_{i=1}^k\ol{m}{}^{e_i}_i$ is the image of 
the generator 
of $(\ol{g}{}^*(M_{\ol{S(T)^{\nat}}}/{\mathcal O}^*_{\ol{S(T)^{\nat}}}))_{x}$ 
in $(M_{\ol{\mathcal Q}}
/{\mathcal O}_{\ol{\mathcal Q}}^*)_x$. 
Let $m_i$ be a lift of $\ol{m}_i$ to $M_{\ol{\mathcal Q},x}$ $(1\leq i\leq k)$. 
Let $\al \col M_{\ol{\mathcal Q}}\lo {\mathcal O}_{\ol{\mathcal Q}}$ be the structural morphism. 
Assume furthermore that 
$\Om^1_{(\os{\circ}{\ol{\mathcal Q}},N_{\ol{\mathcal Q}})/\os{\circ}{T}}$ is locally free 
and that $\{d\al(m_i)\}_{i=1}^k$ is a part of a basis of 
$\Om^1_{(\os{\circ}{\ol{\mathcal Q}},N_{\ol{\mathcal Q}})/\os{\circ}{T}}$. 
Assume also that 
$(\os{\circ}{\ol{\mathcal Q}},L_{\ol{\mathcal Q}})$ is log smooth over $\os{\circ}{T}$.  
Then the following natural morphism 
\begin{equation*} 
\Om^{\bul}_{(\os{\circ}{\ol{\mathcal Q}},N_{\ol{\mathcal Q}})/\os{\circ}{T}}
\lo \Om^{\bul}_{\ol{\mathcal Q}/\os{\circ}{T}}    
\tag{4.2.2}\label{eqn:yxppdppd}
\end{equation*} 
is injective. 
\par 
$(5)$ Let the assumptions be as in $(2)$ and $(4)$.  
Then the injective morphism 
$\Om^{\bul}_{\ol{\mathcal Q}/\os{\circ}{T}}
(-\os{\circ}{\mathcal Q}) \lo 
\Om^{\bul}_{\ol{\mathcal Q}/\os{\circ}{T}}$ 
factors through the following injective morphism$:$
%Then there exists the following natural injective morphism$:$
\begin{equation*}  
\Om^{\bul}_{\ol{\mathcal Q}/\os{\circ}{T}}
(-\os{\circ}{\mathcal Q}) \lo 
\Om^{\bul}_{(\os{\circ}{\ol{\mathcal Q}},N_{\ol{\mathcal Q}})/\os{\circ}{T}}.    
\tag{4.2.3}\label{eqn:yxpplnpd}
\end{equation*} 
\par 
%$(6)$ $($for our memory$)$  
%The following natural morphism 
%\begin{equation*} 
%{\mathcal O}_{{\mathfrak E}_T}
%\otimes_{{\mathcal O}_{{\mathcal Q}^{\rm ex}}}
%P_k\Om^{\bul}_{{\mathcal Q}{}^{\rm ex}/\os{\circ}{S}}=
%{\mathcal O}_{{\mathfrak E}_T}\otimes_{{\mathcal O}_{\ol{\mathcal Q}{}^{\rm ex}}}
%P_k\Om^{\bul}_{\ol{\cal Q}{}^{\rm ex}/\os{\circ}{S}}\lo 
%{\cal O}_{{\mathfrak E}_T}
%\otimes_{{\cal O}_{\ol{\cal Q}{}^{\rm ex}}}
%\Om^{\bul}_{\ol{\cal Q}{}^{\rm ex}/\os{\circ}{S}}=
%{\cal O}_{{\mathfrak E}_T}
%\otimes_{{\cal O}_{{\cal Q}{}^{\rm ex}}}
%\Om^{\bul}_{{\cal Q}{}^{\rm ex}/\os{\circ}{S}}.   
%\tag{4.9.4}\label{eqn:ydpd}
%\end{equation*} 
%is injective. 
\end{prop}
\begin{proof} 
(We give the proof for the completeness of this paper.) 
\par 
(1): Let $x$ be a point of $\os{\circ}{\ol{\mathcal Q}}$.  
Let $t$ be a local section of $M_{\ol{S(T)^{\nat}},\ol{g}(x)}$ which 
gives a generator of $(M_{\ol{S(T)^{\nat}}}/{\mathcal O}_{\ol{S(T)^{\nat}}}^*)_{\ol{g}(x)}$. 
By abuse of notation, we denote the image of $t$ 
in ${\mathcal O}_{\ol{S(T)^{\nat}},\ol{g}(x)}$ by $t$.  
The morphism 
$\os{\circ}{\ol{\mathcal Q}}\lo  \os{\circ}{\ol{S(T)^{\nat}}}$ is flat. 
Consequently the natural morphism 
${\mathcal I}_{\ol{\mathcal Q}}\lo {\mathcal O}_{\ol{\mathcal Q}}$ is injective.  
Hence the natural morphism 
$\Om^i_{\ol{\mathcal Q}/\os{\circ}{T}}(-\os{\circ}{\mathcal Q}) 
\lo \Om^i_{\ol{\mathcal Q}/\os{\circ}{T}}$ 
$(i\in {\mab N})$ is injective. 
For a local section 
$\om \in {\Om}^{\bul}_{\ol{\mathcal Q}/\os{\circ}{T}}$ 
around $x$, 
$d(t \om)=t d\log t\wedge \om+t d\om$. 
Hence 
$\Om^{\bul}_{\ol{\mathcal Q}/\os{\circ}{T}}
(-\os{\circ}{\mathcal Q})$ 
is a subcomplex of 
$\Om^{\bul}_{\ol{\mathcal Q}/\os{\circ}{T}}$.  
\par 
(2): 
%First note that $\Om^i_{\star,\ol{T}}
%\subset \Om^i_{\ol{\mathcal Q}{}^{\star}_{\ol{T}}/\os{\circ}{T}}$ by definition. 
The question is local.
Let $\al \col M_{\ol{\mathcal Q},x}\lo {\mathcal O}_{\ol{\mathcal Q},x}$ 
be the structural morphism. 
Let $m_i$ be a lift of $\ol{m}_i$ to 
$M_{\ol{\mathcal Q},x}$ as stated in (4). 
Consider a section 
$d\log m_{i_1}\wedge \cdots 
\wedge d\log m_{i_l}\wedge \om$ 
$(1\leq i_1< \cdots <i_{l}\leq k)$ 
with $\om\in 
P^{M_{\ol{\mathcal Q}}\setminus N_{\ol{\mathcal Q}}}_0\Om^j_{\ol{\mathcal Q}/\os{\circ}{T}}$ $(j\in {\mab N})$.  
By the assumption, we may assume that 
$\prod_{i=1}^k\al(m_i^{e_i})=t$. 
Then 
$$t d\log m_{i_1}\wedge \cdots 
\wedge d\log m_{i_l}\wedge \om
=d\al(m_{i_1})\wedge \cdots 
\wedge d\al(m_{i_l})\wedge \om'$$  
with $\om'\in 
P^{M_{\ol{\mathcal Q}}\setminus N_{\ol{\mathcal Q}}}_0\Om^j_{\ol{\mathcal Q}/\os{\circ}{T}}$. 
Hence 
$\Om^{\bul}_{{\ol{\mathcal Q}}/\os{\circ}{T}}
(-\os{\circ}{\mathcal Q})$
is a subcomplex of 
$P^{M_{\ol{\mathcal Q}}\setminus N_{\ol{\mathcal Q}}}_0\Om^{\bul}_{\ol{\mathcal Q}/\os{\circ}{T}}$.  
\par 
(3): By (2) the complex $\Om^{\bul}_{\ol{\mathcal Q}/\os{\circ}{T}}(-\os{\circ}{\mathcal Q})$
is a subcomplex of $P^{M_{\ol{\mathcal Q}}\setminus 
N_{\ol{\mathcal Q}}}_k\Om^{\bul}_{\ol{\mathcal Q}/\os{\circ}{T}}$ $(k\in {\mab N})$.  
By the definition of 
$P^{M_{\mathcal Q}\setminus N_{\mathcal Q}}_k{\Om}^{\bul}_{{\mathcal Q}/\os{\circ}{T}}$, 
we have the following natural surjective morphism 
\begin{equation*} 
P^{M_{\ol{\mathcal Q}}\setminus N_{\ol{\mathcal Q}}}_k{\Om}^{\bul}_{\ol{\mathcal Q}/\os{\circ}{T}}
/{\Om}^{\bul}_{\ol{\mathcal Q}/\os{\circ}{T}}(-\os{\circ}{\mathcal Q}) 
\lo 
P^{M_{\mathcal Q}\setminus N_{\mathcal Q}}_k{\Om}^{\bul}_{{\mathcal Q}/\os{\circ}{T}}.
\end{equation*}  
The following diagram shows that this morphism 
is injective: 
\begin{equation*} 
\begin{CD} 
P^{M_{\ol{\mathcal Q}}\setminus N_{\ol{\mathcal Q}}}_k{\Om}^{\bul}_{\ol{\mathcal Q}/\os{\circ}{T}}/
{\Om}^{\bul}_{\ol{\mathcal Q}/\os{\circ}{T}}(-\os{\circ}{\mathcal Q})  
@>>> 
P^{M_{\mathcal Q}\setminus N_{\mathcal Q}}_k{\Om}^{\bul}_{{\mathcal Q}/\os{\circ}{T}}\\ 
@V{\bigcap}VV  @VV{\bigcap}V \\ 
{\Om}^{\bul}_{\ol{\mathcal Q}/\os{\circ}{T}}/
{\Om}^{\bul}_{\ol{\mathcal Q}/\os{\circ}{T}}
(-\os{\circ}{\mathcal Q})  
@>{\sim}>> 
{\Om}^{\bul}_{{\mathcal Q}/\os{\circ}{T}}. 
\end{CD} 
\tag{4.2.4}\label{eqn:icor}
\end{equation*} 
%Similarly we can show the injectivity of 
%the natural morphism 
%${\mathcal O}_{\ol{\mathfrak E}}
%\otimes_{{\mathcal O}_{\ol{\mathcal Q}{}^{\rm ex}}}
%{\Om}^{\bul}_{\ol{\mathcal Q}{}^{\rm ex}/\os{\circ}{S}}
%(-\os{\circ}{\mathcal Q}{}^{\rm ex})
%\lo {\mathcal O}_{\ol{\mathfrak E}}
%\otimes_{{\mathcal O}_{\ol{\mathcal Q}{}^{\rm ex}}}
%\Om^{\bul}_{\ol{\mathcal Q}{}^{\rm ex}{}'/\os{\circ}{S}}$. 
%Consequently the morphism (\ref{eqn:yxpplnpd}) is injective. 
Hence we obtain (\ref{eqn:pops}). 
\par 
(4):  
Take a local chart 
$(\{1\} \lo {\mathcal O}^*_T,P \lo L_{\ol{\mathcal Q}},
\{1\}\os{\subset}{\lo} P)$ of the morphism 
$(\os{\circ}{\ol{\mathcal Q}},L_{\ol{\mathcal Q}})\lo  \os{\circ}{T}$ 
on a neighborhood of $x$.  
%such that $\rho$ is injective, 
%such that ${\rm Coker}(\rho^{\rm gp})$ is torsion free 
%and that the natural homomorphism 
%${\mathcal O}_{Y_{T_0},y} \otimes_{{\mab Z}} 
%(P^{{\rm gp}}/Q^{\rm gp}) \lo \Om^1_{Y_{T_0}/T_0,y}$ is an isomorphism. 
Because $(\os{\circ}{\ol{\mathcal Q}},L_{\ol{\mathcal Q}})$
is (formally) log smooth over $\os{\circ}{T}$, 
we can take the $P$ such that  
${\mathcal O}_{\ol{\mathcal Q}}$ 
is \'{e}tale over ${\mathcal O}_{T}[P]$, 
in particular, flat over ${\mathcal O}_{T}[P]$. 
Since $P$ is integral, any element $a\in P$ defines an injective multiplication 
$$a\cdot \col  {\mathcal O}_{T}[P] \os{\sus}{\lo} {\mathcal O}_{T}[P].$$
Hence the morphism 
\begin{equation*} 
a\cdot \col {\mathcal O}_{\ol{\mathcal Q}} 
\lo {\mathcal O}_{\ol{\mathcal Q}}
\tag{4.2.5}\label{eqn:injm}
\end{equation*} 
is injective.  
%Hence any nonzero local section of the sheaf 
%${\cal O}_{\ol{\cal Q}{}^}{}'}$ is not a zero divisor. 
We may assume that $m_i$ is the image of an element of $P$ 
and that 
$$\Om^1_{(\os{\circ}{\ol{\mathcal Q}},N_{\ol{\mathcal Q}})/\os{\circ}{T}}
=\bigoplus_{j=1}^k{\mathcal O}_{\ol{\mathcal Q}}
d\al(m_j)
\oplus \bigoplus_{j=1}^l {\mathcal O}_{\ol{\mathcal Q}}\eta_j 
\quad (l\in {\mab N}),$$ 
$$\Om^1_{\ol{\mathcal Q}/\os{\circ}{T}}
=\bigoplus_{j=1}^k{\mathcal O}_{\ol{\mathcal Q}}
d\log m_j
\oplus \bigoplus_{j=1}^l {\mathcal O}_{\ol{\mathcal Q}}\eta_j 
\quad (l\in {\mab N})$$ 
with $\eta_j \in \Om^1_{(\os{\circ}{\ol{\mathcal Q}},N_{\ol{\mathcal Q}})/\os{\circ}{T}}$.  
Let $i$ be a nonnegative integer and consider a local section 
$\om \in \Om^i_{(\os{\circ}{\ol{\mathcal Q}},N_{\ol{\mathcal Q}})/\os{\circ}{T}}$. 
Set $\om_j:=d\log m_j$ $(0\leq j\leq k)$ and 
$\om_j:=\eta_{j-k}$ $(k+1\leq j\leq k+l)$.  
Then we have the following expression 
$$\om  =\sum_{j_1< \cdots <j_i}b_{j_1\cdots j_i}
d\al(m_{j_1}) \wedge \cdots \wedge d\al(m_{j_{n(j_1\cdots j_i)}}) 
\wedge \om_{j_{n(j_1\cdots j_i)+1}}
\wedge \cdots \wedge \om_{j_i}$$ for some $b_{j_1\cdots j_i}\in {\mathcal O}_{\ol{\mathcal Q}}$,  
%in $\Om^i_{(\os{\circ}{\ol{\mathcal Q}},N_{\ol{\mathcal Q}})/\os{\circ}{T}}$ 
%by the assumption about the locally freeness of  
%$\Om^1_{(\os{\circ}{\ol{\mathcal Q}},N_{\ol{\mathcal Q}})/\os{\circ}{T}}$.  
where $n(j_1,\ldots,j_i)$ is an integer 
such that $j_n \leq k$ and $j_{n+1} >k$ for $j_1< \cdots <j_i$. 
The image ${\rm im}(\om)$ of $\om$ in 
$\Om^i_{\ol{\mathcal Q}/\os{\circ}{T}}$ 
is the following form 
$${\rm im}(\om) =
\sum_{j_1< \cdots <j_i}b_{j_1\cdots j_i}\al(m_{j_1}) \cdots \al(m_{j_{n(j_1,\ldots,j_i)}})
\om_{j_1} \wedge \cdots \wedge \om_{j_i}.$$ 
%for some $b_{j_1\cdots j_i}\in {\mathcal O}_{\ol{\mathcal Q}}$. 
Assume that ${\rm im}(\om)=0$.  
%$\om$ in $\Om^i_{\ol{\mathcal Q}/\os{\circ}{T}}$ is zero. 
Then 
$b_{j_1\cdots j_i}\al(m_{j_1}) \cdots \al(m_{j_n})=0$. 
By the injectivity of the morphism (\ref{eqn:injm}),  
we see that $b_{j_1\cdots j_i}=0$. Hence $\om=0$ 
and 
we have shown that the natural morphism 
$\Om^{\bul}_{(\os{\circ}{\ol{\mathcal Q}},N_{\ol{\mathcal Q}})/\os{\circ}{T}}
\lo \Om^{\bul}_{\ol{\mathcal Q}/\os{\circ}{T}}$ 
is injective. 
%The rest of the proof the injectivity of the morphism 
%(\ref{eqn:yxppdppd}) is the same as that of 
%the injectivity of the morphism (\ref{eqn:nik}).
\par 
(5): By (1) the morphism 
${\Om}^{\bul}_{\ol{\mathcal Q}/\os{\circ}{T}}
(-\os{\circ}{\mathcal Q}) \lo \Om^{\bul}_{\ol{\mathcal Q}/\os{\circ}{T}}$
is injective. 
%Because  
%${\Om}^{\bul}_{\ol{\mathcal Q}{}^{\star}_{\ol{T}}/\os{\circ}{T}}
%(-\os{\circ}{\mathcal Q}{}^{\star})$ is a subcomplex of 
%${\Om}^{\bul}_{\ol{\mathcal Q}{}^{\star}_{\ol{T}}/\os{\circ}{T}}$, 
%it suffices to prove that (\ref{eqn:yxppdppd}) is injective. 
By (2) 
the morphism 
${\Om}^{\bul}_{\ol{\mathcal Q}/\os{\circ}{T}}(-\os{\circ}{\mathcal Q}) \lo
\Om^{\bul}_{\ol{\mathcal Q}/\os{\circ}{T}}$ 
factors through the morphism 
${\Om}^{\bul}_{\ol{\mathcal Q}/\os{\circ}{T}}(-\os{\circ}{\mathcal Q}) \lo
{\rm Im}
(\Om^{\bul}_{(\os{\circ}{\ol{\mathcal Q}},N_{\ol{\mathcal Q}})/\os{\circ}{T}}   
\lo 
\Om^{\bul}_{\ol{\mathcal Q}/\os{\circ}{T}})$.  
The target of the last morphism is isomorphic to 
$\Om^{\bul}_{(\os{\circ}{\ol{\mathcal Q}},N_{\ol{\mathcal Q}})/\os{\circ}{T}}$ by (4). 
\end{proof}

\par 
In the following we assume that $(\os{\circ}{T},{\mathcal J},\del)$ is a $p$-adic formal PD-scheme. 
Let $(X_{\os{\circ}{T}_0},D_{\os{\circ}{T}_0})\os{\sus}{\lo} ({\mathcal X},{\mathcal D})$ 
be an exact immersion into a formal SNCL scheme over $S(T)^{\nat}$ 
with a relative SNCD on ${\mathcal X}/S(T)^{\nat}$. 
%such that the immersion 
%$\os{\circ}{X}_{T_0}\os{\sus}{\lo} \os{\circ}{\mathcal X}$ is an isomorphism of 
%topological spaces. 
Let ${\mathfrak D}$ be the log PD-envelope of this immersion over 
$(S(T)^{\nat},{\mathcal J},\del)$. 
Care must be taken when defining the following filtrations 
on $\Om^i_{({\mathcal X},{\mathcal D})/\os{\circ}{T}}$
because errors can easily occur  
(see (\ref{rema:estzmis1}) and (\ref{rema:estzmis2}) below for this). 
%, \cite[I (3.1), (3.2)]{ezth}). 
\par 
Let $P$ be a filtration on $\Om^i_{({\mathcal X},{\mathcal D})/\os{\circ}{T}}$ 
defined by the following: 
\begin{equation*} 
P_k\Om^i_{({\mathcal X},{\mathcal D})/\os{\circ}{T}} 
=
\begin{cases} 
0 & (k<0), \\
{\rm Im}(\Om^k_{({\mathcal X},{\mathcal D})/\os{\circ}{T}}
{\otimes}_{{\mathcal O}_X} 
\Om^{i-k}_{\os{\circ}{\mathcal X}/\os{\circ}{T}}
\lo 
\Om^i_{({\mathcal X},{\mathcal D})/\os{\circ}{T}})
& (0\leq k\leq i), \\
\Om^i_{({\mathcal X},{\mathcal D})/\os{\circ}{T}} & (k > i).
\end{cases}
\tag{4.2.6}\label{eqn:pomw}
\end{equation*} 
Set $(\os{\circ}{\mathcal X},\os{\circ}{\mathcal D}):=(\os{\circ}{\mathcal X},M({\mathcal D}))$. 
Let $P^{\mathcal X}$ be a filtration on $\Om^i_{({\mathcal X},{\mathcal D})/\os{\circ}{T}}$ 
defined by the following: 
\begin{equation*} 
P^{\mathcal X}_k\Om^i_{({\mathcal X},{\mathcal D})/\os{\circ}{T}} =
\begin{cases} 
0 & (k<0), \\
{\rm Im}(\Om^k_{{\mathcal X}/\os{\circ}{T}}{\otimes}_{{\mathcal O}_X}
\Om^{i-k}_{\os{\circ}{\mathcal X}/\os{\circ}{T}} 
\lo 
\Om^i_{({\mathcal X},{\mathcal D})/\os{\circ}{T}})
& (0\leq k\leq i), \\
{\rm Im}(\Om^i_{{\mathcal X}/\os{\circ}{T}}
\lo 
\Om^i_{({\mathcal X},{\mathcal D})/\os{\circ}{T}}) & (k > i).
\end{cases}
\tag{4.2.7}\label{eqn:pkdefw}
\end{equation*} 
Let $P^{{\mathcal D}/{\mathcal X}}$ be a filtration on 
$\Om^i_{({\mathcal X},{\mathcal D})/\os{\circ}{T}}$  
defined by the following: 
\begin{equation*} 
P^{{\mathcal D}/{\mathcal X}}_k
\Om^i_{({\mathcal X},{\mathcal D})/\os{\circ}{T}}=
\begin{cases} 
0 & (k<0), \\
{\rm Im}(\Om^k_{({\mathcal X},{\mathcal D})/\os{\circ}{T}}
{\otimes}_{{\mathcal O}_X}
\Om^{i-k}_{{\mathcal X}/\os{\circ}{T}}
\lo 
\Om^i_{({\mathcal X},{\mathcal D})/\os{\circ}{T}})
& (0\leq k\leq i), \\
\Om^i_{({\mathcal X},{\mathcal D})/\os{\circ}{T}} & (k > i).
\end{cases}
\tag{4.2.8}\label{eqn:pkdw}
\end{equation*} 
Let $P^{{\mathcal X}/{\mathcal D}}$ be a filtration on 
$\Om^i_{({\mathcal X},{\mathcal D})/\os{\circ}{T}}$  
defined by the following: 
\begin{equation*} 
P^{{\mathcal X}/{\mathcal D}}_k
\Om^i_{({\mathcal X},{\mathcal D})/\os{\circ}{T}}=
\begin{cases} 
0 & (k<0), \\
{\rm Im}(\Om^k_{({\mathcal X},{\mathcal D})/\os{\circ}{T}}
{\otimes}_{{\mathcal O}_X}
\Om^{i-k}_{(\os{\circ}{\mathcal X},\os{\circ}{\mathcal D})/\os{\circ}{T}}
\lo 
\Om^i_{({\mathcal X},{\mathcal D})/\os{\circ}{T}})
& (0\leq k\leq i), \\
\Om^i_{({\mathcal X},{\mathcal D})/\os{\circ}{T}} & (k > i).
\end{cases}
\tag{4.2.9}\label{eqn:ptxkdw}
\end{equation*} 
%Here note that, in (\ref{eqn:pkdw}), we consider sheaves of differential forms over 
%$S(T)^{\nat}$ not over $\os{\circ}{T}$, which is unusual. 
We use the same notations $P$, $P^{\mathcal X}$, $P^{{\mathcal D}/{\mathcal X}}$  and $P^{{\mathcal X}/{\mathcal D}}$
for the induced filtrations   
on ${\mathcal O}_{\mathfrak D}\otimes_{{\mathcal O}_{\mathcal X}}
\Om^i_{({\mathcal X},{\mathcal D})/\os{\circ}{T}}$, 
${\mathcal O}_{\mathfrak D}\otimes_{{\mathcal O}_{\mathcal X}}
\Om^i_{({\mathcal X},{\mathcal D})/\os{\circ}{T}}$
and ${\mathcal O}_{\mathfrak D}\otimes_{{\mathcal O}_{\mathcal X}}
\Om^i_{({\mathcal X},{\mathcal D})/\os{\circ}{T}}$
by the filtrations $P$, $P^{\mathcal X}$, $P^{{\mathcal D}/{\mathcal X}}$ and 
$P^{{\mathcal X}/{\mathcal D}}$ on 
$\Om^i_{({\mathcal X},{\mathcal D})/\os{\circ}{T}}$, respectively.

Let 
$$a^{(l,m)} \col \os{\circ}{X}{}^{(l)}_{T_0}\cap \os{\circ}{D}{}^{(m)}_{T_0} \lo \os{\circ}{X}_{T_0}, 
\quad b^{(l,m)} \col \os{\circ}{\mathcal X}{}^{(l)}\cap \os{\circ}{\mathcal D}{}^{(m)}
\lo \os{\circ}{\mathcal X}$$  
and 
$$c^{(k)} \col D^{(k)}_{\os{\circ}{T}_0} \lo 
X_{\os{\circ}{T}_0}, \quad d^{(k)} \col {\mathcal D}^{(k)} \lo {\mathcal X}\quad (l, m,k\in {\mab N})$$
be natural morphisms. 
Set $a^{(l)}:=a^{(l,0)}$ and $b^{(l)}:=b^{(l,0)}$. 
Let $\{\os{\circ}{X}_{T_0\lam}\}_{\lam \in \Lam}$ and 
$\{\os{\circ}{D}_{T_0\mu}\}{\mu \in M}$ be smooth components of 
$\os{\circ}{X}_{T_0}$ and $\os{\circ}{D}_{T_0}$, respectively. 
For subsets $\ul{\lam}=\{\lam_0,\ldots,\lam_{l-1}\}$ of $\Lam$ 
and $\ul{\mu}=\{\mu_1,\ldots, \mu_m\}$ of $M$, 
set $\os{\circ}{X}_{T_0\ul{\lam}}
:=\os{\circ}{X}_{T_0{\{\lam_0, \ldots, \lam_{l-1}\}}}$ 
and $\os{\circ}{D}_{T_0\ul{\mu}}:=
\os{\circ}{D}_{T_0\{\mu_1, \ldots, \mu_{m}\}}$ as in \S\ref{sec:snclv} 
and \S\ref{sec:snrdlv}, and 
let 
$$a_{\ul{\lam}\ul{\mu}} \col \os{\circ}{X}_{T_0\ul{\lam}}\cap \os{\circ}{D}_{T_0\ul{\mu}} 
\os{\sus}{\lo} \os{\circ}{X}_{T_0}, \quad  
b_{\ul{\lam}\ul{\mu}} \col \os{\circ}{\mathcal X}_{\ul{\lam}}\cap \os{\circ}{\mathcal D}_{\ul{\mu}} 
\os{\sus}{\lo} \os{\circ}{\mathcal X}$$   
and 
$$c_{\ul{\mu}} \col \os{\circ}{D}_{T_0\ul{\mu}} \os{\sus}{\lo} \os{\circ}{X}_{T_0}, \quad  
d_{\ul{\mu}} \col \os{\circ}{\mathcal D}_{\ul{\mu}} 
\os{\sus}{\lo} \os{\circ}{\mathcal X}$$ 
be the natural closed immersions.

\par
In the following we define orientation sheaves 
for $\os{\circ}{X}_{T_0}$ and $\os{\circ}{D}_{T_0}$. 
\par 
Let $E$ be a finite set with cardinality $k\geq 0$.  
Set $\vp_E:=\bigwedge^k{\mab Z}^E$ if $k \geq 1$ 
and $\vp_E:={\mab Z}$ if $k =0$ (\cite[(3.1.4)]{dh2}).
\par 
Let $l$ be a positive integer and 
let $m$ be a nonnegative integer. 
%For simplicity of notation, set 
%$\ul{\lam}:=\{\lam_0, \ldots, \lam_{l-1}\}$ 
%and 
%$\ul{\mu}:=\{\mu_1, \ldots, \mu_{m}\}$. 
Let $P$ be a point of 
$\os{\circ}{X}{}^{(l)}_{T_0}\cap \os{\circ}{D}{}^{(m)}_{T_0}$.
Let $\os{\circ}{X}_{T_0\lam_0}, \ldots, \os{\circ}{X}_{T_0\lam_{l-1}}$ 
(resp.~$\os{\circ}{D}_{T_0\mu_1}, \ldots, \os{\circ}{D}_{T_0\mu_{m}}$) 
be distinct smooth components of $\os{\circ}{X}_{T_0}/\os{\circ}{T}_0$ 
(resp.~$\os{\circ}{D}_{T_0}/\os{\circ}{T}_0$) such that 
$\os{\circ}{X}_{T_0\ul{\lam}}\cap \os{\circ}{D}_{T_0\ul{\mu}}$ 
contains $P$.    
Then the set 
$E:=\{\os{\circ}{X}_{T_0\lam_0},\ldots,\os{\circ}{X}_{T_0\lam_{l-1}}\}$ 
gives an abelian sheaf 
\begin{equation*} 
\vp_{\ul{\lam}{\rm zar}}(\os{\circ}{X}_{T_0}/\os{\circ}{T}_0):=
\vp_{(\lam_0\cdots \lam_{l-1})
{\rm zar}}
(\os{\circ}{X}_{T_0}/\os{\circ}{T}_0):=
\bigwedge^l
{\mab Z}^{E}_{\os{\circ}{X}_{T_0\ul{\lam}}} 
\end{equation*} 
on a local neighborhood of $P$ in $\os{\circ}{X}_{T_0\ul{\lam}}$; 
the set $F:=\{\os{\circ}{D}_{T_0\mu_1},\ldots,\os{\circ}{D}_{T_0\mu_{m}}\}$ 
gives an abelian sheaf 
\begin{equation*} 
\vp_{\ul{\mu}{\rm zar}}(\os{\circ}{D}_{T_0}/\os{\circ}{T}_0):=
\vp_{(\mu_1 \cdots \mu_{m}){\rm zar}}
(\os{\circ}{D}_{T_0}/\os{\circ}{T}_0):=
\bigwedge^m
{\mab Z}^{F}_{\os{\circ}{D}_{T_0\ul{\mu}}}
\end{equation*} 
on a local neighborhood of $P$ 
in $\os{\circ}{D}_{T_0\ul{\mu}}$. 
Set 
\begin{equation*} 
\vp_{\ul{\lam}\ul{\mu}{\rm zar}}((\os{\circ}{X}_{T_0},\os{\circ}{D}_{T_0})/\os{\circ}{T}_0):=
\vp_{(\lam_0 \cdots \lam_{l-1}){\rm zar}}(\os{\circ}{X}_{T_0}/\os{\circ}{T}_0)
\vert_{\os{\circ}{X}_{\ul{\lam}}\cap {\os{\circ}{D}_{\ul{\mu}}}}
\otimes_{\mab Z}
\vp_{(\mu_1 \cdots \mu_{m}){\rm zar}}
(\os{\circ}{D}_{T_0}/\os{\circ}{T}_0)\vert_{\os{\circ}{X}_{\ul{\lam}}\cap {\os{\circ}{D}_{\ul{\mu}}}}.  
\end{equation*} 
We denote a local section 
of 
$\vp_{\ul{\lam}\ul{\mu}{\rm zar}}(\os{\circ}{X}_{T_0}/\os{\circ}{T}_0)$ 
by $n(\ul{\lam} \ul{\mu})=n(\lam_0\cdots \lam_{l-1})\otimes (\mu_1\cdots \mu_{m})$.  
$(n \in {\mab Z})$. 
The sheaf 
$\vp_{\ul{\lam}\ul{\mu}{\rm zar}}(\os{\circ}{X}_{T_0}/\os{\circ}{T}_0)$ 
is globalized 
on $\os{\circ}{X}{}^{(l)}_{T_0}\cap \os{\circ}{D}{}^{(m)}_{T_0}$;  
we denote this globalized abelian 
sheaf by the same symbol
$\vp_{\ul{\lam}\ul{\mu}{\rm zar}}(\os{\circ}{X}_{T_0}/\os{\circ}{T}_0)$. 
%Analogously the set 
%$F:=\{\os{\circ}{D}_{\mu_0},\ldots,\os{\circ}{D}_{\mu_{m-1}}\}$ 
%gives an abelian sheaf 
%\begin{equation*} 
%\vp_{\ul{\lam}\ul{\mu}{\rm zar}}
%(\os{\circ}{D}_{T_0}/\os{\circ}{T}_0):=
%\vp_{(\lam_0\cdots \lam_{l-1};\mu_0 \cdots \mu_{m-1}){\rm zar}}
%(\os{\circ}{D}_{T_0}/\os{\circ}{T}_0):=
%\bigwedge^k
%{\mab Z}^{F}_{\os{\circ}{X}_{\ul{\lam}}\cap \os{\circ}{D}_{\ul{\mu}}} 
%\end{equation*} 
%on a local neighborhood of $P$ in 
%$\os{\circ}{X}_{\ul{\lam}}\cap \os{\circ}{D}_{\ul{\mu}}$. 
%We denote a local section of 
%$\vp_{\ul{\lam}\ul{\mu}{\rm zar}}(\os{\circ}{D}_{T_0}/\os{\circ}{T}_0)$ 
%by $n(\ul{\mu})=n(\mu_0\cdots \mu_{m-1})$ 
%$(n \in {\mab Z})$. 
%The sheaf 
%$\vp_{\ul{\lam}\ul{\mu}{\rm zar}}
%(\os{\circ}{D}_{T_0}/\os{\circ}{T}_0)$ 
%is globalized 
%on $\os{\circ}{X}{}^{(l)}\cap \os{\circ}{D}{}^{(m)}$;  
%we denote this globalized abelian 
%sheaf by the same symbol
%$\vp_{\ul{\lam}\ul{\mu}{\rm zar}}
%(\os{\circ}{D}_{T_0}/\os{\circ}{T}_0)$. 
%Set 
%\begin{equation*}
%\begin{split} 
%\vp_{\ul{\lam}\ul{\mu}{\rm zar}}
%((\os{\circ}{X}_{T_0}+\os{\circ}{D}_{T_0})/\os{\circ}{T}_0)& 
%:=
%\vp_{(\lam_0\cdots \lam_{l-1};\mu_0\ \cdots \ \mu_{m-1})
%{\rm zar}}((\os{\circ}{X}_{T_0}+\os{\circ}{D}_{T_0})/\os{\circ}{T}_0) \\ 
%{} & := 
%\vp_{(\lam_0\cdots \lam_{l-1};\mu_0 \cdots \mu_{m-1}){\rm zar}}
%(\os{\circ}{X}_{T_0}/\os{\circ}{T}_0)\otimes_{\mab Z}\\
%{} & \phantom{:=\vp} 
%\vp_{(\lam_0\cdots \lam_{l-1};\mu_0 \cdots \mu_{m-1}){\rm zar}}
%(\os{\circ}{D}_{T_0}/\os{\circ}{T}_0). 
%\end{split} 
%\end{equation*}
Set 
\begin{align*}
\vp^{(l,m)}_{\rm zar}
((\os{\circ}{X}_{T_0},\os{\circ}{D}_{T_0})/\os{\circ}{T}_0) & := 
\bigoplus_{\{\ul{\lam}, \ul{\mu}\}}
\vp_{\ul{\lam}\ul{\mu}{\rm zar}}
((\os{\circ}{X}_{T_0},\os{\circ}{D}_{T_0})/\os{\circ}{T}_0)\\ 
{} &:= 
\bigoplus_{\{\lam_0,\ldots \lam_{l-1}, 
\mu_1, \ldots,  \mu_{m}\}}
\vp_{(\lam_0\cdots \lam_{l-1};
\mu_1 \cdots  \mu_{m})
{\rm zar}}((\os{\circ}{X}_{T_0},\os{\circ}{D}_{T_0})/\os{\circ}{T}_0). 
\end{align*} 
The sheaf 
$\vp_{\ul{\lam}\ul{\mu}{\rm zar}}
((\os{\circ}{X}_{T_0},\os{\circ}{D}_{T_0})/\os{\circ}{T}_0)$ 
extends to an abelian sheaf 
$\vp_{\ul{\lam}\ul{\mu}{\rm crys}}
((\os{\circ}{X}_{T_0},\os{\circ}{D}_{T_0})/\os{\circ}{T})$
%=\vp_{(\lam_0\cdots \lam_{l-1};\mu_0 \cdots  \mu_{m-1}){\rm crys}}
%((\os{\circ}{X}_{T_0},\os{\circ}{D}_{T_0})/\os{\circ}{T})$ 
in the crystalline topos 
$((\os{\circ}{X}{}^{(l)}_{T_0}\cap \os{\circ}{D}{}^{(m)}_{T_0})/\os{\circ}{T})_{\rm crys}$, 
and 
$\vp^{(l,m)}_{\rm zar}((\os{\circ}{X}_{T_0},\os{\circ}{D}_{T_0})/\os{\circ}{T}_0)$ extends 
to an abelian sheaf 
$\vp^{(l,m)}_{{\rm crys}}
((\os{\circ}{X}_{T_0},\os{\circ}{D}_{T_0})/\os{\circ}{T})$. 
%The sheaf 
%$\vp^{(l,m)}_{{\rm crys}}
%((\os{\circ}{X}_{T_0},\os{\circ}{D}_{T_0})/\os{\circ}{T})$ 
%also extends to an abelian sheaf 
%$\vp^{(l,m)}_{{\rm crys}}
%((\os{\circ}{X}_{T_0}+\os{\circ}{D}_{T_0})/\os{\circ}{T})$ 
%in the log crystalline topos 
%$((\os{\circ}{X}{}^{(l)}}_{T_0}\cap \os{\circ}{D}{}^{(m)}}_{T_0})/\os{\circ}{T})_{\rm crys}$, 
%where $X^{(l)}\cap \os{\circ}{D}{}^{(m)}$ 
%is the log scheme whose underlying scheme is 
%$\os{\circ}{X}{}^{(l)}\cap \os{\circ}{D}{}^{(m)}$ 
%and whose log structure is the pull-back of $X^{(l)}$. 
Set $\vp^{(l)}_{{\rm zar}}
(\os{\circ}{X}_{T_0}/\os{\circ}{T}_0)
:=\vp^{(l,0)}_{{\rm zar}}
((\os{\circ}{X}_{T_0},\os{\circ}{D}_{T_0})/\os{\circ}{T}_0)$ 
and $\vp^{(l)}_{{\rm crys}}(\os{\circ}{X}_{T_0}/\os{\circ}{T})
:=\vp^{(l,0)}_{{\rm crys}}
((\os{\circ}{X}_{T_0},\os{\circ}{D}_{T_0})/\os{\circ}{T})$. 
Set also 
$\vp^{(m)}_{{\rm zar}}
(\os{\circ}{D}_{T_0}/\os{\circ}{T}_0)
:=\vp^{(0,m)}_{{\rm zar}}
((\os{\circ}{X}_{T_0},\os{\circ}{D}_{T_0})/\os{\circ}{T}_0)$ 
and $\vp^{(m)}_{{\rm crys}}(\os{\circ}{D}_{T_0}/\os{\circ}{T})
:=\vp^{(0,m)}_{{\rm crys}}
((\os{\circ}{X}_{T_0},\os{\circ}{D}_{T_0})/\os{\circ}{T})$. 
%The sheaf $\vp^{(l)}_{{\rm crys}}
%(\os{\circ}{X}_{T_0}/\os{\circ}{T})$ 
%extends to an abelian sheaf 
%$\vp^{(l)}_{{\rm crys}}
%(\os{\circ}{X}/\os{\circ}{S};\os{\circ}{D})$ 
%in the log crystalline topos 
%$((\os{\circ}{X}{}^{(l)},\os{\circ}{D}\vert_{\os{\circ}{X}{}^{(l)})/\os{\circ}{S}})_{\rm crys}$.  
%Analogously we have abelian sheaves $\vp_{\mu_0 \cdots  \mu_{m-1}{\rm crys}}(D/\os{\circ}{S};X)$ 
%and we also have abelian sheaves 
%$\vp^{(m)}_{{\rm crys}}(D/\os{\circ}{S})$ 
%in the  crystalline topos 
%$(\os{\circ}{D}{}^{(m)}/\os{\circ}{S})_{\rm crys}$, 
%and  
%$\vp_{\ul{\mu}{\rm crys}}(D/\os{\circ}{S};X)=
%\vp_{\mu_0 \cdots  \mu_{m-1}{\rm crys}}(D/\os{\circ}{S};X)$ 
%and 
%$\vp^{(m)}_{{\rm crys}}(D/\os{\circ}{S};X)$ 
%in the log crystalline topos 
%$(D^{(m)}/\os{\circ}{S})_{\rm crys}$. 
%Ignoring $\os{\circ}{D}$ (=the case $\os{\circ}{D}=\emptyset$), 
%we obtain sheaves 
%$\vp^{(k)}_{\rm zar}(\os{\circ}{X}/\os{\circ}{S}_0)$ 
%and 
%$\vp^{(k)}_{\rm crys}(\os{\circ}{X}/\os{\circ}{S}_0)$  

\begin{defi}
We call  
\begin{align*} 
\vp^{(l,m)}_{\rm zar}
((\os{\circ}{X}_{T_0},\os{\circ}{D}_{T_0})/\os{\circ}{T}_0) 
\quad {\rm and} \quad 
\vp^{(l,m)}_{\rm crys}
((\os{\circ}{X}_{T_0},\os{\circ}{D}_{T_0})/\os{\circ}{T})
\end{align*} 
the {\it zariskian orientation sheaf} of 
$\os{\circ}{X}{}^{(l)}_{T_0}\cap \os{\circ}{D}{}^{(m)}_{T_0}/\os{\circ}{T}_0$ 
and the {\it crystalline orientation sheaf} of 
$\os{\circ}{X}{}^{(l)}_{T_0}\cap \os{\circ}{D}{}^{(m)}_{T_0}
/(\os{\circ}{T},{\mathcal J},\del)$, respectively. 
We also call $\vp^{(l)}_{{\rm zar}}(\os{\circ}{X}_{T_0}/\os{\circ}{T})$ and 
$\vp^{(l)}_{{\rm crys}}(\os{\circ}{X}_{T_0}/\os{\circ}{T})$ 
the {\it zariskian orientation sheaf} of $\os{\circ}{X}{}^{(l)}_{T_0}/\os{\circ}{T}_0$ 
and the {\it crystalline orientation sheaf} of  $\os{\circ}{X}{}^{(l)}_{T_0}/\os{\circ}{T}$.
We also call $\vp^{(m)}_{{\rm zar}}(\os{\circ}{D}_{T_0}/\os{\circ}{T})$ and 
$\vp^{(m)}_{{\rm crys}}(\os{\circ}{D}_{T_0}/\os{\circ}{T})$ 
the {\it zariskian orientation sheaf} of $\os{\circ}{D}{}^{(m)}_{T_0}/\os{\circ}{T}_0$ 
and the {\it crystalline orientation sheaf} of  $\os{\circ}{D}{}^{(m)}_{T_0}/\os{\circ}{T}$.
\end{defi}

\begin{lemm}\label{lemm:dlm} 
$(1)$ Let $\os{\circ}{\mathfrak D}{}^{(l,m)}$ 
be the PD-envelopes of 
$\os{\circ}{X}{}^{(l)}_{T_0}\cap \os{\circ}{D}{}^{(m)}_{T_0} \os{\sus}{\lo} 
\os{\circ}{\mathcal X}{}^{(l)}\cap \os{\circ}{\mathcal D}{}^{(m)}$ over 
$(\os{\circ}{T},{\mathcal J},\del)$.  
%and $(X_{\os{\circ}{T}_0},D_{\os{\circ}{T}_0}) \os{\sus}{\lo} ({\mathcal X},{\mathcal D})$ 
%over $(S(T)^{\nat},{\mathcal J},\del)$. 
%respectively.   
Then $\os{\circ}{\mathfrak D}{}^{(l,m)}=\os{\circ}{\mathfrak D}
\times_{\os{\circ}{\mathcal X}}(\os{\circ}{\mathcal X}{}^{(l)}\cap \os{\circ}{\mathcal D}{}^{(m)})$. 
\par 
$(2)$ Let ${\mathfrak D}({\mathcal D}^{(k)})$ be the log PD-envelope of 
$D^{(k)}_{\os{\circ}{T}_0} \os{\sus}{\lo} {\mathcal D}^{(k)}$ over $(S(T)^{\nat},{\mathcal J},\del)$.   
$($The log scheme ${\mathfrak D}({\mathcal D}^{(k)})$ is a log scheme whose underlying scheme is 
$\os{\circ}{\mathfrak D}{}^{(0,k)}$ and whose log structure is the inverse image 
of ${\mathcal D}^{(k)}$.$)$ 
Then ${\mathfrak D}({\mathcal D}^{(k)})={\mathfrak D}\times_{({\mathcal X},{\mathcal D})}{\mathcal D}^{(k)}$. 
\end{lemm} 
\begin{proof}(1), (2): 
Because we have natural morphisms 
$\os{\circ}{\mathfrak D}{}^{(l,m)}\lo \os{\circ}{\mathfrak D}$ and 
${\mathfrak D}({\mathcal D}^{(k)})\lo{\mathfrak D}$, the questions are local on $X_{\os{\circ}{T}_0}$. 
Hence we may assume that $M_{S(T)^{\nat}}$ is free of rank $1$ and that 
there exists the following cartesian diagram   
\begin{equation*}
\begin{CD} 
(X_{\os{\circ}{T}_0},D_{\os{\circ}{T}_0}) @>{\sus}>> ({\mathcal X},{\mathcal D}) \\
@VVV @VVV \\
{\mab A}_{S_{\os{\circ}{T}_0}}(a,b,d,e) 
@>{\sus}>> {\mab A}_{S(T)^{\nat}}(a,b,d',e) \\ 
@VVV @VVV \\
S_{\os{\circ}{T}_0}@>{\sus}>> S(T)^{\nat},  
\end{CD} 
\tag{4.4.1}\label{cd:pnwtp} 
\end{equation*}
where $d'\geq d$ and the vertical morphism 
$({\mathcal X},{\mathcal D}) \lo {\mab A}_{S(T)^{\nat}}(a,b,d',e)$ is solid and \'{e}tale and 
the immersion ${\mab A}_{S_{\os{\circ}{T}_0}}(a,b,d,e) 
\os{\sus}{\lo} {\mab A}_{S(T)^{\nat}}(a,b,d',e)$ is defined by 
the ``extra coordinates'' $x_{d+1}, \ldots, x_{d'}$ of ${\mab A}_{S(T)^{\nat}}(a,b,d',e)$. 
Now (1) and (2) follow from the local descriptions of 
${\mathfrak D}^{(l,m)}$, ${\mathfrak D}$ and 
${\mathfrak D}({\mathcal D}^{(k)})$. 
\end{proof}

The following is a generalization of \cite[(1.3.14)]{nb}: 

\begin{prop}\label{prop:grem}
Identify the points of $\os{\circ}{\mathfrak D}$ 
with those of $\os{\circ}{X}$. Let $z$ be a point of $\os{\circ}{\mathfrak D}$. 
%Identify also the images of the points of $\os{\circ}{X}$ 
%in $\os{\circ}{\mathcal X}$ with the points of $\os{\circ}{X}$. 
Then the following hold$:$
\par 
$(1)$ 
%Let $r$ be a nonnegative integer such that 
%$M_{(X,D),z}/{\mathcal O}_{X,z}^*\simeq {\mab N}^r$. 
Let $x_{\lam_i}$ and $y_{\mu_j}$ be 
the corresponding local coordinates to 
${\mathcal X}_{\lam_i}\not= \emptyset$ and 
${\mathcal D}_{\mu_j}\not= \emptyset$, respectively,  around $z$.  
Let 
\begin{equation*} 
b^{(l,m)}_{\mathfrak D} \col \os{\circ}{\mathfrak D}{}^{(l,m)}
\lo \os{\circ}{\mathfrak D}  \quad (l,m\in {\mab N})
%\tag{4.5.1}\label{eqn:adtn}
\end{equation*} 
be the natural morphism. 
Then, for a positve integer $k$, 
the following morphism  
\begin{align*} 
&{\rm Res} \col 
P_k({\mathcal O}_{\mathfrak D}\otimes_{{\mathcal O}_{\mathcal X}}
\Om^{\bul}_{({\mathcal X},{\mathcal D})/\os{\circ}{T}}) \lo 
\tag{4.5.1}\label{eqn:mprarn}\\
& {\mathcal O}_{\mathfrak D}\otimes_{{\mathcal O}_{\mathcal X}}
\bigoplus_{l+m=k}b^{(l-1),(m)}_*
(\Om^{\bul}_{\os{\circ}{\mathcal X}{}^{(l-1)}\cap 
\os{\circ}{\mathcal D}{}^{(m)}/\os{\circ}{T}}\otimes_{\mab Z}
\vp^{(l-1),(m)}_{\rm zar}((\os{\circ}{\mathcal X},\os{\circ}{\mathcal D})/\os{\circ}{T}))[-k]
\end{align*} 
\begin{align*} 
&\sig \otimes  
d\log x_{\lam_0}\wedge \cdots \wedge d\log x_{\lam_{l-1}}
\wedge d\log y_{\mu_1}\wedge \cdots \wedge d\log y_{\mu_{m}} \wedge \omega  \lom \\
&\sig \otimes b^*_{\lam_0,\cdots, \lam_{l-1},\mu_1\cdots \mu_{m}}(\omega)
\otimes({\rm orientation}~(\lam_0\cdots \lam_{l-1})
\otimes (\mu_1\cdots \mu_{m})) \\
& (\sig \in {\mathcal O}_{\mathfrak D}, 
\om \in P_0{\Om}^{\bul}_{({\mathcal X},{\mathcal D})/\os{\circ}{T}}) 
\end{align*} 
{\rm (cf.~\cite[(3.1.5)]{dh2})} 
around $z$ 
induces the following ``Poincar\'{e} residue isomorphism''$:$  
\begin{align*} 
&{\rm Res}
\col {\rm gr}_k^P({\mathcal O}_{\mathfrak D}\otimes_{{\mathcal O}_{\mathcal X}}
\Om^{\bul}_{({\mathcal X},{\mathcal D})/\os{\circ}{T}}) \os{\sim}{\lo} 
\tag{4.5.2}\label{eqn:rspys}\\
&\bigoplus_{l+m=k}b^{(l-1),(m)}_{{\mathfrak D}*}
({\mathcal O}_{\os{\circ}{\mathfrak D}{}^{(l-1),(m)}}
\otimes_{{\mathcal O}_{\os{\circ}{\mathcal X}{}^{(l-1)}\cap \os{\circ}{\mathcal D}{}^{(m)}}}
\Om^{\bul}_{\os{\circ}{\mathcal X}{}^{(l-1)}\cap \os{\circ}{\mathcal D}{}^{(m)}
/\os{\circ}{T}}\otimes_{\mab Z}
\vp^{(l-1),(m)}_{\rm zar}((\os{\circ}{\mathcal X},\os{\circ}{\mathcal D})/\os{\circ}{T}))[-k]. 
\end{align*}  
\par 
$(2)$ 
%Let $r$ be a nonnegative integer such that 
%$M(D)_{X,z}/{\mathcal O}_{X,z}^*\simeq {\mab N}^r$. 
Let $y_{\mu_j}$ be the corresponding local coordinate to 
${\mathcal D}_{\mu_j}\not= \emptyset$ around $z$. 
Let 
\begin{equation*} 
d^{(k)}_{\mathfrak D}({\mathcal D}) \col {\mathfrak D}^{(k)}({\mathcal D})
\lo {\mathfrak D}  \quad (k\in {\mab N})
%\tag{4.5.1}\label{eqn:adtn}
\end{equation*} 
be the natural morphism. 
Then, for a nonnegative integer $k$, 
the following morphism  
\begin{align*} 
&{\rm Res}^{{\mathcal D}/{\mathcal X}} \col 
P^{{\mathcal D}/{\mathcal X}}_k({\mathcal O}_{\mathfrak D}\otimes_{{\mathcal O}_{\mathcal X}}
\Om^{\bul}_{({\mathcal X},{\mathcal D})/\os{\circ}{T}}) \lo 
d^{(k)}_{{\mathfrak D}}({\mathcal D})_*
({\mathcal O}_{{\mathfrak D}^{(k)}({\mathcal D})}\otimes_{{\mathcal O}_{{\mathcal D}^{(k)}}}
\Om^{\bul}_{{\mathcal D}^{(k)}/\os{\circ}{T}}
\otimes_{\mab Z}\vp^{(k)}_{\rm zar}
(\os{\circ}{\mathcal D}/\os{\circ}{T}))[-k]\tag{4.5.3}\label{eqn:mpdrarn}\\
&\sig \otimes d\log y_{\mu_1}\wedge \cdots \wedge d\log y_{\mu_{k}}\wedge \omega
\lom 
\sig \otimes c^*_{\mu_1\cdots \mu_{k}}(\omega)
\otimes({\rm orientation}~(\mu_1\cdots \mu_{k})) \\
&(\sig \in {\mathcal O}_{\mathfrak D}, 
\om \in P^{{\mathcal D}/{\mathcal X}}_0{\Om}^{\bul}_{({\mathcal X},{\mathcal D})/\os{\circ}{T}})
\end{align*} 
{\rm (cf.~\cite[(3.1.5)]{dh2})} 
around $z$ 
induces the following ``Poincar\'{e} residue isomorphism''  
\begin{align*} 
{\rm gr}^{P^{{\mathcal D}/{\mathcal X}}}_k
({\mathcal O}_{\mathfrak D}\otimes_{{\mathcal O}_{\mathcal X}}
{\Om}^{\bul}_{({\mathcal X},{\mathcal D})/\os{\circ}{T}}) 
& \os{\sim}{\lo} 
d^{(k)}_{{\mathfrak D}}({\mathcal D})_*
({\mathcal O}_{{\mathfrak D}^{(k)}({\mathcal D})}\otimes_{{\mathcal O}_{{\mathcal D}^{(k)}}}
\Om^{\bul}_{{\mathcal D}^{(k)}/\os{\circ}{T}}
\otimes_{\mab Z}\vp^{(k)}_{\rm zar}
(\os{\circ}{\mathcal D}/\os{\circ}{T}))[-k]. 
\tag{4.5.4}\label{eqn:prgvin} \\ 
\end{align*} 
\par 
$(3)$  
%Let $r$ be a nonnegative integer such that 
%$M_{X,z}/{\mathcal O}_{X,z}^*\simeq {\mab N}^r$. 
Let $x_{\lam_l}$ be the corresponding local coordinate to 
${\mathcal X}_{\lam_l}\not= \emptyset$ around $z$. 
Let ${\mathfrak D}^{(k)}(\os{\circ}{\mathcal X})$ $(k\in {\mab N})$ be the log scheme 
whose underlying scheme is $\os{\circ}{\mathfrak D}{}^{(k),(0)}$ 
and whose log structure is the inverse image of 
$(\os{\circ}{\mathcal X}{}^{(k)},{\mathcal D}\vert_{\os{\circ}{\mathcal X}{}^{(k)}})$. 
Let 
\begin{equation*} 
d^{(k)}_{{\mathfrak D}}(\os{\circ}{\mathcal X}) \col {\mathfrak D}^{(k)}(\os{\circ}{\mathcal X})
\lo \os{\circ}{\mathfrak D}  \quad (k\in {\mab N})
%\tag{4.5.1}\label{eqn:adtn}
\end{equation*} 
be the natural morphism. 
Then, for a positive integer $k$, 
the following morphism  
\begin{align*}
&{\rm Res}^{{\mathcal X}/{\mathcal D}} \col 
P^{{\mathcal X}/{\mathcal D}}_k({\mathcal O}_{\mathfrak D}\otimes_{{\mathcal O}_{\mathcal X}}
\Om^{\bul}_{({\mathcal X},{\mathcal D})/\os{\circ}{T}}) \lo \tag{4.5.5}\label{eqn:mpxdrarn}\\
& {\mathcal O}_{\mathfrak D}\otimes_{{\mathcal O}_{\mathcal X}}
d^{(k-1)}_{\mathfrak D}(\os{\circ}{\mathcal X})_*(\Om^{\bul}_{(\os{\circ}{\mathcal X}{}^{(k-1)},
\os{\circ}{\mathcal D}\vert_{\os{\circ}{\mathcal X}{}^{(k-1)}})/\os{\circ}{T}}
\otimes_{\mab Z}\vp^{(k-1)}_{\rm zar}
(\os{\circ}{\mathcal X}/\os{\circ}{T}))[-k] \\
& \sig \otimes d\log x_{\lam_1}\wedge 
\cdots \wedge d\log x_{\lam_{k}}\wedge \omega
\lom 
\sig \otimes b^*_{\lam_1\cdots \lam_k}(\omega)
\otimes({\rm orientation}~(\lam_1\cdots \lam_k)) \\
& (\sig \in {\mathcal O}_{\mathfrak D}, 
\om \in P^{{\mathcal X}/{\mathcal D}}_0{\Om}^{\bul}_{({\mathcal X},{\mathcal D})/\os{\circ}{T}})
\end{align*} 
{\rm (cf.~\cite[(3.1.5)]{dh2})} 
around $z$ 
induces the following ``Poincar\'{e} residue isomorphism''  
\begin{align*} 
&{\rm gr}^{P^{{\mathcal X}/{\mathcal D}}}_k
({\mathcal O}_{\mathfrak D}\otimes_{{\mathcal O}_{\mathcal X}}
{\Om}^{\bul}_{({\mathcal X},{\mathcal D})/\os{\circ}{T}}) 
\os{\sim}{\lo}  \tag{4.5.6}\label{eqn:prgxvin}\\
&d^{(k-1)}_{{\mathfrak D}}(\os{\circ}{\mathcal X})_*
({\mathcal O}_{{\mathfrak D}^{(k-1)}(\os{\circ}{\mathcal X})}\otimes_{{\mathcal O}_{\os{\circ}{\mathcal X}{}^{(k-1)}}}
\Om^{\bul}_{(\os{\circ}{\mathcal X}{}^{(k-1)},
\os{\circ}{\mathcal D}\vert_{\os{\circ}{\mathcal X}^{(k-1)}})/\os{\circ}{T}}
\otimes_{\mab Z}\vp^{(k-1)}_{\rm zar}
(\os{\circ}{\mathcal X}/\os{\circ}{T}))[-k]. \\ 
\end{align*} 
\par 
$(4)$ Denote by $P^{{\mathcal D}/{\mathcal X}}$ the induced filtration on 
${\mathcal O}_{\mathfrak D}\otimes_{{\mathcal O}_{\mathcal X}}
{\Om}^{\bul}_{({\mathcal X},{\mathcal D})/\os{\circ}{T}}/
P^{{\mathcal X}}_l({\mathcal O}_{\mathfrak D}\otimes_{{\mathcal O}_{\mathcal X}}
{\Om}^{\bul}_{({\mathcal X},{\mathcal D})/\os{\circ}{T}})$ by the filtration 
{\rm (\ref{eqn:pkdw})}. 
Then the Poincar\'{e} residue morphism in $(2)$ induces the following isomorphism$:$ 
\begin{align*} 
&{\rm gr}^{P^{{\mathcal D}/{\mathcal X}}}_k
({\mathcal O}_{\mathfrak D}\otimes_{{\mathcal O}_{\mathcal X}}
{\Om}^{\bul}_{({\mathcal X},{\mathcal D})/\os{\circ}{T}}/
P^{{\mathcal X}}_l({\mathcal O}_{\mathfrak D}\otimes_{{\mathcal O}_{\mathcal X}}
{\Om}^{\bul}_{({\mathcal X},{\mathcal D})/\os{\circ}{T}})) \os{\sim}{\lo} 
\tag{4.5.7}\label{eqn:prplkvin} \\
&d^{(k)}_{{\mathfrak D}}({\mathcal D})_*
({\mathcal O}_{{\mathfrak D}^{(k)}({\mathcal D})}\otimes_{{\mathcal O}_{{\mathcal D}^{(k)}}}
\Om^{\bul}_{{\mathcal D}^{(k)}/\os{\circ}{T}}/
P_l^{{\mathcal D}^{(k)}}({\mathcal O}_{{\mathfrak D}^{(k)}({\mathcal D})}\otimes_{{\mathcal O}_{{\mathcal D}^{(k)}}}
\Om^{\bul}_{{\mathcal D}^{(k)}/\os{\circ}{T}}
\otimes_{\mab Z}\vp^{(k)}_{\rm zar}
(\os{\circ}{\mathcal D}/\os{\circ}{T})))[-k] 
\end{align*} 
for $l\in {\mab Z}$. 
$($Note that ${\mathcal D}^{(k)}$ is an SNCL scheme over $S(T)^{\nat}$.$)$ 
\end{prop}
\begin{proof} 
Because the proofs of (1), (2) and (4) are similar to and simpler than that of (3),  
we give only the proof of (3). 
\par 
As in the usual case, we can easily check that 
the morphism (\ref{eqn:mpxdrarn}) is well-defined and surjective. 
Because the question is local on $\os{\circ}{X}_{T_0}$, 
we may assume that $M_{S(T)^{\nat}}$ is free of rank $1$
and that there exists the following cartesian diagram: 
%(\ref{cd:pnwtp}). 
\begin{equation*}
\begin{CD} 
(X_{\os{\circ}{T}_0},D_{\os{\circ}{T}_0}) @>{\sus}>> ({\mathcal X},{\mathcal D}) @>{\sus}>> 
(\ol{\mathcal X},\ol{\mathcal D})\\
@VVV @VVV @VVV \\
{\mab A}_{S_{\os{\circ}{T}_0}}(a,b,d,e) 
@>{\sus}>> {\mab A}_{S(T)^{\nat}}(a,b,d',e) @>{\sus}>> {\mab A}_{\ol{S(T)^{\nat}}}(a,b,d',e)\\ 
@VVV @VVV@VVV \\
S_{\os{\circ}{T}_0}@>{\sus}>> S(T)^{\nat}@>{\sus}>> \ol{S(T)^{\nat}},   
\end{CD} 
\tag{4.5.8}\label{cd:pnbtp} 
\end{equation*}
where $d'\geq d$ and the vertical morphism 
$({\mathcal X},{\mathcal D}) \lo {\mab A}_{\ol{S(T)^{\nat}}}(a,b,d',e)$ is solid and \'{e}tale and 
the immersion ${\mab A}_{S_{\os{\circ}{T}_0}}(a,b,d,e) 
\os{\sus}{\lo} {\mab A}_{S(T)^{\nat}}(a,b,d',e)$ is defined by 
the extra coordinates $x_{d+1}, \ldots, x_{d'}$ of ${\mab A}_{\ol{S(T)^{\nat}}}(a,b,d',e)$. 
Set 
$\ol{\mathcal P}=(\ol{\mathcal X},\ol{\mathcal D})$, ${\mathcal P}=({\mathcal X},{\mathcal D})$, 
$\os{\circ}{\ol{\mathcal P}}=(\os{\circ}{\ol{\mathcal X}},\os{\circ}{\ol{\mathcal D}})$
and 
$\os{\circ}{\mathcal P}{}^{(k)}:=(\os{\circ}{\mathcal X}{}^{(k)},\os{\circ}{\mathcal D}\vert_{\os{\circ}{\mathcal X}{}^{(k)}})$ 
for simplicity of notation in this proof (and (\ref{prop:csrl}) below).  
Set ${\mathcal P}:=\ol{\mathcal P}\times_{\ol{S(T)^{\nat}}}S(T)^{\nat}$. 
Let 
$\ol{b}{}^{(k)}\col \os{\circ}{\mathcal P}{}^{(k)}\lo \ol{\mathcal P}$ 
be the natural morphism. 
Let $P^{\ol{\mathcal X}/\ol{\mathcal D}}$ be the filtration on $\Om^i_{\ol{\mathcal P}/\os{\circ}{T}}$ defined by 
the following: 
\begin{equation*} 
P^{\ol{\mathcal X}/\ol{\mathcal D}}_k\Om^i_{\ol{\mathcal P}/\os{\circ}{T}} =
\begin{cases} 
0 & (k<0), \\
{\rm Im}(\Om^k_{\ol{\mathcal P}/\os{\circ}{T}}\otimes_{{\mathcal O}_{\ol{\mathcal X}}}
\Om^{i-k}_{\os{\circ}{\ol{\mathcal P}}/\os{\circ}{T}}
\lo 
\Om^i_{\ol{\mathcal P}/\os{\circ}{T}})
& (0\leq k\leq i), \\
\Om^i_{\ol{\mathcal P}/\os{\circ}{T}} & (k > i).
\end{cases}
\tag{4.5.9}\label{eqn:ptfw}
\end{equation*} 
Then, as in \cite[(2.2.21.3)]{nh2} (the usual Poincar\'{e} residue isomorphism), 
we have the following isomorphism 
\begin{equation*}  
{\rm Res} \col 
{\rm gr}_k^{P^{\ol{\mathcal X}/\ol{\mathcal D}}}
{\Om}^{\bul}_{\ol{\mathcal P}/\os{\circ}{T}}
\os{\sim}{\lo} 
\ol{b}{}^{(k-1)}_*
(\Om^{\bul}_{\os{\circ}{\mathcal P}{}^{(k-1)}/\os{\circ}{T}}
\otimes_{\mab Z}\vp^{(k-1)}_{\rm zar}(\os{\circ}{\mathcal X}/\os{\circ}{T}))[-k]
\quad (k\geq 1). 
\tag{4.5.10}\label{eqn:gurps}
\end{equation*}  
Because 
${\rm gr}_k^{P^{\ol{\mathcal X}/\ol{\mathcal D}}}{\Om}^{\bul}_{\ol{\mathcal P}/\os{\circ}{T}}
=(P^{\ol{\mathcal X}/\ol{\mathcal D}}_k{\Om}^{\bul}_{\ol{\mathcal P}/\os{\circ}{T}}/
{\Om}^{\bul}_{\ol{\mathcal P}/\os{\circ}{T}}(-\os{\circ}{\mathcal P}))
/(P^{\ol{\mathcal X}/\ol{\mathcal D}}_{k-1}{\Om}^{\bul}_{\ol{\mathcal P}/\os{\circ}{T}}/
{\Om}^{\bul}_{\ol{\mathcal P}/\os{\circ}{T}}(-\os{\circ}{\mathcal P}))
={\rm gr}_k^{P^{{\mathcal X}/{\mathcal D}}}{\Om}^{\bul}_{{\mathcal P}/\os{\circ}{T}}$ 
by (\ref{eqn:pops}), 
we have the following isomorphism 
\begin{equation*}  
{\rm Res} \col {\rm gr}_k^{P^{{\mathcal X}/{\mathcal D}}}{\Om}^{\bul}_{{\mathcal P}/\os{\circ}{T}}
\os{\sim}{\lo} 
b^{(k-1)}_*
(\Om^{\bul}_{\os{\circ}{\mathcal P}{}^{(k-1)}/\os{\circ}{T}}
\otimes_{\mab Z}\vp^{(k-1)}_{\rm zar}(\os{\circ}{\mathcal X}/\os{\circ}{T}))[-k] 
\quad (k\geq 1).  
\tag{4.5.11}\label{eqn:grps}
\end{equation*} 
By (\ref{prop:injpf}) and (\ref{eqn:grps}), 
we have the following exact sequence 
\begin{align*} 
0 \lo & P^{{\mathcal X}/{\mathcal D}}_{k-1}
({\mathcal O}_{\mathfrak D}\otimes_{{\mathcal O}_{\mathcal P}}
{\Om}^{\bul}_{{\mathcal P}/\os{\circ}{T}}) 
 \lo 
P^{{\mathcal X}/{\mathcal D}}_k
({\mathcal O}_{\mathfrak D}\otimes_{{\mathcal O}_{\mathcal P}}
{\Om}^{\bul}_{{\mathcal P}/\os{\circ}{T}})  
\tag{4.5.12}\label{eqn:ypbgd} \\
{} & \lo 
{\mathcal O}_{\mathfrak D}\otimes_{{\mathcal O}_{\mathcal P}}
b^{(k-1)}_*(\Om^{\bul}_{\os{\circ}{\mathcal P}{}^{(k-1)}/\os{\circ}{T}}
\otimes_{\mab Z}\vp^{(k-1)}_{\rm zar}(\os{\circ}{\mathcal X}/\os{\circ}{T}))[-k] \lo 0. 
\end{align*}
Let ${\mathcal J}$ be the defining ideal sheaf of the immersion 
$X\os{\sus}{\lo}{\mathcal P}$. 
%By \cite[3.20 Remarks 7)]{bob}, 
%${\mathcal O}_{\mathfrak D}\otimes_{{\mathcal O}_{\mathcal P}}
%={\mathcal O}_{\mathfrak D}\otimes_{{\mathcal O}_{\mathcal P}/{\mathcal J}^n}$ 
%for some $n\in {\mab N}$. 
By (\ref{lemm:dlm}) (1), the fourth term of (\ref{eqn:ypbgd}) 
is equal to the right hand side of (\ref{eqn:prgxvin}). 
\end{proof}  

\begin{rema}
The statement (4) is a statement of new type. 
We cannot find the corresponding statement in the $\infty$-adic case in 
\cite{ezth} and \cite{stz}. 
\end{rema}

\begin{prop}[{\rm {\bf cf.~\cite[Lemma 3.15.1]{msemi}, \cite[(6.29)]{ndw}, 
\cite[(1.3.21)]{nb}}}]\label{prop:csrl}  
Recall that $\os{\circ}{\mathcal P}{}^{(k)}:=
(\os{\circ}{\mathcal X}{}^{(k)},\os{\circ}{\mathcal D}
\vert_{\os{\circ}{\mathcal X}{}^{(k)}})$ $(k\in {\mab N})$. 
Set $\os{\circ}{\mathcal P}:=
(\os{\circ}{\mathcal X},\os{\circ}{\mathcal D})$. 
Let 
$
%e^{(k)}\col (\os{\circ}{X}{}^{(k)}_{T_0},\os{\circ}{D}_{T_0}\vert_{\os{\circ}{X}{}^{(k)}_{T_0}}) 
%\lo (\os{\circ}{X}_{T_0}, \os{\circ}{D}_{T_0}), 
%\quad 
e^{(k)} \col \os{\circ}{\mathcal P}{}^{(k)}\lo \os{\circ}{\mathcal P}$ 
be the natural morphism. 
%Set ${\mathfrak D}^{(k)}:={\mathfrak D}\times_{\os{\circ}{\mathcal X}}\os{\circ}{\mathcal X}{}^{(k)}$  
%$(k\in {\mab N})$. 
Fix a total order on $\Lam$ once and for all. 
For an subset $\ul{\lam}=\{\lam_0,\ldots,\lam_k\}$ 
$(\lam_i <\lam_j~{\rm if}~i< j, \lam_i\in \Lam)$ of $\Lam$, 
set $\ul{\lam}_j:=\ul{\lam}\setminus \{\lam_j\}$
and let $\iota_{\ul{\lam}_j,\ul{\lam}} \col \os{\circ}{\mathcal X}_{\ul{\lam}}\os{\sus}{\lo} 
\os{\circ}{\mathcal X}_{\ul{\lam}_j}$
be the natural closed immersion. 
Let ${\mathfrak D}_{\ul{\lam}}$ be the log PD-envelope of the immersion 
$X_{\ul{\lam}}\os{\sus}{\lo} {\mathcal X}_{\ul{\lam}}$ over $(S(T)^{\nat},{\mathcal J},\del)$ 
and let $b_{\os{\circ}{\mathfrak D}_{\ul{\lam}}} \col 
\os{\circ}{\mathfrak D}_{\ul{\lam}}\lo \os{\circ}{\mathfrak D}$ 
be the natural morphism. 
Let 
\begin{align*} 
\iota^{*}_{\ul{\lam}_j,\ul{\lam}} 
\col &
b_{\os{\circ}{\mathfrak D}_{\ul{\lam}_j}*}(
{\mathcal O}_{\os{\circ}{\mathfrak D}_{\ul{\lam}_j}}
\otimes_{{\mathcal O}_{\os{\circ}{\mathcal X}_{\ul{\lam}_j}}}
\Om^{\bul}_{(\os{\circ}{\mathcal X}_{\ul{\lam}_j},
\os{\circ}{\mathcal D}\vert_{{\mathcal X}_{\ul{\lam}_j}})/\os{\circ}{T}})
\lo 
b_{\os{\circ}{\mathfrak D}_{\ul{\lam}}*}(
{\mathcal O}_{\os{\circ}{\mathfrak D}_{\ul{\lam}}}
\otimes_{{\mathcal O}_{\os{\circ}{\mathcal X}_{\ul{\lam}}}}
\Om^{\bul}_{(\os{\circ}{\mathcal X}_{\ul{\lam}},\os{\circ}{\mathcal D}
\vert_{\os{\circ}{\mathcal X}_{\ul{\lam}}})/\os{\circ}{T}}) 
\end{align*} 
be the induced morphism by $\iota_{\ul{\lam}_j,\ul{\lam}}$. 
Let $b^{(k)}_{\mathfrak D}\col 
({\mathfrak D}^{(k)}
(\os{\circ}{\mathcal X}))^{\circ}\lo 
\os{\circ}{\mathfrak D}$ be the natural morphism.   
Set 
\begin{align*} 
\iota^{(k)*}:=\us{\{\ul{\lam} \vert \# \ul{\lam}=k+1\}}{\sum}
\sum_{j=0}^k(-1)^j\iota^{*}_{\ul{\lam}_j,\ul{\lam}} 
\col &
b^{(k-1)}_{{\mathfrak D}*}
({\mathcal O}_{{\mathfrak D}^{(k-1)}(\os{\circ}{\mathcal X})}
\otimes_{{\mathcal O}_{\os{\circ}{\mathcal X}{}^{(k-1)}}}
\Om^{\bul}_{\os{\circ}{\mathcal P}{}^{(k-1)}/\os{\circ}{T}})
\tag{4.7.1}\label{ali:ulm}\\
& \lo 
b^{(k)}_{{\mathfrak D}*}(
{\mathcal O}_{{\mathfrak D}^{(k)}(\os{\circ}{\mathcal X})}
\otimes_{{\mathcal O}_{(\os{\circ}{\mathcal X}{}^{(k)}}}
\Om^{\bul}_{\os{\circ}{\mathcal P}{}^{(k)}/\os{\circ}{T}}). 
\end{align*} 
Then the following sequence 
\begin{align*} 
0 &\lo P^{{\mathcal X}/{\mathcal D}}_0
({\mathcal O}_{{\mathfrak D}}
\otimes_{{\mathcal O}_{\mathcal X}}
\Om^{\bul}_{({\mathcal X},{\mathcal D})/\os{\circ}{T}}) 
\lo b^{(0)}_{{\mathfrak D}*}
({\mathcal O}_{{\mathfrak D}^{(0)}(\os{\circ}{\mathcal X})}
\otimes_{{\mathcal O}_{{\os{\circ}{\mathcal X}{}^{(0)}}}}
\Om^{\bul}_{\os{\circ}{\mathcal P}{}^{(0)}/\os{\circ}{T}}
\otimes_{\mab Z}
\vp^{(0)}_{\rm zar}(\os{\circ}{\mathcal X}
/\os{\circ}{T})) \tag{4.7.2}\label{ali:useq}\\
&\os{\iota^{(0)*}}{\lo} {b}{}^{(1)}_{{\mathfrak D}*}
({\mathcal O}_{{\mathfrak D}^{(1)}(\os{\circ}{\mathcal X})}
\otimes_{{\mathcal O}_{\os{\circ}{\mathcal X}{}^{(1)}}}
\Om^{\bul}_{\os{\circ}{\mathcal P}{}^{(1)}/\os{\circ}{T}}
\otimes_{\mab Z}
\vp^{(1)}_{\rm zar}(\os{\circ}{\mathcal X}/\os{\circ}{T})) 
\os{\iota^{(1)*}}{\lo} \cdots 
\end{align*} 
is exact.
\end{prop}
\begin{proof} 
(The proof is the same as that of \cite[(1.3.21)]{nb}.)  
Because the question is local on $\os{\circ}{X}_{T_0}$, 
%We may assume that $E={\cal O}_{\os{\circ}{X}_{T_0}/\os{\circ}{T}}$.  
%and $T=S(T)$.  
%Consequently 
we may assume that $M_{S(T)^{\nat}}$ is free of rank $1$ 
and that  there exists the following cartesian diagram 
(\ref{cd:pnbtp}).  
Because the morphisms 
$$P^{{\mathcal X}/{\mathcal D}}_0
({\mathcal O}_{{\mathfrak D}}
\otimes_{{\mathcal O}_{\mathcal X}}
\Om^i_{({\mathcal X},{\mathcal D})/\os{\circ}{T}}) 
\lo b^{(0)}_{{\mathfrak D}*}
({\mathcal O}_{{\mathfrak D}^{(0)}(\os{\circ}{\mathcal X})}
\otimes_{{\mathcal O}_{{\os{\circ}{\mathcal X}{}^{(0)}}}}
\Om^{\bul}_{\os{\circ}{\mathcal P}{}^{(0)}/\os{\circ}{T}}
\otimes_{\mab Z}
\vp^{(0)}_{\rm zar}(\os{\circ}{\mathcal X}
/\os{\circ}{T}))$$ 
and $\iota^{(k)*}$ are ${\mathcal O}_{\mathcal X}$-linear, we may assume that 
$\ol{\mathcal P}:=(\ol{\mathcal X},\ol{\mathcal D})={\mab A}_{\ol{S(T)^{\nat}}}(a,b,d',e)$. 
Set ${\mathcal P}:=\ol{\mathcal P}\times_{\ol{S(T)^{\nat}}}S(T)^{\nat}$, 
$\ol{\mathcal Q}:={\mab A}_{\ol{S(T)^{\nat}}}(a,d-b)$, 
${\mathcal Q}:=\ol{\mathcal Q}\times_{\ol{S(T)^{\nat}}}S(T)^{\nat}$ and  
${\mathcal R}:=({\mab A}^b_{\os{\circ}{T}},(y_1\cdots y_b=0))$. 
Let 
$b'{}^{(k)}\col \os{\circ}{\mathcal Q}{}^{(k)} \lo \os{\circ}{\mathcal Q}$ 
be the natural morphism. 
By \cite[(4.2.2) (c)]{di} we have the following exact sequence 
\begin{equation*}
0 \lo 
\Om^{\bul}_{\os{\circ}{\ol{\mathcal Q}}/\os{\circ}{T}}/
{\Om}^{\bul}_{{\ol{\mathcal Q}}/\os{\circ}{T}}(-\os{\circ}{\mathcal Q})  
\lo b^{(0)}_{*}(\Om^{\bul}_{\os{\circ}{\mathcal Q}{}^{(0)}/\os{\circ}{T}}
\otimes_{\mab Z}\vp^{(0)}_{\rm zar}
(\os{\circ}{\mathcal X}/\os{\circ}{T})) 
\lo \cdots.  
\end{equation*} 
Because  
\begin{align*} 
{\mathcal O}_{\mathfrak D}\otimes_{{\mathcal O}_{{\mathcal P}}}
(\Om^{\bul}_{\os{\circ}{\ol{\mathcal P}}/\os{\circ}{T}}/
{\Om}^{\bul}_{{\ol{\mathcal P}}/\os{\circ}{T}}
(-\os{\circ}{\mathcal P}))
\os{\sim}{\lo} & 
(\Om^{\bul}_{\os{\circ}{\ol{\mathcal Q}}/\os{\circ}{T}}/
\Om^{\bul}_{\ol{\mathcal Q}/\os{\circ}{T}}(-\os{\circ}{\mathcal Q})
\otimes_{{\mathcal O}_T} \\
{} & 
{\mathcal O}_T\langle x_{d+1},\ldots, x_{d'} \rangle 
\otimes_{{\mathcal O}_T}
\Om^{\bul}_{{\mathcal R}/\os{\circ}{T}},   
\end{align*}
because 
\begin{align*}  
{\mathcal O}_{\mathfrak D}\otimes_{{\mathcal O}_{\mathcal P}}
b{}^{(k)}_{*}(\Om^{\bul}_{\os{\circ}{\mathcal P}{}^{(k)}/\os{\circ}{T}}
\otimes_{\mab Z}\vp^{(k)}_{\rm zar}(\os{\circ}{\mathcal X}/\os{\circ}{T})) 
\os{\sim}{\lo} & 
b'{}^{(k)}_{*}
(\Om^{\bul}_{\os{\circ}{\mathcal Q}{}^{(k)}/\os{\circ}{T}}
\otimes_{\mab Z}\vp^{(k)}_{\rm zar}(\os{\circ}{\mathcal X}/\os{\circ}{T}))  
\otimes_{{\mathcal O}_T} \\
{} & 
{\mathcal O}_T\langle x_{d+1},\ldots, x_{d'} \rangle 
\otimes_{{\mathcal O}_T}\Om^{\bul}_{{\mathcal R}/\os{\circ}{T}}   
\end{align*} 
and because the complex 
${\mathcal O}_T
\langle x_{d+1},\ldots, x_{d'} \rangle 
\otimes_{{\mathcal O}_T}\Om^{\bul}_{{\mathcal R}/\os{\circ}{T}}$ 
consists of free ${\mathcal O}_T$-modules, 
we see that the following sequence is exact: 
\begin{equation*} 
0 \lo  
{\mathcal O}_{\mathfrak D}
\otimes_{{\mathcal O}_{{\mathcal P}}}
(\Om^{\bul}_{\os{\circ}{\ol{\mathcal P}}/\os{\circ}{T}}/
{\Om}^{\bul}_{{\ol{\mathcal P}}/\os{\circ}{T}}
(-\os{\circ}{\mathcal P}))
\lo 
{\mathcal O}_{\mathfrak D}\otimes_{{\mathcal O}_{\mathcal P}}
b^{(0)}_{*}
(\Om^{\bul}_{\os{\circ}{\mathcal P}{}^{(0)}
/\os{\circ}{T}}\otimes_{\mab Z}\vp^{(0)}_{\rm zar}
(\os{\circ}{\mathcal X}/\os{\circ}{T})) 
\lo \cdots.
\tag{4.7.3}\label{ali:opot}    
\end{equation*} 
By (\ref{prop:incol}) (4) and (\ref{eqn:pops}) we have 
the following isomorphism: 
\begin{align*}
\Om^{\bul}_{\os{\circ}{\ol{\mathcal P}}/\os{\circ}{T}}/
\Om^{\bul}_{{\ol{\mathcal P}}/\os{\circ}{T}}(-\os{\circ}{\mathcal P})
\simeq P_0^{{\mathcal X}/{\mathcal D}}\Om^{\bul}_{\ol{\mathcal P}/\os{\circ}{T}}/
\Om^{\bul}_{{\ol{\mathcal P}}/\os{\circ}{T}}(-\os{\circ}{\mathcal P})
=P_0^{{\mathcal X}/{\mathcal D}}\Om^{\bul}_{{\mathcal P}/\os{\circ}{T}}. 
\tag{4.7.4}\label{ali:opt} 
\end{align*}  
By (\ref{prop:injpf}), (\ref{lemm:dlm}) (1), (\ref{ali:opot}) and 
(\ref{ali:opt}), we see that the sequence (\ref{ali:useq}) is exact. 
\end{proof}  

%\begin{rema}
%The underlying scheme of ${\mathfrak D}^{(k)}(\os{\circ}{\mathcal X})$
%is equal to that of ${\mathfrak D}^{(k),(0)}$. 
%\end{rema}

\begin{lemm}\label{lemm:ti} 
Let $k$ be a positive integer.   
Let $\theta_{({\mathcal X},{\mathcal D})/\os{\circ}{T}}
\in \Om^1_{({\mathcal X},{\mathcal D})/\os{\circ}{T}}$ be the image of 
$d\log t\in \Om^1_{S(T)^{\nat}/\os{\circ}{T}}$, where $t$ is 
a local section of $M_{S(T)^{\nat}}$ 
whose image in $M_{S(T)^{\nat}}/{\mathcal O}_T^*$ is the generator. 
$($The local section $d\log t$ is independent of the choice of $t.)$
Set 
{\footnotesize{
\begin{align*} 
&\iota^{(k)*}:=\us{\{\ul{\lam} \vert \# \ul{\lam}=k+1\}}{\sum}
\sum_{j=0}^k(-1)^j\iota^{*}_{\ul{\lam}_j,\ul{\lam}} 
\col 
{\mathcal O}_{\mathfrak D}\otimes_{{\mathcal O}_{\mathcal X}}
\bigoplus_{l+m=k}b^{(l-1),(m)}_*
(\Om^{i-k}_{\os{\circ}{\mathcal X}{}^{(l-1)}\cap 
\os{\circ}{\mathcal D}{}^{(m)}/\os{\circ}{T}}\otimes_{\mab Z}
\vp^{(l-1),(m)}_{\rm zar}((\os{\circ}{\mathcal X},\os{\circ}{\mathcal D})/\os{\circ}{T}))
\tag{4.8.1}\label{ali:uliim}\\
& \lo 
{\mathcal O}_{\mathfrak D}\otimes_{{\mathcal O}_{\mathcal X}}
\bigoplus_{l+m=k}b^{(l-1),(m)}_*
(\Om^{i-k}_{\os{\circ}{\mathcal X}{}^{(l-1)}\cap 
\os{\circ}{\mathcal D}{}^{(m)}/\os{\circ}{T}}\otimes_{\mab Z}
\vp^{(l-1),(m)}_{\rm zar}((\os{\circ}{\mathcal X},\os{\circ}{\mathcal D})/\os{\circ}{T})). 
\end{align*}}} 
Then the following diagram is commutative$:$
\begin{equation*} 
\begin{CD} 
{\rm gr}^{P}_{k+1}
({\mathcal O}_{\mathfrak D}\otimes_{{\mathcal O}_{{\mathcal X}}}
{\Om}^{i+1}_{({\mathcal X},{\mathcal D})/\os{\circ}{T}})
@>{\simeq}>>  
\\
@A{\theta_{({\mathcal X},{\mathcal D})/\os{\circ}{T}}\wedge}AA  \\
{\rm gr}^{P}_k
({\mathcal O}_{\mathfrak D}
\otimes_{{\mathcal O}_{\mathcal X}}\Om^{i}_{({\mathcal X},{\mathcal D})/\os{\circ}{T}})
@>{\simeq}>> 
\end{CD}
\tag{4.8.2}\label{eqn:xdxgras}
\end{equation*} 
\begin{equation*} 
\begin{CD} 
{\mathcal O}_{\mathfrak D}\otimes_{{\mathcal O}_{\mathcal X}}
\bigoplus_{l+m=k+1}b^{(l-1),(m)}_*
(\Om^{i-k}_{\os{\circ}{\mathcal X}{}^{(l-1)}\cap 
\os{\circ}{\mathcal D}{}^{(m)}/\os{\circ}{T}}\otimes_{\mab Z}
\vp^{(l-1),(m)}_{\rm zar}((\os{\circ}{\mathcal X},\os{\circ}{\mathcal D})/\os{\circ}{T}))[-k-1]
\\
@AA{\iota^{(k-1)*}}A \\
{\mathcal O}_{\mathfrak D}\otimes_{{\mathcal O}_{\mathcal X}}
\bigoplus_{l+m=k}b^{(l-1),(m)}_*
(\Om^{i-k}_{\os{\circ}{\mathcal X}{}^{(l-1)}\cap 
\os{\circ}{\mathcal D}{}^{(m)}/\os{\circ}{T}}\otimes_{\mab Z}
\vp^{(l-1),(m)}_{\rm zar}((\os{\circ}{\mathcal X},\os{\circ}{\mathcal D})/\os{\circ}{T}))[-k]. 
\end{CD}
\end{equation*} 
\end{lemm}
\begin{proof} 
The proof is the same as that of \cite[4.12]{msemi} (cf.~\cite[(10.1.16)]{ndw}). 
\end{proof}

\begin{prop}\label{prop:pbe}
The following morphism 
\begin{align*} 
\theta_{{\mathcal P}/\os{\circ}{T}} \wedge  & 
\col 
{\mathcal O}_{{\mathfrak D}}
\otimes_{{\mathcal O}_{{\mathcal P}}}
{\Om}^i_{{\mathcal P}/S(T)^{\nat}}
\lo  \tag{4.9.1}\label{eqn:prppn} \\
&\{({\mathcal O}_{\mathfrak D}\otimes_{{\mathcal O}_{{\mathcal P}}}
{\Om}^{i+1}_{{\mathcal P}/\os{\circ}{T}}
/P^{\mathcal X}_0  
\os{\theta_{{\mathcal P}/\os{\circ}{T}}}{\lo} {\mathcal O}_{\mathfrak D}
\otimes_{{\mathcal O}_{{\mathcal P}}}
{\Om}^{i+2}_{{\mathcal P}/\os{\circ}{T}}/P^{{\mathcal X}}_1 
\os{\theta_{{\mathcal P}/\os{\circ}{T}}}{\lo}\cdots )\}
\end{align*} 
is a quasi-isomorphism. 
\end{prop} 
\begin{proof} 
It suffices to prove that 
\begin{align*} 
& {\mathcal O}_{{\mathfrak D}}
\otimes_{{\mathcal O}_{{\mathcal P}}}
{\Om}^{i-1}_{{\mathcal P}/\os{\circ}{T}}
\os{\theta_{{\mathcal P}/\os{\circ}{T}}}{\lo}  {\mathcal O}_{{\mathfrak D}}
\otimes_{{\mathcal O}_{{\mathcal P}}}
{\Om}^i_{{\mathcal P}/\os{\circ}{T}}
\os{\theta_{{\mathcal P}/\os{\circ}{T}}}{\lo} 
{\mathcal O}_{\mathfrak D}\otimes_{{\mathcal O}_{{\mathcal P}}}
{\Om}^{i+1}_{{\mathcal P}/\os{\circ}{T}}
/P^{\mathcal X}_0  \\
&\os{\theta_{{\mathcal P}/\os{\circ}{T}}}{\lo} {\mathcal O}_{\mathfrak D}
\otimes_{{\mathcal O}_{{\mathcal P}}}
{\Om}^{i+2}_{{\mathcal P}/\os{\circ}{T}}/P^{{\mathcal X}}_1 
\os{\theta_{{\mathcal P}/\os{\circ}{T}}}{\lo}\cdots 
\end{align*} 
is exact. 
%Following the proof of \cite[(3.15)]{msemi}, 
%set $K_{-2}:=
The sheaves ${\mathcal O}_{{\mathfrak D}}
\otimes_{{\mathcal O}_{{\mathcal P}}}
{\Om}^{i-1}_{{\mathcal P}/\os{\circ}{T}}$, 
${\mathcal O}_{{\mathfrak D}}
\otimes_{{\mathcal O}_{{\mathcal P}}}
{\Om}^{i}_{{\mathcal P}/\os{\circ}{T}}$ 
and 
${\mathcal O}_{{\mathfrak D}}
\otimes_{{\mathcal O}_{{\mathcal P}}}
{\Om}^{i+j+1}_{{\mathcal P}/\os{\circ}{T}}/P_j^{\mathcal X}$ 
have induced filtrations by $P^{{\mathcal D}/{\mathcal X}}$ on 
${\mathcal O}_{{\mathfrak D}}
\otimes_{{\mathcal O}_{{\mathcal P}}}
{\Om}^{\bul}_{{\mathcal P}/\os{\circ}{T}}$, which we denote by 
$P^{{\mathcal D}/{\mathcal X}}$ again. 
It suffices to prove that 
\begin{align*} 
& {\rm gr}_k^{P^{{\mathcal D}/{\mathcal X}}}({\mathcal O}_{{\mathfrak D}}
\otimes_{{\mathcal O}_{{\mathcal P}}}
{\Om}^{i-1}_{{\mathcal P}/\os{\circ}{T}})
\os{\theta_{{\mathcal P}/\os{\circ}{T}}}{\lo}  
{\rm gr}_k^{P^{{\mathcal D}/{\mathcal X}}}({\mathcal O}_{{\mathfrak D}}
\otimes_{{\mathcal O}_{{\mathcal P}}}
{\Om}^i_{{\mathcal P}/\os{\circ}{T}})
\os{\theta_{{\mathcal P}/\os{\circ}{T}}}{\lo} 
{\rm gr}_k^{P^{{\mathcal D}/{\mathcal X}}}
({\mathcal O}_{\mathfrak D}\otimes_{{\mathcal O}_{{\mathcal P}}}
{\Om}^{i+1}_{{\mathcal P}/\os{\circ}{T}}
/P^{\mathcal X}_0)  \\
&\os{\theta_{{\mathcal P}/\os{\circ}{T}}}{\lo} 
{\rm gr}_k^{P^{{\mathcal D}/{\mathcal X}}}({\mathcal O}_{\mathfrak D}
\otimes_{{\mathcal O}_{{\mathcal P}}}
{\Om}^{i+2}_{{\mathcal P}/\os{\circ}{T}}/P^{{\mathcal X}}_1)  
\os{\theta_{{\mathcal P}/\os{\circ}{T}}}{\lo}\cdots 
\end{align*} 
is exact for any $k\in {\mab Z}$. 
By (\ref{eqn:prplkvin}) this sequence is isomorphic to 
\begin{align*} 
& d^{(k)}_{{\mathfrak D}}({\mathcal D})_*
({\mathcal O}_{{\mathfrak D}^{(k)}({\mathcal D})}\otimes_{{\mathcal O}_{{\mathcal D}^{(k)}}}
\Om^{i-k-1}_{{\mathcal D}^{(k)}/\os{\circ}{T}}
\otimes_{\mab Z}\vp^{(k)}_{\rm zar}
(\os{\circ}{\mathcal D}/\os{\circ}{T})) \tag{4.9.2}\label{eqn:prtvipn}\\
& \os{\theta_{{\mathcal D}^{(k)}/\os{\circ}{T}}}{\lo} 
d^{(k)}_{{\mathfrak D}}({\mathcal D})_*
({\mathcal O}_{{\mathfrak D}^{(k)}({\mathcal D})}\otimes_{{\mathcal O}_{{\mathcal D}^{(k)}}}
\Om^{i-k}_{{\mathcal D}^{(k)}/\os{\circ}{T}}
\otimes_{\mab Z}\vp^{(k)}_{\rm zar}
(\os{\circ}{\mathcal D}/\os{\circ}{T})) \\
&
\os{\theta_{{\mathcal D}^{(k)}/\os{\circ}{T}}}{\lo} 
d^{(k)}_{{\mathfrak D}}({\mathcal D})_*
({\mathcal O}_{{\mathfrak D}^{(k)}({\mathcal D})}\otimes_{{\mathcal O}_{{\mathcal D}^{(k)}}}
\Om^{i+1-k}_{{\mathcal D}^{(k)}/\os{\circ}{T}}
\otimes_{\mab Z}\vp^{(k)}_{\rm zar}
(\os{\circ}{\mathcal D}/\os{\circ}{T})/P_0^{{\mathcal D}^{(k)}}) \\
& \os{\theta_{{\mathcal D}^{(k)}/\os{\circ}{T}}}{\lo}  d^{(k)}_{{\mathfrak D}}({\mathcal D})_*
({\mathcal O}_{{\mathfrak D}^{(k)}({\mathcal D})}\otimes_{{\mathcal O}_{{\mathcal D}^{(k)}}}
\Om^{i+2-k}_{{\mathcal D}^{(k)}/\os{\circ}{T}}
\otimes_{\mab Z}\vp^{(k)}_{\rm zar}
(\os{\circ}{\mathcal D}/\os{\circ}{T})/P_1^{{\mathcal D}^{(k)}})
\os{\theta_{{\mathcal D}^{(k)}/\os{\circ}{T}}}{\lo}  \cdots. 
\end{align*} 
The exactness of this sequence is equivalent to the exactness of the following sequence: 
\begin{align*} 
& 0 \lo
d^{(k)}_{{\mathfrak D}}({\mathcal D})_*
({\mathcal O}_{{\mathfrak D}^{(k)}({\mathcal D})}\otimes_{{\mathcal O}_{{\mathcal D}^{(k)}}}
\Om^{i-k}_{{\mathcal D}^{(k)}/S(T)^{\nat}}
\otimes_{\mab Z}\vp^{(k)}_{\rm zar}
(\os{\circ}{\mathcal D}/\os{\circ}{T})) \tag{4.9.3}\label{eqn:prcvipn}\\
&
\os{\theta_{{\mathcal D}^{(k)}/\os{\circ}{T}}}{\lo} 
d^{(k)}_{{\mathfrak D}}({\mathcal D})_*
({\mathcal O}_{{\mathfrak D}^{(k)}({\mathcal D})}\otimes_{{\mathcal O}_{{\mathcal D}^{(k)}}}
\Om^{i+1-k}_{{\mathcal D}^{(k)}/\os{\circ}{T}}
\otimes_{\mab Z}\vp^{(k)}_{\rm zar}
(\os{\circ}{\mathcal D}/\os{\circ}{T})/P_0^{{\mathcal D}^{(k)}}) \\
& \os{\theta_{{\mathcal D}^{(k)}/\os{\circ}{T}}}{\lo}  d^{(k)}_{{\mathfrak D}}({\mathcal D})_*
({\mathcal O}_{{\mathfrak D}^{(k)}({\mathcal D})}\otimes_{{\mathcal O}_{{\mathcal D}^{(k)}}}
\Om^{i+2-k}_{{\mathcal D}^{(k)}/\os{\circ}{T}}
\otimes_{\mab Z}\vp^{(k)}_{\rm zar}
(\os{\circ}{\mathcal D}/\os{\circ}{T})/P_1^{{\mathcal D}^{(k)}})
\os{\theta_{{\mathcal D}^{(k)}/\os{\circ}{T}}}{\lo}  \cdots. 
\end{align*} 
This has been proved in \cite[(1.4.3)]{nb} since 
${\mathcal D}^{(k)}$ is an SNCL scheme over $S(T)^{\nat}$. 
\end{proof}

\section{Zariskian $p$-adic bifiltered El Zein-Steenbrink-Zucker complexes}\label{sec:psc}
Let $S$, $(T,{\mathcal J},\del)$, $T_0\lo S$, $S_{\os{\circ}{T}_0}$ and $S(T)^{\nat}$  
be as in the previous sections. 
Let $(X,D)/S$ be an SNCL scheme with a relative SNCD on $X/S$. 
Let 
$f\col (X_{\os{\circ}{T}_0},D_{\os{\circ}{T}_0}) \lo S_{\os{\circ}{T}_0}$ 
be the structural morphism. 
By abuse of notation, let us also denote the structural morphism 
$(X_{\os{\circ}{T}_0},D_{\os{\circ}{T}_0})\lo S(T)^{\nat}$ by $f$. 
Let $E$ be a flat quasi-coherent crystal of 
${\mathcal O}_{\os{\circ}{X}_{T_0}/\os{\circ}{T}}$-modules.  
\par 
The aim of this section is to construct a bifiltered complex
\begin{equation*} 
(A_{\rm zar}((X_{\os{\circ}{T}_0},D_{\os{\circ}{T}_0})/S(T)^{\nat},E),P^{D_{\os{\circ}{T}_0}},P)
\in {\rm D}^+{\rm F}^2(f^{-1}({\mathcal O}_T)), 
\end{equation*}  
which we call 
the {\it zariskian $p$-adic bifiltered El Zein-Steenbrink-Zucker complex} 
of $E$ for $X_{\os{\circ}{T}_0}/S(T)^{\nat}$.

\par 
For the time being, assume that 
there exists an immersion $(X_{\os{\circ}{T}_0},D_{\os{\circ}{T}_0})
\os{\sus}{\lo} \ol{\mathcal P}$ into a log smooth scheme over $\ol{S(T)^{\nat}}$. 
Denote $\ol{\mathcal P}{}^{\rm ex}=(\ol{\mathcal X},\ol{\mathcal D})$, where 
$(\ol{\mathcal X},\ol{\mathcal D})$ is a strictly semistable formal scheme 
with a relative SNCD over $\ol{S(T)^{\nat}}$ ((\ref{prop:neoxeo}) (2)). 
Set ${\mathcal P}^{\rm ex}:=\ol{\mathcal P}{}^{\rm ex}\times_{\ol{S(T)^{\nat}}}S(T)^{\nat}$ 
and $({\mathcal X},{\mathcal D}):=(\ol{\mathcal X},\ol{\mathcal D})\times_{\ol{S(T)^{\nat}}}S(T)^{\nat}
(={\mathcal P}^{\rm ex})$.  
Let $\ol{\mathfrak D}$ be the log PD-envelope of 
the immersion $(X_{\os{\circ}{T}_0},D_{\os{\circ}{T}_0}) \os{\subset}{\lo} \ol{\mathcal P}$ over 
$(\os{\circ}{T},{\mathcal J},\del)$.  
Let ${\mathfrak D}(\ol{S(T)^{\nat}})$ be the log PD-envelope of 
the immersion $S_{\os{\circ}{T}_0}\os{\sus}{\lo} \ol{S(T)^{\nat}}$ over 
$(\os{\circ}{T},{\mathcal J},\del)$. 
Set ${\mathfrak D}:=\ol{\mathfrak D}\times_{{\mathfrak D}(\ol{S(T)^{\nat}})}S(T)^{\nat}$. 
This is the log PD-envelope of the immersion 
$(X_{\os{\circ}{T}_0},D_{\os{\circ}{T}_0}) \os{\subset}{\lo} {\mathcal P}^{\rm ex}$
over $(S(T)^{\nat},{\mathcal J},\del)$. 
Let 
\begin{align*} 
\eps_{(X_{\os{\circ}{T}_0},D_{\os{\circ}{T}_0})/\os{\circ}{T}}
\col 
(((X_{\os{\circ}{T}_0},D_{\os{\circ}{T}_0})/\os{\circ}{T})_{\rm crys},
{\mathcal O}_{(X_{\os{\circ}{T}_0},D_{\os{\circ}{T}_0})/\os{\circ}{T}}) 
\lo 
((\os{\circ}{X}_{T_0}/\os{\circ}{T})_{\rm crys},
{\mathcal O}_{\os{\circ}{X}_{T_0}/\os{\circ}{T}}) 
\tag{5.0.1}\label{eqn:olpsapc}
\end{align*} 
be the morphism of ringed topoi induced by the morphism 
$\eps_{(X_{\os{\circ}{T}_0},D_{\os{\circ}{T}_0})
/\os{\circ}{T}_0}\col (X_{\os{\circ}{T}_0},D_{\os{\circ}{T}_0})\lo \os{\circ}{X}_{T_0}$ 
over $\os{\circ}{T}$ forgetting the log structure of $(X_{\os{\circ}{T}_0},D_{\os{\circ}{T}_0})$. 
Let $(\ol{\mathcal E},\ol{\nabla})$ be the quasi-coherent ${\mathcal O}_{\ol{\mathfrak D}}$-module 
with the integrable connection 
associated to $\eps^*_{(X_{\os{\circ}{T}_0},D_{\os{\circ}{T}_0})/\os{\circ}{T}}(E)$: 
$\ol{\nabla}\col \ol{\mathcal E}\lo 
\ol{\mathcal E}\otimes_{{\mathcal O}_{\ol{\mathcal P}}}\Om^1_{\ol{\mathcal P}/\os{\circ}{T}}$. 
Set ${\mathcal E}:=\ol{\mathcal E}\otimes_{{\mathcal O}_{\ol{\mathfrak D}}}{\mathcal O}_{\mathfrak D}$.  
As in \cite[pp.~41--42]{nhir}, it is easy to check that  $\ol{\nabla}$ 
induces the following integrable connection 
\begin{equation*} 
\nabla \col {\mathcal E}
\lo {\mathcal E}\otimes_{{\mathcal O}_{{\mathcal P}}}\Om^1_{{\mathcal P}/\os{\circ}{T}} 
=
{\mathcal E}\otimes_{{\mathcal O}_{{\mathcal P}^{\rm ex}}}
{\Om}^1_{{\mathcal P}^{\rm ex}/\os{\circ}{T}}.  
\tag{5.0.2}\label{eqn:nitpltd}
\end{equation*}  
%Set $\theta:=\theta_{{\mathcal P}^{\rm ex}}$. 
Set  
\begin{align*} 
A_{\rm zar}({\mathcal P}^{\rm ex}/S(T)^{\nat},{\mathcal E})^{ij}
& :={\mathcal E}\otimes_{{\mathcal O}_{\mathcal X}}
{\Om}^{i+j+1}_{{\mathcal P}^{\rm ex}/\os{\circ}{T}}/P_j^{\mathcal X} 
\tag{5.0.3}\label{cd:accef} \\
& :={\mathcal E}\otimes_{{\mathcal O}_{{\mathcal X}}}
{\Om}^{i+j+1}_{{\mathcal P}^{\rm ex}/\os{\circ}{T}}/
P^{\mathcal X}_j({\mathcal E}\otimes_{{\mathcal O}_{\mathcal X}}
{\Om}^{i+j+1}_{{\mathcal P}^{\rm ex}/\os{\circ}{T}})  
\quad (i,j \in {\mab N}). 
\end{align*}   
The sheaf 
$A_{\rm zar}({\mathcal P}^{\rm ex}/S(T)^{\nat},{\mathcal E})^{ij}$ 
has quotient filtrations $P^{{\mathcal D}/{\mathcal X}}$ and $P$ obtained by 
the filtrations $P^{{\mathcal D}/{\mathcal X}}$ and $P$ on 
${\mathcal E}\otimes_{{\mathcal O}_{{\mathcal X}}}
{\Om}^{i+j+1}_{{\mathcal P}^{\rm ex}/\os{\circ}{T}}$ defined 
in the previous section.  
We consider the following boundary morphisms of 
double complexes: 
\begin{equation*}
\begin{CD}
A_{\rm zar}({\mathcal P}^{\rm ex}/S(T)^{\nat}
,{\mathcal E})^{i,j+1}  @.  \\ 
@A{\theta_{{\mathcal P}^{\rm ex}/\os{\circ}{T}} \wedge}AA  @. \\
A_{\rm zar}({\mathcal P}^{\rm ex}/S(T)^{\nat},{\mathcal E})^{ij}
@>{-\nabla}>> 
A_{\rm zar}({\mathcal P}^{\rm ex}/S(T)^{\nat}
,{\mathcal E})^{i+1,j}.\\
\end{CD}
\tag{5.0.4}\label{cd:lccbd} 
\end{equation*}  
(We think that these are the best boundary morphisms with respect to 
the signs.)
Then we have the double complex 
$A_{\rm zar}({\mathcal P}^{\rm ex}/S(T)^{\nat},{\mathcal E})^{\bul \bul}$. 
The double complex 
$A_{\rm zar}({\mathcal P}^{\rm ex}/S(T)^{\nat},{\mathcal E})^{\bul \bul}$ 
has filtrations $P^{{\mathcal D}/{\mathcal X}}=\{P^{{\mathcal D}/{\mathcal X}}_k\}_{k \in {\mab Z}}$ 
and $P=\{P_k\}_{k \in {\mab Z}}$ 
defined by the following formulas: 
\begin{equation*} 
P^{{\mathcal D}/{\mathcal X}}_kA_{\rm zar}({\mathcal P}^{\rm ex}/S(T)^{\nat},{\mathcal E})^{\bul \bul}
:=(\cdots 
(P^{{\mathcal D}/{\mathcal X}}_k+P_j^{\mathcal X})({\mathcal E}\otimes_{{\mathcal O}_{\mathcal X}}
{\Om}^{i+j+1}_{{\mathcal P}^{\rm ex}/\os{\circ}{T}})
/P_j^{\mathcal X} 
\cdots)\in C^+(f^{-1}({\mathcal O}_T))   
\tag{5.0.5}\label{eqn:lpcad}
\end{equation*} 
and 
\begin{equation*} 
P_kA_{\rm zar}({\mathcal P}^{\rm ex}/S(T)^{\nat},{\mathcal E})^{\bul \bul}
:=
(\cdots 
(P_{2j+k+1}+P_j^{\mathcal X})({\mathcal E}\otimes_{{\mathcal O}_{\mathcal X}}
{\Om}^{i+j+1}_{{\mathcal P}^{\rm ex}/\os{\circ}{T}})
/P_j^{\mathcal X} \cdots)
\in C^+(f^{-1}({\mathcal O}_T)).    
\tag{5.0.6}\label{eqn:lpbcad}
\end{equation*} 
Let $(A_{\rm zar}({\mathcal P}^{\rm ex}/S(T)^{\nat},{\mathcal E}),P^{{\mathcal D}/{\mathcal X}},P)$ 
be the bifiltered single complex of the bifiltered double complex 
$(A_{\rm zar}({\mathcal P}^{\rm ex}/S(T)^{\nat},{\mathcal E})^{\bul \bul},P^{{\mathcal D}/{\mathcal X}},P)$.  
\par 
Let $(Y,E)$ be an SNCL scheme over $S$ with a relative SNCD on $Y/S$ and 
assume that there exists an immersion 
$(Y_{\os{\circ}{T}_0},E_{\os{\circ}{T}_0}) \os{\subset}{\lo} \ol{\mathcal Q}$  
into a log smooth scheme over $\ol{S(T)^{\nat}}$. 
Assume that there exist morphisms 
$g\col (X_{\os{\circ}{T}_0},D_{\os{\circ}{T}_0}) \lo (Y_{\os{\circ}{T}_0},E_{\os{\circ}{T}_0})$ and 
$\ol{g} \col \ol{\mathcal P} \lo \ol{\mathcal Q}$  
making the following diagram commutative: 
\begin{equation*} 
\begin{CD} 
(X_{\os{\circ}{T}_0},D_{\os{\circ}{T}_0})@>{\sus}>> \ol{\mathcal P} \\ 
@V{g}VV @VV{\ol{g}}V \\ 
(Y_{\os{\circ}{T}_0},E_{\os{\circ}{T}_0})@>{\sus}>> \ol{\mathcal Q}.  
\end{CD}
\end{equation*} 
Set ${\mathcal Q}:=\ol{\mathcal Q}\times_{\ol{S(T)^{\nat}}}S(T)^{\nat}$. 
Let ${\mathfrak D}$ and ${\mathfrak E}$ be the log PD-envelopes of 
the immersion $(X_{\os{\circ}{T}_0},D_{\os{\circ}{T}_0}) \os{\subset}{\lo} {\mathcal P}$ 
and $(Y_{\os{\circ}{T}_0},E_{\os{\circ}{T}_0}) \os{\subset}{\lo} {\mathcal Q}$ 
over $(S(T)^{\nat},{\mathcal J},\del)$, respectively.  
Let $\wt{g}\col {\mathcal P}\lo{\mathcal Q}$ be 
the base change morphism of $\ol{g}$ 
with respect to the morphism 
$S(T)^{\nat}\os{\sus}{\lo} \ol{S(T)^{\nat}}$.  
Let 
$g^{\rm PD}\col 
{\mathfrak D}\lo {\mathfrak E}$ 
be the natural morphism induced by $\wt{g}$. 
Let 
$$\os{\circ}{g}_{\rm crys} \col 
((\os{\circ}{X}_{T_0}/\os{\circ}{T})_{\rm crys},
{\mathcal O}_{\os{\circ}{X}_{T_0}/\os{\circ}{T}})\lo 
((\os{\circ}{Y}_{T_0}/\os{\circ}{T})_{\rm crys},
{\mathcal O}_{\os{\circ}{Y}_{T_0}/\os{\circ}{T}})$$  
be the induced morphism of ringed topoi by 
$\os{\circ}{g} \col \os{\circ}{X}_{T_0}\lo \os{\circ}{Y}_{T_0}$. 
Let $E$ (resp.~$F$) be a flat quasi-coherent crystal of  
${\mathcal O}_{\os{\circ}{X}_{T_0}/\os{\circ}{T}}$-modules  
(resp.~a flat quasi-coherent crystal of 
${\mathcal O}_{\os{\circ}{Y}_{T_0}/\os{\circ}{T}}$-modules). 
Assume that we are given 
a morphism $F\lo \os{\circ}{g}_{{\rm crys}*}(E)$. 
Let $({\mathcal F},\nabla)$ be 
the ${\mathcal O}_{\mathfrak D}$-module 
with integrable connection obtained in (\ref{eqn:nitpltd}) for $F$. 
By (\ref{prop:neoxeo}) (1), ${\mathcal Q}^{\rm ex}$ is 
a formal SNCL scheme 
$({\mathcal Y},{\mathcal E)}$ over $S(T)^{\nat}$ with a relative formal 
SNCD on ${\mathcal Y}$ over $(S(T)^{\nat},{\mathcal J},\del)$, respectively. 
%and the morphism 
%${\mathcal P}\lo{\mathcal Q}$ is equal to $({\mathcal X},{\mathcal D})\lo ({\mathcal Y},{\mathcal E})$ 
%over $(S(T)^{\nat},{\mathcal J},\del)$. 
Then we have the following morphism of trifiltered complexes: 
\begin{equation*}
({\mathcal F}\otimes_{{\mathcal O}_{{\mathcal Q}^{\rm ex}}}
{\Om}^{\bul}_{{\mathcal Q}^{\rm ex}/\os{\circ}{T}},P^{{\mathcal E}/{\mathcal Y}},P,P^{\mathcal Y})
\lo  
g^{\rm PD}_*(({\mathcal E}\otimes_{{\mathcal O}_{{\mathcal P}^{\rm ex}}}
{\Om}^{\bul}_{{\mathcal P}^{\rm ex}/\os{\circ}{T}},P^{{\mathcal D}/{\mathcal X}},P,P^{\mathcal X})). 
\tag{5.0.7}\label{eqn:keyxy}
\end{equation*}
Hence we have the following  morphism of bifiltered complexes: 
\begin{equation*}
(A_{\rm zar}({\mathcal Q}^{\rm ex}/S(T)^{\nat},{\mathcal F}),P^{{\mathcal E}/{\mathcal Y}},P)
\lo  
g^{\rm PD}_*((A_{\rm zar}({\mathcal P}^{\rm ex}/S(T)^{\nat},{\mathcal E}),P^{{\mathcal D}/{\mathcal X}},P)). 
\tag{5.0.8}\label{eqn:kepqyxy}
\end{equation*}

\par 
Now we come back to the situation in the beginning of this section. 
\par 
Let $(X'_{\os{\circ}{T}_0},D'_{\os{\circ}{T}_0})$ be 
the disjoint union of an affine open covering of 
$(X_{\os{\circ}{T}_0},D_{\os{\circ}{T}_0})$.   
Let $(X'_{\os{\circ}{T}_0},D'_{\os{\circ}{T}_0}) \os{\sus}{\lo} \ol{\mathcal P}{}'$ 
be an immersion into a log smooth scheme over $\ol{S(T)^{\nat}}$ 
(this immersion exists if 
each log affine subscheme of 
$X_{\os{\circ}{T}_0}$ is sufficiently small). 
Set $X_{{\os{\circ}{T}_0}n}:={\rm cosk}_0^{X_{\os{\circ}{T}_0}}(X'_{\os{\circ}{T}_0})_n$  
and 
$\ol{\mathcal P}_{n}
:={\rm cosk}_0^{\ol{S(T)^{\nat}}}(\ol{\mathcal P}{}')_n$ $(n\in{\mab N})$.
Then we have a natural immersion 
$(X_{\os{\circ}{T}_0 \bul},D_{\os{\circ}{T}_0 \bul}) \os{\sus}{\lo} \ol{\mathcal P}_{\bul}$ 
over $S_{\os{\circ}{T}_0}\os{\sus}{\lo} \ol{S(T)^{\nat}}$. 
Let $\ol{\mathfrak D}_{\bul}$ 
be the log PD-envelope of the immersion 
$(X_{\os{\circ}{T}_0\bul},D_{\os{\circ}{T}_0\bul}) \os{\sus}{\lo} \ol{\mathcal P}_{\bul}$ 
over $(\os{\circ}{T},{\mathcal J},\gam)$. 
Set ${\mathfrak D}_{\bul}:=\ol{\mathfrak D}_{\bul}
\times_{{\mathfrak D}(\ol{S(T)^{\nat}})}S(T)^{\nat}$. 
Let $f_{\bul}\col X_{\os{\circ}{T}_0\bul} \lo S(T)^{\nat}$ 
be the structural morphism. 
Set ${\mathcal P}_{\bul}:=\ol{\mathcal P}_{\bul}\times_{\ol{S(T)^{\nat}}}S(T)^{\nat}$.  
We have a natural immersion $X_{\os{\circ}{T}_0\bul} \os{\sus}{\lo} {\mathcal P}_{\bul}$ 
over $S_{\os{\circ}{T}_0}\os{\sus}{\lo} S(T)^{\nat}$. 
Denote $\ol{\mathcal P}{}^{\rm ex}_{\bul}=(\ol{\mathcal X}_{\bul},\ol{\mathcal D}_{\bul})$, where 
$(\ol{\mathcal X}_{\bul},\ol{\mathcal D}_{\bul})$ is a simplicial strictly semistable formal scheme 
with a simplicial relative SNCD over $\ol{S(T)^{\nat}}$. 
Set ${\mathcal P}^{\rm ex}_{\bul}:=\ol{\mathcal P}{}^{\rm ex}_{\bul}\times_{\ol{S(T)^{\nat}}}S(T)^{\nat}$ 
and $({\mathcal X}_{\bul},{\mathcal D}_{\bul}):=(\ol{\mathcal X}_{\bul},\ol{\mathcal D}_{\bul})
\times_{\ol{S(T)^{\nat}}}S(T)^{\nat}
(={\mathcal P}^{\rm ex}_{\bul})$.  
Let $E^{\bul}$ be the flat quasi-coherent crystal 
of ${\mathcal O}_{\os{\circ}{X}_{T_0\bul}/\os{\circ}{T}}$-modules obtained by $E$. 
Let 
\begin{align*} 
\eps_{(X_{\os{\circ}{T}_0\bul},D_{\os{\circ}{T}_0\bul})/\os{\circ}{T}}
\col 
(((X_{\os{\circ}{T}_0\bul},D_{\os{\circ}{T}_0\bul})/\os{\circ}{T})_{\rm crys},
{\mathcal O}_{(X_{\os{\circ}{T}_0\bul},D_{\os{\circ}{T}_0\bul})/\os{\circ}{T}}) 
\lo 
((\os{\circ}{X}_{T_0\bul}/\os{\circ}{T})_{\rm crys},
{\mathcal O}_{\os{\circ}{X}_{T_0\bul}/\os{\circ}{T}}) 
\tag{5.0.9}\label{eqn:olpbpc}
\end{align*} 
be the morphism of ringed topoi induced by the morphism 
$\eps_{(X_{\os{\circ}{T}_0\bul},D_{\os{\circ}{T}_0\bul})
/\os{\circ}{T}_0}\col (X_{\os{\circ}{T}_0\bul},D_{\os{\circ}{T}_0\bul})\lo \os{\circ}{X}_{T_0\bul}$ 
over $\os{\circ}{T}$ forgetting the log structure of 
$(X_{\os{\circ}{T}_0\bul},D_{\os{\circ}{T}_0\bul})$. 
Let $(\ol{\mathcal E}{}^{\bul},\ol{\nabla}{}^{\bul})$ 
be the quasi-coherent ${\mathcal O}_{\ol{\mathfrak D}_{\bul}}$-module 
with the integrable connection 
associated to  
$\eps^*_{(X_{\os{\circ}{T}_0\bul},D_{\os{\circ}{T}_0\bul})/\os{\circ}{T}}(E^{\bul})$: 
$\ol{\nabla}{}^{\bul}\col \ol{\mathcal E}{}^{\bul}\lo 
\ol{\mathcal E}{}^{\bul}\otimes_{{\mathcal O}_{\ol{\mathcal P}_{\bul}}}
\Om^1_{\ol{\mathcal P}_{\bul}/\os{\circ}{T}}$. 
Set ${\mathcal E}^{\bul}:={\mathcal O}_{{\mathfrak D}_{\bul}}
\otimes_{{\mathcal O}_{\ol{\mathfrak D}_{\bul}}}\ol{\mathcal E}{}^{\bul}$.  
The connection $\ol{\nabla}{}^{\bul}$ 
induces the following integrable connection 
\begin{equation*} 
\nabla^{\bul} \col {\mathcal E}^{\bul} 
\lo {\mathcal E}^{\bul}\otimes_{{\mathcal O}_{{\mathcal P}_{\bul}}}
{\Om}^1_{{\mathcal P}_{\bul}/\os{\circ}{T}} 
=
{\mathcal E}^{\bul}\otimes_{{\mathcal O}_{{\mathcal P}^{\rm ex}_{\bul}}}
{\Om}^1_{{\mathcal P}^{\rm ex}_{\bul}/\os{\circ}{T}}. 
\tag{5.0.10}\label{eqn:nidopltd}
\end{equation*}  
%by the argument in \cite[pp.~41--42]{nhir}. 
By (\ref{prop:neoxeo}) (1),  
${\mathcal P}^{\rm ex}_{\bul}$ is equal to a simplicial formal SNCL scheme 
$({\mathcal X}_{\bul},{\mathcal D}_{\bul})$ with a simplicial relative SNCD over $S(T)^{\nat}$. 
Hence we have the following cosimplicial trifiltered complex by (\ref{eqn:keyxy}): 
\begin{equation*} 
({\mathcal E}^{\bul}\otimes_{{\mathcal O}_{{\mathcal P}^{\rm ex}_{\bul}}}
{\Om}^{\bul}_{{\mathcal P}^{\rm ex}_{\bul}
/\os{\circ}{T}},P^{{\mathcal D}_{\bul}/{\mathcal X}_{\bul}},P,P^{{\mathcal X}_{\bul}}).   
\tag{5.0.11}\label{cd:afil} 
\end{equation*}

Set  
\begin{equation*} 
A_{\rm zar}({\mathcal P}^{\rm ex}_{\bul}/S(T)^{\nat},{\mathcal E}^{\bul})^{ij} 
:=({\mathcal E}^{\bul}
\otimes_{{\mathcal O}_{{\mathcal P}^{\rm ex}_{\bul}}}
{\Om}^{i+j+1}_{{\mathcal P}^{\rm ex}_{\bul}/\os{\circ}{T}})/P^{{\mathcal X}_{\bul}}_j 
\quad (i,j \in {\mab N}). 
\tag{5.0.12}\label{cd:adef} 
\end{equation*}   

\begin{rema}\label{rema:estzmis1}
It seems to me that the $\infty$-adic analogue of (\ref{cd:adef}) in the case of 
the trivial coefficient is different from the definition \cite[(5.5)]{stz}
because the definition of 
$P^{{\mathcal X}_{\bul}}$ has no poles with respect to ${\mathcal D}_{\bul}$. 
%different from the definition $W^Y$ in [loc.~cit., p.~121]; 
%our definition of $P^{{\mathcal X}_{\bul}}$ has all poles with respect to ${\mathcal D}_{\bul}$; 
It seems to me that 
the definition of $W(Y)$ has poles with respect to $D$ in $\Om^{\bul}_X(\log (D+Y))$ 
in [loc.~cit.]. (See also [loc.~cit., (5.4)].) 
%The $\infty$-adic analogue of (\ref{cd:adef}) in the case of 
%the trivial coefficient is the definition in \cite[p.~121]{ezth} 
%It seems to me that the definition $W(Y)$ in \cite[(5.5)]{stz} is incorrect, at least unclear. 
The $\infty$-adic analogue of (\ref{cd:adef}) in the case of 
the trivial coefficient is the definition in \cite[p.~122, (3.3)]{ezth}. 
\end{rema} 

We consider the following boundary morphisms of 
the following double complex: 
\begin{equation*}
\begin{CD}
A_{\rm zar}({\mathcal P}^{\rm ex}_{\bul}/S(T)^{\nat},{\mathcal E}^{\bul})^{i,j+1} @.  \\ 
@A{\theta_{{\mathcal P}^{\rm ex}_{\bul}/\os{\circ}{T}} \wedge}AA 
@. \\
A_{\rm zar}({\mathcal P}^{\rm ex}_{\bul}/S(T)^{\nat},{\mathcal E}^{\bul})^{ij} 
@>{-\nabla}>> A_{\rm zar}({\mathcal P}^{\rm ex}_{\bul}/S(T)^{\nat},
{\mathcal E}^{\bul})^{i+1,j}.\\
\end{CD}
\tag{5.1.1}\label{cd:locstbd} 
\end{equation*}  
Then we have the cosimplicial double complex 
$A_{\rm zar}({\mathcal P}^{\rm ex}_{\bul}/S(T)^{\nat},{\mathcal E}^{\bul})^{\bul \bul}$. 
The double complex 
$A_{\rm zar}({\mathcal P}^{\rm ex}_{\bul}/S(T)^{\nat},{\mathcal E}^{\bul})^{\bul \bul}$ 
has filtrations $P^{{\mathcal D}_{\bul}/{\mathcal X}_{\bul}}=
\{P^{{\mathcal D}_{\bul}/{\mathcal X}_{\bul}}_k\}_{k \in {\mab Z}}$ and 
$P=\{P_k\}_{k \in {\mab Z}}$
defined by the following formulas: 
\begin{equation*} 
P^{{\mathcal D}_{\bul}/{\mathcal X}_{\bul}}_k
A_{\rm zar}({\mathcal P}^{\rm ex}_{\bul}/S(T)^{\nat},{\mathcal E}^{\bul})^{\bul \bul}
:=(\cdots P^{{\mathcal D}_{\bul}/{\mathcal X}_{\bul}}_k
A_{\rm zar}({\mathcal P}^{\rm ex}_{\bul}/S(T)^{\nat},{\mathcal E}^{\bul})^{ij} 
\cdots)    
\tag{5.1.2}\label{eqn:dbldad}
\end{equation*} 
and 
\begin{equation*} 
P_kA_{\rm zar}({\mathcal P}^{\rm ex}_{\bul}/S(T)^{\nat},{\mathcal E}^{\bul})^{\bul \bul}
:=(\cdots P_{2j+k+1}A_{\rm zar}({\mathcal P}^{\rm ex}_{\bul}/S(T)^{\nat},{\mathcal E}^{\bul})^{ij} 
\cdots).    
\tag{5.1.3}\label{eqn:dblad}
\end{equation*} 

\begin{rema}\label{rema:estzmis2}
The $\infty$-adic analogue of the filtration 
$P^{{\mathcal D}_{\bul}/{\mathcal X}_{\bul}}$ is different from $W^f$ in \cite[p.~122 (3.3)]{ezth} 
because $W^f$ has no poles along $Y$. 
The $\infty$-adic analogue of the filtration 
$P^{{\mathcal D}_{\bul}/{\mathcal X}_{\bul}}$ is the filtration $W(D)$ in \cite[p.~526]{stz}. 
It seems to me that the filtration $W^f$ in \cite[p.~122]{ezth} is wrong. 
\end{rema}

Let $(A_{\rm zar}({\mathcal P}^{\rm ex}_{\bul}/S(T)^{\nat},{\mathcal E}^{\bul}),
P^{{\mathcal D}_{\bul}/{\mathcal X}_{\bul}},P)$ 
be the cosimplicial single bifiltered complex of the cosimplicial bifiltered double complex 
$(A_{\rm zar}({\mathcal P}^{\rm ex}_{\bul}/S(T)^{\nat},{\mathcal E}^{\bul})^{\bul \bul},
P^{{\mathcal D}_{\bul}/{\mathcal X}_{\bul}},P)$.  
\par 
Let 
\begin{equation*} 
a^{(l,m)} \col \os{\circ}{X}{}^{(l)}_{T_0}\cap \os{\circ}{D}{}^{(m)}_{T_0}
 \lo \os{\circ}{X}_{T_0}  \quad (l,m\in {\mab N})
\tag{5.2.1}\label{eqn:atn}
\end{equation*} 
and 
\begin{equation*} 
a^{(l,m)}_{\bul} \col 
\os{\circ}{X}{}^{(l)}_{T_0\bul}\cap \os{\circ}{D}{}^{(m)}_{T_0\bul} \lo 
\os{\circ}{X}_{T_0\bul}  \quad 
(l,m\in {\mab N})
\tag{5.2.2}\label{eqn:antb}
\end{equation*} 
be the natural morphism of 
schemes and the natural morphism of simplicial schemes, respectively.  
Set $a^{(l)}:= a^{(l,0)}$ and $a^{(l)}_{\bul}:=a^{(l,0)}_{\bul}$.
Let $\os{\circ}{\mathfrak D}{}^{(l,m)}_{\bul}$ and ${\mathfrak D}_{\bul}$ 
be the (log) PD-envelopes of 
$\os{\circ}{X}{}^{(l)}_{T_0\bul}\cap \os{\circ}{D}{}^{(m)}_{T_0\bul} 
\os{\sus}{\lo} 
\os{\circ}{\mathcal X}{}^{(l)}_{\bul}\cap \os{\circ}{\mathcal D}{}^{(m)}_{\bul}$ over 
$(\os{\circ}{T},{\mathcal J},\del)$ 
and $X_{\os{\circ}{T}_0\bul}  \os{\sus}{\lo} {\mathcal P}_{\bul}$ over 
$(S(T)^{\nat},{\mathcal J},\del)$, 
respectively.   
%Set $\os{\circ}{\mathfrak D}{}^{(l)}_{\bul}(\os{\circ}{\mathcal X}_{\bul})
%:=\os{\circ}{\mathfrak D}{}^{(l,0)}_{\bul}$. 
Let ${\mathfrak D}^{(l)}_{\bul}(\os{\circ}{\mathcal X}_{\bul})$ be the log scheme 
whose underlying scheme is 
$\os{\circ}{\mathfrak D}{}^{(l,0)}_{\bul}$ 
and whose log structure is the pull-back of 
$(\os{\circ}{\mathcal X}{}^{(l)}_{\bul},\os{\circ}{\mathcal D}\vert_{\os{\circ}{\mathcal X}{}^{(l)}_{\bul}})$. 
Let 
\begin{equation*} 
b^{(l,m)}_{{\mathfrak D}_{\bul}} \col \os{\circ}{\mathfrak D}{}^{(l,m)}_{\bul}
 \lo \os{\circ}{\mathfrak D}_{\bul}  \quad (l,m\in {\mab N})
\tag{5.2.3}\label{eqn:adtn}
\end{equation*} 
be the natural morphism. 
Set $b^{(l)}_{{\mathfrak D}_{\bul}}:=b^{(l,0)}_{{\mathfrak D}_{\bul}}$. 
Let 
\begin{equation*} 
a^{(l,m)}_{{\rm crys}} \col 
((\os{\circ}{X}{}^{(l)}_{T_0}\cap \os{\circ}{D}{}^{(m)}_{T_0})/\os{\circ}{T})_{\rm crys},
{\mathcal O}_{\os{\circ}{X}{}^{(l)}_{T_0}\cap \os{\circ}{D}{}^{(m)}_{T_0}/\os{\circ}{T}}) 
\lo ((\os{\circ}{X}_{T_0}/\os{\circ}{T})_{\rm crys},
{\mathcal O}_{\os{\circ}{X}_{T_0}/\os{\circ}{T}})  
\quad (l,m\in {\mab N})
\tag{5.2.4}\label{eqn:atnt}
\end{equation*} 
and 
\begin{equation*} 
a^{(l,m)}_{\bul,{\rm crys}} \col 
((\os{\circ}{X}{}^{(l)}_{T_0\bul}\cap 
\os{\circ}{D}{}^{(m)}_{T_0\bul})/\os{\circ}{T})_{\rm crys},
{\mathcal O}_{\os{\circ}{X}{}^{(k)}_{T_0\bul}
/\os{\circ}{T}})  \lo 
((\os{\circ}{X}_{T_0\bul}/\os{\circ}{T})_{\rm crys},
{\mathcal O}_{\os{\circ}{X}_{T_0\bul}/\os{\circ}{T}}) 
\quad (l,m\in {\mab N})
\tag{5.2.5}\label{eqn:antbt}
\end{equation*} 
be the morphism of ringed topoi 
obtained by (\ref{eqn:atn}) and (\ref{eqn:antb}), respectively. 
Let 
\begin{equation*} 
c^{(k)} \col D^{(k)}_{\os{\circ}{T}_0} \lo X_{\os{\circ}{T}_0}  \quad (k\in {\mab N})
\tag{5.2.6}\label{eqn:addtn}
\end{equation*} 
and 
\begin{equation*} 
c^{(k)}_{\bul} \col D^{(k)}_{\os{\circ}{T}_0\bul} \lo X_{\os{\circ}{T}_0\bul}  \quad 
(k\in {\mab N})
\tag{5.2.7}\label{eqn:antddb}
\end{equation*} 
be the natural morphism of 
log schemes and the natural morphism of simplicial log schemes, 
respectively.  
Let ${\mathfrak D}_{\bul}({\mathcal D}^{(k)}_{\bul})$ be the log PD-envelope of 
$D^{(k)}_{\os{\circ}{T}_0\bul} \os{\sus}{\lo} {\mathcal D}^{(k)}_{\bul}$ over $(S(T)^{\nat},{\mathcal J},\del)$.   
Let 
\begin{equation*} 
d^{(k)}_{{\mathfrak D}_{\bul}}({\mathcal D}_{\bul})\col  
{\mathfrak D}_{\bul}^{(k)}({\mathcal D}_{\bul})
\lo {\mathfrak D}_{\bul}  \quad 
(k\in {\mab N})
\tag{5.2.8}\label{eqn:antdb}
\end{equation*} 
be the natural morphism of log schemes.   
Let 
\begin{equation*} 
c^{(l)}_{{\rm crys}} \col ((D^{(l)}_{\os{\circ}{T}_0}
/S(T)^{\nat})_{\rm crys},{\mathcal O}_{D^{(l)}_{T_0}/S(T)^{\nat}}) 
\lo ((X_{\os{\circ}{T}_0}/S(T)^{\nat})_{\rm crys},{\mathcal O}_{X_{\os{\circ}{T}_0}/S(T)^{\nat}})  
\quad (l,m\in {\mab N})
\tag{5.2.9}\label{eqn:abtnt}
\end{equation*} 
and 
\begin{equation*} 
c^{(l)}_{\bul{\rm crys}} \col ((D^{(l)}_{\os{\circ}{T}_0\bul}/S(T)^{\nat})_{\rm crys},
{\mathcal O}_{D^{(l)}_{T_0\bul}/S(T)^{\nat}}) 
\lo ((X_{\os{\circ}{T}_0\bul}/\os{\circ}{T})_{\rm crys},{\mathcal O}_{X_{\os{\circ}{T}_0\bul}/S(T)^{\nat}})  
\quad (l,m\in {\mab N}) 
\tag{5.2.10}\label{eqn:antbbt}
\end{equation*} 
be the morphisms of ringed topoi 
induced by (\ref{eqn:addtn}) and (\ref{eqn:antddb}), 
respectively.

\par 
The following is only a constant simplicial SNCL version 
with a constant relative SNCD version of \cite[(4.14)]{nh3}. 

\begin{lemm}\label{lemm:knit}
Let $k$ be nonnegative integers.  
For the morphism 
$g \col {\mathcal P}^{\rm ex}_{n}=({\mathcal X}_n,{\mathcal D}_n)\lo 
{\mathcal P}^{\rm ex}_{n'}=({\mathcal X}_{n'},{\mathcal D}_{n'})$ 
corresponding to a morphism $[n']\lo [n]$ in $\Del$, 
there exists morphisms  
$\os{\circ}{g}{}^{(k)}_{\mathcal X}
\col \os{\circ}{\mathcal X}{}^{(k)}_{n}\lo \os{\circ}{\mathcal X}{}^{(k)}_{n'}$ 
$(k\in {\mab N})$ 
and 
$g^{(k)}_{\mathcal D}
\col {\mathcal D}^{(k)}_{n}\lo {\mathcal D}{}^{(k)}_{n'}$ 
$(k\in {\mab N})$ 
over the morphism $\os{\circ}{\mathcal X}_n\lo \os{\circ}{\mathcal X}_{n'}$ and 
${\mathcal X}_n\lo {\mathcal X}_{n'}$, respectively. 
Consequently $\{\os{\circ}{\mathcal X}{}^{(k)}_{n}\}_{n\in {\mab N}}$,   
$\{\os{\circ}{\mathcal D}{}^{(k)}_n\}_{n\in {\mab N}}$ and
$\{{\mathcal D}^{(k)}_{n}\}_{n\in {\mab N}}$ give us the $($log$)$ simplicial formal schemes 
$\os{\circ}{\mathcal X}{}^{(l)}_{\bul}\cap \os{\circ}{\mathcal D}{}^{(m)}_{\bul}$ $(l,m\in {\mab N})$ 
and ${\mathcal D}{}^{(k)}_{\bul}$, respectively.   
\end{lemm} 
\begin{proof} 
This immediately follows from the proofs of \cite[(4.14)]{nh3} and \cite[(1.4.1)]{nb}. 
\end{proof}

\begin{lemm}\label{lemm:grc} 
Let $k$ be a nonnegative integer. 
Then the following hold$:$
\par 
$(1)$ There exists an isomorphism 
\begin{align*} 
& ({\rm gr}^{P^{{\mathcal D}_{\bul}/{\mathcal X}_{\bul}}}_kA_{\rm zar}
({\mathcal P}^{\rm ex}_{\bul}/S(T)^{\nat},{\mathcal E}^{\bul}),P^{{\mathcal X}_{\bul}})\os{\sim}{\lo} 
\tag{5.4.1}\label{ali:axdd}\\
&d^{(k)}_{{\mathfrak D}_{\bul}}({\mathcal D}_{\bul})_*
((A_{\rm zar}({\mathcal D}^{(k)}_{\bul}/S(T)^{\nat},{\mathcal E}^{\bul})\otimes_{\mab Z}
\varpi^{(k)}_{\rm zar}(\os{\circ}{D}_{T_0\bul}/\os{\circ}{T})),P^{{\mathcal D}^{(k)}_{\bul}})[-k]
\end{align*}
in $C^+(f^{-1}_{\bul}({\mathcal O}_T))$.
\par 
$(2)$ There exists an isomorphism 
\begin{align*} 
{\rm gr}_k^PA_{\rm zar}
({\mathcal P}^{\rm ex}_{\bul}/S(T)^{\nat},{\mathcal E}^{\bul})
\os{\sim}{\lo} & 
\bigoplus^k_{k'=-\infty}
\bigoplus_{j\geq \max\{-k',0\}}
({\mathcal E}^{\bul}\otimes_{{\mathcal O}_{{\mathcal X}_{\bul}}} 
\Om^{\bul}_{\os{\circ}{\mathcal X}{}^{(2j+k')}_{\bul}
\cap \os{\circ}{\mathcal D}{}^{(k-k')}_{\bul}/\os{\circ}{T}} 
\tag{5.4.2}\label{eqn:axd}\\ 
{} & \quad \otimes_{\mab Z}
\vp^{(2j+k'),(k-k')}_{\rm zar}
((\os{\circ}{X}_{T_0\bul},\os{\circ}{D}_{T_0\bul})/\os{\circ}{T}),\nabla)[-2j-k] 
\end{align*}
in $C^+(f^{-1}_{\bul}({\mathcal O}_T))$.
\par 
$(3)$ There exists an isomorphism 
\begin{align*} 
{\rm gr}_{k'}^P{\rm gr}_k^{P^{{\mathcal D}_{\bul}/{\mathcal X}_{\bul}}}
A_{\rm zar}
({\mathcal P}^{\rm ex}_{\bul}/S(T)^{\nat},{\mathcal E}^{\bul})
\os{\sim}{\lo} &\bigoplus_{j\geq \max\{-(k'-k),0\}}
({\mathcal E}^{\bul}\otimes_{{\mathcal O}_{{\mathcal X}_{\bul}}}
b^{(2j+k'-k),(k)}_{*}(\Om^{\bul }_{\os{\circ}{\mathcal X}{}^{(2j+k'-k)}_{\bul}
\cap \os{\circ}{\mathcal D}{}^{(k)}_{\bul}/\os{\circ}{T}} \tag{5.4.3}\label{eqn:ana}\\ 
{} & \quad \otimes_{\mab Z}
\vp^{(2j+k'-k),(k)}_{\rm et}
((\os{\circ}{X}_{T_0\bul},\os{\circ}{D}_{T_0\bul})/\os{\circ}{T}))[-2j-k'] 
\end{align*}  
in $C^+(f^{-1}_{\bul}({\mathcal O}_T))$. 
\end{lemm} 
\begin{proof} 
(1), (2), (3): 
These follow from (\ref{lemm:knit}) and (\ref{prop:grem}). 
\end{proof}

\par 
%Set $U_0:=S_{\os{\circ}{T}_0}$ or $T_0$ 
%and $U:=S(T)^{\nat}$ or $T$,  respectively. 
%When $U=T$, we always assume that $T$ is restrictively hollow. 
\par 
Let 
\begin{align*} 
\pi_{{\rm zar}} \col &((X_{\os{\circ}{T}_0\bul})_{\rm zar},f^{-1}_{\bul}({\mathcal O}_T))  \lo 
((X_{\os{\circ}{T}_0})_{\rm zar},f^{-1}({\mathcal O}_T)) \tag{5.4.4}\label{ali:pzd} \\
\end{align*} 
be the natural morphism of ringed topoi.  
%\par 
%Let $p_{T_0}\col X_{\bul \,T_0}\lo X_{\bul \,\os{\circ}{T}_0}$ 
%(resp.~$p_{\os{\circ}{T}_0}\col X_{\bul \,T_0}
%\lo X_{\bul \,\os{\circ}{T}_0}$)
%be the first projection over $T_0\lo S_{\os{\circ}{T}_0}$ 
%(resp.~$\os{\circ}{T}_0$).   
%Then $\os{\circ}{p}_{T_0}=\os{\circ}{p}_{\os{\circ}{T}_0}
%={\rm id}_{\os{\circ}{X}_{\bul \,T_0}}$. 
%Let $$p_{T{\rm crys}}\col 
%((X_{\bul \,T_0}/T)_{\rm crys},{\mathcal O}_{X_{\bul \,T_0}/T})
%\lo 
%((X_{\bul \,\os{\circ}{T}_0}/S(T))_{\rm crys},
%{\mathcal O}_{X_{\bul \,\os{\circ}{T}_0}/S(T)})$$  
%and 
%$$p_{\os{\circ}{T}{\rm crys}}
%\col ((X_{\bul \,T_0}/\os{\circ}{T})_{\rm crys},
%{\mathcal O}_{X_{\bul \,T_0}/\os{\circ}{T}})
%\lo 
%((X_{\bul \,\os{\circ}{T}_0}/\os{\circ}{T})_{\rm crys},
%{\mathcal O}_{X_{\bul \,\os{\circ}{T}_0}/\os{\circ}{T}})$$  
%be the induced morphism by  $p_{T_0}$ 
%and $p_{\os{\circ}{T}_0}$, respectively. 
Let 
\begin{align*} 
u_{(X_{\os{\circ}{T}_0},D_{\os{\circ}{T}_0})/S(T)^{\nat}} \col 
(((X_{\os{\circ}{T}_0},D_{\os{\circ}{T}_0})/S(T)^{\nat})_{\rm crys},
{\mathcal O}_{(X_{\os{\circ}{T}_0},D_{\os{\circ}{T}_0})/S(T)^{\nat}}) 
&\lo ((\os{\circ}{X}_{T_0})_{\rm zar},f^{-1}({\mathcal O}_T)) \tag{5.4.5}\label{eqn:prxudef}
\end{align*}   
and 
\begin{align*} 
u_{\os{\circ}{X}_{T_0}/\os{\circ}{T}}
\col 
((\os{\circ}{X}_{T_0}/\os{\circ}{T})_{\rm crys},
{\mathcal O}_{\os{\circ}{X}_{T_0}/\os{\circ}{T}}) 
&\lo ((\os{\circ}{X}_{T_0})_{\rm zar},f^{-1}({\mathcal O}_T))
\tag{5.4.6}\label{eqn:pxuxxef}
\end{align*}   
be the natural projections. 
Let 
\begin{align*} 
\eps_{(X_{\os{\circ}{T}_0},D_{\os{\circ}{T}_0})/S(T)^{\nat}} \col 
(((X_{\os{\circ}{T}_0},D_{\os{\circ}{T}_0})/S(T)^{\nat})_{\rm crys},
{\mathcal O}_{(X_{\os{\circ}{T}_0},D_{\os{\circ}{T}_0})/S(T)^{\nat}}) 
\lo ((\os{\circ}{X}_{T_0}/\os{\circ}{T})_{\rm crys},
{\mathcal O}_{\os{\circ}{X}_{T_0}/\os{\circ}{T}})
\tag{5.4.7}\label{eqn:preoxdef}\\
\end{align*}   
be the morphism forgetting the log structures of 
$(X_{\os{\circ}{T}_0},D_{\os{\circ}{T}_0})$ and $S(T)^{\nat}$. 

\begin{prop}\label{prop:tefc}
There exists the following isomorphism 
\begin{align*} 
\theta:=\theta_{(X_{\os{\circ}{T}_0},D_{\os{\circ}{T}_0})/S(T)^{\nat}}
\wedge \col 
&Ru_{(X_{\os{\circ}{T}_0},D_{\os{\circ}{T}_0})/S(T)^{\nat}*}
(\eps^*_{(X_{\os{\circ}{T}_0},D_{\os{\circ}{T}_0})/S(T)^{\nat}}(E))
\tag{5.5.1}\label{eqn:uz} \\
& \os{\sim}{\lo}R\pi_{{\rm zar}*}
(A_{\rm zar}({\mathcal P}^{\rm ex}_{\bul}/S(T)^{\nat},{\mathcal E}^{\bul}))
\end{align*} 
in $D^+(f^{-1}({\mathcal O}_T))$. 
This isomorphism is independent of the choice of 
an affine simplicial open covering of $X_{\os{\circ}{T}_0}$ 
and the choice of a simplicial immersion 
$X_{\os{\circ}{T}_0\bul} \os{\sus}{\lo} \ol{\mathcal P}_{\bul}$ over $\ol{S(T)^{\nat}}$.
In particular, the complex 
$R\pi_{{\rm zar}*}
(A_{\rm zar}({\mathcal P}^{\rm ex}_{\bul}/S(T)^{\nat},{\mathcal E}^{\bul}))$ 
is independent of the choices above. 
\end{prop}
\begin{proof} 
First we prove that the isomorphism (\ref{eqn:uz}) exists. 
Let 
\begin{align*} 
\pi_{{\rm crys}} 
\col &((X_{\os{\circ}{T}_0\bul},D_{\os{\circ}{T}_0\bul})/S(T)^{\nat})_{\rm crys},
{\mathcal O}_{(X_{\os{\circ}{T}_0\bul},D_{\os{\circ}{T}_0\bul})/S(T)^{\nat}})   
\tag{5.5.2}\label{eqn:pio}\\
&\lo (((X_{\os{\circ}{T}_0},D_{\os{\circ}{T}_0})/S(T)^{\nat})_{\rm crys},
{\mathcal O}_{(X_{\os{\circ}{T}_0},D_{\os{\circ}{T}_0})/S(T)^{\nat}}) \\
\end{align*} 
be the natural morphism of ringed topoi. 
Then, by the cohomological descent, 
we have 
$\eps^*_{(X_{\os{\circ}{T}_0},D_{\os{\circ}{T}_0})/S(T)^{\nat}}(E)=
R\pi_{{\rm crys}*}
(\eps^*_{(X_{\os{\circ}{T}_0\bul},D_{\os{\circ}{T}_0\bul})/S(T)^{\nat}}(E^{\bul}))$. 
Let 
\begin{align*} 
u_{(X_{\os{\circ}{T}_0\bul},D_{\os{\circ}{T}_0\bul})/S(T)^{\nat}} \col 
(((X_{\os{\circ}{T}_0\bul},D_{\os{\circ}{T}_0\bul})/S(T)^{\nat})_{\rm crys},
{\mathcal O}_{(X_{\os{\circ}{T}_0\bul},D_{\os{\circ}{T}_0\bul})/S(T)^{\nat}}) 
&\lo ((X_{\os{\circ}{T}_0\bul})_{\rm zar},f^{-1}_{\bul}({\mathcal O}_T))
\tag{5.5.3}\label{eqn:uudef}\\
\end{align*}     
be the natural projection. 
Then $u_{(X_{\os{\circ}{T}_0},D_{\os{\circ}{T}_0})/S(T)^{\nat}}\circ \pi_{{\rm crys}}
=\pi_{\rm zar}\circ u_{(X_{\os{\circ}{T}_0\bul},D_{\os{\circ}{T}_0\bul})/S(T)^{\nat}}$. 
Hence 
%and the formula in the end of \cite[(1.7)]{klog1}, 
we have the following formula by the log Poincar\'{e} lemma: 
\begin{align*} 
&Ru_{(X_{\os{\circ}{T}_0},D_{\os{\circ}{T}_0})/S(T)^{\nat}*}
(\eps^*_{(X_{\os{\circ}{T}_0},D_{\os{\circ}{T}_0})/S(T)^{\nat}}(E))\tag{5.5.4}\label{eqn:xds}\\
&=
Ru_{(X_{\os{\circ}{T}_0},D_{\os{\circ}{T}_0})/S(T)^{\nat}*}
R\pi_{{\rm crys}*}
(\eps^*_{(X_{\os{\circ}{T}_0\bul},D_{\os{\circ}{T}_0\bul})/S(T)^{\nat}}(E^{\bul})) \\
&=R\pi_{{\rm zar}*}
Ru_{(X_{\os{\circ}{T}_0\bul},D_{\os{\circ}{T}_0\bul})/S(T)^{\nat}*}
(\eps^*_{(X_{\os{\circ}{T}_0\bul},D_{\os{\circ}{T}_0\bul})/S(T)^{\nat}}(E^{\bul})) \\
&=R\pi_{{\rm zar}*}({\mathcal E}^{\bul}\otimes_{{\mathcal O}_{{\mathcal P}^{\rm ex}_{\bul}}}
{\Om}^{\bul}_{{\mathcal P}^{\rm ex}_{\bul}/S(T)^{\nat}}). 
\end{align*}  
Since ${\mathcal E}^{\bul}$ is a flat
${\mathcal O}_{{\mathfrak D}_{\bul}}$-module, 
it suffices to prove that the natural morphism 
\begin{align*} 
\theta_{{\mathcal P}^{\rm ex}_{\bul}/\os{\circ}{T}} \wedge  & 
\col 
{\mathcal O}_{{\mathfrak D}_{\bul}}
\otimes_{{\mathcal O}_{{\mathcal P}^{\rm ex}_{\bul}}}
{\Om}^i_{{\mathcal P}^{\rm ex}_{\bul}/S(T)^{\nat}}
\lo  \\
&\{({\mathcal O}_{{\mathfrak D}_{\bul}}
\otimes_{{\mathcal O}_{{\mathcal P}^{\rm ex}_{\bul}}}
{\Om}^{i+1}_{{\mathcal P}^{\rm ex}_{\bul}/\os{\circ}{T}}
/P^{{\mathcal X}_{\bul}}_0  
\os{\theta_{{\mathcal P}^{\rm ex}_{\bul}/\os{\circ}{T}}}{\lo} 
{\mathcal O}_{{\mathfrak D}_{\bul}}
\otimes_{{\mathcal O}_{{\mathcal P}^{\rm ex}_{\bul}}}
{\Om}^{i+2}_{{\mathcal P}^{\rm ex}_{\bul}/\os{\circ}{T}}/P^{{\mathcal X}_{\bul}}_1 
\os{\theta_{{\mathcal P}^{\rm ex}_{\bul}/\os{\circ}{T}}}{\lo}\cdots )\}
\end{align*} 
is a quasi-isomorphism. 
In (\ref{prop:pbe}) we have already proved this. 
\par 
Next we prove that the isomorphism is independent of 
the choices in (\ref{prop:tefc}).   
\par 
Let $(X''_{\os{\circ}{T}_0},D''_{\os{\circ}{T}_0})$ be another disjoint union 
of an affine simplicial open covering of 
$(X_{\os{\circ}{T}_0},D_{\os{\circ}{T}_0})$. 
Set $(X'''_{\os{\circ}{T}_0},D'''_{\os{\circ}{T}_0}):=
(X'_{\os{\circ}{T}_0},D'_{\os{\circ}{T}_0})
\times_{(X_{\os{\circ}{T}_0},D_{\os{\circ}{T}_0})}(X''_{\os{\circ}{T}_0},D''_{\os{\circ}{T}_0})$. 
Then we have the log scheme 
$(X'''_{\os{\circ}{T}_0},D'''_{\os{\circ}{T}_0})$ which is 
the disjoint union of 
the members of a (not necessarily affine)  simplicial open covering 
of $(X_{\os{\circ}{T}_0},D_{\os{\circ}{T}_0})/S_{\os{\circ}{T}_0}$ 
fitting into the following commutative diagram$:$
\begin{equation*}
\begin{CD}
(X'''_{\os{\circ}{T}_0},D'''_{\os{\circ}{T}_0}) @>>> (X''_{\os{\circ}{T}_0},D''_{\os{\circ}{T}_0})\\
@VVV @VVV  \\
(X'_{\os{\circ}{T}_0},D'_{\os{\circ}{T}_0})@>>> (X_{\os{\circ}{T}_0},D_{\os{\circ}{T}_0}).
\end{CD}
\tag{5.5.5}\label{cd:celcxcov}
\end{equation*} 
Set $(X_{\os{\circ}{T}_0n},D_{\os{\circ}{T}_0n})
:={\rm cosk}^{(X_{\os{\circ}{T}_0},D_{\os{\circ}{T}_0})}_0((X'_{\os{\circ}{T}_0},D'_{\os{\circ}{T}_0}))_n$ and 
$(X'_{\os{\circ}{T}_0n},D'_{\os{\circ}{T}_0n})
:={\rm cosk}^{(X_{\os{\circ}{T}_0},D_{\os{\circ}{T}_0})}_0((X''_{\os{\circ}{T}_0},D''_{\os{\circ}{T}_0}))_n$.   
Let 
$(X''_{\os{\circ}{T}_0\bul},D''_{\os{\circ}{T}_0\bul})
\os{\sus}{\lo}\ol{\mathcal P}{}'_{\bul}$ be another immersion in (\ref{eqn:eipxd}).  
Then, by (\ref{cd:celcxcov}) and considering the fiber product 
$\ol{\mathcal P}_{\bul}\times_{\ol{S(T)^{\nat}}}\ol{\mathcal P}{}'_{\bul}$, 
we may have the following commutative diagram 
\begin{equation*} 
\begin{CD} 
(X_{\os{\circ}{T}_0\bul},D_{\os{\circ}{T}_0\bul}) 
@>{\sus}>> \ol{\mathcal P}_{\bul} \\ 
@VVV @VVV \\ 
(X'_{\os{\circ}{T}_0\bul},D'_{\os{\circ}{T}_0\bul}) 
@>{\sus}>> \ol{\mathcal P}{}'_{\bul}.  
\end{CD} 
\end{equation*} 
Set ${\mathcal P}'_{\bul}:=\ol{\mathcal P}{}'_{\bul}\times_{\ol{S(T)^{\nat}}}S(T)^{\nat}$. 
Let ${\mathcal P}{}'{}^{\rm ex}_{\!\!\!\bul}$  
be the exactification of 
the immersion 
$X'_{\os{\circ}{T}_0\bul} \os{\sus}{\lo}{\mathcal P}'_{\bul}$. 
Let ${\mathcal E}'{}^{\bul}
\otimes_{{\mathcal O}_{{\mathcal P}'^{\rm ex}_{\! \bul}}}
{\Om}^{\bul}_{{\mathcal P}'{}^{\rm ex}_{\! \bul}/\os{\circ}{T}}$ 
be an analogous complex 
to 
${\mathcal E}^{\bul}
\otimes_{{\mathcal O}_{{\mathcal P}^{\rm ex}_{\bul}}}
{\Om}^{\bul}_{{\mathcal P}^{\rm ex}_{\bul}/\os{\circ}{T}}$ for ${\mathcal P}'_{\bul}$.   
Then we have the following morphism 
\begin{equation*} 
R\pi_{{\rm zar}*}
(A_{\rm zar}({\mathcal P}'{}^{\rm ex}_{\! \! \! \bul}/S(T)^{\nat}
,{\mathcal E}'{}^{\bul}))
\lo 
R\pi_{{\rm zar}*}
(A_{\rm zar}({\mathcal P}^{\rm ex}_{\bul}/S(T)^{\nat},{\mathcal E}^{\bul})). 
\tag{5.5.6}\label{eqn:flnfap}
\end{equation*} 
This morphism fits into the following commutative diagram 
\begin{equation*} 
\begin{CD} 
 R\pi_{{\rm zar}*}
(A_{\rm zar}({\mathcal P}'{}^{\rm ex}_{\! \! \!\bul}/S(T)^{\nat},{\mathcal E}'{}^{\bul}))
@>>> 
R\pi_{{\rm zar}*}
(A_{\rm zar}({\mathcal P}^{\rm ex}_{\bul}/S(T)^{\nat},{\mathcal E}^{\bul}))\\
@A{R\pi_{{\rm zar}*}(\theta_{{\mathcal P}'{}^{\rm ex}_{\! \! \! \bul}/\os{\circ}{T}} \wedge)}A{\simeq}A 
@A{\simeq}A{R\pi_{{\rm zar}*}(\theta_{{\mathcal P}^{\rm ex}_{\bul}/\os{\circ}{T}} \wedge)}A \\
R\pi_{{\rm zar}*}
({\mathcal E}'{}^{\bul}\otimes_{{\mathcal O}_{{\mathcal P}'{}^{\rm ex}_{\!\!\!\bul}}}
{\Om}^{\bul}_{{\mathcal P}'{}^{\rm ex}_{\!\!\!\bul}/S(T)^{\nat}})
@=
R\pi_{{\rm zar}*}
({\mathcal E}^{\bul}\otimes_{{\mathcal O}_{{\mathcal P}^{\rm ex}_{\bul}}}
{\Om}^{\bul}_{{\mathcal P}^{\rm ex}_{\bul}/S(T)^{\nat}}). 
\end{CD}
\tag{5.5.7}\label{cd:fnfap}
\end{equation*} 
This diagram tells us the desired independence 
of the choices in (\ref{prop:tefc}).  
\end{proof} 
\par 
Next, by using 
%(\ref{lemm:gcrgr}), 
(\ref{lemm:knit}), (\ref{eqn:ana}) and (\ref{prop:tefc}),   
we prove that the bifiltered complex 
$$R\pi_{{\rm zar}*}
((A_{\rm zar}({\mathcal P}^{\rm ex}_{\bul}/S(T)^{\nat},{\mathcal E}^{\bul}),
P^{{\mathcal D}_{\bul}/{\mathcal X}_{\bul}},P))$$  
depends only on $(X_{\os{\circ}{T}_0},D_{\os{\circ}{T}_0})/(S(T)^{\nat},{\mathcal J},\del)$:  

\begin{theo}\label{theo:indcr} 
The bifiltered complex 
$$R\pi_{{\rm zar}*}
((A_{\rm zar}({\mathcal P}^{\rm ex}_{\bul}/S(T)^{\nat},{\mathcal E}^{\bul}),
P^{{\mathcal D}_{\bul}/{\mathcal X}_{\bul}},P))
\in {\rm D}^+{\rm F}^2(f^{-1}({\mathcal O}_T))$$ 
is independent of the choice of 
an affine  simplicial open covering of 
$(X_{\os{\circ}{T}_0},D_{\os{\circ}{T}_0})$ 
and the choice of a simplicial immersion 
$(X_{\os{\circ}{T}_0\bul},D_{\os{\circ}{T}_0\bul}) \os{\sus}{\lo} 
\ol{\mathcal P}_{\bul}$ over $\ol{S(T)^{\nat}}$.  
\end{theo}
\begin{proof} 
%Let the notations be before this theorem. 
%Especially we use the notations in the proof of (\ref{prop:tefc}).  
Let the notations be as in the proof of (\ref{prop:tefc}). 
Denote ${\mathcal P}{}'^{\rm ex}_{\bul}=({\mathcal X}{}'_{\bul},{\mathcal D}{}'_{\bul})$, where 
$({\mathcal X}_{\bul},{\mathcal D}_{\bul})$ is a simplicial SNCL formal scheme 
with a simplicial relative SNCD over $S(T)^{\nat}$. 
Then we have the following morphism 
\begin{equation*} 
R\pi_{{\rm zar}*}
((A_{\rm zar}({\mathcal P}'{}^{\rm ex}_{\! \! \! \bul}/S(T)^{\nat}
,{\mathcal E}'{}^{\bul}),P^{{\mathcal D}'_{\bul}/{\mathcal X}'_{\bul}},P))
\lo 
R\pi_{{\rm zar}*}
((A_{\rm zar}({\mathcal P}^{\rm ex}_{\bul}/S(T)^{\nat},
{\mathcal E}^{\bul}),P^{{\mathcal D}_{\bul}/{\mathcal X}_{\bul}},P)). 
\tag{5.6.1}\label{eqn:flnflap}
\end{equation*} 
%Set 
%$$(R\pi_{{\rm zar}*}
%(A_{\rm zar}({\mathcal P}^{\rm ex}_{\bul}/S(T)^{\nat},{\mathcal E}^{\bul})),P^{D_{\os{\circ}{T}_0}},P)
%:=R\pi_{{\rm zar}*}
%((A_{\rm zar}({\mathcal P}^{\rm ex}_{\bul}/S(T)^{\nat},{\mathcal E}^{\bul}),P^{{\mathcal D}_{\bul}/
%{\mathcal X}_{\bul}},P)).$$
To prove that this is an isomorphism, it suffices to prove 
that the morphisms  
\begin{equation*} 
R\pi_{{\rm zar}*}
((P^{{\mathcal D}'_{\bul}/{\mathcal X}'_{\bul}}_k
\cap P_l)
A_{\rm zar}
({\mathcal P}'{}^{\rm ex}_{\! \! \! \bul}/S(T)^{\nat},{\mathcal E}'{}^{\bul}))
\lo 
R\pi_{{\rm zar}*}((P^{{\mathcal D}_{\bul}/{\mathcal X}_{\bul}}_k\cap P_l)
A_{\rm zar}({\mathcal P}^{\rm ex}_{\bul}/S(T)^{\nat},{\mathcal E}^{\bul})), 
\tag{5.6.2}\label{eqn:flapwp}
\end{equation*} 
\begin{equation*} 
R\pi_{{\rm zar}*}
((P^{{\mathcal D}'_{\bul}/{\mathcal X}'_{\bul}}_k
A_{\rm zar}
({\mathcal P}'{}^{\rm ex}_{\! \! \! \bul}/S(T)^{\nat},{\mathcal E}'{}^{\bul}))
\lo 
R\pi_{{\rm zar}*}((P^{{\mathcal D}_{\bul}/{\mathcal X}_{\bul}}_k
A_{\rm zar}({\mathcal P}^{\rm ex}_{\bul}/S(T)^{\nat},{\mathcal E}^{\bul}))
\tag{5.6.3}\label{eqn:flapwdp}
\end{equation*} 
and 
\begin{equation*} 
R\pi_{{\rm zar}*}((P_kA_{\rm zar}
({\mathcal P}'{}^{\rm ex}_{\! \! \! \bul}/S(T)^{\nat},{\mathcal E}'{}^{\bul}))
\lo 
R\pi_{{\rm zar}*}((P_k
A_{\rm zar}({\mathcal P}^{\rm ex}_{\bul}/S(T)^{\nat},{\mathcal E}^{\bul}))
\tag{5.6.4}\label{eqn:flanpwdp}
\end{equation*} 
are isomorphisms in $D^+(f^{-1}({\mathcal O}_T))$ by (\ref{prop:tefc}) and the definition of 
$ {\rm D}^+{\rm F}^2(f^{-1}({\mathcal O}_T))$ in \cite{nlf}.  
Because the questions that the morphisms above are isomorphisms are
local on $X_{\os{\circ}{T}_0}$, 
we may assume that $X_{\os{\circ}{T}_0}$ is affine, 
in particular, quasi-compact. 
Hence we may assume that the filtrations $P^D$ and $P$ are biregular 
and it suffices to prove 
that the morphism 
\begin{equation*} 
{\rm gr}_{k'}^P{\rm gr}_k^{P^{D_{\os{\circ}{T}_0}}}R\pi_{{\rm zar}*}
(A_{\rm zar}
({\mathcal P}'{}^{\rm ex}_{\! \! \! \bul}/S(T)^{\nat},{\mathcal E}'{}^{\bul}))
\lo 
{\rm gr}_{k'}^P{\rm gr}_k^{P^{D_{\os{\circ}{T}_0}}}R\pi_{{\rm zar}*}
(A_{\rm zar}({\mathcal P}^{\rm ex}_{\bul}/S(T)^{\nat},{\mathcal E}^{\bul}))
\tag{5.6.5}\label{eqn:flagrpwp}
\end{equation*} 
is an isomorphism in $D^+(f^{-1}({\mathcal O}_T))$ by 
\cite[(2.13)]{nlf}. 
By \cite[(1.3.4.5)]{nh2} and (\ref{eqn:ana}) we have the following: 
\begin{align*} 
& 
{\rm gr}_{k'}^P{\rm gr}_k^{P^{D_{\os{\circ}{T}_0}}}R\pi_{{\rm zar}*}
(A_{\rm zar}({\mathcal P}^{\rm ex}_{\bul}/S(T)^{\nat},{\mathcal E}^{\bul})) 
\os{\sim}{\lo} 
R\pi_{{\rm zar}*} 
({\rm gr}_{k'}^P{\rm gr}_k^{P^{D_{\os{\circ}{T}_0}}}
A_{\rm zar}({\mathcal P}^{\rm ex}_{\bul}/S(T)^{\nat},{\mathcal E}^{\bul}))
\tag{5.6.6}\label{ali:ruogrvp}\\
{} & \os{\sim}{\lo} 
\bigoplus_{j\geq \max\{-(k'-k),0\}}R\pi_{{\rm zar}*}
({\mathcal E}^{\bul}\otimes_{{\mathcal O}_{{\mathcal X}_{\bul}}}
b^{(2j+k'-k),(k)}_{*}(\Om^{\bul}_{\os{\circ}{\mathcal X}{}^{(2j+k'-k)}_{\bul}
\cap \os{\circ}{\mathcal D}{}^{(k)}_{\bul}/\os{\circ}{T}} \\ 
{} & \quad \quad \quad \quad \quad \quad \quad \quad  \otimes_{\mab Z}
\vp^{(2j+k'-k),(k)}_{\rm zar}
((\os{\circ}{X}_{T_0\bul},\os{\circ}{D}_{T_0\bul})/\os{\circ}{T})))[-2j-k']. 
\end{align*} 
By the Poincar\'{e} lemma 
the last complex is equal to 
{\footnotesize{\begin{align*} 
& \bigoplus_{j\geq \max \{-(k'-k),0\}}R\pi_{{\rm zar}*}
(a^{(2j+k'-k),(k)}_{\bul*} 
(Ru_{\os{\circ}{X}{}^{(2j+k'-k)}_{T_0\bul}\cap \os{\circ}{D}{}^{(k)}_{T_0\bul}/\os{\circ}{T}*}
(E\vert_{\os{\circ}{X}{}^{(2j+k'-k)}_{T_0\bul}\cap \os{\circ}{D}{}^{(k)}_{T_0\bul}
/\os{\circ}{T}}\tag{5.6.7}\label{ali:rutd}\\
&\otimes_{\mab Z}
\vp_{\rm crys}^{(2j+k'-k),(k)}((\os{\circ}{X}_{T_0\bul},\os{\circ}{D}_{T_0\bul})/\os{\circ}{T})))[-2j-k'] \\
& =\bigoplus_{j\geq \max \{-k,0\}} 
a^{(2j+k'-k),(k)}_*
(Ru_{\os{\circ}{X}{}^{(2j+k'-k)}_{T_0}\cap \os{\circ}{D}{}^{(k)}_{T_0}/\os{\circ}{T}*}
(E\vert_{{\os{\circ}{X}{}^{(2j+k'-k)}_{T_0}\cap \os{\circ}{D}{}^{(k)}_{T_0}}
/\os{\circ}{T}} \\
&\otimes_{\mab Z}\vp^{(2j+k'-k),(k)}_{\rm crys}
((\os{\circ}{X}_{T_0},\os{\circ}{D}_{T_0})/\os{\circ}{T})))[-2j-k']
\end{align*}}} 
and the analogous formula 
for $R\pi_{{\rm zar}*}((A_{\rm zar}({\mathcal P}'{}^{\rm ex}_{\bul}/S(T)^{\nat}
,{\mathcal E}'{}^{\bul}),P^{{\mathcal D}'_{\bul}/{\mathcal X}'_{\bul}},P))$. 
\par 
We complete the proof of (\ref{theo:indcr}).  
\end{proof}

\begin{defi}\label{defi:fdirpd}
We call the bifiltered direct image 
$R\pi_{{\rm zar}*}
((A_{\rm zar}({\mathcal P}^{\rm ex}_{\bul}/S(T)^{\nat},{\mathcal E}^{\bul}),
P^{{\mathcal D}_{\bul}/{\mathcal X}_{\bul}}.P))$ 
the {\it zariskian $p$-adic bifiltered El-Zein-Steenbrink-Zucker complex} of 
$E$ for $(X_{\os{\circ}{T}_0},D_{\os{\circ}{T}_0})/(S(T)^{\nat},{\mathcal J},\del)$. 
We denote it  by
$(A_{\rm zar}((X_{\os{\circ}{T}_0},D_{\os{\circ}{T}_0})/S(T)^{\nat},E),
P^{D_{\os{\circ}{T}_0}},P)
\in {\rm D}^+{\rm F}^2(f^{-1}({\mathcal O}_T))$. 
When $E={\mathcal O}_{\os{\circ}{X}_{T_0}/\os{\circ}{T}}$, we denote 
$(A_{\rm zar}((X_{\os{\circ}{T}_0},D_{\os{\circ}{T}_0})/S(T)^{\nat},E),P^{D_{\os{\circ}{T}_0}},P)$ 
by 
$$(A_{\rm zar}((X_{\os{\circ}{T}_0},D_{\os{\circ}{T}_0})/S(T)^{\nat}),P^{D_{\os{\circ}{T}_0}},P).$$ 
We call $(A_{\rm zar}((X_{\os{\circ}{T}_0},D_{\os{\circ}{T}_0})/S(T)^{\nat}),P^{D_{\os{\circ}{T}_0}},P)$ 
the {\it zariskian $p$-adic bifiltered El-Zein-Steenbrink-Zucker complex} of 
$(X_{\os{\circ}{T}_0}D_{\os{\circ}{T}_0})/(S(T)^{\nat},{\mathcal J},\del)$. 
\end{defi} 

\begin{coro}\label{coro:pwspp}
Let $E_{\os{\circ}{D}{}^{(k)}/\os{\circ}{T}}$ be the inverse image of 
$E$ to $({\os{\circ}{D}{}^{(k)}/\os{\circ}{T}})_{\rm crys}$ 
and let 
$$f_{D^{(k)}_{\os{\circ}{T}_0}/S(T)^{\nat}}\col 
(D^{(k)}_{\os{\circ}{T}_0}/S(T)^{\nat})_{\rm crys}\lo \os{\circ}{T}$$ 
be the structural morphism.  
Then there exist the following spectral sequences$:$
\begin{align*} 
E_1^{k,q-k}&=R^{q-k}f_{D^{(k)}_{\os{\circ}{T}_0}/S(T)^{\nat}*}(
\eps^*_{D^{(k)}_{\os{\circ}{T}_0}/S(T)^{\nat}}
(E_{\os{\circ}{D}{}^{(k)}/\os{\circ}{T}})\otimes_{\mab Z}
\eps^{-1}_{D^{(k)}_{\os{\circ}{T}_0}/S(T)^{\nat}}\vp^{(k)}_{\rm crys}
((\os{\circ}{D}_{T_0}/\os{\circ}{T}_0)))\tag{5.8.1}\label{ali:dks}\\
&\Lo 
R^qf_{(X_{\os{\circ}{T}_0},D_{\os{\circ}{T}_0})/S(T)^{\nat}*}
(\eps^*_{(X_{\os{\circ}{T}_0},D_{\os{\circ}{T}_0})/S(T)^{\nat}}(E)), 
\end{align*} 
\begin{align*} 
E_1^{-k,q+k}&=
\bigoplus_{k'\leq k}
\bigoplus_{j\geq \max\{-k',0\}}
R^{q-2j-k}f_{\os{\circ}{X}{}^{(2j+k')}_{\os{\circ}{T}_0}\cap D^{(k-k')}_{\os{\circ}{T}_0}/\os{\circ}{T}*}
(E_{\os{\circ}{X}{}^{(2j+k')}_{\os{\circ}{T}_0}\cap 
\os{\circ}{D}{}^{(k-k')}_{\os{\circ}{T}_0}/\os{\circ}{T}}\tag{5.8.2}\label{ali:dwks}\\
&\otimes_{\mab Z}\vp^{(2j+k',k-k')}_{\rm crys}
((\os{\circ}{X}_{T_0},\os{\circ}{D}_{T_0})/\os{\circ}{T}_0))
\\
&\Lo 
R^qf_{(X_{\os{\circ}{T}_0},D_{\os{\circ}{T}_0})/S(T)^{\nat}*}
(\eps^*_{(X_{\os{\circ}{T}_0},D_{\os{\circ}{T}_0})/S(T)^{\nat}}(E)). 
\end{align*} 
%Here the Tate twists $(-k)$ and $(-j-k)$ are the Tate twists with respect to 
%the abrelative Frobenius morphism defined in {\rm \cite{nb}} and {\rm \cite{nhir}}. 
\end{coro}
\begin{proof} 
Let $f^{(k)}\col D^{(k)}_{\os{\circ}{T}_0} \lo S(T)^{\nat}$ and 
$f^{(k)}_{\bul} \col D^{(k)}_{\os{\circ}{T}_0\bul}\lo S(T)^{\nat}$ 
be the structural morphisms. 
Let 
\begin{align*} 
\pi^{(k)}_{{\rm zar}} \col &((D^{(k)}_{\os{\circ}{T}_0\bul})_{\rm zar},f^{(k)-1}_{\bul}({\mathcal O}_T))
\lo ((D^{(k)}_{\os{\circ}{T}_0})_{\rm zar},f^{-1}({\mathcal O}_T)) 
%\tag{5.4.4}\label{ali:pzd} \\
\end{align*} 
be the natural morphism of ringed topoi.  
Let 
\begin{align*} 
\pi^{(k)}_{{\rm crys}} 
\col &((D^{(k)}_{\os{\circ}{T}_0\bul}/S(T)^{\nat})_{\rm crys},
{\mathcal O}_{D^{(k)}_{\os{\circ}{T}_0\bul}/S(T)^{\nat}})   
\lo ((D^{(k)}_{\os{\circ}{T}_0}/S(T)^{\nat})_{\rm crys},
{\mathcal O}_{D^{(k)}_{\os{\circ}{T}_0}/S(T)^{\nat}})
%\tag{5.5.2}\label{eqn:pio} \\
\end{align*} 
be also the natural morphism of ringed topoi. 
(\ref{ali:dks}) follows from the following: 
\begin{align*} 
&R\pi_{{\rm zar}*}({\rm gr}_k^{P^{D_{\os{\circ}{T}_0}}}
A_{\rm zar}({\mathcal P}^{\rm ex}_{\bul}/S(T)^{\nat},{\mathcal E}^{\bul}))
\tag{5.8.3}\label{ali:rzrvp}\\
& 
= 
R\pi_{{\rm zar}*}
(d^{(k)}_{{\mathfrak D}_{\bul}}({\mathcal D}_{\bul})_*(A_{\rm zar}({\mathcal D}^{(k)}_{\bul}/\os{\circ}{T},
{\mathcal E}^{\bul})\otimes_{\mab Z}
\varpi^{(k)}_{\rm zar}(\os{\circ}{D}_{T_0\bul}/\os{\circ}{T})))[-k]\\
&\os{\sim}{\longleftarrow} R\pi_{{\rm zar}*}
(d^{(k)}_{{\mathfrak D}_{\bul}}({\mathcal D}_{\bul})_*
({\mathcal E}^{\bul}\otimes_{{\mathcal O}_{{\mathcal X}_{\bul}}}\Om^{\bul}_{{\mathcal D}{}^{(k)}_{\bul}/S(T)^{\nat}} 
\otimes_{\mab Z}
\vp^{(k)}_{\rm zar}(\os{\circ}{D}_{T_0\bul}/\os{\circ}{T})))[-k]\\
& =
R\pi_{{\rm zar}*}Rc^{(k)}_{\bul *}Ru_{D^{(k)}_{\os{\circ}{T}_0\bul}/S(T)^{\nat}*}
(\eps^*_{D^{(k)}_{\os{\circ}{T}_0\bul}/S(T)^{\nat}}
(E_{\os{\circ}{D}{}^{(k)}_{T_0\bul}/\os{\circ}{T}})\otimes_{\mab Z}
\eps^{-1}_{D^{(k)}_{\os{\circ}{T}_0\bul}/S(T)^{\nat}}\vp^{(k)}_{\rm crys}
(\os{\circ}{D}_{T_0\bul}/\os{\circ}{T}_0))[-k]\\
& \os{\sim}{\lo} 
Rc^{(k)}_*Ru_{D^{(k)}_{\os{\circ}{T}_0}/S(T)^{\nat}*}R\pi^{(k)}_{{\rm crys}*}
(\eps^*_{D^{(k)}_{\os{\circ}{T}_0\bul}/S(T)^{\nat}}
(E_{\os{\circ}{D}{}^{(k)}_{T_0\bul}/\os{\circ}{T}})\otimes_{\mab Z}
\eps^{-1}_{D^{(k)}_{\os{\circ}{T}_0\bul}/S(T)^{\nat}}\vp^{(k)}_{\rm crys}
(\os{\circ}{D}_{T_0\bul}/\os{\circ}{T}_0))[-k]\\
& \os{\sim}{\lo} 
Rc^{(k)}_*Ru_{D^{(k)}_{\os{\circ}{T}_0}/S(T)^{\nat}*}
(\eps^*_{D^{(k)}_{\os{\circ}{T}_0}/S(T)^{\nat}}
(E_{\os{\circ}{D}{}^{(k)}_{T_0}/\os{\circ}{T}})\otimes_{\mab Z}
\eps^{-1}_{D^{(k)}_{\os{\circ}{T}_0}/S(T)^{\nat}}\vp^{(k)}_{\rm crys}
(\os{\circ}{D}_{T_0}/\os{\circ}{T}_0))[-k]. 
\end{align*} 
(\ref{ali:dwks}) follows from the following: 
\begin{align*} 
&R\pi_{{\rm zar}*}({\rm gr}_k^P
A_{\rm zar}({\mathcal P}^{\rm ex}_{\bul}/S(T)^{\nat},{\mathcal E}^{\bul}))
{}  \os{\sim}{\lo} \tag{5.8.4}\label{ali:rzrwp}\\
&R\pi_{{\rm zar}*}
(\bigoplus_{k'\leq k}
\bigoplus_{j\geq \max\{-k',0\}}
({\mathcal E}^{\bul}\otimes_{{\mathcal O}_{{\mathcal X}_{\bul}}} 
\Om^{\bul}_{\os{\circ}{\mathcal X}{}^{(2j+k')}_{\bul}
\cap \os{\circ}{\mathcal D}{}^{(k-k')}_{\bul}/\os{\circ}{T}} 
\otimes_{\mab Z}
\vp^{(2j+k'),(k-k')}_{\rm zar}
((\os{\circ}{X}_{T_0\bul},\os{\circ}{D}_{T_0\bul})/\os{\circ}{T}))[-2j-k] \\
&  \os{\sim}{\lo} 
\bigoplus_{k'\leq k}
\bigoplus_{j\geq \max\{-k',0\}}
R\pi_{{\rm zar}*}(a^{(2j+k'),(k-k')}_{\bul*}
Ru_{\os{\circ}{X}{}^{(2j+k')}_{\os{\circ}{T}_0\bul}\cap D^{(k-k')}_{\os{\circ}{T}_0\bul}/\os{\circ}{T}*}
(E_{\os{\circ}{X}{}^{(2j+k')}_{\os{\circ}{T}_0\bul}\cap 
\os{\circ}{D}{}^{(k-k')}_{\os{\circ}{T}_0\bul}/\os{\circ}{T}}\\
&\otimes_{\mab Z}\vp^{(2j+k',k-k')}_{\rm crys}
((\os{\circ}{X}_{T_0\bul},\os{\circ}{D}_{T_0\bul})/\os{\circ}{T}_0)))[-2j-k]\\
& \os{\sim}{\lo} 
\bigoplus_{k'\leq k}
\bigoplus_{j\geq \max\{-k',0\}}
a^{(2j+k'),(k-k')}_{*}
Ru_{\os{\circ}{X}{}^{(2j+k')}_{\os{\circ}{T}_0}\cap D^{(k-k')}_{\os{\circ}{T}_0}/\os{\circ}{T}*}
(E_{\os{\circ}{X}{}^{(2j+k')}_{\os{\circ}{T}_0}\cap 
\os{\circ}{D}{}^{(k-k')}_{\os{\circ}{T}_0}/\os{\circ}{T}}\\
&\otimes_{\mab Z}\vp^{(2j+k',k-k')}_{\rm crys}
((\os{\circ}{X}_{T_0},\os{\circ}{D}_{T_0})/\os{\circ}{T}_0)))[-2j-k]. 
\end{align*}
\end{proof}

%\tag{5.2.2}\label{eqn:axd}\\ 
%{} & \quad \otimes_{\mab Z}
%\vp^{(2j+k'),(k-k')}_{\rm zar}
%((\os{\circ}{X}_{T_0\bul},\os{\circ}{D}_{T_0\bul})/\os{\circ}{T}),-d). 

\begin{defi}
We call the spectral sequences (\ref{ali:dks}) and (\ref{ali:dwks}) 
the {\it Poincar\'{e} spectral sequence} of   
$R^qf_{(X_{\os{\circ}{T}_0},D_{\os{\circ}{T}_0})/S(T)^{\nat}*}
(\eps^*_{(X_{\os{\circ}{T}_0},D_{\os{\circ}{T}_0})/S(T)^{\nat}}(E))$ 
{\it relative to} $D_{\os{\circ}{T}_0}/S(T)^{\nat}$ 
and the {\it Poincar\'{e} spectral sequence} of   
$R^qf_{(X_{\os{\circ}{T}_0},D_{\os{\circ}{T}_0})/S(T)^{\nat}*}
(\eps^*_{(X_{\os{\circ}{T}_0},D_{\os{\circ}{T}_0})/S(T)^{\nat}}(E))$, respectively. 
We denote by $P^{D_{\os{\circ}{T}_0}}$ and $P$ the filtrations on 
$R^qf_{(X_{\os{\circ}{T}_0},D_{\os{\circ}{T}_0})/S(T)^{\nat}*}
(\eps^*_{(X_{\os{\circ}{T}_0},D_{\os{\circ}{T}_0})/S(T)^{\nat}}(E))$ 
obtained by (\ref{ali:dks}) and (\ref{ali:dwks}), respectively. 
If $E$ is trivial, then 
we call $P^{D_{\os{\circ}{T}_0}}$ and $P$ the {\it weight filtration}  
on $R^qf_{(X_{\os{\circ}{T}_0},D_{\os{\circ}{T}_0})/S(T)^{\nat}*}
(\eps^*_{(X_{\os{\circ}{T}_0},D_{\os{\circ}{T}_0})/S(T)^{\nat}}(E))$
{\it relative to} $D_{\os{\circ}{T}_0}/S(T)^{\nat}$  and the {\it weight filtration} on 
$R^qf_{(X_{\os{\circ}{T}_0},D_{\os{\circ}{T}_0})/S(T)^{\nat}*}
(\eps^*_{(X_{\os{\circ}{T}_0},D_{\os{\circ}{T}_0})/S(T)^{\nat}}(E))$, 
respectively. 
\end{defi}

%\begin{prop}\label{prop:lbdl}
%The edge morphism $d^{-k, q+k}_1 \col E_{1,l}^{-k, q+k} \lo E_{1,l}^{-k+1, q+k}$ 
%of the spectral sequence  {\rm (\ref{eqn:escssp})} is 
%identified with the following morphism$:$
%\begin{equation*}
%\sum_{k'\leq k}\sum_{j\geq {\rm max}\{-k', 0\}}
%\{-G^{(k),(k')}+\iota^{(k),(k')*}+
%(-1)^{2j+k'+1}G^{(k),(k')}(\os{\circ}{D})\}. 
%\tag{5.8.1}\label{eqn:glbd}
%\end{equation*}
%\end{prop}
%\begin{proof} 
%The proof is the same as that of \cite[(10.1)]{ndw}. 
%\end{proof} 

\section{Contravariant functorialities of zariskian 
$p$-adic bifiltered El Zein-Steenbrink-Zucker complexes}\label{sec:fcuc} 
Let the notations be as in the previous section. 
\par 
Let $S'$ be another family of log points.   
Let $(T',{\mathcal J}',\del')$ be a log PD-enlargement over $S'$. 
Assume that $p$ is locally nilpotent on $\os{\circ}{T}{}'$. 
Let $u\col (S(T)^{\nat},{\mathcal J},\del) \lo (S'(T')^{\nat},{\mathcal J}',\del')$ be 
a morphism of fine log schemes. 
Set $T_0:=\ul{\rm Spec}^{\log}_T({\mathcal O}_T/{\mathcal J})$ and 
$T'_0:=\ul{\rm Spec}^{\log}_{T'}({\mathcal O}_{T'}/{\mathcal J}')$.   
By the definition of $\deg(u)_x$ ((\ref{defi:ddef})), 
we have the following equality: 
\begin{equation*} 
u_{x}^*(\theta_{S'(T')^{\nat},\os{\circ}{u}(x)})= \deg(u)_x\theta_{S(T)^{\nat},x} 
\quad (x\in \os{\circ}{T}). 
\tag{6.0.1}\label{eqn:uta}  
\end{equation*}
It is easy to check that $\deg(u)_x\not=0$ for any point $x\in \os{\circ}{T}$. 
Let $(X,D)$ and $(Y,C)$ be SNCL schemes with SNCD's 
over $S$ and $S'$, respectively. 
Let ${\mathfrak D}(\ol{S'(T')^{\nat}})$ be the log PD-envelope of 
the immersion $S'(T')^{\nat}\os{\sus}{\lo} \ol{S'(T')^{\nat}}$ 
over $(\os{\circ}{T}{}',{\mathcal J}',\del')$. 
Let 
\begin{equation*} 
\begin{CD} 
(X_{\os{\circ}{T}_0},D_{\os{\circ}{T}_0}) 
@>{g}>> (Y_{\os{\circ}{T}{}'_0},C_{\os{\circ}{T}{}'_0})\\
@VVV @VVV \\ 
S_{\os{\circ}{T}_0} @>>> S'_{\os{\circ}{T}{}'_0} \\ 
@V{\bigcap}VV @VV{\bigcap}V \\ 
S(T)^{\nat} @>{u}>> S'(T')^{\nat}
\end{CD}
\tag{6.0.2}\label{eqn:xdxduss}
\end{equation*} 
be a commutative diagram of SNCL schemes with SNCD's 
over $S_{\os{\circ}{T}_0}$ and $S'_{\os{\circ}{T}{}'_0}$.   
Let  $(X'_{\os{\circ}{T}_0},D'_{\os{\circ}{T}_0})$ and 
$(Y'_{\os{\circ}{T}{}'_0},C'_{\os{\circ}{T}{}'_0})$ 
be the disjoint unions of 
affine open coverings of $(X_{\os{\circ}{T}_0},D_{\os{\circ}{T}_0})$ 
and $(Y_{\os{\circ}{T}{}'_0},C_{\os{\circ}{T}_0})$ 
respectively, 
fitting into the following commutative diagram 
\begin{equation*} 
\begin{CD} 
(X'_{\os{\circ}{T}_0},D'_{\os{\circ}{T}_0})  
@>{g'}>> (Y'_{\os{\circ}{T}{}'_0},C'_{\os{\circ}{T}{}'_0})  \\
@VVV @VVV \\ 
(X_{\os{\circ}{T}_0},D_{\os{\circ}{T}_0}) @>{g}>> (Y_{\os{\circ}{T}{}'_0},C_{\os{\circ}{T}{}'_0}). 
\end{CD}
\tag{6.0.3}\label{cd:xygxy}
\end{equation*} 
Set $(X_{\os{\circ}{T}_0\bul},D_{\os{\circ}{T}_0\bul})
:={\rm cosk}^{(X_{\os{\circ}{T}_0},D_{\os{\circ}{T}_0})}_0
((X'_{\os{\circ}{T}_0},D'_{\os{\circ}{T}_0}))$ and 
$(Y_{\os{\circ}{T}{}'_0\bul},C_{\os{\circ}{T}{}'_0\bul})
:={\rm cosk}^{(Y_{\os{\circ}{T}{}'_0},C_{\os{\circ}{T}{}'_0})}_0
((Y'_{\os{\circ}{T}{}'_0},C'_{\os{\circ}{T}{}'_0}))$. 
Let $(X_{\os{\circ}{T}_0\bul},D_{\os{\circ}{T}_0\bul}) 
\os{\sus}{\lo} \ol{\mathcal P}{}'_{\bul}$
and $(Y_{\os{\circ}{T}{}'_0\bul},C_{\os{\circ}{T}{}'_0\bul}) 
\os{\sus}{\lo} \ol{\mathcal Q}_{\bul}$ 
be immersions into simplicial log smooth schemes over 
$\ol{S(T)^{\nat}}$ and $\ol{S'(T')^{\nat}}$, respectively. 
Indeed, these immersions exist by (\ref{eqn:eipxd}). 
Set 
$$\ol{\mathcal P}_{\bul}:=\ol{\mathcal P}{}'_{\bul}\times_{\ol{S(T)^{\nat}}}
(\ol{\mathcal Q}_{\bul}\times_{\ol{S'(T')^{\nat}}}\ol{S(T)^{\nat}})
=\ol{\mathcal P}{}'_{\bul}\times_{\ol{S'(T')^{\nat}}}\ol{\mathcal Q}_{\bul}.$$ 
Let $\ol{g}_{\bul}\col \ol{\mathcal P}_{\bul} \lo \ol{\mathcal Q}_{\bul}$
be the second projection. 
Then we have the following commutative diagram 
\begin{equation*} 
\begin{CD} 
(X_{\os{\circ}{T}_0\bul},D_{\os{\circ}{T}_0\bul})
@>{\sus}>> \ol{\mathcal P}_{\bul}\\
@V{g_{\bul}}VV 
@VV{\ol{g}_{\bul}}V \\ 
(Y_{\os{\circ}{T}{}'_0\bul},C_{\os{\circ}{T}_0\bul}) 
@>{\sus}>> \ol{\mathcal Q}_{\bul}
\end{CD} 
\tag{6.0.4}\label{cd:xpnxp} 
\end{equation*} 
over 
\begin{equation*} 
\begin{CD} 
S_{\os{\circ}{T}_0} @>{\subset}>> \ol{S(T)^{\nat}}\\ 
@VVV @VVV \\ 
S'_{\os{\circ}{T}{}'_0} @>{\subset}>> \ol{S'(T')^{\nat}}. 
\end{CD}
\end{equation*} 
Let $\ol{\mathfrak E}_{\bul}$ be the log PD-envelope of  
the immersion 
$(Y_{\os{\circ}{T}{}'_0\bul},C_{\os{\circ}{T}{}'_0\bul}) 
\os{\sus}{\lo}\ol{\mathcal Q}_{\bul}$ over 
$(\os{\circ}{T}{}',{\mathcal J}',\del')$.     
Set ${\mathfrak E}_{\bul}:=
\ol{\mathfrak E}_{\bul}\times_{{\mathfrak D}(\ol{S'(T')^{\nat}})}S'(T')^{\nat}$. 
By (\ref{cd:xpnxp}) 
we have the following natural morphism  
\begin{equation*} 
\ol{g}{}^{\rm PD}_{\bul}\col 
\ol{\mathfrak D}_{\bul}\lo \ol{\mathfrak E}_{\bul}.   
\tag{6.0.5}\label{eqn:gpdf} 
\end{equation*}  
Hence we have the following natural morphism 
\begin{equation*} 
g^{\rm PD}_{\bul}\col 
{\mathfrak D}_{\bul}\lo {\mathfrak E}_{\bul}.   
\tag{6.0.6}\label{eqn:gpndf} 
\end{equation*}

Let $E$ and $F$ be flat quasi-coherent crystals of 
${\mathcal O}_{\os{\circ}{X}_{T_0}/\os{\circ}{T}}$-modules  
and ${\mathcal O}_{\os{\circ}{Y}_{T_0'}/\os{\circ}{T}{}'}$-modules, respectively.     
Let 
\begin{align*} 
\os{\circ}{g}{}^*_{\rm crys}(F)\lo E
\tag{6.0.7}\label{ali:gnfe} 
\end{align*} 
be a morphism of 
${\mathcal O}_{\os{\circ}{X}_{T_0}/\os{\circ}{T}}$-modules. 
%Henceforth, we assume that 
%the morphism $u\col (S(T)^{\nat},{\mathcal J},\del) \lo (S'(T')^{\nat},{\mathcal J}',\del')$ 
%extends to a morphism $u\col (T,{\mathcal J},\del) \lo (T',{\mathcal J}',\del')$ 
%of fine log schemes. 

\begin{theo}[{\bf Contravariant functoriality I of $A_{\rm zar}$}]\label{theo:funas} 
$(1)$ Assume that 
$\deg(u)_x$ is not divisible by $p$ for any point $x \in \os{\circ}{T}$. 
Then $g\col X_{\os{\circ}{T}{}_0}\lo Y_{\os{\circ}{T}{}'_0}$ induces the following 
well-defined pull-back morphism 
\begin{equation*}  
g^* \col 
(A_{\rm zar}((Y_{\os{\circ}{T}{}'_0},C_{\os{\circ}{T}{}'_0})/S'(T')^{\nat},F),P^{C_{\os{\circ}{T}{}'_0}},P)
\lo Rg_*((A_{\rm zar}((X_{\os{\circ}{T}_0},D_{\os{\circ}{T}_0})/S(T)^{\nat},E),P^{D_{\os{\circ}{T}_0}},P)) 
\tag{6.1.1}\label{eqn:fzaxd}
\end{equation*} 
fitting into the following commutative diagram$:$
\begin{equation*} 
\begin{CD}
A_{\rm zar}((Y_{\os{\circ}{T}{}'_0},C_{\os{\circ}{T}{}'_0})/S'(T')^{\nat},F)@>{g^*}>>
Rg_{*}(A_{\rm zar}((X_{\os{\circ}{T}_0},D_{\os{\circ}{T}_0})/S(T)^{\nat},E))\\ 
@A{\theta_{(Y_{\os{\circ}{T}{}'_0},C_{\os{\circ}{T}{}'_0})/S'(T')^{\nat}}\wedge}A{\simeq}A 
@A{Rg_{*}(\theta_{(X_{\os{\circ}{T}_0},D_{\os{\circ}{T}_0})/S(T)^{\nat}} \wedge)}A{\simeq}A\\
Ru_{(Y_{\os{\circ}{T}{}'_0},C_{\os{\circ}{T}{}'_0})/S'(T')^{\nat}*}
(\eps^*_{(Y_{\os{\circ}{T}{}'_0},C_{\os{\circ}{T}{}'_0})/S'(T')^{\nat}}(F))
@>{g^*}>>Rg_{*}Ru_{(X_{\os{\circ}{T}_0},D_{\os{\circ}{T}_0})/S(T)^{\nat}*}
(\eps^*_{(X_{\os{\circ}{T}_0},D_{\os{\circ}{T}_0})/S(T)^{\nat}}(E)).
\end{CD}
\tag{6.1.2}\label{cd:pssfccz} 
\end{equation*} 
\par 
$(2)$ Let $S''$ be a family of log points. 
Let $(T'',{\mathcal J}'',\del'')$ be a log PD-enlargement of $S''$. 
Set  
$T''_0:=\ul{\rm Spec}^{\log}_{T''}({\mathcal O}_{T''}/{\mathcal J}'')$. 
Let $v\col (S'(T')^{\nat},{\mathcal J},\del) \lo (S''(T'')^{\nat},{\mathcal J}',\del')$ and 
$h\col (Y_{\os{\circ}{T}{}'_0},C_{\os{\circ}{T}{}'_0})\lo (Z_{\os{\circ}{T}{}''_0},B_{\os{\circ}{T}{}''_0})$ 
be similar morphisms to $u$ and $g$, respectively. 
%Let $(T',{\mathcal J}',\del') \lo (T'',{\mathcal J}'',\del'')$ be a morphism of log PD-enlargements 
%over $v_{S''S'}$.   
Assume that $\deg(v)_x$ is not divisible by $p$ for any point  $x \in \os{\circ}{T}{}'$. 
Let $G$ be a flat quasi-coherent crystal of 
${\mathcal O}_{\os{\circ}{Z}_{T''_0}/\os{\circ}{T}{}''}$-modules.    
Let 
\begin{align*} 
\os{\circ}{h}{}^*_{{\rm crys}}(G)\lo F
\tag{6.1.3}\label{ali:gznfe} 
\end{align*} 
be a morphism of 
${\mathcal O}_{\os{\circ}{Y}_{T'_0}/\os{\circ}{T}{}'}$-modules. 
Then 
\begin{align*} 
(h\circ g)^*  =& Rh_{*}(g^*)\circ h^*    
\col (A_{\rm zar}((Z_{\os{\circ}{T}{}''_0},B_{\os{\circ}{T}{}''_0})/S''(T'')^{\nat},F),
P^{B_{\os{\circ}{T}{}''_0}},P)
 \tag{6.1.4}\label{ali:pdpp}\\ 
& \lo Rh_*Rg_*((A_{\rm zar}((X_{\os{\circ}{T}_0},D_{\os{\circ}{T}_0})
/S(T)^{\nat},E),P^{D_{\os{\circ}{T}_0}},P)) \\
& =R(h\circ g)_*((A_{\rm zar}((X_{\os{\circ}{T}_0},D_{\os{\circ}{T}_0})
/S(T)^{\nat},E),P^{D_{\os{\circ}{T}_0}},P)).
\end{align*}  
\par 
$(3)$  
\begin{equation*} 
{\rm id}_{(X_{\os{\circ}{T}_0},D_{\os{\circ}{T}_0})}^*={\rm id} 
_{(A_{\rm zar}(X_{\os{\circ}{T}_0}/S(T)^{\nat},E),P^{D_{\os{\circ}{T}_0}},P)}. 
%\lo (A_{\rm zar}(X_{\os{\circ}{T}_0}/S(T)^{\nat},E),P^{D_{\os{\circ}{T}_0}},P).  
\tag{6.1.5}\label{eqn:fzidd}
\end{equation*} 
\end{theo}
\begin{proof} 
(1): For $F$, let $(\ol{\mathcal F}{}^{\bul},\nabla)$ 
and $({\mathcal F}{}^{\bul},\nabla)$ 
be the similar objects to 
$(\ol{\mathcal E}{}^{\bul},\nabla)$ and 
$({\mathcal E}{}^{\bul},\nabla)$ in \S\ref{sec:psc}, respectively.  
For simplicity of notation, 
denote $\theta_{{\mathcal P}^{\rm ex}_{\bul}/\os{\circ}{T}}\in 
{\mathcal O}_{{\mathfrak D}_{\bul}}
\otimes_{{\mathcal O}_{{\mathcal P}^{\rm ex}_{\bul}}}
{\Om}^1_{{\mathcal P}^{\rm ex}_{\bul}/\os{\circ}{T}}$ 
(resp.~$\theta_{{\mathcal Q}^{\rm ex}_{\bul}/\os{\circ}{T}}\in 
{\mathcal O}_{{\mathfrak E}_{\bul}}
\otimes_{{\mathcal O}_{{\mathcal Q}^{\rm ex}_{\bul}}}
{\Om}^1_{{\mathcal Q}^{\rm ex}_{\bul}/\os{\circ}{T}{}'}$) 
simply by $\theta$ (resp.~$\theta'$). 
\par 
First we would like to construct a morphism  
\begin{equation*} 
{\mathcal F}^{\bul}
\otimes_{{\mathcal O}_{{\mathcal Q}{}^{\rm ex}_{\bul}}}
\Om^{\bul}_{{\mathcal Q}{}^{\rm ex}_{\bul}/\os{\circ}{T}{}'}
\lo 
g^{\rm PD}_{\bul*}({\mathcal E}^{\bul}
\otimes_{{\mathcal O}_{{\mathcal P}^{\rm ex}_{\bul}}}
\Om^{\bul}_{{\mathcal P}^{\rm ex}_{\bul}/\os{\circ}{T}})
\end{equation*} 
of complexes. 
Because we are given the morphism  
$\os{\circ}{g}{}^*_{{\rm crys}}(F)
\lo E$, we have a morphism 
\begin{align*} 
\ol{g}{}^{{\rm PD}*}_{\bul}
\col \ol{\mathcal F}{}^{\bul}
\lo 
\ol{g}{}^{{\rm PD}}_{\bul*}(\ol{\mathcal E}{}^{\bul})
\end{align*} 
fitting into the following commutative diagram$:$ 
\begin{equation*} 
\begin{CD}  
\ol{\mathcal F}{}^{\bul}
@>{\ol{g}{}^{{\rm PD}*}_{\bul}}>>
\ol{g}{}^{\rm PD}_{\bul*}(\ol{\mathcal E}{}^{\bul})\\
@VVV @VVV\\
\ol{\mathcal F}{}^{\bul}
\otimes_{{\mathcal O}_{\ol{\mathcal Q}{}^{\rm ex}_{\bul}}}
\Om^1_{\ol{\mathcal Q}{}^{\rm ex}_{\bul}/\os{\circ}{T}{}'}
@>{\ol{g}{}^{{\rm PD}*}_{\bul}}>>
\ol{g}{}^{\rm PD}_{\bul*}
(\ol{\mathcal E}^{\bul}
\otimes_{{\mathcal O}_{\ol{\mathcal P}{}^{\rm ex}_{\bul}}}
\Om^1_{\ol{\mathcal P}{}^{\rm ex}_{\bul}/\os{\circ}{T}})
\end{CD} 
\tag{6.1.6}\label{eqn:dolm}
\end{equation*} 
by using (\ref{cd:xpnxp}). Hence we have a morphism 
\begin{align*} 
g^{{\rm PD}*}_{\bul}\col {\mathcal F}^{\bul}\lo g^{\rm PD}_{{\bul}*}({\mathcal E}^{\bul})
\tag{6.1.7}\label{eqn:defm}
\end{align*} 
fitting into the following commutative diagram$:$ 
\begin{equation*} 
\begin{CD}  
{\mathcal F}^{\bul}
@>{g^{{\rm PD}*}_{\bul}}>>
g^{\rm PD}_{\bul*}({\mathcal E}^{\bul})\\
@VVV @VVV\\
{\mathcal F}^{\bul}
\otimes_{{\mathcal O}_{{\mathcal Q}{}^{\rm ex}_{\bul}}}
\Om^1_{{\mathcal Q}{}^{\rm ex}_{\bul}/\os{\circ}{T}{}'}
@>{g^{{\rm PD}*}_{\bul}}>>
g^{\rm PD}_{\bul*}
({\mathcal E}^{\bul}
\otimes_{{\mathcal O}_{{\mathcal P}{}^{\rm ex}_{\bul}}}
\Om^1_{{\mathcal P}{}^{\rm ex}_{\bul}/\os{\circ}{T}}).
\end{CD} 
\tag{6.1.8}\label{eqn:dgm}
\end{equation*} 
Express ${\mathcal P}^{\rm ex}:=({\mathcal X}_{\bul},{\mathcal D}_{\bul})$ and 
${\mathcal Q}^{\rm ex}:=({\mathcal Y}_{\bul},{\mathcal C}_{\bul})$, 
where $({\mathcal X}_{\bul},{\mathcal D}_{\bul})$ and $({\mathcal Y}_{\bul},{\mathcal C}_{\bul})$
are SNCL schemes with SNCD's over $S(T)^{\nat}$ and $S'(T')^{\nat}$, respectively.  
By using (\ref{cd:xpnxp}) again, 
we have the following morphism 
\begin{equation*} 
g^{{\rm PD}*}_{\bul}
\col {\mathcal O}_{{\mathfrak E}_{\bul}}
\otimes_{{\mathcal O}_{{\mathcal Q}{}^{\rm ex}_{\bul}}}
\Om^{\bul}_{{\mathcal Q}{}^{\rm ex}_{\bul}/\os{\circ}{T}{}'}
\lo 
g^{\rm PD}_{\bul*}({\mathcal O}_{{\mathfrak D}_{\bul}}
\otimes_{{\mathcal O}_{{\mathcal P}^{\rm ex}_{\bul}}}
\Om^{\bul}_{{\mathcal P}^{\rm ex}_{\bul}/\os{\circ}{T}}). 
\tag{6.1.9}\label{eqn:dgppm}
\end{equation*}  
Set 
$$g^{{\rm PD}*}_{\bul}(e\otimes \om):=
g^{{\rm PD}*}_{\bul}(e)\otimes g^{{\rm PD}*}_{\bul}(\om)
\quad (e\in {\mathcal F}^{\bul},~ 
\om \in \Om^i_{{\mathcal Q}^{\rm ex}_{\bul}/\os{\circ}{T}{}'}~(i\in {\mab N})).$$ 
This $g^{{\rm PD}*}_{\bul}$ induces 
the following morphism of bifiltered complexes: 
\begin{equation*} 
({\mathcal F}^{\bul}
\otimes_{{\mathcal O}_{{\mathcal Q}^{\rm ex}_{\bul}}}
\Om^{\bul}_{{\mathcal Q}^{\rm ex}_{\bul}/\os{\circ}{T}{}'},P^{{\mathcal C}_{\bul}},P^{{\mathcal Y}_{\bul}})
\lo 
g^{\rm PD}_{\bul*} 
(({\mathcal E}^{\bul}
\otimes_{{\mathcal O}_{{\mathcal P}^{\rm ex}_{\bul}}}
\Om^{\bul}_{{\mathcal P}^{\rm ex}_{\bul}/\os{\circ}{T}},P^{{\mathcal D}_{\bul}},P^{{\mathcal X}_{\bul}})).  
\tag{6.1.10}\label{eqn:dogm}
\end{equation*} 
Because the following diagram
\begin{equation*} 
\begin{CD} 
g^{{\rm PD}*}_{\bul} 
({\mathcal F}^{\bul}
\otimes_{{\mathcal O}_{{\mathcal Q}^{\rm ex}_{\bul}}}
\Om^{\bul}_{{\mathcal Q}^{\rm ex}_{\bul}/\os{\circ}{T}{}'})[1] 
@>{g^{{\rm PD}*}_{\bul}}>> 
{\mathcal E}^{\bul}
\otimes_{{\mathcal O}_{{\mathcal P}^{\rm ex}_{\bul}}}
{\Om}^{\bul}_{{\mathcal P}^{\rm ex}_{\bul}/\os{\circ}{T}}[1] \\   
@A{g^{{\rm PD}*}_{\bul} 
(\deg(u)^{-1}\theta'\wedge)}AA @AA{\theta \wedge}A \\ 
g^{{\rm PD}*}_{\bul}
({\mathcal F}^{\bul}
\otimes_{{\mathcal O}_{{\mathcal Q}^{\rm ex}_{\bul}}}
{\Om}^{\bul}_{{\mathcal Q}^{\rm ex}_{\bul}/\os{\circ}{T}{}'})
@>{g^{{\rm PD}*}_{\bul}}>> 
{\mathcal E}^{\bul}
\otimes_{{\mathcal O}_{{\mathcal P}^{\rm ex}_{\bul}}}
{\Om}^{\bul}_{{\mathcal P}^{\rm ex}_{\bul}/\os{\circ}{T}}    
\end{CD} 
\tag{6.1.11}\label{cd:fblssod}
\end{equation*} 
is commutative ($\deg(u)^{-1}$ has a meaning by 
the assumption of the $p$-nondivisibility of $\deg(u)$) 
and because 
we have the following commutative diagram for 
$i,j\in {\mab N}$ 
\begin{equation*} 
\begin{CD}
g^{\rm PD}_{\bul*}
(P^{{\mathcal X}_{\bul}}_j({\mathcal E}^{\bul}
\otimes_{{\mathcal O}_{{\mathcal P}^{\rm ex}_{\bul}}}
{\Om}^{i+j+1}_{{\mathcal P}^{\rm ex}_{\bul}/\os{\circ}{T}{}}))
@>{\subset}>>  g^{\rm PD}_{\bul*}({\mathcal E}^{\bul}
\otimes_{{\mathcal O}_{{\mathcal P}^{\rm ex}_{\bul}}}
{\Om}^{i+j+1}_{{\mathcal P}^{\rm ex}_{\bul}/\os{\circ}{T}{}}) \\
@A{({\rm deg}(u))^{-(j+1)}
g^{{\rm PD}*}_{\bul}}AA @AA{({\rm deg}(u))^{-(j+1)}
g^{{\rm PD}*}_{\bul}}A \\
P^{{\mathcal Y}_{\bul}}_j({\mathcal F}^{\bul}
\otimes_{{\mathcal O}_{{\mathcal Q}^{\rm ex}_{\bul}}}
{\Om}^{i+j+1}_{{\mathcal Q}^{\rm ex}_{\bul}/\os{\circ}{T}{}'})
@>{\subset}>> {\mathcal F}^{\bul}
\otimes_{{\mathcal O}_{{\mathcal Q}^{\rm ex}_{\bul}}}
{\Om}^{i+j+1}_{{\mathcal Q}^{\rm ex}_{\bul}/\os{\circ}{T}{}'}, 
\end{CD}
\tag{6.1.12}\label{cd:axfpdwt}
\end{equation*} 
we can define the pull-back morphism 
\begin{align*} 
g^*_{\bul} \col &
A_{\rm zar}({\mathcal Q}^{\rm ex}_{\bul}/S'(T')^{\nat},{\mathcal F}^{\bul})
\lo 
g^{\rm PD}_{\bul*}A_{\rm zar}({\mathcal P}^{\rm ex}_{\bul}/S(T)^{\nat},
{\mathcal E}^{\bul})
\tag{6.1.13}\label{eqn:axfdwt}
\end{align*} 
by the following formula 
\begin{align*} 
g^*_{\bul}
:= \deg(u)^{-(j+1)}g^{{\rm PD}*}_{\bul} 
\col 
A_{\rm zar}({\mathcal Q}^{\rm ex}_{\bul}/S'(T')^{\nat},{\mathcal F}^{\bul})^{ij}
\lo 
g^{\rm PD}_{\bul*}A_{\rm zar}({\mathcal P}^{\rm ex}_{\bul}/S(T)^{\nat},{\mathcal E}^{\bul})^{ij}. 
\tag{6.1.14}\label{eqn:axdwt}
\end{align*} 
In fact we have the following bifiltered morphism 
\begin{align*} 
g^*_{\bul}
:= \deg(u)^{-(j+1)}g^{{\rm PD}*}_{\bul} 
\col &
(A_{\rm zar}({\mathcal Q}^{\rm ex}_{\bul}/S'(T')^{\nat},{\mathcal F}^{\bul})^{ij},
P^{{\mathcal C}_{\bul}/{\mathcal Y}_{\bul}},P) \tag{6.1.15}\label{eqn:axfit}\\
& \lo 
g^{\rm PD}_{\bul*}
((A_{\rm zar}({\mathcal P}^{\rm ex}_{\bul}/S(T)^{\nat},{\mathcal E}^{\bul})^{ij},
P^{{\mathcal D}_{\bul}/{\mathcal X}_{\bul}},P)). 
\end{align*} 
Let 
\begin{equation*}
g^*\col 
(A_{\rm zar}((Y_{\os{\circ}{T}{}'_0},C_{\os{\circ}{T}{}'_0})/S'(T')^{\nat},F),P^{C_{\os{\circ}{T}{}'_0}},P)
\lo Rg_*((A_{\rm zar}((X_{\os{\circ}{T}_0},D_{\os{\circ}{T}_0})/S(T)^{\nat},E),P^{D_{\os{\circ}{T}_0}},P))
\tag{6.1.16}\label{eqn:axnwt}
\end{equation*} 
be the induced morphism by $g^*_{\bul}$.  
\par  
Next we check that the morphism (\ref{eqn:axnwt})
is independent of the choice of the diagrams 
(\ref{cd:xygxy}) and (\ref{cd:xpnxp}).  
Let $(X''_{\os{\circ}{T}_0},D''_{\os{\circ}{T}_0})$ and 
$(Y''_{\os{\circ}{T}{}'_0},C''_{\os{\circ}{T}{}'_0})$ be 
the disjoint unions of the members of affine open coverings of 
$(X_{\os{\circ}{T}_0},D_{\os{\circ}{T}_0})$ and 
$(Y_{\os{\circ}{T}{}'_0},C_{\os{\circ}{T}{}'_0})$, 
respectively, fitting into the following commutative diagram 
(\ref{cd:xygxy}). 
Set $(X'_{\os{\circ}{T}_0\bul},D'_{\os{\circ}{T}_0\bul})
:={\rm cosk}^{(X_{\os{\circ}{T}_0},D_{\os{\circ}{T}_0})}_0
((X''_{\os{\circ}{T}_0},D''_{\os{\circ}{T}_0}))$ 
and 
$(Y'_{\os{\circ}{T}{}'_0\bul},C'_{\os{\circ}{T}{}'_0\bul})
:={\rm cosk}^{(Y_{\os{\circ}{T}{}'_0},C_{\os{\circ}{T}{}'_0})}_0
((Y''_{\os{\circ}{T}{}'_0},C''_{\os{\circ}{T}{}'_0}))$.  
Then we have the following commutative diagram 
\begin{equation*} 
\begin{CD} 
(X'_{\os{\circ}{T}_0\bul},D'_{\os{\circ}{T}_0\bul})@>{\sus}>> 
\ol{\mathcal P}{}'^{\rm ex}_{\bul} \\ 
@V{g'_{\bul}}VV 
@VV{\ol{g}{}'_{\bul}}V \\ 
(Y'_{\os{\circ}{T}_0\bul},C'_{\os{\circ}{T}_0\bul}) 
@>{\sus}>> \ol{\mathcal Q}{}'^{\rm ex}_{\bul}, 
\end{CD} 
\tag{6.1.17}\label{cd:xqxp} 
\end{equation*} 
where the two horizontal exact immersions above are the exactifications 
of immersions 
$(X'_{\os{\circ}{T}_0\bul},D'_{\os{\circ}{T}_0\bul})\os{\sus}{\lo} \ol{\mathcal P}{}'_{\bul}$
and 
$(Y'_{\os{\circ}{T}_0\bul},C'_{\os{\circ}{T}_0\bul})  \os{\sus}{\lo} \ol{\mathcal Q}{}'_{\bul}$ 
into simplicial 
log smooth schemes over $\ol{S(T)^{\nat}}$ and $\ol{S'(T')^{\nat}}$, respectively.  
Set $(X'''_{\os{\circ}{T}_0},D'''_{\os{\circ}{T}_0}):=
(X'_{\os{\circ}{T}_0},D'_{\os{\circ}{T}_0})
\times_{(X_{\os{\circ}{T}_0},D_{\os{\circ}{T}_0})}
(X''_{\os{\circ}{T}_0},D_{\os{\circ}{T}_0})$ 
and $(Y'''_{\os{\circ}{T}{}'_0},C'''_{\os{\circ}{T}{}'_0})
:=(Y'_{\os{\circ}{T}{}'_0},C'_{\os{\circ}{T}{}'_0})
\times_{(Y_{\os{\circ}{T}{}'_0},C_{\os{\circ}{T}{}'_0})}
(Y''_{\os{\circ}{T}{}'_0},C''_{\os{\circ}{T}{}'_0})$. 
Set $(X''_{\os{\circ}{T}_0\bul},D''_{\os{\circ}{T}_0\bul})
:={\rm cosk}^{(X_{\os{\circ}{T}_0},D_{\os{\circ}{T}_0})}_0
((X'''_{\os{\circ}{T}_0},D'''_{\os{\circ}{T}_0}))$ 
and 
$(Y''_{\os{\circ}{T}{}'_0\bul},C''_{\os{\circ}{T}{}'_0\bul})
:={\rm cosk}^{(Y_{\os{\circ}{T}{}'_0},C_{\os{\circ}{T}{}'_0})}_0
((Y'''_{\os{\circ}{T}{}'_0},C'''_{\os{\circ}{T}{}'_0}))$.  
Set $\ol{\mathcal P}{}''_{\bul}:=\ol{\mathcal P}_{\bul}\times_{\ol{S(T)^{\nat}}}\ol{\mathcal P}{}'_{\bul}$ 
and 
$\ol{\mathcal Q}{}''_{\bul}
:=\ol{\mathcal Q}_{\bul}\times_{\ol{S'(T')^{\nat}}}\ol{\mathcal Q}{}'_{\bul}$.  
%Let $\ol{\mathfrak D}{}''_{\bul}$ 
%be the log PD-envelope of the diagonal immersion 
%$X''_{\bul} \os{\sus}{\lo}
%\ol{\mathcal P}{}''_{\bul}$ over $(\os{\circ}{T},{\mathcal J},\del)$ 
Let $\ol{\mathcal P}{}''_{\bul}{}^{\rm ex}$ 
(resp.~$\ol{\mathcal Q}{}''_{\bul}{}^{\rm ex}$) 
be the exactification of 
the immersion 
$(X''_{\os{\circ}{T}_0\bul},D''_{\os{\circ}{T}_0\bul})  \os{\sus}{\lo}
\ol{\mathcal P}{}''_{\bul}$
(resp.~$(Y''_{\os{\circ}{T}_0\bul},C''_{\os{\circ}{T}_0\bul})  \os{\sus}{\lo}\ol{\mathcal Q}{}''_{\bul}$). 
Then we have the following commutative diagram 
\begin{equation*} 
\begin{CD} 
(X''_{\os{\circ}{T}_0\bul},D''_{\os{\circ}{T}_0\bul})@>{\sus}>> 
\ol{\mathcal P}{}''^{\rm ex}_{\bul}\\ 
@V{g''_{\bul}}VV 
@VV{\ol{g}{}''_{\bul}}V \\ 
(Y''_{\os{\circ}{T}_0\bul},C''_{\os{\circ}{T}_0\bul}) @>{\sus}>> \ol{\mathcal Q}{}''^{\rm ex}_{\bul} 
\end{CD} 
\tag{6.1.18}\label{cd:xexp} 
\end{equation*} 
over (\ref{cd:xpnxp}) and (\ref{cd:xqxp}). 
The rest of the proof of 
the well-definedness of the morphism (\ref{eqn:axnwt})
is similar to the proof of (\ref{theo:indcr}); we leave the detail of the rest to the reader. 
\par 
By (\ref{prop:tefc}), (\ref{eqn:axdwt}) and (\ref{eqn:axnwt}), 
we have the commutative diagram (\ref{cd:pssfccz}).  
\par 
(2): (2) is clear from the construction of the morphism (\ref{eqn:fzaxd}). 
%and the relation (\ref{ali:dpvv}).  
\par 
(3): This is obvious. 
%\par 
%(4): This follows from (1) and (\ref{eqn:utz}).  
\end{proof}

\par 
Next we consider the case where 
$\deg (u)_x$ may be divisible by $p$ for 
some point $x\in \os{\circ}{S}$. 
Let $e_p(x)$ $(x\in  \os{\circ}{S})$ 
be the exponent of $\deg (u)_x$ with respect to $p$: 
$p^{e_p(x)}\vert \vert \deg(u_x)$. 
Then we have a function 
\begin{equation*} 
e_p \col \os{\circ}{S}\owns x \lom e_p(x) \in {\mab N}.   
\tag{6.1.19}\label{eqn:expf}
\end{equation*} 

\begin{defi}\label{defi:efu} 
We call $e_p$ the {\it exponent function} of $u$ 
with respect to $p$. 
\end{defi}

Now we assume that 
%$\os{\circ}{S}$, $\os{\circ}{S}{}'$, 
$\os{\circ}{T}$ and $\os{\circ}{T}{}'$ are $($not necessarily affine$)$ 
$p$-adic formal schemes   
such that $p\in {\mathcal J}$ and $p\in {\mathcal J}'$ 
(\cite[7.17, Definition]{bob}). 
Assume that ${\mathcal O}_T$ is $p$-torsion-free. 
Let 
$f_{\bul}\col {\mathcal P}_{\bul}\lo S(T)^{\nat}$ 
be the structural morphism. 
Furthermore, for any $i,j\in {\mab N}$, 
assume that the induced morphism 
\begin{align*} 
g^{{\rm PD}*}_{\bul} 
\col &
{\mathcal F}^{\bul}
\otimes_{{\mathcal O}_{{\mathcal Q}^{\rm ex}_{\bul}}}
{\Om}^{i+j+1}_{{\mathcal Q}^{\rm ex}_{\bul}/\os{\circ}{T}{}'}/P^{{\mathcal Y}_{\bul}}_j 
\lo 
g^{\rm PD}_{\bul*}({\mathcal E}^{\bul}
\otimes_{{\mathcal O}_{{\mathcal P}^{\rm ex}_{\bul}}}
{\Om}^{i+j+1}_{{\mathcal P}^{\rm ex}_{\bul}/\os{\circ}{T}{}}/P^{{\mathcal X}_{\bul}}_j) 
\tag{6.2.1}\label{eqn:odnl}
\end{align*}
by the following morphism 
\begin{equation*} 
g^{{\rm PD}*}_{\bul} 
\col {\mathcal F}^{\bul}\otimes_{{\mathcal O}_{{\mathcal Q}^{\rm ex}_{\bul}}}
{\Om}^{i+j+1}_{{\mathcal Q}^{\rm ex}_{\bul}/\os{\circ}{T}{}'}
\lo 
g^{\rm PD}_{\bul*}({\mathcal E}^{\bul}
\otimes_{{\mathcal O}_{{\mathcal P}^{\rm ex}_{\bul}}}
{\Om}^{i+j+1}_{{\mathcal P}^{\rm ex}_{\bul}/\os{\circ}{T}{}}) 
\tag{6.2.2}\label{eqn:onfdnl}
\end{equation*}
is divisible by $p^{e_p(j+1)}$, that is, 
$g^{{\rm PD}*}_{\bul}$ 
at any point $x\in {\mathcal P}^{\rm ex}_{\bul}$
is divisible by $p^{e_p(f_{\bul}(x))(j+1)}$. 
Because the target of (\ref{eqn:odnl}) is ${\mathcal O}_T$-flat, 
the following morphism
\begin{align*} 
{\rm deg}(u)^{-(j+1)}g^{{\rm PD}*}_{\bul} 
\col &
{\mathcal F}^{\bul}
\otimes_{{\mathcal O}_{{\mathcal Q}^{\rm ex}_{\bul}}}
{\Om}^{i+j+1}_{{\mathcal Q}^{\rm ex}_{\bul}/\os{\circ}{T}{}'}/P^{{\mathcal Y}_{\bul}}_j 
\lo g^{\rm PD}_{\bul*}
({\mathcal E}^{\bul}
\otimes_{{\mathcal O}_{{\mathcal P}^{\rm ex}_{\bul}}}
{\Om}^{i+j+1}_{{\mathcal P}^{\rm ex}_{\bul}/\os{\circ}{T}{}}/P^{{\mathcal X}_{\bul}}_j)
\tag{6.2.3}\label{eqn:oddnl}
\end{align*}
for $j\in {\mab N}$ is well-defined.  
In fact, the following morphisms 
\begin{align*} 
{\rm deg}(u)^{-(j+1)}g^{{\rm PD}*}_{\bul} 
\col &
(P^{{\mathcal C}_{\bul}/{\mathcal Y}_{\bul}}_k+P^{{\mathcal Y}_{\bul}}_j) 
({\mathcal F}^{\bul}
\otimes_{{\mathcal O}_{{\mathcal Q}^{\rm ex}_{\bul}}}
{\Om}^{i+j+1}_{{\mathcal Q}^{\rm ex}_{\bul}/\os{\circ}{T}{}'})/P^{{\mathcal Y}_{\bul}}_j 
\tag{6.2.4}\label{eqn:odpdnl}\\
& \lo g^{\rm PD}_{\bul*}
((P^{{\mathcal D}_{\bul}}_k+P^{{\mathcal X}_{\bul}}_j)({\mathcal E}^{\bul}
\otimes_{{\mathcal O}_{{\mathcal P}^{\rm ex}_{\bul}}}
{\Om}^{i+j+1}_{{\mathcal P}^{\rm ex}_{\bul}/\os{\circ}{T}{}}/P^{{\mathcal X}_{\bul}}_j))
\end{align*}
and 
%\begin{align*} 
%{\rm deg}(u)^{-(j+1)}g^{{\rm PD}*}_{\bul} \col &
%(P^{{\mathcal Y}_{\bul}}_k+P^{{\mathcal Y}_{\bul}}_j) ({\mathcal F}^{\bul}
%\otimes_{{\mathcal O}_{{\mathcal Q}^{\rm ex}_{\bul}}}
%{\Om}^{i+j+1}_{{\mathcal Q}^{\rm ex}_{\bul}/\os{\circ}{T}{}'})/P^{{\mathcal Y}_{\bul}}_j 
%\tag{6.2.5}\label{eqn:odpdxnl}\\
%& \lo g^{\rm PD}_{\bul*}
%((P^{{\mathcal X}_{\bul}}_k+P^{{\mathcal X}_{\bul}}_j)({\mathcal E}^{\bul}
%\otimes_{{\mathcal O}_{{\mathcal P}^{\rm ex}_{\bul}}}
%{\Om}^{i+j+1}_{{\mathcal P}^{\rm ex}_{\bul}/\os{\circ}{T}{}}/P^{{\mathcal X}_{\bul}}_j))
%\end{align*}
%for $k\in {\mab N}$ are well-defined. 
the following morphism 
\begin{align*} 
{\rm deg}(u)^{-(j+1)}g^{{\rm PD}*}_{\bul} 
\col &
(P_{2j+k+1}+P^{{\mathcal Y}_{\bul}}_j)({\mathcal F}^{\bul}
\otimes_{{\mathcal O}_{{\mathcal Q}^{\rm ex}_{\bul}}}
{\Om}^{i+j+1}_{{\mathcal Q}^{\rm ex}_{\bul}/\os{\circ}{T}{}'})/P^{{\mathcal Y}_{\bul}}_j 
\tag{6.2.5}\label{eqn:odpdbnl}\\
&\lo 
g^{\rm PD}_{\bul*}
((P_{2j+k+1}+P^{{\mathcal X}_{\bul}}_j)({\mathcal E}^{\bul}
\otimes_{{\mathcal O}_{{\mathcal P}^{\rm ex}_{\bul}}}
{\Om}^{i+j+1}_{{\mathcal P}^{\rm ex}_{\bul}/\os{\circ}{T}{}})/P^{{\mathcal X}_{\bul}}_j) 
\end{align*}
for $k\in {\mab Z}$ and $j\in {\mab N}$ is also well-defined.

\begin{lemm}\label{lemm:divp}  
The divisibility assumption for the morphism 
{\rm (\ref{eqn:odnl})} is independent of the choices of 
the  affine open coverings of $X_{\os{\circ}{T}_0\bul}$ 
and $Y_{\os{\circ}{T}{}'_0\bul}$ and the choices of 
the simplicial immersions of 
$(X_{\os{\circ}{T}\bul},D_{\os{\circ}{T}_0\bul}) \os{\sus}{\lo} 
\ol{\mathcal P}_{\bul}$ over $\ol{S(T)^{\nat}}$ and 
$(Y_{\os{\circ}{T}{}'_0\bul},C_{\os{\circ}{T}{}'_0\bul}) 
\os{\sus}{\lo} \ol{\mathcal Q}_{\bul}$ over $\ol{S'(T')^{\nat}}$ 
giving the commutative diagram {\rm (\ref{cd:xpnxp})}. 
\end{lemm} 
\begin{proof} 
Let $(X'_{\os{\circ}{T}_0\bul},D'_{\os{\circ}{T}_0\bul})\os{\sus}{\lo} \ol{\mathcal P}{}'_{\bul}$ 
and 
$(Y'_{\os{\circ}{T}{}'_0\bul},C'_{\os{\circ}{T}{}'_0\bul})\os{\sus}{\lo} \ol{\mathcal Q}{}'_{\bul}$ 
be other immersions into log smooth simplicial schemes over $\ol{S(T)^{\nat}}$ 
and $\ol{S'(T')^{\nat}}$, respectively,  fitting another commutative diagram 
\begin{equation*} 
\begin{CD} 
(X'_{\os{\circ}{T}_0\bul},D'_{\os{\circ}{T}\bul})
@>{\subset}>> \ol{\mathcal P}{}'{}^{\rm ex}_{\! \! \!\bul} \\
@VVV @VVV  \\
(Y'_{\os{\circ}{T}{}'_0\bul}, C'_{\os{\circ}{T}{}'_0\bul}) 
@>{\subset}>> \ol{\mathcal Q}{}'{}^{\rm ex}_{\bul}.    
\end{CD}
\end{equation*}  
Then, by considering the products as in the proof of (\ref{theo:funas}), 
we may assume that the following commutative diagram 
\begin{equation*} 
\begin{CD} 
(X_{\os{\circ}{T}_0\bul},D_{\os{\circ}{T}_0\bul}) @>{\subset}>> 
\ol{\mathcal P}{}^{\rm ex}_{\bul} \\
@VVV @VVV  \\
(Y_{\os{\circ}{T}{}'_0\bul},C_{\os{\circ}{T}{}'_0\bul}) 
@>{\subset}>> \ol{\mathcal Q}{}^{\rm ex}_{\bul}   
\end{CD}
\end{equation*}   
is over 
\begin{equation*} 
\begin{CD} 
(X'_{\os{\circ}{T}_0\bul},D'_{\os{\circ}{T}_0\bul}) @>{\subset}>> 
\ol{\mathcal P}{}'{}^{\rm ex}_{\! \! \!\bul} \\
@VVV @VVV  \\
(Y'_{\os{\circ}{T}{}'_0\bul},C'_{\os{\circ}{T}{}'_0\bul}) 
@>{\subset}>> \ol{\mathcal Q}{}'{}^{\rm ex}_{\bul}.    
\end{CD}
\end{equation*}   
Because the question is local, we may assume that 
there exists an immersion $(X_{\os{\circ}{T}_0},D_{\os{\circ}{T}_0})
\os{\sus}{\lo} \ol{\mathcal P}$ 
(resp.~$(Y_{\os{\circ}{T}{}'_0},C_{\os{\circ}{T}{}'_0})\os{\sus}{\lo} \ol{\mathcal Q}$) 
into a log smooth scheme 
over $\ol{S(T)^{\nat}}$  
with morphism $\ol{\mathcal P}\lo \ol{\mathcal P}{}'$ 
(resp.~$\ol{\mathcal Q}\lo \ol{\mathcal Q}{}'$) of log smooth schemes over $\ol{S(T)^{\nat}}$ 
(resp.~$\ol{S'(T')^{\nat}}$)
such that the composite morphism 
$X_{\os{\circ}{T}_0}\os{\sus}{\lo} \ol{\mathcal P}\lo \ol{\mathcal P}{}'$ 
(resp.~$Y_{\os{\circ}{T}{}'_0}\os{\sus}{\lo} \ol{\mathcal Q}\lo \ol{\mathcal Q}{}'$) is also an immersion.  
%We may assume that the following commutative diagram 
%\begin{equation*} 
%\begin{CD} 
%(X_{\os{\circ}{T}_0},D_{\os{\circ}{T}_0}) @>{\subset}>> \ol{\mathcal P}\\
%@VVV @VVV  \\
%(Y_{\os{\circ}{T}{}'_0},C_{\os{\circ}{T}{}'_0}) 
%@>{\subset}>> \ol{\mathcal Q}{}^{\rm ex}
%\end{CD}
%\end{equation*}   
%is over 
%\begin{equation*} 
%\begin{CD} 
%(X'_{\os{\circ}{T}_0},D'_{\os{\circ}{T}_0}) @>{\subset}>> \ol{\mathcal P}{}'{}\\
%@VVV @VVV  \\
%(Y'_{\os{\circ}{T}{}'_0},C'_{\os{\circ}{T}{}'_0}) @>{\subset}>> \ol{\mathcal Q}{}'.    
%\end{CD}
%\end{equation*}
Let $\ol{\mathfrak D}$ and $\ol{\mathfrak D}{}'$ 
be the log PD-envelopes of the immersions $X_{\os{\circ}{T}_0}\os{\sus}{\lo} \ol{\mathcal P}$ 
and $(X_{\os{\circ}{T}_0},D_{\os{\circ}{T}_0})\os{\sus}{\lo} \ol{\mathcal P}{}'$ 
over $(\os{\circ}{T},{\mathcal J},\del)$, respectively.  
Let $\ol{\mathfrak E}$ and $\ol{\mathfrak E}{}'$ 
be the log PD-envelopes of the immersions 
$(Y_{\os{\circ}{T}{}'_0},C_{\os{\circ}{T}{}'_0})\os{\sus}{\lo} \ol{\mathcal Q}$ 
and $(Y_{\os{\circ}{T}{}'_0},C_{\os{\circ}{T}{}'_0})\os{\sus}{\lo} \ol{\mathcal Q}{}'$ 
over $(\os{\circ}{T}{}',{\mathcal J}',\del')$, respectively.  
We may also assume that 
there exists the following commutative diagram
\begin{equation*} 
\begin{CD} 
(X_{\os{\circ}{T}_0},D_{\os{\circ}{T}_0}) @>{\subset}>> \ol{\mathcal P}{}^{\rm ex}=({\mathcal X},{\mathcal D})
@>>>\ol{\mathcal P}{}'{}^{\rm ex}=({\mathcal X}',{\mathcal D}')\\
@VVV @VVV @VVV \\
(Y_{\os{\circ}{T}{}'_0},C_{\os{\circ}{T}{}'_0}) @>{\subset}>> \ol{\mathcal Q}{}^{\rm ex}=({\mathcal Y},{\mathcal C}) @>>> 
\ol{\mathcal Q}{}'{}^{\rm ex}=({\mathcal Y}',{\mathcal C}').   
\end{CD}
\tag{6.3.1}\label{cd:xyppq}
\end{equation*} 
Set ${\mathfrak D}:=\ol{\mathfrak D}\times_{{\mathfrak D}(\ol{S(T)^{\nat}})}S(T)^{\nat}$ 
and ${\mathfrak D}':=\ol{\mathfrak D}{}'\times_{{\mathfrak D}(\ol{S(T)^{\nat}})}S(T)^{\nat}$.   
Set also 
${\mathfrak E}:=\ol{\mathfrak E}\times_{{\mathfrak D}(\ol{S'(T')^{\nat}})}S'(T')^{\nat}$ 
and ${\mathfrak E}':=\ol{\mathfrak E}{}'\times_{{\mathfrak D}(\ol{S'(T')^{\nat}})}S'(T')^{\nat}$.  
Let $(\ol{\mathcal E},\ol{\nabla})$ and $(\ol{\mathcal E}{}',\ol{\nabla}{}')$ be 
an ${\mathcal O}_{\ol{\mathfrak D}}$-module with integrable connection 
and an ${\mathcal O}_{\ol{\mathfrak D}{}'}$-module 
with integrable connection obtained by $E$, respectively. 
Let $(\ol{\mathcal F},\ol{\nabla})$ and $(\ol{\mathcal F}{}',\ol{\nabla}{}')$ be 
an ${\mathcal O}_{\ol{\mathfrak E}}$-module with integrable connection 
and an ${\mathcal O}_{\ol{\mathfrak E}{}'}$-module 
with integrable connection obtained by $F$, respectively. 
Set 
${\mathcal E}:=\ol{\mathcal E}\otimes_{{\mathcal O}_{\mathfrak D}(\ol{S(T)^{\nat}})}{\mathcal O}_T$ 
and ${\mathcal E}':=\ol{\mathcal E}{}'\otimes_{{\mathcal O}_{\mathfrak D}(\ol{S'(T')^{\nat}})}{\mathcal O}_{T}$. 
Set 
${\mathcal F}:=\ol{\mathcal F}\otimes_{{\mathcal O}_{\mathfrak D}(\ol{S'(T')^{\nat}})}{\mathcal O}_{T'}$ 
and ${\mathcal F}':=\ol{\mathcal F}{}'\otimes_{{\mathcal O}_{\mathfrak D}(\ol{S'(T')^{\nat}})}{\mathcal O}_{T'}$. 
By the local structures of the exact immersions ((\ref{prop:adla})), 
we may assume that 
$\ol{\mathcal P}{}^{\rm ex}=\ol{\mathcal P}{}'^{\rm ex}\times_{\ol{S(T)^{\nat}}}{\mab A}_{\ol{S(T)^{\nat}}}^c$ 
and 
$\ol{\mathcal Q}{}^{\rm ex}=\ol{\mathcal Q}{}'^{\rm ex}\times_{\ol{S'(T')^{\nat}}}{\mab A}_{\ol{S'(T')^{\nat}}}^{c'}$ 
for nonnegative integers $c$ and $c'$. 
Since $E$ and $F$ are crystals, 
$\ol{\mathcal E}=\ol{\mathcal E}{}'\otimes_{{\mathcal O}_{\ol{S(T)^{\nat}}}}
{\mathcal O}_{\ol{S(T)^{\nat}}}\langle x_1,\ldots,x_c\rangle$ 
and 
$\ol{\mathcal F}=\ol{\mathcal F}{}'\otimes_{{\mathcal O}_{\ol{S'(T')^{\nat}}}}
{\mathcal O}_{\ol{S'(T')^{\nat}}}\langle y_1,\ldots,y_{c'}\rangle$. 
Hence ${\mathcal E}={\mathcal E}'\otimes_{{\mathcal O}_T}
{\mathcal O}_T\langle x_1,\ldots,x_c\rangle$ 
and 
${\mathcal F}={\mathcal F}'\otimes_{{\mathcal O}_T}
{\mathcal O}_T\langle y_1,\ldots,y_{c'}\rangle$. 
Because ${\mathcal O}_T\langle x_1,\ldots,x_c\rangle \otimes_{{\mathcal O}_T} 
\Om^{i}_{{\mab A}^c_{\os{\circ}{T}}/\os{\circ}{T}} $ 
$(i\in {\mab N})$ is a free ${\mathcal O}_T$-module, 
the morphism ${\mathcal O}_T\lo 
{\mathcal O}_T\langle x_1,\ldots,x_c\rangle \otimes_{{\mathcal O}_T} 
\Om^{i}_{{\mab A}^c_{\os{\circ}{T}}/\os{\circ}{T}}$ 
is faithfully flat.  
By (\ref{eqn:dopx}) we have the following commutative diagram:  
{\footnotesize{\begin{equation*} 
\begin{CD} 
{\mathcal E}
\otimes_{{\mathcal O}_{{\mathcal P}^{\rm ex}}}
\Om^{i+j+1}_{{\mathcal P}^{\rm ex}/\os{\circ}{T}}/P^{\mathcal X}_j
@>{\sim}>> 
\bigoplus_{i'+i''=i+j+1}
({\mathcal E}'\otimes_{{\mathcal O}_{{\mathcal P}'{}^{\rm ex}}}
\Om^{i'}_{{\mathcal P}'{}^{\rm ex}/\os{\circ}{T}}/P^{{\mathcal X}'}_j)
\otimes_{{\mathcal O}_T}
{\mathcal O}_T\langle x_1,\ldots, x_c\rangle \otimes_{{\mathcal O}_T}
\Om^{i''}_{{\mab A}^c_{\os{\circ}{T}}/\os{\circ}{T}}  \\
@AAA @AAA \\
{\mathcal F}
\otimes_{{\mathcal O}_{{\mathcal Q}^{\rm ex}}}
\Om^{i+j+1}_{{\mathcal Q}^{\rm ex}/\os{\circ}{T}}/P^{\mathcal Y}_j
@>{\sim}>>
\bigoplus_{i'+i''=i+j+1}
({\mathcal F}'\otimes_{{\mathcal O}_{{\mathcal Q}'{}^{\rm ex}}}
\Om^{i'}_{{\mathcal Q}'{}^{\rm ex}/\os{\circ}{T}}/P^{{\mathcal Y}'}_j)
\otimes_{{\mathcal O}_{T'}}
{\mathcal O}_{T'}\langle y_1,\ldots, y_{c'}\rangle \otimes_{{\mathcal O}_{T'}}
\Om^{i''}_{{\mab A}^{c'}_{\os{\circ}{T}{}'}/\os{\circ}{T}{}'}.  
\end{CD} 
\tag{6.3.2}\label{eqn:dpdpx} 
\end{equation*}}}
Here we omit to write the direct images in the lower isomorphism in this 
commutative diagram. 
Hence the divisibilities in 
${\mathcal E}\otimes_{{\mathcal O}_{{\mathcal P}^{\rm ex}}}
{\Om}^{i+j+1}_{{\mathcal P}^{\rm ex}/\os{\circ}{T}{}}/
P^{\mathcal X}_j$ $(\forall i\in {\mab N})$ and  
${\mathcal E}'\otimes_{{\mathcal O}_{{\mathcal P}'{}^{\rm ex}}}
{\Om}^{i+j+1}_{{\mathcal P}'{}^{\rm ex}/\os{\circ}{T}{}}/P^{{\mathcal X}'}_j$ 
$(\forall i\in {\mab N})$ are equivalent. 
\end{proof}

\begin{theo}[{\bf Contravariant functoriality II of $A_{\rm zar}$}]\label{theo:funpas} 
$(1)$ Let the notations and the assumptions be as above. 
Then $g\col (X_{\os{\circ}{T}{}_0},D_{\os{\circ}{T}{}_0})
\lo (Y_{\os{\circ}{T}{}'_0},C_{\os{\circ}{T}{}'_0})$ induces the following 
well-defined pull-back morphism 
\begin{equation*}  
g^* \col 
(A_{\rm zar}((Y_{\os{\circ}{T}{}'_0},C_{\os{\circ}{T}{}'_0})/S'(T')^{\nat},F),P^{C_{\os{\circ}{T}{}'_0}},P)
\lo Rg_*
((A_{\rm zar}((X_{\os{\circ}{T}_0},D_{\os{\circ}{T}_0})/S(T)^{\nat},E),P^{D_{\os{\circ}{T}_0}},P))  
\tag{6.4.1}\label{eqn:fzapxd}
\end{equation*} 
fitting into the following commutative diagram$:$
\begin{equation*}  
\begin{CD}
A_{\rm zar}((Y_{\os{\circ}{T}{}'_0},C_{\os{\circ}{T}{}'_0})/S'(T')^{\nat},F)
@>{g^*}>>  Rg_*(A_{\rm zar}((X_{\os{\circ}{T}_0},D_{\os{\circ}{T}_0})/S(T)^{\nat},E)) \\ 
@A{\theta_{Y_{\os{\circ}{T}{}'_0}/S'(T')^{\nat}}\wedge}A{\simeq}A 
@A{Rg_{*}(\theta_{X_{\os{\circ}{T}_0}/S(T)^{\nat}}\wedge)}A{\simeq}A\\
Ru_{(Y_{\os{\circ}{T}{}'_0},C_{\os{\circ}{T}{}'_0})/S'(T')^{\nat}*}
(\eps^*_{(Y_{\os{\circ}{T}{}'_0},C_{\os{\circ}{T}{}'_0})/S'(T')^{\nat}}(F))
@>{g^*}>>
Rg_*Ru_{(X_{\os{\circ}{T}_0},D_{\os{\circ}{T}_0})/S(T)^{\nat}*}
(\eps^*_{(X_{\os{\circ}{T}_0},D_{\os{\circ}{T}_0})/S(T)^{\nat}}(E)). 
\end{CD}
\tag{6.4.2}\label{cd:pssfpccz} 
\end{equation*}
\par 
$(2)$ The similar relation to {\rm (\ref{ali:pdpp})} holds.  
%\par 
%$(3)$ There exists the similar commutative diagram to {\rm (\ref{cd:tmfccz})}.
\end{theo}
\begin{proof} 
The proof is the same as that of (\ref{theo:funas}). 
\end{proof}

\begin{rema}
Let the notations be as in \cite{msemi}. 
Let $F_{W^{\times}}\col W^{\times}\lo W^{\times}$ be the absolute Frobenius endomorphism 
of $W^{\times}$. 
For the absolute Frobenius endomorphism 
$F_X\col X\lo X$, the Frobenius action $F^*_X$ on $(WA_X,P)$ is not 
well-defined in [loc.~cit.] 
because $\deg (F_{W^{\times}})^{-(j+1)}(=p^{-(j+1)})$ $(j\in {\mab N})$ 
for $WA_X^{\bul j}$ has not been considered in [loc.~cit.] 
and because it has not been proved that 
$F^*_X(P_kWA_X)\subset P_kWA_X$ $(k\in {\mab N})$. 
See also \cite[\S9]{ndw} for this incompleteness. 
\end{rema}

\par

\par 
Let the notations and the assumptions be 
as in (\ref{theo:funas}) or (\ref{theo:funpas}). 
\par 
Consider the morphism 
\begin{align*} 
{\rm gr}^P_k(g^*) 
\col & {\rm gr}^P_kA_{\rm zar}((Y_{\os{\circ}{T}{}'_0},C_{\os{\circ}{T}{}'_0})/S'(T')^{\nat},F) 
\lo Rg_*({\rm gr}^P_kA_{\rm zar}((X_{\os{\circ}{T}_0},D_{\os{\circ}{T}_0})/S(T)^{\nat},E)). 
\tag{6.5.1}\label{ali:gdkcl}
\end{align*}

\par
The following conditions are open SNCL versions of 
the conditions \cite[(2.9.2.3), (2.9.2.4)]{nh2} for the open log case.  
\par   
Assume that the following two conditions hold:

\medskip
\parno 
$(6.5.2)$: for any smooth component 
$\os{\circ}{X}_{\lam}$ of $\os{\circ}{X}_{T_0}$ over $\os{\circ}{T}_0$, 
there exists a unique smooth component 
$\os{\circ}{Y}_{\lam'}$ of 
$\os{\circ}{Y}_{T'_0}$ over $\os{\circ}{T}{}'_0$ such that $g$ 
induces a morphism 
$\os{\circ}{X}_{\lam} \lo \os{\circ}{Y}_{\lam'}$. 
(Let $\Lam$ and $\Lam'$ be the sets of indexes 
of the $\lam$'s and the $\lam'$'s, respectively. 
Then we obtain a function 
$\phi \col \Lam \owns \lam \lom \lam' \in \Lam'$.)
\medskip
\parno
$(6.5.3)$: for any smooth component 
$\os{\circ}{D}_{\mu}$ of $\os{\circ}{X}_{T_0}$ over $\os{\circ}{T}_0$, 
there exists a unique smooth component 
$\os{\circ}{C}_{\mu'}$ of 
$\os{\circ}{Y}_{T'_0}$ over $\os{\circ}{T}{}'_0$ such that $g$ 
induces a morphism 
$\os{\circ}{D}_{\mu} \lo \os{\circ}{C}_{\mu'}$. 
(Let $M$ and $M'$ be the sets of indexes 
of the $\mu$'s and the $\mu'$'s, respectively. 
Then we obtain a function 
$\psi \col M \owns \mu \lom \mu' \in M'$.)

By (6.5.2) there exist a unique positive integer $e_{\lam}$   
for each $\lam \in \Lam$ such that 
there exist local equations $x_{\lam}=0$ and 
$y_{\lam}=0$ of 
$\os{\circ}{X}_{\lam}$ and $\os{\circ}{Y}_{\lam}$, 
respectively,  
such that $g^*(y_{\lam})=u_{\lam}x^{e_{\lam}}_{\lam}$ for some local section 
$u_{\lam}\in {\mathcal O}_{X}^*$.  
%\medskip
%\par 

\begin{prop}\label{prop:xxle} 
Let the assumptions and the notations be as above. 
Set $\Lam(x):=
\{\lam \in \Lam~\vert~x \in \os{\circ}{X}_{\lam}\}$.  
Then $\deg(u)_x=e({\lam})$ for $\lam \in \Lam(x)$.  
In particular, 
$e({\lam})$'s are independent of
the choice of an element of $\Lam(x)$.  
\end{prop}
\begin{proof} 
The proof is completely the same as that of \cite[(1.5.7)]{nb}. 
\end{proof}

\begin{prop}\label{prop:dvok}
Let $n$ be a positive integer. 
Assume that ${\mathcal J}\subset p^n{\mathcal O}_T$. 
Assume that $e_p\geq 1$. 
Set $m:=\min\{n,e_p\}$ $(m$ is a function on $\os{\circ}{X}_{T_0})$. 
Then the morphism {\rm (\ref{eqn:odnl})} is divisible by $p^{m(j+1)}$. 
$($As a result, if $e_p\leq n$, then  the morphism {\rm (\ref{eqn:odnl})}
satisfies the divisibility assumption after {\rm (\ref{defi:efu})}.$)$ 
\end{prop}
\begin{proof} 
The proof is completely the same as that of \cite[(1.5.8)]{nb}. 
\end{proof}

\par 
For a nonnegative integer $k$, 
let $\Lam^{(k)}_X(\os{\circ}{g})$ 
be the sets of subsets $I$'s of $\Lam$ 
such that $\sharp I= \sharp \phi (I)=k+1$.  
Let $M^{(k)}_D(\os{\circ}{g})$ be the sets of 
subsets $J$'s of $M$ 
such that $\sharp J= \sharp \psi (J)=k$. 
For $\ul{\lam}=\{\lam_0,\ldots, \lam_k\}$ 
$(\ul{\lam}\in \Lam^{(k)}(\os{\circ}{g}))$, 
set 
$\os{\circ}{X}_{\ul{\lam}}:=
\os{\circ}{X}_{\lam_0}\cap \cdots \cap \os{\circ}{X}_{\lam_{k}}$ 
and $\os{\circ}{Y}_{\phi(\ul{\lam})}:=
\os{\circ}{Y}_{\phi(\lam_0)}\cap \cdots 
\cap \os{\circ}{Y}_{\phi(\lam_{k})}$.  
For $\ul{\mu}:=
\{\mu_1,\ldots, \mu_{k}\}$ 
$(\ul{\mu}\in \Lam^{(k)}_D(\os{\circ}{g}))$, 
set 
$\os{\circ}{D}_{\ul{\mu}}:=
\os{\circ}{D}_{\mu_1}\cap \cdots \cap \os{\circ}{D}_{\mu_{k}}$ 
and 
$\os{\circ}{C}_{\psi(\ul{\mu})}:=
\os{\circ}{C}_{\psi(\mu_1)}\cap \cdots \cap \os{\circ}{C}_{\psi(\mu_{k})}$. 
Let $\os{\circ}{g}_{\ul{\lam}\ul{\mu}}\col \os{\circ}{X}_{\ul{\lam}}\cap 
\os{\circ}{D}_{\ul{\mu}}\lo \os{\circ}{Y}_{\phi(\ul{\lam})}\cap 
\os{\circ}{C}_{\psi(\ul{\mu})}$ be the natural morphism of schemes. 
%Let 
%$\os{\circ}{g}_{\ul{\lam}} \col \os{\circ}{X}_{\ul{\lam}} \lo \os{\circ}{Y}_{\phi(\ul{\lam})}$ 
%and 
%$\os{\circ}{g}_{\ul{\mu}} \col \os{\circ}{D}_{\ul{\mu}} \lo \os{\circ}{C}_{\psi(\ul{\mu})}$ 
%be the induced morphism by $g$.
\par 
For integers $j$, $k'$ and $k$ such that $j\geq \max \{-k',0\}$ and $k'\leq k$, 
set $\ul{\lam}:=\{\lam_0, \ldots, \lam_{2j+k'}\}\in 
\Lam^{(2j+k')}_X(\os{\circ}{g})$ and $\ul{\mu}:=\{\mu_1,\ldots,\mu_{k-k'}\}\in 
\Lam^{(k-k')}_D(\os{\circ}{g})$. 
Let $a_{\ul{\lam}\ul{\mu}} \col 
\os{\circ}{X}_{\ul{\lam}}\cap \os{\circ}{D}_{\ul{\mu}} 
\os{\sus}{\lo} \os{\circ}{X}_{T_0}$  
and 
$a_{\phi(\ul{\lam})\phi(\ul{\mu})} \col \os{\circ}{Y}_{\phi(\ul{\lam})}
\cap \os{\circ}{C}_{\phi(\ul{\mu})} 
\os{\sus}{\lo}  \os{\circ}{Y}_{T'_0}$ 
be the natural closed immersions.  
Let 
$$a_{\ul{\lam}\ul{\mu}{\rm crys}} 
\col 
(((\os{\circ}{X}_{\ul{\lam}}\cap \os{\circ}{D}_{\ul{\mu}})/{\os{\circ}{T}})_{\rm crys},
{\mathcal O}_{\os{\circ}{X}_{\ul{\lam}}/{\os{\circ}{T}}})
\lo 
((\os{\circ}{X}_{T_0}/{\os{\circ}{T}})_{\rm crys},
{\mathcal O}_{\os{\circ}{X}_{T_0}/\os{\circ}{T}})$$  
and 
$$a_{\phi(\ul{\lam})\phi(\ul{\mu}){\rm crys}} 
\col 
(((\os{\circ}{Y}_{\phi(\ul{\lam})}\cap \os{\circ}{C}_{\phi(\ul{\mu})})/{\os{\circ}{T}{}'})_{\rm crys},
{\mathcal O}_{\os{\circ}{Y}_{\phi(\ul{\lam}),T'_0}/
{\os{\circ}{T}{}'}})
\lo 
((\os{\circ}{Y}_{T'_0}/{\os{\circ}{T}{}'})_{\rm crys},
{\mathcal O}_{\os{\circ}{Y}_{T'_0}/\os{\circ}{T}{}'})$$  
be the induced morphisms of ringed topoi by 
$a_{\ul{\lam}\ul{\mu}}$ and $a_{\phi(\ul{\lam})\phi(\ul{\mu})}$, respectively. 
Set $E_{\ul{\lam}\ul{\mu}}:=a^*_{\ul{\lam}\ul{\mu}{\rm crys}}(E)$ 
and $F_{\phi(\ul{\lam})\phi(\ul{\mu})}:=a^*_{\phi(\ul{\lam})\phi(\ul{\mu}){\rm crys}}(F)$.   
Let 
$\vp_{\ul{\lam}{\rm crys}}(\os{\circ}{X}_{T_0}/\os{\circ}{T})$ 
(resp.~
$\vp_{\phi(\ul{\lam}){\rm crys}}(\os{\circ}{Y}_{T'_0}/\os{\circ}{T}{}')$) 
be the crystalline orientation sheaf in 
$(\os{\circ}{X}_{\ul{\lam}}/{\os{\circ}{T}})_{\rm crys}$ 
(resp.~$(\os{\circ}{Y}_{\phi(\ul{\lam})}/{\os{\circ}{T}{}'})_{\rm crys}$) for the set 
$\{\os{\circ}{X}_{\lam_0}, \ldots, \os{\circ}{X}_{\lam_{2j+k'}}\}$ 
(resp.~$\{\os{\circ}{Y}_{\phi(\lam_0)},\ldots,\os{\circ}{Y}_{\phi(\lam_{2j+k'})}\}$). 
Let $\vp_{\ul{\mu}{\rm crys}}(\os{\circ}{D}_{T_0}/\os{\circ}{T})$ 
and 
$\vp_{\phi(\ul{\mu}){\rm crys}}(\os{\circ}{C}_{T'_0}/\os{\circ}{T}{}')$ 
be the crystalline orientation sheaves in 
$(\os{\circ}{D}_{\ul{\mu}}/{\os{\circ}{T}})_{\rm crys}$ 
and $(\os{\circ}{C}_{\psi(\ul{\mu})}/{\os{\circ}{T}{}'})_{\rm crys}$) for the sets  
$\{\os{\circ}{D}_{\mu_0}, \ldots, \os{\circ}{D}_{\mu_{k-k'}}\}$ 
and $\{\os{\circ}{C}_{\psi(\mu_0)}, \ldots, \os{\circ}{C}_{\psi(\mu_{k-k'})}\}$, 
respectively.

Assume that the divisibility condition for the morphism 
(\ref{eqn:odnl}) holds. 
Then consider the following direct factor of the morphism (\ref{ali:gdkcl}): 
\begin{align*}
g^*_{\ul{\lam}\ul{\mu}} \col & 
a_{\phi(\ul{\lam})\psi(\ul{\mu})*}
Ru_{\os{\circ}{Y}_{\phi(\ul{\lam})}\cap \os{\circ}{C}_{\psi(\ul{\mu})}/\os{\circ}{T}{}'*}
(F_{\phi(\ul{\lam})\psi(\ul{\mu})}\otimes_{\mab Z}
\vp_{\phi(\ul{\lam})\psi(\ul{\mu}){\rm crys}}
((\os{\circ}{Y}_{T'_0},\os{\circ}{C}_{T'_0})/\os{\circ}{T}{}')))[-2j-k]  
\tag{6.7.1}\label{ali:grgm}\\
{} & \lo a_{\phi(\ul{\lam})\phi(\ul{\mu})*}
R\os{\circ}{g}_{\ul{\lam}\ul{\mu}*}
Ru_{\os{\circ}{X}_{\ul{\lam}}\cap 
\os{\circ}{D}_{\ul{\mu}}/\os{\circ}{T}*}(E_{\ul{\lam}\ul{\mu}}\otimes_{\mab Z}
\vp_{\ul{\lam}{\rm crys}}
((\os{\circ}{X}_{T_0},\os{\circ}{D}_{T_0})/\os{\circ}{T}))[-2j-k]. 
\end{align*}

\begin{defi}
Let $v \col {\mathcal E} \lo {\mathcal F}$ 
be a morphism of  
$f^{-1}({\mathcal O}_T)$-modules (resp.~${\mathcal O}_T$-modules). 
The $D$-{\it twist}(:=degree twist) by $k$  
$$v(-k) \col {\mathcal E}(-k,u)\lo {\mathcal F}(-k,u)$$  
of $v$ with respect to $u$ 
is, by definition, the morphism  
$\deg(u)^kv \col {\mathcal E} \lo {\mathcal F}$.  
This definition is well-defined for morphisms of objects of 
the derived category 
$D^+(f^{-1}({\mathcal O}_T))$ (resp.~$D^+({\mathcal O}_T)$). 
\end{defi}

\par 
By the assumption (6.5.3), 
there exist positive integers $e(\mu)$'s  
$(\mu \in M)$ such that 
there exist local equations $y_{\mu}=0$ and 
$y'_{\psi(\mu)}=0$ of 
$\os{\circ}{D}_{\mu}$ and $\os{\circ}{C}_{\psi(\mu)}$, 
respectively,  such that $g^*(y'_{\psi(\mu)})=y^{e({\mu})}_{\mu}$.

%\par
%The following new definition is a generalization of the usual Tate twist and 
%the $l$-adic analogue of the $D$-twist(degree-twist) defined in \cite[(1.5.10)]{nb} 
%and \cite[(5.1.5)]{nhir}: 

The following definition is the $p$-adic analogue of that 
in \cite[(11.6)]{nlf}:

\begin{defi}
\label{defi:wcel}
(1) We call 
$\{e(\mu)\}_{\mu \in M}\in {\mab Z}^M_{>0}$ the {\it multi-degree} of 
$g$ with respect to a decomposition 
$\Del:=\{D_{\mu}\}_{\mu \in M}$ 
and $\Del':=\{C_{\mu'}\}_{\mu'\in \psi(M)}$ of $D$ and $C$, respectively.   
We  denote it by ${\rm deg}_{\Del,\Del'}(g) \in {\mab Z}^M_{>0}$. 
%If $e(\mu)$'s for all $\mu$'s are equal, we also denote 
%$e(\mu)\in {\mab Z}_{>0}$ by ${\rm deg}_{\Del,\Del'}(g) \in {\mab Z}_{>0}$.
\par
(2) Let $f'\col (Y_{\os{\circ}{T}{}'_0},C_{\os{\circ}{T}{}'_0})\lo S'(T')^{\nat}$ 
be the structural morphism. 
Let ${\mathcal E}_{\mu_1\cdots \mu_k}$ 
(resp.~${\mathcal F}_{\mu'_1\cdots \mu'_k}$) be an 
${\mathcal O}_{D_{\mu_1,\cdots, \mu_k}}$-module 
(resp.~an ${\mathcal O}_{C_{\mu'_1\cdots \mu'_k}}$-module). 
%Assume that 
%$e(\mu)$'s for all $\mu$'s are equal. 
Let 
$$u \col {\mathcal F}=
\bigoplus_{\{\mu_1,\ldots, \mu_k\}}{\mathcal F}_{\psi(\mu_1)\cdots \psi(\mu_k)}\lo 
g_*({\mathcal E}):=g_*(\bigoplus_{\{\mu_1,\ldots, \mu_k\}}{\mathcal E}_{\mu_1\cdots \mu_k})$$  
be a morphism of   
$f'{}^{-1}({\mathcal O}_{T'})$-modules. 
Let us consider the following composite morphism:  
$$u \col \bigoplus_{\{\mu'_1,\ldots, \mu'_k\}}{\mathcal F}_{\mu'_1,\ldots,\mu'_k}
\os{{\rm proj}.}{\lo} 
\bigoplus_{\{\mu_1,\ldots, \mu_k\}}{\mathcal F}_{\psi(\mu_1)\cdots \psi(\mu_k)}\lo 
g_*({\mathcal E}):=g_*(\bigoplus_{\{\mu_1,\ldots, \mu_k\}}{\mathcal E}_{\mu_1\cdots \mu_k}).$$ 
Let $k$ be a nonnegative integer. 
The $k$-{\it twist}\index{twist} 
$$u(-k) \col {\mathcal F}(-k;g;\Del,\Del'):=
\bigoplus_{\{\mu'_1,\ldots, \mu'_k\}}{\mathcal F}_{\mu'_1,\ldots,\mu'_k} 
\lo g_*({\mathcal E}(-k;g;\Del,\Del')):=
g_*(\bigoplus_{\{\mu_1,\ldots, \mu_k\}}{\mathcal E}_{\mu_1,\ldots, \mu_k})$$  
of $u$ with respect to $v$, $\Del$ and $\Del'$ 
is, by definition, the induced morphism by the direct sum of the morphisms  
$e_{\mu_1}\cdots e_{\mu_k}u 
\col \allowbreak {\mathcal F}_{\psi(\mu_1)\cdots \psi(\mu_k)} 
\allowbreak \lo {\mathcal E}_{\mu_1\cdots \mu_k}$'s.
\end{defi}

\parno

\begin{prop}\label{prop:grloc}
Let the notations 
and the assumptions be as above. 
%Let $a^{(l),(m)'}\col \os{\circ}{Y}{}^{(l)}_{T_0}\cap \os{\circ}{C}{}^{(m)}_{T_0}\lo 
%\os{\circ}{Y}_{T_0}$ be the analogous morphism to 
%$a^{(l),(m)}\col \os{\circ}{X}{}^{(l)}_{T_0}\cap \os{\circ}{D}{}^{(m)}_{T_0}\lo 
%\os{\circ}{Y}_{T_0}$. 
Then the morphism $g^*_{\ul{\lam}\ul{\mu}}$ in {\rm (\ref{ali:grgm})} 
is equal to 
$\deg(u)^{j+k'}
e_{\mu_1}\cdots e_{\mu_{k-k'}}a_{\phi(\ul{\lam})\phi(\ul{\mu})*}
\os{\circ}{g}{}^*_{\ul{\lam}\ul{\mu}}$ for $j\geq \max\{-k',0\}$. 
\end{prop}
\begin{proof}
The proof is essentially 
the same as that of \cite[(2.9.3)]{nh2} and \cite[(1.5.9)]{nb}. 
\end{proof}

\begin{coro}\label{coro:fuu}
The morphism 
\begin{align*} 
g^*\col 
{\rm gr}^P_kA_{\rm zar}((Y_{\os{\circ}{T}{}'_0},C_{\os{\circ}{T}{}'_0})/S'(T')^{\nat},F)
\lo 
Rg_{*}({\rm gr}^P_kA_{\rm zar}((X_{\os{\circ}{T}_0},D_{\os{\circ}{T}_0})/S(T)^{\nat},E))
\tag{6.11.1}\label{ali:cgrvp} 
\end{align*} 
is equal to 
\begin{align*} 
& \bigoplus_{k'\leq k} 
\bigoplus_{j\geq \max \{-k,0\}} 
a^{(2j+k'),(k-k')}_{*} 
(Ru_{\os{\circ}{Y}{}^{(2j+k')}_{T'_0}\cap \os{\circ}{C}{}^{(k-k')}_{T'_0}/\os{\circ}{T}{}'*}
(F_{\os{\circ}{Y}{}^{(2j+k)}_{T'_0}\cap \os{\circ}{C}{}^{(k-k')}_{T'_0}/\os{\circ}{T}{}'} 
\otimes_{\mab Z}\tag{6.11.2}\label{ali:cgrp}\\ 
&\vp_{\rm crys}^{(2j+k'),(k-k')}
((\os{\circ}{Y}_{T'_0},\os{\circ}{C}_{T'_0})/T')))(-j-k';u)(-(k-k');g,\Del,\Del')[-2j-k]
\\ 
%& = \\
%&\bigoplus_{j\geq \max \{-k,0\}} 
%\bigoplus_{\ul{\lam}\in \Lam^{(k)}(\os{\circ}{g}),}
%a_{\phi(\ul{\lam})\phi(\ul{\mu})*} 
%(Ru_{\os{\circ}{Y}_{\phi(\ul{\lam})}\cap \os{\circ}{C}_{\phi(\ul{\mu})}/\os{\circ}{T}{}'*}
%(F_{\os{\circ}{Y}_{\phi(\ul{\lam})}\cap \os{\circ}{C}_{\phi(\ul{\mu})}/\os{\circ}{T}{}'} 
%\otimes_{\mab Z}\vp_{\phi(\ul{\lam})\phi(\ul{\mu}),{\rm crys}}
%(\os{\circ}{Y}_{T'_0}/T')))\\
%&(-j-k';u)(-(k-k');g,\Del,\Del')[-2j-k] \\ 
%& \os{\sum_{\ul{\lam}\in \Lam^{(k)}_X(\os{\circ}{g})}
%\os{\circ}{g}{}^*_{\ul{\lam}}(-j-k,u)}{\lo} \\
%&\bigoplus_{j\geq \max \{-k,0\}} 
%\bigoplus_{{\ul{\lam}\in \Lam^{(2j+k')}(\os{\circ}{g})}}
%a_{\ul{\lam}\ul{\mu}*} 
%(Ru_{\os{\circ}{X}_{\ul{\lam}}\cap \os{\circ}{D}_{\ul{\mu}}/\os{\circ}{T}*}
%(E_{\os{\circ}{X}_{\ul{\lam}}\cap \os{\circ}{D}_{\ul{\mu}}/\os{\circ}{T}} 
%\otimes_{\mab Z}\vp_{\ul{\lam}\ul{\mu},{\rm crys}}
%(\os{\circ}{X}_{T_0}/T)))\\
%&(-j-k,u)[-2j-k] \\ 
& \lo \\
&\bigoplus_{k'\leq k} \bigoplus_{j\geq \max \{-k,0\}} 
a^{(2j+k'),(k-k')}_{*} 
(Ru_{\os{\circ}{X}{}^{(2j+k')}_{T_0}\cap \os{\circ}{D}{}^{(k-k')}_{T_0}/\os{\circ}{T}*}
(E_{\os{\circ}{X}{}^{(2j+k')}_{T_0}\cap \os{\circ}{D}{}^{(k-k')}_{T_0}/\os{\circ}{T}} 
\otimes_{\mab Z}\\
&\vp_{\rm crys}^{(2j+k'),(k-k')}
((\os{\circ}{X}_{T_0},\os{\circ}{D}_{T_0})/T)))(-j-k';u)(-(k-k');g,\Del,\Del')[-2j-k]. 
\end{align*} 
\end{coro} 
\begin{proof} 
This immediately follows from (\ref{prop:grloc}).  
\end{proof}

\begin{coro}\label{coro:indg}
Let  $h\col X_{\os{\circ}{T}_0}\lo Y_{\os{\circ}{T}{}'_0}$ 
be another morphism satisfying the condition {\rm (\ref{cd:xygxy})}, {\rm (6.5.2)} and 
{\rm (6.5.3)}. 
Assume that $\os{\circ}{g}=\os{\circ}{h}$. 
Then 
\begin{align*} 
{\mathcal H}^q(h^*)={\mathcal H}^q(g^*) \col &
{\mathcal H}^q(P_kA_{\rm zar}((Y_{\os{\circ}{T}{}'_0},C_{\os{\circ}{T}{}'_0})/S'(T')^{\nat},F)) \\
&\lo {\mathcal H}^q(Rg_{*}(P_kA_{\rm zar}((X_{\os{\circ}{T}_0},D_{\os{\circ}{T}_0})/
S(T)^{\nat},E))) \quad (q\in {\mab N}).  
\end{align*} 
\end{coro}
\begin{proof} 
%It suffices to prove that $g^*=h^*$. 
This is a local question on $\os{\circ}{Y}_{T'_0}$. 
Hence we may assume that $\os{\circ}{Y}_{T'_0}$ is quasi-compact. 
It suffices to prove that 
${\mathcal H}^q({\rm gr}_k^P(h^*))={\mathcal H}^q({\rm gr}_k^P(g^*))$ $(q\in {\mab N},k\in {\mab Z})$. 
This follows from (\ref{coro:fuu}). 
\end{proof}

\par   
Let $f \col (X_{\os{\circ}{T}_0},D_{\os{\circ}{T}_0}) \lo S(T)^{\nat}$ 
be the structural morphism. 
By (\ref{ali:cgrvp}) 
we have the following spectral sequence 
\begin{align*} 
&E_1^{-k,q+k}:=E_1^{-k,q+k}((X_{\os{\circ}{T}_0},D_{\os{\circ}{T}_0})/S(T)^{\nat}) 
:=\bigoplus^k_{k'=-\infty}
\bigoplus_{j\geq \max\{-k',0\}} 
R^{q-2j-k}f_{\os{\circ}{X}{}^{(2j+k')}_{T_0}\cap \os{\circ}{D}{}^{(k-k')}_{T_0}
/\os{\circ}{T}*}\tag{6.12.1}\label{eqn:escssp} \\
&(E_{\os{\circ}{X}{}^{(2j+k')}_{T_0}\cap 
\os{\circ}{D}{}^{(k-k')}_{T_0}/\os{\circ}{T}}
\otimes_{\mab Z} 
\vp^{(2j+k'),(k-k')}_{\rm crys}(
(\os{\circ}{X}_{T_0},\os{\circ}{D}_{T_0})/\os{\circ}{T}))(-j-k';u)(-(k-k');g,\Del,\Del') \\
&\Lo R^qf_{(X_{\os{\circ}{T}_0},D_{\os{\circ}{T}_0})/S(T)^{\nat}*}
(\eps^*_{(X_{\os{\circ}{T}_0},D_{\os{\circ}{T}_0})/S(T)^{\nat}}(E))  
\quad (q\in {\mab Z}).  
\end{align*}
The precise meaning of this spectral sequence is the following:

\begin{coro}[{\bf Contravariant functoriality 
of the (pre)weight spectral sequence(a generalization of the $p$-adic analogue of 
{\rm \cite[Corollary 2.12]{stwsl}})}]\label{prop:contwt}
Let the notations and the assumption be as above 
and in {\rm (\ref{theo:funas})} or {\rm (\ref{theo:funpas})}.  
Assume that the two conditions 
$(6.5.2)$ and $(6.5.3)$ hold. 
Let $\os{\circ}{g}{}^{(l),(m)*}$ be the following  morphism$:$
\begin{align*}  
\os{\circ}{g}{}^{(l),(m)*}
& :=\sum_{\ul{\lam}\in \Lam^{(l)}_X(\os{\circ}{g}),\ul{\mu}\in \Lam^{(m)}_D(\os{\circ}{g})}
\os{\circ}{g}{}^*_{\ul{\lam}\ul{\mu}}
\col 
R^qf_{\os{\circ}{Y}{}^{(l)}_{T'_0}\cap \os{\circ}{C}{}^{(m)}_{T'_0}/\os{\circ}{T}{}'*}
(F_{\os{\circ}{Y}{}^{(l)}_{T'_0}\cap \os{\circ}{C}{}^{(m)}_{T'_0}/\os{\circ}{T}{}'}
\otimes_{\mab Z}
\vp^{(l,m)}_{\rm crys}((\os{\circ}{Y}_{T'_0},\os{\circ}{C}_{T'_0})/\os{\circ}{T}{}')) 
\tag{6.13.1}\label{eqn:rglm} \\
& \lo 
R^qf_{\os{\circ}{X}{}^{(l)}_{T_0}\cap \os{\circ}{D}{}^{(l)}_{T_0}/\os{\circ}{T}{}*}
(E_{\os{\circ}{X}{}^{(l)}_{T_0}\cap \os{\circ}{D}{}^{(m)}_{T_0}/\os{\circ}{T}}
\otimes_{\mab Z}
\vp^{(l,m)}_{\rm crys}((\os{\circ}{X}_{T_0},\os{\circ}{D}_{T_0})/T)). 
\end{align*} 
Then there exists the following morphism of 
$($pre$)$weight spectral sequences, 
where the left arrow below is the induced morphism by 
$\os{\circ}{g}{}^{(l),(m)*}$'s$:$ 
\footnotesize{\begin{align*} 
& E_1^{-k,q+k} = \bigoplus_{k'\leq k}\bigoplus_{j\geq 
\max \{-k',0\}}R^{q-2j-k}f_{\os{\circ}{X}{}^{(2j+k')}_{T_0}\cap 
\os{\circ}{D}{}^{(k-k')}_{T_0}/\os{\circ}{T}*}
(E_{\os{\circ}{X}{}^{(2j+k')}_{T_0}\cap \os{\circ}{D}{}^{(k-k')}_{T_0}
/\os{\circ}{T}}
\otimes_{\mab Z} \tag{6.13.2}\label{eqn:espwfsp} \\
& \vp^{(2j+k'),(k-k')}_{\rm crys}
((\os{\circ}{X}_{T_0},\os{\circ}{D}_{T_0})/\os{\circ}{T}))(-j-k';u)(-(k-k');g,\Del,\Del') 
\Lo 
R^qf_{(X_{\os{\circ}{T}_0},D_{\os{\circ}{T}_0})/S(T)^{\nat}*}
(\eps^*_{(X_{\os{\circ}{T}_0},D_{\os{\circ}{T}_0})/S(T)^{\nat}}(E))
\end{align*}}
\begin{equation*} 
\begin{CD} 
{ \quad \quad \quad \quad \quad \quad \quad \quad \quad} @. 
{ \quad \quad \quad \quad \quad \quad}\\ 
@AAA  
@. @. @AA{g^*}A \\
\end{CD} 
\end{equation*}
\begin{align*} 
& E_1^{-k,q+k} = \bigoplus_{k'\leq k}\bigoplus_{j\geq 
\max \{-k',0\}}R^{q-2j-k}f_{\os{\circ}{Y}{}^{(2j+k')}_{T_0}\cap 
\os{\circ}{C}{}^{(k-k')}_{T_0}/\os{\circ}{T}*}
(F_{\os{\circ}{Y}{}^{(2j+k')}_{T_0}\cap \os{\circ}{C}{}^{(k-k')}_{T'_0}
/\os{\circ}{T}{}'}\otimes_{\mab Z}  \\
& \vp^{(2j+k'),(k-k')}_{\rm crys}
((\os{\circ}{Y}_{T'_0},\os{\circ}{C}_{T'_0})/\os{\circ}{T}{}'))(-j-k';u)(-(k-k');g,\Del,\Del') 
\Lo 
R^qf_{(Y_{\os{\circ}{T}{}'_0},C_{\os{\circ}{T}{}'_0})/S(T')^{\nat}*}
(\eps^*_{(Y_{\os{\circ}{T}{}'_0},D_{\os{\circ}{T}{}'_0})/S(T')^{\nat}}(F)). 
\end{align*}
\end{coro}
\begin{proof} 
This follows from (\ref{theo:funpas}), (\ref{prop:grloc}) 
and from the definition of $\os{\circ}{g}{}^{(l,m)}$ 
$(l,m\in {\mab N})$.  
\end{proof}

\begin{defi} 
(1) We call the spectral sequence (\ref{eqn:escssp}) 
the {\it Poincar\'{e} spectral sequence} of 
$\eps^*_{(X_{\os{\circ}{T}_0},D_{\os{\circ}{T}_0})/S(T)^{\nat}}(E)$.
If $E={\mathcal O}_{\os{\circ}{X}_{T_0}/\os{\circ}{T}}$ 
and if $p$ is locally nilpotent on $\os{\circ}{T}$,  
then we call the spectral sequence (\ref{eqn:escssp}) 
the {\it preweight spectral sequence}  of 
$(X_{\os{\circ}{T}_0},D_{\os{\circ}{T}_0})/S(T)^{\nat}$. 
If $E={\mathcal O}_{\os{\circ}{X}_{T_0}/\os{\circ}{T}}$ and 
if $\os{\circ}{T}$ is a flat formal ${\mab Z}_p$-scheme,  
then we call the spectral sequence (\ref{eqn:escssp}) 
the {\it weight spectral sequence} of 
$(X_{\os{\circ}{T}_0},D_{\os{\circ}{T}_0})/S(T)^{\nat}$.
\par 
(2) We usually denote by $P$ the induced filtration on 
$R^qf_{(X_{\os{\circ}{T}_0},D_{\os{\circ}{T}_0})/S(T)^{\nat}*}
(\eps^*_{(X_{\os{\circ}{T}_0},D_{\os{\circ}{T}_0})/S(T)^{\nat}}(E))$
by the spectral sequence (\ref{eqn:escssp}) 
twisted by $q$ by abuse of notation. 
We call $P$ the {\it Poincar\'{e} filtration} on 
$R^qf_{(X_{\os{\circ}{T}_0},D_{\os{\circ}{T}_0})/S(T)^{\nat}*}
(\eps^*_{X_{\os{\circ}{T}_0},D_{\os{\circ}{T}_0})/S(T)^{\nat}}(E))$. 
If $E
={\mathcal O}_{\os{\circ}{X}_{T_0}/\os{\circ}{T}}$ 
and if $p$ is locally nilpotent on $\os{\circ}{T}$,  
then we call $P$ the {\it preweight filtration} on 
$$R^qf_{(X_{\os{\circ}{T}_0},D_{\os{\circ}{T}_0})/S(T)^{\nat}*}
({\mathcal O}_{(X_{\os{\circ}{T}_0},D_{\os{\circ}{T}_0})/S(T)^{\nat}}).$$ 
If $E={\mathcal O}_{\os{\circ}{X}_{T_0}/\os{\circ}{T}}$ 
and if $\os{\circ}{T}$ is a flat formal ${\mab Z}_p$-scheme,  
then we call $P$ the {\it weight filtration} on 
$R^qf_{(X_{\os{\circ}{T}_0},D_{\os{\circ}{T}_0})/S(T)^{\nat}*}
({\mathcal O}_{(X_{\os{\circ}{T}_0},D_{\os{\circ}{T}_0})/S(T)^{\nat}})$.   
\end{defi}

%{\bigoplus_{k'\leq k}\bigoplus_{j\geq \max \{-k',0\}}{\rm deg}(u)^{j+k'}\os{\circ}{g}{}^{(2j+k'),(k-k')*}}

We also obtain the following:

\begin{prop}\label{prop:e1sp}
There exists the following spectral sequence 
\begin{align*} 
E_1^{k,q-k}&=R^{q-k}f_{D^{(k)}_{\os{\circ}{T}_0}/S(T)^{\nat}*}
(\eps^*_{D^{(k)}_{\os{\circ}{T}_0}/S(T)^{\nat}}
(E_{\os{\circ}{D}{}^{(k)}/\os{\circ}{T}})\otimes_{\mab Z}
\eps^{-1}_{D^{(k)}_{\os{\circ}{T}_0}/S(T)^{\nat}}\vp^{(k)}_{\rm crys}
((\os{\circ}{D}_{T_0}/\os{\circ}{T}_0)))(-k;g,\Del,\Del') 
\tag{6.15.1}\label{ali:dkfs}\\
&\Lo 
R^qf_{(X_{\os{\circ}{T}_0},D_{\os{\circ}{T}_0})/S(T)^{\nat}*}
(\eps^*_{(X_{\os{\circ}{T}_0},D_{\os{\circ}{T}_0})/S(T)^{\nat}}(E)). 
\end{align*} 
\end{prop}
\begin{proof} 
By (\ref{coro:pwspp}) and the similar argument to the argument above, we  
obtain (\ref{prop:e1sp}). 
\end{proof}

As in \cite[(1.5.4)]{nb} and 
\cite[(5.1.6)]{nhir} 
let us recall the following morphisms. $S_{\os{\circ}{T}_0}\lo S^{[p]}_{\os{\circ}{T}{}'_0}$ 
and 
$(S(T)^{\nat},{\mathcal J},\del)\lo (S^{[p]}(T')^{\nat},{\mathcal J}',\del')$ 
the abrelative Frobenius morphism of base log schemes
and the abrelative Frobenius morphism of base log PD-schemes, respectively.

\begin{defi}[{\bf  Abrelative Frobenius morphism (\cite[(1.5.4)]{nb}, \cite[(5.1.6)]{nhir})}]\label{defi:rwd}  
(1) Assume that $\os{\circ}{S}$ is of characteristic $p>0$.  
Let $\os{\circ}{F}_{S}\col \os{\circ}{S} \lo \os{\circ}{S}$ 
be the Frobenius endomorphism of $\os{\circ}{S}$.   
Set $S^{[p]}:=S\times_{\os{\circ}{S},\os{\circ}{F}_{S}}\os{\circ}{S}$. 
Then we have the following natural morphisms  
\begin{align*} 
F_{S/\os{\circ}{S}}\col S\lo S^{[p]}
\end{align*}  
and 
\begin{align*} 
W_{S/\os{\circ}{S}} \col S^{[p]}\lo S.
\end{align*} 
(The underlying morphism of the former morphism is ${\rm id}_{\os{\circ}{S}}$.)
Let $(T,{\mathcal J},\del)\lo (T',{\mathcal J}',\del')$ be a morphism of 
$p$-adic formal log PD-enlargements over the morphism $S\lo S^{[p]}$.  
Then we have the following natural morphisms  
\begin{align*} 
S_{\os{\circ}{T}_0}\lo S^{[p]}_{\os{\circ}{T}{}'_0}
\end{align*} 
and 
\begin{align*} 
(S(T)^{\nat},{\mathcal J},\del)\lo (S^{[p]}(T')^{\nat},{\mathcal J}',\del').
\end{align*}  
As in \cite[(1.5.4)]{nb} and 
\cite[(5.1.6)]{nhir} 
we call the morphisms $S_{\os{\circ}{T}_0}\lo S^{[p]}_{\os{\circ}{T}{}'_0}$ 
and 
$(S(T)^{\nat},{\mathcal J},\del)\lo (S^{[p]}(T')^{\nat},{\mathcal J}',\del')$ 
the {\it abrelative Frobenius morphism of base log schemes}
and the {\it abrelative Frobenius morphism of base log PD-schemes}, respectively.  
In particular, when $(T',{\mathcal J}',\del')=(T,{\mathcal J},\del)$ 
with morphism $T_0\lo S$, 
we have the following natural morphisms 
\begin{align*} 
S_{\os{\circ}{T}_0}\lo S^{[p]}_{\os{\circ}{T}_0}
\end{align*} 
and 
\begin{align*} 
(S(T)^{\nat},{\mathcal J},\del)\lo (S^{[p]}(T)^{\nat},{\mathcal J},\del)
\end{align*}   
by using a composite morphism 
$T_0\lo S\os{W_{S/\os{\circ}{S}}}{\lo} S^{[p]}$.  
\par 
(2) Let the notations be as in (1).  
Set $(X^{[p]},D^{[p]}):=(X,D)\times_SS^{[p]}$ 
and 
\begin{align*} 
(X^{[p]}_{\os{\circ}{T}{}'_0},D^{[p]}_{\os{\circ}{T}{}'_0})&:=
(X^{[p]},D^{[p]})\times_{S^{[p]}}S^{[p]}_{\os{\circ}{T}{}'_0}. 
\end{align*} 
Then $(X^{[p]}_{\os{\circ}{T}{}'_0},D^{[p]}_{\os{\circ}{T}{}'_0})/S^{[p]}_{\os{\circ}{T}{}'_0}$ 
is an SNCL scheme with a relative SNCD over $S^{[p]}_{\os{\circ}{T}{}'_0}$. 
Let 
$$F^{\rm ar}_{(X_{\os{\circ}{T}_0},D_{\os{\circ}{T}_0})
/S_{\os{\circ}{T}_0},S^{[p]}_{\os{\circ}{T}{}'_0}}\col 
(X_{\os{\circ}{T}_0},D_{\os{\circ}{T}_0})\lo (X^{[p]}_{\os{\circ}{T}{}'_0},D^{[p]}_{\os{\circ}{T}{}'_0})$$ 
and 
$$F^{\rm ar}_{(X_{\os{\circ}{T}_0},D_{\os{\circ}{T}_0})/S(T)^{\nat},S^{[p]}(T')^{\nat}}
\col 
(X_{\os{\circ}{T}_0},D_{\os{\circ}{T}_0})
\lo (X^{[p]}_{\os{\circ}{T}{}'_0},D^{[p]}_{\os{\circ}{T}{}'_0})$$ 
be the natural morphisms   
over $S_{\os{\circ}{T}_0}\lo (S^{[p]})_{\os{\circ}{T}{}'_0}$ and 
$(S(T)^{\nat},{\mathcal J},\del)\lo (S^{[p]}(T')^{\nat},{\mathcal J}',\del')$. 
We call 
$F^{\rm ar}_{(X_{\os{\circ}{T}_0},D_{\os{\circ}{T}_0})/S_{\os{\circ}{T}_0},S^{[p]}_{\os{\circ}{T}{}'_0}}$ 
and 
$F^{\rm ar}_{(X_{\os{\circ}{T}_0},D_{\os{\circ}{T}_0})/S(T)^{\nat},S^{[p]}(T')^{\nat}}$ 
the {\it abrelative Frobenius morphisms}   
of $(X_{\os{\circ}{T}_0},D_{\os{\circ}{T}_0})$ over $S_{\os{\circ}{T}_0}\lo S^{[p]}_{\os{\circ}{T}{}'_0}$ 
and $(S(T)^{\nat},{\mathcal J},\del)\lo (S^{[p]}(T')^{\nat},{\mathcal J}',\del')$, respectively. 
\par 
Assume that ${\mathcal O}_T$ is $p$-torsion-free and that ${\mathcal J}= p{\mathcal O}_T$.  
Let $E$ and $E'$ be a flat quasi-coherent crystal of 
${\mathcal O}_{\os{\circ}{X}_{T_0}/\os{\circ}{T}}$-modules and 
a flat quasi-coherent crystal of 
${\mathcal O}_{\os{\circ}{X}{}^{[p]}_{T'_0}/\os{\circ}{T}{}'}$-modules, 
respectively.  
Let 
\begin{align*} 
\Phi^{\rm ar}\col 
\os{\circ}{F}{}^{{\rm ar}*}_{(X_{\os{\circ}{T}_0},D_{\os{\circ}{T}_0})/
S(T)^{\nat},S^{[p]}(T')^{\nat},{\rm crys}}(E')\lo E
\tag{6.16.1}\label{ali:sppts}
\end{align*} 
be a morphism of crystals in 
$(\os{\circ}{X}_{T_0}/\os{\circ}{T})_{\rm crys}$.   
Since $\deg (F_{S(T)^{\nat}/S^{[p]}(T')^{\nat}})=p$, 
the divisibility of the morphism (\ref{eqn:odnl}) holds by (\ref{prop:dvok}) 
if ${\mathcal I}=p{\mathcal O}_T $.  
We call the following induced morphism by $\Phi^{\rm ar}$   
\begin{align*} 
\Phi^{\rm ar} \col &
(A_{\rm zar}((X^{[p]}_{\os{\circ}{T}{}'_0},D^{[p]}_{\os{\circ}{T}{}'_0})/S^{[p]}(T')^{\nat},E'),
P^{D^{[p]}_{\os{\circ}{T}{}'_0}},P) \tag{6.16.2}\label{ali:spapts} \\
&\lo 
RF^{\rm ar}_{(X_{\os{\circ}{T}_0},D_{\os{\circ}{T}{}'_0})/S(T)^{\nat},S^{[p]}(T')^{\nat}*}
((A_{\rm zar}((X_{\os{\circ}{T}_0},D_{\os{\circ}{T}_0})/S(T)^{\nat},E),P^{D_{\os{\circ}{T}_0}},P)) 
\end{align*}
the {\it abrelative Frobenius morphism} of 
$$(A_{\rm zar}((X_{\os{\circ}{T}_0},D_{\os{\circ}{T}_0})/S(T)^{\nat},E),
P^{D_{\os{\circ}{T}_0}},P)\quad 
{\rm and} \quad  
(A_{\rm zar}((X^{[p]}_{\os{\circ}{T}{}'_0},D^{[p]}_{\os{\circ}{T}{}'_0})/S^{[p]}(T')^{\nat},E'),
P^{D^{[p]}_{\os{\circ}{T}{}'_0}},P).$$ 
When $E'={\mathcal O}_{\os{\circ}{X}{}^{[p]}_{T'_0}/\os{\circ}{T}{}'}$, 
we set 
$$(A_{\rm zar}((X^{[p]}_{\os{\circ}{T}{}'_0},D^{[p]}_{\os{\circ}{T}{}'_0})/S^{[p]}(T')^{\nat}),
P^{D^{[p]}_{\os{\circ}{T}{}'_0}}.P) 
:=(A_{\rm zar}((X^{[p]}_{\os{\circ}{T}{}'_0},D^{[p]}_{\os{\circ}{T}{}'_0})/S^{[p]}(T')^{\nat},E'),
P^{D^{[p]}_{\os{\circ}{T}{}'_0}},P).$$  
Then we have the following {\it  abrelative Frobenius morphism} 
\begin{align*} 
\Phi^{\rm  ar} \col &
(A_{\rm zar}((X^{[p]}_{\os{\circ}{T}{}'_0},D^{[p]}_{\os{\circ}{T}{}'_0})/S^{[p]}(T')^{\nat}),
P^{D^{[p]}_{\os{\circ}{T}{}'_0}},P) \lo \tag{6.16.3}\\
& RF^{\rm  ar}_{(X_{\os{\circ}{T}_0},D_{\os{\circ}{T}_0})/S(T)^{\nat},S^{[p]}(T')^{\nat}*}
((A_{\rm zar}((X_{\os{\circ}{T}_0},D_{\os{\circ}{T}_0})/S(T)^{\nat}),P^{D^{[p]}_{\os{\circ}{T}_0}}.P)) 
\end{align*}
of $(A_{\rm zar}((X_{\os{\circ}{T}_0},D_{\os{\circ}{T}_0})/S(T)^{\nat}),P^{D_{\os{\circ}{T}_0}},P)$ and 
$(A_{\rm zar}((X^{[p]}_{\os{\circ}{T}{}'_0},D^{[p]}_{\os{\circ}{T}{}'_0})/S^{[p]}(T')^{\nat}),
P^{D^{[p]}_{\os{\circ}{T}{}'_0}},P)$. 
\end{defi}

\begin{prop}[{\bf Frobenius compatibility I}]\label{prop:fcbar} 
The following diagram is commutative$:$ 
\begin{equation*} 
\begin{CD} 
A_{\rm zar}((X^{[p]}_{\os{\circ}{T}{}'_0},D^{[p]}_{\os{\circ}{T}{}'_0})/S^{[p]}(T')^{\nat},E')
@>{\Phi^{\rm ar}}>>
 \\
@A{\theta_{(X^{[p]}_{\os{\circ}{T}{}'_0},D^{[p]}_{\os{\circ}{T}{}'_0})/S^{[p]}(T')^{\nat}} \wedge}A{\simeq}A  \\
Ru_{(X^{[p]}_{\os{\circ}{T}{}'_0},D^{[p]}_{\os{\circ}{T}{}'_0})/S^{[p]}(T')^{\nat}*}
(\eps^*_{(X^{[p]}_{\os{\circ}{T}{}'_0},D^{[p]}_{\os{\circ}{T}{}'_0})/S^{[p]}(T')^{\nat}}(E'))
@>{\Phi^{\rm ar}}>>
\end{CD}
\tag{6.17.1}\label{cd:rukt}
\end{equation*} 
\begin{equation*} 
\begin{CD} 
RF^{\rm ar}_{(X_{\os{\circ}{T}_0},D_{\os{\circ}{T}_0})/S(T)^{\nat},S^{[p]}(T')^{\nat}*}
(A_{\rm zar}((X_{\os{\circ}{T}_0},D_{\os{\circ}{T}_0})/S(T)^{\nat},E)) \\
@A{RF^{\rm ar}_{X_{\os{\circ}{T}_0/S(T)^{\nat},S^{[p]}(T')^{\nat}}*}
(\theta_{(X_{\os{\circ}{T}_0},D_{\os{\circ}{T}_0})/S(T)^{\nat}})\wedge}A{\simeq}A 
\\
RF^{\rm ar}_{(X_{\os{\circ}{T}_0},D_{\os{\circ}{T}_0})/S(T)^{\nat},S^{[p]}(T')^{\nat}*}
Ru_{(X_{\os{\circ}{T}_0},D_{\os{\circ}{T}_0})/S(T)^{\nat}*}
(\eps^*_{(X_{\os{\circ}{T}_0},D_{\os{\circ}{T}_0})/S(T)^{\nat}}(E)). 
\end{CD}
\end{equation*} 
This is contravariantly functorial 
for the morphism {\rm (\ref{eqn:xdxduss})} satisfying {\rm (\ref{cd:xygxy})} 
and for the morphism of $F$-crystals 
\begin{equation*} 
\begin{CD} 
\os{\circ}{F}{}^{{\rm ar}*}_{(X_{\os{\circ}{T}_0},D_{\os{\circ}{T}_0})/S(T)^{\nat},S^{[p]}(T')^{\nat},{\rm crys}}
(E')
@>{\Phi^{\rm ar}}>> E\\
@AAA @AAA \\
\os{\circ}{g}{}^*
\os{\circ}{F}{}^{{\rm ar}*}_{(Y_{\os{\circ}{T}_0},C_{\os{\circ}{T}_0})/S(T)^{\nat},S^{[p]}(T')^{\nat},{\rm crys}}
(F')
@>{\os{\circ}{g}{}^*(\Phi^{\rm ar})}>> 
\os{\circ}{g}{}^*(F), 
\end{CD}
\end{equation*} 
where $F$ $($resp.~$F')$ is 
a similar quasi-coherent 
${\mathcal O}_{\os{\circ}{Y}_{T_0}/\os{\circ}{T}}$-module to $E$ 
$($resp.~a similar quasi-coherent 
${\mathcal O}_{\os{\circ}{Y}{}^{[p]}_{T_0}/\os{\circ}{T}}$-module to $E')$. 
\end{prop} 
\begin{proof} 
This is a special case of (\ref{theo:funpas}). 
\end{proof}

\begin{defi}[{\bf Absolute Frobenius endomorphism}]\label{defi:btd}  
Let the notations and the assumptions be as in (\ref{defi:rwd}).  
Let $F_{S} \col S \lo S$ be the Frobenius endomorphism of $S$, 
that is, $\os{\circ}{F}_{S}\col \os{\circ}{S} \lo \os{\circ}{S}$ 
is induced by the $p$-th power endomorphism of ${\mathcal O}_{S}$ 
and the multiplication by $p$ of the log structure of $S$.   
Let $F_{S_{\os{\circ}{T}_0}} \col S_{\os{\circ}{T}_0} 
\lo S_{\os{\circ}{T}_0}$ be 
the Frobenius endomorphism of $S_{\os{\circ}{T}_0}$.  
Assume that there exists a lift $F_{S(T)^{\nat}} \col S(T)^{\nat}\lo S(T)^{\nat}$ of 
$F_{S_{\os{\circ}{T}_0}}$  
which gives a PD-morphism 
$F_{S(T)^{\nat}}\col (S(T)^{\nat},{\mathcal J},\del)\lo (S(T)^{\nat},{\mathcal J},\del)$.  
Let 
$$F^{\rm abs}_{(X_{\os{\circ}{T}_0},D_{\os{\circ}{T}_0})/S(T)^{\nat}} 
\col (X_{\os{\circ}{T}_0},D_{\os{\circ}{T}_0})  \lo (X_{\os{\circ}{T}_0},D_{\os{\circ}{T}_0})$$ 
be the absolute Frobenius endomorphism over $F_{S(T)^{\nat}}$.   
Let 
\begin{align*} 
\Phi^{\rm abs} \col 
\os{\circ}{F}{}^{{\rm abs}*}_{(X_{\os{\circ}{T}_0},D_{\os{\circ}{T}_0})/S(T)^{\nat},{\rm crys}}(E)\lo E
\end{align*} 
be a morphism of crystals in 
$(\os{\circ}{X}_{T_0}/\os{\circ}{T})_{\rm crys}$.   
Then the divisibility of the morphism (\ref{eqn:odnl}) 
holds in this situation by (\ref{prop:dvok}) for the case $n=1$ if ${\mathcal I}= p{\mathcal O}_T$.  
Then we call the induced morphism by $\Phi^{\rm abs}$ and $F_{S(T)^{\nat}}$
\begin{align*} 
\Phi^{\rm abs} \col &
(A_{\rm zar}((X_{\os{\circ}{T}_0},D_{\os{\circ}{T}_0})/S(T)^{\nat},E),P^{D_{\os{\circ}{T}_0}},P) 
\lo \tag{6.18.1}\label{eqn:ebnp}\\
& RF^{\rm abs}_{(X_{\os{\circ}{T}_0},D_{\os{\circ}{T}_0})/S(T)^{\nat}*}
((A_{\rm zar}((X_{\os{\circ}{T}_0},D_{\os{\circ}{T}_0})/S(T)^{\nat},E),P^{D_{\os{\circ}{T}_0}},P)) 
\end{align*}
the {\it absolute Frobenius endomorphism} of 
$(A_{\rm zar}((X_{\os{\circ}{T}_0},D_{\os{\circ}{T}_0})/S(T)^{\nat},E),P^{D_{\os{\circ}{T}_0}},P)$ 
with respect to $F_{S(T)^{\nat}}$.
When $E={\mathcal O}_{\os{\circ}{X}_{T_0}/\os{\circ}{T}}$, 
we have the following {\it absolute Frobenius endomorphism} 
\begin{align*} 
\Phi^{\rm abs}{}^* \col & 
(A_{\rm zar}((X_{\os{\circ}{T}_0},D_{\os{\circ}{T}_0})/S(T)^{\nat}),P^{D_{\os{\circ}{T}_0}},P) 
\lo \tag{6.18.2}\label{eqn:abst}\\
&  RF^{\rm abs}_{(X_{\os{\circ}{T}_0},D_{\os{\circ}{T}_0})/S(T)^{\nat}*}
((A_{\rm zar}((X_{\os{\circ}{T}_0},D_{\os{\circ}{T}_0})/S(T)^{\nat}),P^{D_{\os{\circ}{T}_0}},P)) 
\end{align*}
of $(A_{\rm zar}(X_{\os{\circ}{T}_0}/S(T)^{\nat}),P^{D_{\os{\circ}{T}_0}},P)$ 
with respect to $F_{S(T)^{\nat}}$.  
\end{defi}

\begin{prop}[{\bf Frobenius compatibility II}]\label{prop:fcabbar} 
The following diagram is commutative$:$ 
\begin{equation*} 
\begin{CD} 
A_{\rm zar}((X_{\os{\circ}{T}_0},D_{\os{\circ}{T}_0})/S(T)^{\nat},E)
@>{\Phi^{\rm abs}}>>
 \\
@A{\theta_{(X_{\os{\circ}{T}_0},D_{\os{\circ}{T}_0})/S(T)^{\nat}} \wedge}A{\simeq}A  \\
Ru_{(X_{\os{\circ}{T}_0},D_{\os{\circ}{T}_0})/S(T)^{\nat}*}
(\eps^*_{(X_{\os{\circ}{T}_0},D_{\os{\circ}{T}_0})/S(T')^{\nat}}(E))
@>{\Phi^{\rm abs}}>>
\end{CD}
\tag{6.19.1}\label{cd:ruabkt}
\end{equation*} 
\begin{equation*} 
\begin{CD} 
RF^{\rm abs}_{(X_{\os{\circ}{T}_0},D_{\os{\circ}{T}_0})/S(T)^{\nat},S(T)^{\nat}*}
(A_{\rm zar}((X_{\os{\circ}{T}_0},D_{\os{\circ}{T}_0})/S(T)^{\nat},E)) \\
@A{RF^{\rm abs}_{X_{\os{\circ}{T}_0/S(T)^{\nat}}*}
(\theta_{(X_{\os{\circ}{T}_0},D_{\os{\circ}{T}_0})/S(T)^{\nat}})\wedge}A{\simeq}A 
\\
RF^{\rm abs}_{(X_{\os{\circ}{T}_0},D_{\os{\circ}{T}_0})/S(T)^{\nat}*}
Ru_{(X_{\os{\circ}{T}_0},D_{\os{\circ}{T}_0})/S(T)^{\nat}*}
(\eps^*_{(X_{\os{\circ}{T}_0},D_{\os{\circ}{T}_0})/S(T)^{\nat}}(E)). 
\end{CD}
\end{equation*} 
This is contravariantly functorial 
for the morphism {\rm (\ref{eqn:xdxduss})} satisfying {\rm (\ref{cd:xygxy})} 
and for the morphism of $F$-crystals 
\begin{equation*} 
\begin{CD} 
\os{\circ}{F}{}^{{\rm abs}*}_{(X_{\os{\circ}{T}_0},D_{\os{\circ}{T}_0})/S(T)^{\nat},{\rm crys}}
(E)
@>{\Phi^{\rm abs}}>> E\\
@AAA @AAA \\
\os{\circ}{g}{}^*
\os{\circ}{F}{}^{{\rm abs}*}_{(Y_{\os{\circ}{T}_0},C_{\os{\circ}{T}_0})/S(T)^{\nat},{\rm crys}}
(F)
@>{\os{\circ}{g}{}^*(\Phi^{\rm abs})}>> 
\os{\circ}{g}{}^*(F), 
\end{CD}
\end{equation*} 
where $F$ is a similar quasi-coherent 
${\mathcal O}_{\os{\circ}{Y}_{T_0}/\os{\circ}{T}}$-module to $E$. 
\end{prop} 
\begin{proof} 
This is a special case of (\ref{theo:funpas}). 
\end{proof}

\par 
%Let the notations be as in (\ref{eqn:esasp}). 
%\par 
When $\os{\circ}{S}$ is of characteristic $p>0$, when 
$u\col (S(T)^{\nat},{\mathcal J},\del)\lo (S'(T')^{\nat},{\mathcal J}',\del')$ 
is a lift of the abrelative Frobenius morphism 
$F_{S_{\os{\circ}{T}_0}}\col S_{\os{\circ}{T}_0}\lo S^{[p]}_{\os{\circ}{T}{}'_0}$ 
and 
when $g$ is 
the abrelative Frobenius morphism  
$F^{\rm ar}_{(X_{\os{\circ}{T}_0},D_{\os{\circ}{T}_0})/S(T)^{\nat},S^{[p]}(T')^{\nat}}\col 
(X_{\os{\circ}{T}_0},D_{\os{\circ}{T}_0}) \lo (X^{[p]}_{\os{\circ}{T}_0},D^{[p]}_{\os{\circ}{T}_0})$  or 
the absolute Frobenius endomorphism of $(X_{\os{\circ}{T}_0},D_{\os{\circ}{T}_0})$, 
we denote $(-j-k-m,u)$ in (\ref{eqn:escssp}) by $(-j-k-m)$ as usual.

\section{Monodromy operators}\label{sec:mod}
Let $S$, $((T,{\mathcal J},\del),z)$ and $T_0$ be as in \S\ref{sec:snrdlv}.  
In this section we recall the mornodromy operator defined in 
\cite{hk} and \cite{nb} quickly. 
Let $Y$ be a log smooth scheme over $S$.  
Let $Y'_{\os{\circ}{T}_0}$ be the disjoint union of the member of 
an affine open covering of $Y_{\os{\circ}{T}_0}$. 
By abuse of notation, we also denote by $g$ 
the composite morphism 
$g \col Y_{\os{\circ}{T}_0} \lo S_{\os{\circ}{T}_0} \lo S(T)^{\nat}$.  
Let $Y_{\os{\circ}{T}_0\bul}$ be the \v{C}ech diagram of 
$Y'_{\os{\circ}{T}_0}$ over $Y_{\os{\circ}{T}_0}$: 
$Y_{\os{\circ}{T}_0,n}={\rm cosk}_0^{Y_{\os{\circ}{T}_0}}(Y'_{\os{\circ}{T}_0})_n$ $(n\in {\mab N})$. 
Let $g_{\bul}\col  Y_{\os{\circ}{T}_0\bul}\lo S(T)^{\nat}$ 
be the structural morphism. 
For $U=S(T)^{\nat}$ or $\os{\circ}{T}$, 
let 
$$\eps_{Y_{\os{\circ}{T}_0\bul}/U}\col 
((Y_{\os{\circ}{T}_0\bul}/U)_{\rm crys},
{\mathcal O}_{Y_{\os{\circ}{T}_0\bul}/U})\lo 
((\os{\circ}{Y}_{T_0\bul}/\os{\circ}{T})_{\rm crys},
{\mathcal O}_{\os{\circ}{Y}_{T_0\bul}/\os{\circ}{T}})$$  
be the morphism forgetting the log structures of $Y_{\os{\circ}{T}_0\bul}$ and $U$. 
Let 
$Y_{\os{\circ}{T}_0\bul} \os{\sus}{\lo} \ol{\mathcal Q}_{\bul}$ 
be an immersion into a log smooth scheme over $\ol{S(T)^{\nat}}$. 
Set ${\mathcal Q}_{\bul}
:=\ol{\mathcal Q}_{\bul}\times_{\ol{S(T)^{\nat}}}S(T)^{\nat}$.  
Let $\ol{\mathfrak E}_{\bul}$ 
be the log PD-envelope of 
the immersion $Y_{\os{\circ}{T}_0\bul} \os{\sus}{\lo} \ol{{\mathcal Q}}_{\bul}$ over 
$(\os{\circ}{T},{\mathcal J},\del)$. 
Set ${\mathfrak E}_{\bul}
=\ol{\mathfrak E}_{\bul}
\times_{{\mathfrak D}(\ol{S(T)^{\nat}})}S(T)^{\nat}$.  
Let $\theta_{{\mathcal Q}_{\bul}}\in 
{\Om}^1_{{{\mathcal Q}}_{\bul}/\os{\circ}{T}}$ 
be the pull-back of $\theta =d\log t\in  
{\Om}^1_{S(T)^{\nat}/\os{\circ}{T}}$ by the structural morphism 
${\mathcal Q}_{\bul}\lo S(T)^{\nat}$.  
Let $\ol{F}$ be a flat quasi-coherent crystal of 
${\mathcal O}_{Y_{\os{\circ}{T}_0}/\os{\circ}{T}}$-modules. 
Let $\ol{F}{}^{\bul}$ be the crystal of 
${\mathcal O}_{Y_{\os{\circ}{T}_0\bul}/\os{\circ}{T}}$-modules obtained by $\ol{F}$. 
Let $(\ol{\mathcal F}{}^{\bul},\ol{\nabla})$ be the quasi-coherent 
${\mathcal O}_{\ol{\mathfrak E}_{\bul}}$-module 
with integrable connection corresponding to 
$\ol{F}{}^{\bul}$:
\begin{equation*} 
\ol{\nabla} \col \ol{\mathcal F}{}^{\bul} 
\lo \ol{\mathcal F}{}^{\bul}\otimes_{{\mathcal O}_{{{\mathcal Q}}_{\bul}}}
{\Om}^1_{\ol{\mathcal Q}_{\bul}/\os{\circ}{T}}.  
%\tag{7.0.1}\label{eqn:nidfopltd}
\end{equation*}  
Set 
$({\mathcal F}^{\bul},\nabla)=
(\ol{\mathcal F}{}^{\bul},\ol{\nabla})
\otimes_{{\mathcal O}_{\ol{\mathfrak E}_{\bul}}}
{\mathcal O}_{{\mathfrak E}_{\bul}}$.   
Let 
\begin{equation*} 
\nabla \col {\mathcal F}^{\bul} 
\lo {\mathcal F}^{\bul}\otimes_{{\mathcal O}_{{{\mathcal Q}}_{\bul}}}
{\Om}^1_{{\mathcal Q}_{\bul}/\os{\circ}{T}} 
\tag{7.0.1}\label{eqn:nidfopltd}
\end{equation*}  
be the induced connection by $\ol{\nabla}$ (\cite[pp.~41--42]{nhir}). 
Since ${\Om}^1_{{\mathcal Q}_{\bul}/S(T)^{\nat}}$ 
is a quotient of 
${\Om}^1_{{\mathcal Q}_{\bul}/\os{\circ}{T}}$, 
we also have the induced connection 
\begin{equation*} 
\nabla_{/S(T)^{\nat}} 
\col {\mathcal F}^{\bul} \lo {\mathcal F}^{\bul}
\otimes_{{\mathcal O}_{{{\mathcal Q}}_{\bul}}}
{\Om}^1_{{\mathcal Q}_{\bul}/S(T)^{\nat}}  
\tag{7.0.2}\label{eqn:nidfqtopltd}
\end{equation*}   
by $\nabla$.  
The object 
$({\mathcal F}^{\bul},\nabla_{/S(T)^{\nat}})$ 
corresponds to the log crystal 
$F^{\bul}:=\eps_{Y_{\os{\circ}{T}_0\bul}/S(T)^{\nat}}^*
(\ol{F}{}^{\bul})$ of ${\mathcal O}_{Y_{\os{\circ}{T}_0\bul}/S(T)^{\nat}}$-modules. 

\par
In \cite{nb} we have proved the following whose proof is not difficult: 

\begin{prop}[{\bf \cite[(1.7.22)]{nb}}]\label{prop:mce}
The following sequence 
\begin{align*} 
0 & \lo {\mathcal F}^{\bul}
\otimes_{{\mathcal O}_{{{\mathcal Q}}_{\bul}}}
{\Om}^{\bul}_{{{\mathcal Q}}_{\bul}/S(T)^{\nat}}[-1] 
\os{\theta_{{\mathcal Q}_{{\bul}} \wedge}}{\lo} 
{\mathcal F}^{\bul}
\otimes_{{\mathcal O}_{{{\mathcal Q}}_{\bul}}}
{\Om}^{\bul}_{{{\mathcal Q}}_{\bul}/\os{\circ}{T}} 
\lo {\mathcal F}^{\bul}
\otimes_{{\mathcal O}_{{{\mathcal Q}}_{\bul}}}
{\Om}^{\bul}_{{{\mathcal Q}}_{\bul}/S(T)^{\nat}} \lo 0
\tag{7.1.1}\label{eqn:gsflxd}\\ 
\end{align*} 
is exact. 
\end{prop}

\parno 
Let 
\begin{equation*} 
{\mathcal F}^{\bul}
\otimes_{{\mathcal O}_{{{\mathcal Q}}_{\bul}}}
{\Om}^{\bul}_{{{\mathcal Q}}_{\bul}/S(T)^{\nat}} 
\lo 
{\mathcal F}^{\bul}
\otimes_{{\mathcal O}_{{{\mathcal Q}}_{\bul}}}
{\Om}^{\bul}_{{{\mathcal Q}}_{\bul}/S(T)^{\nat}} 
\tag{7.1.2}\label{eqn:gyglxd}
\end{equation*} 
be the boundary morphism of (\ref{eqn:gsflxd}) 
in the derived category 
$D^+(f^{-1}({\mathcal O}_T))$. 
(We make the convention on the sign of 
the boundary morphism as in \cite[p.~12 (4)]{nh2}.)  
By using the formula 
$Ru_{Y_{\os{\circ}{T}_0\bul}/S(T)^{\nat}*}(F^{\bul})
={\mathcal F}^{\bul}
\otimes_{{\mathcal O}_{{{\mathcal Q}}_{\bul}}}
{\Om}^{\bul}_{{{\mathcal Q}}_{\bul}/S(T)^{\nat}}$  
(\cite[(6.4)]{klog1}, (cf.~\cite[(2.2.7)]{nh2})), 
we have the following morphism  
\begin{equation*} 
Ru_{Y_{\os{\circ}{T}_0\bul}/S(T)^{\nat}*}(F^{\bul})
\lo 
Ru_{Y_{\os{\circ}{T}_0\bul}/S(T)^{\nat}*}(F^{\bul}).   
\tag{7.1.3}\label{eqn:uyons}
\end{equation*}  
Let 
\begin{equation*} 
\pi_{{\rm crys}} \col 
((Y_{\os{\circ}{T}_0\bul}/S(T)^{\nat})_{\rm crys},{\mathcal O}_{Y_{\os{\circ}{T}_0\bul}/S(T)^{\nat}})
\lo 
((Y_{\os{\circ}{T}_0}/S(T)^{\nat})_{\rm crys},{\mathcal O}_{Y_{\os{\circ}{T}_0}/S(T)^{\nat}})
\tag{7.1.4}\label{eqn:tzcar}
\end{equation*} 
and 
\begin{equation*} 
\pi_{{\rm zar}} \col 
((\os{\circ}{Y}_{T_0\bul})_{\rm zar},
g^{-1}_{\bul}({\mathcal O}_T)) \lo 
((\os{\circ}Y_{T_0})_{\rm zar},g^{-1}({\mathcal O}_T)) 
\tag{7.1.5}\label{eqn:tzar}
\end{equation*} 
be the natural morphisms of ringed topoi. 
Applying $R\pi_{{\rm zar}*}$ to (\ref{eqn:uyons}) 
and using the formula 
$\pi_{{\rm zar}*}\circ u_{Y_{\os{\circ}{T}_0\bul}/S(T)^{\nat}}
=u_{Y_{\os{\circ}{T}_0}/S(T)^{\nat}}\circ \pi_{{\rm crys}}$ 
and using the cohomological descent,  
we have the following morphism   
\begin{equation*}
N_{\rm zar} \col Ru_{Y_{\os{\circ}{T}_0}/S(T)^{\nat}*}(F) 
\lo 
Ru_{Y_{\os{\circ}{T}_0}/S(T)^{\nat}*}(F), 
\tag{7.1.6}\label{eqn:nzgslyne}
\end{equation*}
where $F:=\eps_{Y_{\os{\circ}{T}_0}/S(T)^{\nat}}^*(\ol{F})$.

In \cite{nb} we have proved the following whose proof is not difficult: 

\begin{prop}[{\bf \cite[(1.7.26), (1.7.30)]{nb}}] 
The morphism {\rm (\ref{eqn:nzgslyne})} 
is independent of the choices of an affine open covering of 
$Y$ and a simplicial immersion 
$Y_{\bul} \os{\sus}{\lo} 
\ol{{\mathcal Q}}_{\bul}$ over $\ol{S(T)^{\nat}}$. 
\end{prop}

\par 
Let 
$v \col (S(T)^{\nat},{\mathcal J},\del) \lo (S(T)^{\nat},{\mathcal J},\del)$ be 
an endomorphism of $(S(T)^{\nat},{\mathcal J},\del)$.    
Let 
\begin{equation*} 
\begin{CD} 
Y_{\os{\circ}{T}_0} @>{h}>> Y_{\os{\circ}{T}_0} \\ 
@VVV @VVV \\ 
S_{\os{\circ}{T}_0} @>{v_0}>> S_{\os{\circ}{T}_0} \\ 
@V{\bigcap}VV @VV{\bigcap}V \\ 
S(T)^{\nat} @>{v}>> S(T)^{\nat}
\end{CD}
\tag{7.2.1}\label{eqn:ydxduss}
\end{equation*} 
be a commutative diagram of log schemes.  
The morphism (\ref{eqn:nzgslyne}) 
is nothing but a morphism 
\begin{equation*}
N_{\rm zar} \col Ru_{Y_{\os{\circ}{T}_0}/S(T)^{\nat}*}(F) 
\lo 
Ru_{Y_{\os{\circ}{T}_0}/S(T)^{\nat}*}(F)(-1;v)
\tag{7.2.2}\label{eqn:mlcepglynl}
\end{equation*}
since $v^*(\theta_{{\mathcal Q}_{\bul}})=\deg(v)\theta_{{\mathcal Q}_{\bul}}$.
In \cite{nb} we have called the morphism (\ref{eqn:mlcepglynl}) the 
{\it  zariskian monodromy operator} 
of $Y_{\os{\circ}{T}_0}/S(T)^{\nat}$. 
In particular we obtain the following monodromy operator: 
\begin{equation*}
N_{\rm zar} \col Ru_{(X_{\os{\circ}{T}_0},D_{\os{\circ}{T}_0})/S(T)^{\nat}*}(F) 
\lo 
Ru_{(X_{\os{\circ}{T}_0},D_{\os{\circ}{T}_0})/S(T)^{\nat}*}(F)(-1;v)
\tag{7.2.3}\label{eqn:mlxglynl}
\end{equation*}

In \cite[(1.7.27)]{nb} we have proved the following:

\begin{prop-defi}[{\bf \cite[(1.7.27)]{nb}}]\label{prop:hkt}
{\rm The complex $R\pi_{{\rm zar}*}({\mathcal F}^{\bul}
\otimes_{{\mathcal O}_{{\mathcal Q}^{\rm ex}_{\bul}}}
{\Om}^{\bul}_{{\mathcal Q}^{\rm ex}_{\bul}/\os{\circ}{T}})$ 
is independent of the choices of 
an open covering of $Y_{\os{\circ}{T}_0}$ and a simplicial immersion 
$Y_{\os{\circ}{T}_0\bul} \os{\sus}{\lo} \ol{\mathcal Q}_{\bul}$ 
over $\ol{S(T)^{\nat}}$. 
Set 
\begin{align*} 
&\wt{R}u_{Y_{\os{\circ}{T}_0}/\os{\circ}{T}*}(\ol{F})
:=R\pi_{{\rm zar}*}
({\mathcal F}^{\bul}\otimes_{{\mathcal O}_{{\mathcal Q}{}^{{\rm ex}}_{\bul}}}
\Om^{\bul}_{{\mathcal Q}{}^{{\rm ex}}_{\bul}/\os{\circ}{T}})
\in D^+(g^{-1}({\mathcal O}_T)). 
\end{align*} 
In \cite{nb} we have called 
$\wt{R}u_{Y_{\os{\circ}{T}_0}/\os{\circ}{T}*}(\ol{F})$ 
the {\it modified log crystalline complex} of $\ol{F}$ 
on $Y_{\os{\circ}{T}_0}/S(T)^{\nat}$. 
%When $\ol{F}={\mathcal O}_{Y_{\os{\circ}{T}_0}/\os{\circ}{T}}$, we have called 
%$\wt{R}u_{Y_{\os{\circ}{T}_0}/\os{\circ}{T}*}({\mathcal O}_{Y_{\os{\circ}{T}_0}/\os{\circ}{T}})$ 
%the {\it modified log crystalline complex} of 
%$Y_{\os{\circ}{T}_0}/S(T)^{\nat}/\os{\circ}{T}$.
}
\end{prop-defi}

\par 
Next we express the monodromy operator (\ref{eqn:mlxglynl}) by using 
the isomorphism (\ref{eqn:uz}) and the complex 
$A_{\rm zar}((X_{\os{\circ}{T}_0},D_{\os{\circ}{T}_0})/S(T)^{\nat},E)$. 
\par 
Let 
\begin{equation*} 
(A_{\rm zar}({\mathcal P}^{\rm ex}_{\bul}/S(T)^{\nat},{\mathcal E}^{\bul})^{\bul \bul},
P^{{\mathcal D}_{\bul}/{\mathcal X}_{\bul}},P)
\lo (I^{\bul \bul \bul},P^{{\mathcal D}_{\bul}/{\mathcal X}_{\bul}},P)
%\tag{7.2.1}\label{eqn:apfi}
\end{equation*} 
be the cosimplicial bifiltered Godement resolution, 
where the left degree is the cosimplicial degree. 
Then we have a resolution 
$$(A_{\rm zar}({\mathcal P}^{\rm ex}_{\bul}/S(T)^{\nat}, {\mathcal E}^{\bul}),
P^{{\mathcal D}_{\bul}/{\mathcal X}_{\bul}},P)
\lo  
(s(I^{\bul \bul \bul}),P^{{\mathcal D}_{\bul}/{\mathcal X}_{\bul}},P),$$   
where $s$ means the single complex 
of the double complex with respect to the last two degrees. 
Set 
$$(A_{\rm zar}
((X_{\os{\circ}{T}_0},D_{\os{\circ}{T}_0})/S(T)^{\nat},E)^{\bul \bul},P^{D_{\os{\circ}{T}_0}},P):=
\pi_{{\rm zar}*}((I^{\bul \bul \bul},P^{{\mathcal  D}_{\bul}/{\mathcal X}_{\bul}},P)).$$ 
Then 
$$s((A_{\rm zar}((X_{\os{\circ}{T}_0},D_{\os{\circ}{T}_0})/S(T)^{\nat},E)^{\bul \bul},
P^{D_{\os{\circ}{T}_0}},P))
=(A_{\rm zar}(X_{\os{\circ}{T}_0}/S(T)^{\nat},E),P^{D_{\os{\circ}{T}_0}},P)$$ 
in ${\rm D}^+{\rm F}^2(f^{-1}({\mathcal O}_T))$.  
Consider the following natural projection 
\begin{align*} 
&\nu_{\rm zar}({\mathcal P}^{\rm ex}_{\bul}/S(T)^{\nat},
{\mathcal E}^{\bul})^{ij}:={\rm proj}  \col 
{\mathcal E}^{\bul}\otimes_{{\mathcal O}_{{\mathcal P}^{\rm ex}_{\bul}}}
{\Om}^{i+j+1}_{{\mathcal P}^{\rm ex}_{\bul}/\os{\circ}{T}}/P^{{\mathcal X}_{\bul}}_{j+1}  
\lo {\mathcal E}^{\bul}\otimes_{{\mathcal O}_{{\mathcal P}^{\rm ex}_{\bul}}}
{\Om}^{i+j+1}_{{\mathcal P}^{\rm ex}_{\bul}/\os{\circ}{T}}/P^{{\mathcal X}_{\bul}}_{j+2}.   
\tag{7.3.1}\label{eqn:reslbb} \\ 
\end{align*}  
It is easy to check that 
the morphism above actually induces  
a morphism of complexes 
with the boundary morphisms in (\ref{cd:locstbd}).  
Since $(I^{\bul \bul \bul},P^{{\mathcal D}_{\bul}/{\mathcal X}_{\bul}},P)$ is 
the bifilteredly Godement resolution of 
$(A_{\rm zar}({\mathcal P}^{\rm ex}_{\bul}/S(T)^{\nat}, {\mathcal E}^{\bul}))^{\bul \bul},
P^{{\mathcal D}_{\bul}/{\mathcal X}_{\bul}},P)$, 
the morphism (\ref{eqn:reslbb}) 
%$\nu_{\rm zar}({\mathcal P}^{\rm ex}_{\bul \\bul}/S(T)^{\nat},
%{\mathcal E}^{\bul \\bul})^{ij}$ 
induces the morphism  
\begin{equation*} 
\wt{\nu}^{ij}_{{\rm zar}}:={\rm proj}. \col 
A_{\rm zar}((X_{\os{\circ}{T}_0},D_{\os{\circ}{T}_0})/S(T)^{\nat},E)^{ij}\lo 
A_{\rm zar}((X_{\os{\circ}{T}_0},D_{\os{\circ}{T}_0})/S(T)^{\nat},E)^{i-1,j+1},
\end{equation*}  
of sheaves of $f^{-1}({\mathcal O}_T)$-modules. 
Set 
\begin{align*} 
\wt{\nu}_{{\rm zar}}:=
s(\oplus_{i,j\in {\mab N}}\wt{\nu}^{ij}_{{\rm zar}}).
\tag{7.3.2}\label{ali:ntzd}
\end{align*}  
Let 
\begin{equation*} 
\nu_{S(T)^{\nat},{\rm zar}}
\col (A_{\rm zar}((X_{\os{\circ}{T}_0},D_{\os{\circ}{T}_0})/S(T)^{\nat},E),P^{D_{\os{\circ}{T}_0}},P)
\lo 
(A_{\rm zar}((X_{\os{\circ}{T}_0},D_{\os{\circ}{T}_0})/S(T)^{\nat},E),P^{D_{\os{\circ}{T}_0}},P\langle -2\rangle) 
\tag{7.3.3}\label{eqn:naxgdxdn}
\end{equation*} 
be a morphism of bifiltered complexes induced by 
$\{\nu^{ij}_{S(T)^{\nat},{\rm zar}}\}_{i,j \in {\mab N}}$. 
The morphism 
$$\theta_{{\mathcal P}^{\rm ex}_{\bul}/\os{\circ}{T}} \wedge \col 
(A_{\rm zar}({\mathcal P}^{\rm ex}_{\bul}/S(T)^{\nat}, 
{\mathcal E}^{\bul})^{ij},P^{D_{\os{\circ}{T}_0}},P)
\lo 
(A_{\rm zar}({\mathcal P}^{\rm ex}_{\bul}/S(T)^{\nat}, 
{\mathcal E}^{\bul})^{i,j+1},P^{D_{\os{\circ}{T}_0}},P)$$  
in (\ref{cd:locstbd}) induces a morphism 
$$\pi_{{\rm zar}*}(\theta_{{\mathcal P}^{\rm ex}_{\bul}/\os{\circ}{T}}\wedge) 
\col A_{\rm zar}((X_{\os{\circ}{T}_0},D_{\os{\circ}{T}_0})/S(T)^{\nat},E)^{ij}\lo 
A_{\rm zar}((X_{\os{\circ}{T}_0},D_{\os{\circ}{T}_0})/S(T)^{\nat},E)^{i,j+1}.$$ 

\par 
Let the notations be as in \S\ref{sec:fcuc}. 
Since the following diagram  
{\footnotesize{\begin{equation*}
\begin{CD}
A_{\rm zar}((Y_{\os{\circ}{T}{}'_0},C_{\os{\circ}{T}{}'_0})/S'(T')^{\nat},F)^{ij}
@>{{\rm proj}.}>> 
A_{\rm zar}((Y_{\os{\circ}{T}{}'_0},C_{\os{\circ}{T}{}'_0})/S'(T')^{\nat},F)^{i-1,j+1}\\ 
@V{\pi_{{\rm zar}*}((g^{{\rm PD}*}_{\bul})^{(ij)})}VV 
@VV{\deg(u)(\pi_{{\rm zar}*}(g^{{\rm PD}*}_{\bul}))^{(i-1,j+1)}}V  \\
\pi_{{\rm zar}*}(g^{{\rm PD}*}_{\bul*})(A_{\rm zar}((X_{\os{\circ}{T}_0},D_{\os{\circ}{T}_0})/S(T)^{\nat},E))^{\bul ij} 
@>{{\rm proj}.}>>
\pi_{{\rm zar}*}(g^{{\rm PD}*}_{\bul*})
(A_{\rm zar}((X_{\os{\circ}{T}_0},D_{\os{\circ}{T}_0})/S(T)^{\nat},E)^{i-1,j+1}) 
\end{CD}
\tag{7.3.4}\label{cd:phipnu}
\end{equation*}}}
is commutative, the morphism (\ref{eqn:naxgdxdn}) is 
the following morphism 
\begin{align*} 
\nu_{S(T)^{\nat},{\rm zar}} & \col 
(A_{\rm zar}((Y_{\os{\circ}{T}{}'_0},C_{\os{\circ}{T}{}'_0})/S'(T')^{\nat},E),
P^{C_{\os{\circ}{T}{}'_0}},P)\tag{7.3.5}\label{eqn:axdxdn}\\
& \lo 
Rg_*((A_{\rm zar}((X_{\os{\circ}{T}_0},D_{\os{\circ}{T}_0})
/S(T)^{\nat},E),P^{D_{\os{\circ}{T}_0}},P\langle -2\rangle))(-1,u). 
\end{align*} 

\par 
Now assume that $(T',{\mathcal J}',\del')=(T,{\mathcal J},\del)$, $S'=S$ 
and $g$ is an endomorphism of $(X_{\os{\circ}{T}_0},D_{\os{\circ}{T}_0})$. 
Consider (\ref{cd:xygxy}) in this case: 
\begin{equation*} 
\begin{CD} 
(X'_{\os{\circ}{T}_0},D'_{\os{\circ}{T}_0}) @>{g'}>> (X''_{\os{\circ}{T}_0},D''_{\os{\circ}{T}_0}) \\
@VVV @VVV \\ 
(X_{\os{\circ}{T}_0},D_{\os{\circ}{T}_0}) @>{g}>> (X_{\os{\circ}{T}_0},D_{\os{\circ}{T}_0}) \\
@VVV @VVV \\ 
S_{\os{\circ}{T}_0} @>>> S_{\os{\circ}{T}{}_0} \\ 
@V{\bigcap}VV @VV{\bigcap}V \\ 
S(T)^{\nat} @>{u}>> S(T)^{\nat}, 
\end{CD}
\tag{7.3.6}\label{cd:xgspxy}
\end{equation*}
where $(X''_{\os{\circ}{T}{}'_0},D''_{\os{\circ}{T}_0})$ 
is another disjoint union 
of the member of an affine  open covering of 
$(X_{\os{\circ}{T}_0},D_{\os{\circ}{T}_0})$. 
Assume that $\deg(u)$ is not divisible by $p$ or 
that  $\os{\circ}{T}$ is a $p$-adic formal scheme and 
that the morphism (\ref{eqn:odnl}) is divisible by $p^{e_p(j+1)}$.  

\par  
Next we consider a $p$-adic analogue of 
a generalization of the (double) complex in \cite[(4.22)]{st1}. 
For simplicity of notations, denote 
$A_{\rm zar}({\mathcal P}^{\rm ex}_{\bul}/S(T)^{\nat},{\mathcal E}^{\bul})^{\bul \bul}$ 
and 
$A_{\rm zar}((X_{\os{\circ}{T}_0},D_{\os{\circ}{T}_0})/S(T)^{\nat},E)^{\bul \bul}$ by 
$A^{\bul \bul \bul}$ and $A^{\bul \bul}$, respectively. 
Set 
\begin{equation*} 
B^{\bul ij}
:= A^{\bul ij} \oplus A^{\bul i-1,j}(-1,u)
\quad ( i,j\in {\mab N})
\tag{7.3.7}\label{eqn:axmdn}
\end{equation*} 
and 
\begin{equation*} 
B^{\bul \ij}
:= A^{\bul ij} \oplus A^{\bul i-1,j}(-1,u)
\quad ( i,j\in {\mab N}). 
\tag{7.3.8}\label{eqn:anxdn}
\end{equation*} 
The horizontal boundary morphism
$d' \col B^{\bul ij} \lo B^{\bul i+1,j}$ 
is, by definition, the induced morphism 
$d''{}^{\bul} \col B^{\bul ij} \lo B^{\bul i+1,j}$ 
defined by the following formula: 
\begin{align*} 
d'(\om_1,\om_2)=(-\nabla \om_1,\nabla \om_2) 
\tag{7.3.9}\label{ali:sddclxn}
\end{align*} 
and the vertical one 
$d'' \col B^{\bul ij} \lo B^{\bul i,j+1}$ is the induced morphism 
by a morphism  $d''{}^{\bul} \col B^{\bul ij} \lo B^{\bul i,j+1}$ 
defined by the following formula: 
\begin{align*} 
d''{}^{\bul}(\om_1,\om_2)=
(\theta_{{\mathcal P}^{\rm ex}_{\bul}/\os{\circ}{T}}\wedge \om_1,
-\theta_{{\mathcal P}^{\rm ex}_{\bul}/\os{\circ}{T}}\wedge \om_2+\nu_{S(T)^{\nat},{\rm zar}}(\om_1)).
\tag{7.3.10}\label{ali:sddpxn}
\end{align*}  
It is easy to check that 
$B^{\bul \bul \bul}$ is 
actually a cosimplicial double complex. 
Let $B^{\bul \bul}$ be the single complex of $B^{\bul \bul \bul}$ 
with respect to the last two degrees. 
The cosimplicial complex $B^{\bul \bul}$ is nothing but the mapping fiber of 
$\wt{\nu}_{\rm zar}$: $B^{\bul \bul}={\rm MF}(\wt{\nu}_{\rm zar})=
A^{\bul \bul}\oplus A^{\bul \bul}[-1]$.  
Set $B^{\bul}:=R\pi_{{\rm zar}*}(B^{\bul \bul})$. 
Then $B^{\bul}$ is the ``mapping fiber'' of 
$\nu_{S(T)^{\nat},{\rm zar}}$: 
$$B^{\bul}\simeq A_{\rm zar}((X_{\os{\circ}{T}_0},D_{\os{\circ}{T}_0})/S(T)^{\nat},E)
\oplus 
A_{\rm zar}((X_{\os{\circ}{T}_0},D_{\os{\circ}{T}_0})/S(T)^{\nat},E)[-1].$$
\par
Let 
\begin{align*} 
\mu_{{(X_{\os{\circ}{T}_0},D_{\os{\circ}{T}_0})/\os{\circ}{T}}} \col 
\wt{R}u_{(X_{\os{\circ}{T}_0},D_{\os{\circ}{T}_0})/\os{\circ}{T}*}
(\eps^*_{(X_{\os{\circ}{T}_0},D_{\os{\circ}{T}_0})/\os{\circ}{T}}(E))
=&R\pi_{{\rm zar}*}({\mathcal E}^{\bul}
\otimes_{{\mathcal O}_{{\mathcal P}^{\rm ex}_{\bul}}}
{\Om}^{\bul}_{{\mathcal P}^{\rm ex}_{\bul}/\os{\circ}{T}})
\tag{7.3.11}\label{ali:sgsclxn}\\
&\lo B^{\bul}
\end{align*} 
be a morphism
of complexes induced by the following morphisms
\begin{align*} 
\mu^i_{\bul} \col {\mathcal E}^{\bul}
\otimes_{{\mathcal O}_{{\mathcal P}^{\rm ex}_{\bul}}}
{\Om}^i_{{\mathcal P}^{\rm ex}_{\bul}/\os{\circ}{T}} 
\lo  
{\mathcal E}^{\bul}
\otimes_{{\mathcal O}_{{\mathcal P}^{\rm ex}_{\bul}}}
{\Om}^{i+1}_{{\mathcal P}^{\rm ex}_{\bul}/\os{\circ}{T}}/P^{{\mathcal X}_{\bul}}_0 \oplus  
{\mathcal E}^{\bul}
\otimes_{{\mathcal O}_{{\mathcal P}^{\rm ex}_{\bul}}}
{\Om}^i_{{\mathcal P}^{\rm ex}_{\bul}/\os{\circ}{T}}/P^{{\mathcal X}_{\bul}}_0 
\quad (i\in {\mab N})
\end{align*} 
defined by the following formula 
\begin{align*} 
\mu^i_{\bul}(\om):=(\theta_{{\mathcal P}^{\rm ex}_{\bul}/\os{\circ}{T}}
\wedge \om~{\rm mod}~P_0,\om~{\rm mod}~P_0) \quad 
(\om \in {\mathcal E}^{\bul}
\otimes_{{\mathcal O}_{{\mathcal P}^{\rm ex}_{\bul}}}
{\Om}^i_{{\mathcal P}^{\rm ex}_{\bul}/\os{\circ}{T}}).
\end{align*} 
Then we have the following 
morphism of triangles:
\begin{equation*}
\begin{CD}
@>>> 
A_{\rm zar}((X_{\os{\circ}{T}_0},D_{\os{\circ}{T}_0})/S(T)^{\nat},E)[-1] @>{}>>\\
@. @A{(\theta_{(X_{\os{\circ}{T}_0},D_{\os{\circ}{T}_0})/S(T)^{\nat}}\wedge *)[-1]}AA  \\
@>>> 
R\pi_{{\rm zar}*}({\mathcal E}^{\bul}
\otimes_{{\mathcal O}_{{\mathcal P}^{\rm ex}_{\bul}}}
\Om^{\bul}_{{\mathcal P}^{\rm ex}_{\bul}/S(T)^{\nat}})[-1]   
@>{R\pi_{{\rm zar}*}(\theta_{{\mathcal P}^{\rm ex}_{\bul}/\os{\circ}{T}} \wedge)}>> 
\end{CD}
\tag{7.3.12}\label{cd:sgsctexn}
\end{equation*} 
\begin{equation*}
\begin{CD}
B^{\bul} @>>> 
A_{\rm zar}((X_{\os{\circ}{T}_0},D_{\os{\circ}{T}_0})/S(T)^{\nat},E)
@>{+1}>>  \\ 
@A{\mu_{{(X_{\os{\circ}{T}_0},D_{\os{\circ}{T}_0})/\os{\circ}{T}}}}AA  
@A{\theta_{(X_{\os{\circ}{T}_0},D_{\os{\circ}{T}_0})/S(T)^{\nat}} \wedge}AA \\ 
R\pi_{{\rm zar}*}({\mathcal E}^{\bul}\otimes_{{\mathcal O}_{{\mathcal P}^{\rm ex}_{\bul}}}
{\Om}^{\bul}_{{\mathcal P}^{\rm ex}_{\bul}/\os{\circ}{T}}) @>>> 
R\pi_{{\rm zar}*}({\mathcal E}^{\bul}\otimes_{{\mathcal O}_{{\mathcal P}^{\rm ex}_{\bul}}}
{\Om}^{\bul}_{{\mathcal P}^{\rm ex}_{\bul}/S(T)^{\nat}}) 
@>{+1}>>.
\end{CD}
\end{equation*} 
This is nothing but the following diagram of triangles: 
\begin{equation*}
\begin{CD}
@>>> 
A_{\rm zar}((X_{\os{\circ}{T}_0},D_{\os{\circ}{T}_0})/S(T)^{\nat},E)[-1] @>{}>>\\
@. @A{(\theta_{(X_{\os{\circ}{T}_0},D_{\os{\circ}{T}_0})/S(T)^{\nat}}\wedge )[-1]}AA  \\
@>>> 
Ru_{(X_{\os{\circ}{T}_0},D_{\os{\circ}{T}_0})/S(T)^{\nat}*}
(\eps^*_{(X_{\os{\circ}{T}_0},D_{\os{\circ}{T}_0})/S(T)^{\nat}}(E))[-1]   
@>{}>> 
\end{CD}
\end{equation*} 
\begin{equation*}
\begin{CD}
B^{\bul} @>>> A_{\rm zar}((X_{\os{\circ}{T}_0},D_{\os{\circ}{T}_0})/S(T)^{\nat},E)
@>{+1}>>  \\  
@A{\mu_{{(X_{\os{\circ}{T}_0},D_{\os{\circ}{T}_0})/\os{\circ}{T}}}}AA 
@A{\theta_{(X_{\os{\circ}{T}_0},D_{\os{\circ}{T}_0})/S(T)^{\nat}} \wedge}AA \\ 
\wt{R}u_{(X_{\os{\circ}{T}_0},D_{\os{\circ}{T}_0})/\os{\circ}{T}*}
(\eps^*_{(X_{\os{\circ}{T}_0},D_{\os{\circ}{T}_0})/\os{\circ}{T}}(E)) @>>> 
Ru_{(X_{\os{\circ}{T}_0},D_{\os{\circ}{T}_0})/S(T)^{\nat}*}
(\eps^*_{(X_{\os{\circ}{T}_0},D_{\os{\circ}{T}_0})/S(T)^{\nat}}(E))@>{+1}>>.
\end{CD}
\tag{7.3.13}\label{cd:monodcom}
\end{equation*}

%\parno 
%Hence we obtain the following as in \cite[p.~246]{st1} 
%and \cite[(11.10)]{ndw}: 

\begin{prop}\label{prop:cmzoqm}
The zariskian monodromy operator
\begin{align*} 
N_{S(T)^{\nat},{\rm zar}} \col &
Ru_{(X_{\os{\circ}{T}_0},D_{\os{\circ}{T}_0})/S(T)^{\nat}*}
(\eps^*_{(X_{\os{\circ}{T}_0},D_{\os{\circ}{T}_0})/S(T)^{\nat}}(E))
\tag{7.4.1}\label{eqn:minbv}\\
& \lo  
 Ru_{(X_{\os{\circ}{T}_0},D_{\os{\circ}{T}_0})/S(T)^{\nat}*}
(\eps^*_{(X_{\os{\circ}{T}_0},D_{\os{\circ}{T}_0})/S(T)^{\nat}}(E))(-1,u)
\end{align*}  
is equal to  
\begin{align*} 
\nu_{S(T)^{\nat},{\rm zar}} \col 
A_{\rm zar}((X_{\os{\circ}{T}_0},D_{\os{\circ}{T}_0})/S(T)^{\nat},E) 
\lo A_{\rm zar}((X_{\os{\circ}{T}_0},D_{\os{\circ}{T}_0})/S(T)^{\nat},E)(-1,u) 
\tag{7.4.2}\label{eqn:minsbv}
\end{align*} 
via the isomorphism {\rm (\ref{eqn:uz})}. 
\end{prop}
\begin{proof} 
Since $B^{\bul}=\pi_{{\rm zar}*}(I^{\bul \bul})\oplus \pi_{{\rm zar}*}(I^{\bul \bul})[-1]$  
is the mapping fiber of 
$$\nu_{S(T)^{\nat},{\rm zar}} \col 
\pi_{{\rm zar}*}(I^{\bul \bul})\lo \pi_{{\rm zar}*}(I^{\bul \bul})(-1,u)$$ 
by (\ref{ali:sddpxn}), we obtain (\ref{prop:cmzoqm}). 
\end{proof}

\begin{prop}\label{prop:nst}
The morphism {\rm (\ref{eqn:minsbv})} 
is an underlying morphism of the following morphism 
\begin{align*} 
\nu_{S(T)^{\nat},{\rm zar}} \col &
(A_{\rm zar}((X_{\os{\circ}{T}_0},D_{\os{\circ}{T}_0})/S(T)^{\nat},E),P^{D_{\os{\circ}{T}_0}},P)
\tag{7.5.1}\label{eqn:min1usbv} \\
& \lo (A_{\rm zar}((X_{\os{\circ}{T}_0},D_{\os{\circ}{T}_0})/S(T)^{\nat},E),
P^{D_{\os{\circ}{T}_0}},P\langle -2\rangle)(-1,u).  
\end{align*} 
\end{prop}
\begin{proof} 
This follows from the definition of $P^{D_{\os{\circ}{T}_0}}$, $P$ and 
$\nu_{S(T)^{\nat},{\rm zar}}$. 
Here $P\langle -2\rangle_k:=P_{k-2}$. 
\end{proof}

\begin{coro}\label{coro:mn}
Let $q$ be a nonnegative integer. 
The monodromy operator 
{\footnotesize{\begin{equation*}
N_{\rm zar}\col R^qf_{(X_{\os{\circ}{T}_0},D_{\os{\circ}{T}_0})/S(T)^{\nat}*}
(\eps^*_{(X_{\os{\circ}{T}_0},D_{\os{\circ}{T}_0})/S(T)^{\nat}}(E))  
\lo 
R^qf_{(X_{\os{\circ}{T}_0},D_{\os{\circ}{T}_0})/S(T)^{\nat}*}
(\eps^*_{(X_{\os{\circ}{T}_0},D_{\os{\circ}{T}_0})/S(T)^{\nat}}(E))(-1;u)
\tag{7.6.1}\label{eqn:mlxdlynl}
\end{equation*}}}
induced by {\rm (\ref{eqn:nzgslyne})} 
induces the following morphisms 
\begin{align*} 
N\col &P^D_kR^qf_{(X_{\os{\circ}{T}_0},D_{\os{\circ}{T}_0})/S(T)^{\nat}*}
(\eps^*_{(X_{\os{\circ}{T}_0},D_{\os{\circ}{T}_0})/S(T)^{\nat}}(E))
\tag{7.6.2}\label{ali:ondexd}\\
&\lo 
P^D_kR^qf_{(X_{\os{\circ}{T}_0},D_{\os{\circ}{T}_0})/S(T)^{\nat}*}
(\eps^*_{(X_{\os{\circ}{T}_0},D_{\os{\circ}{T}_0})/S(T)^{\nat}}(E))(-1;u)\quad (k\in  {\mab Z})
\end{align*} 
and 
\begin{align*} 
N\col & P_kR^qf_{(X_{\os{\circ}{T}_0},D_{\os{\circ}{T}_0})/S(T)^{\nat}*}
(\eps^*_{(X_{\os{\circ}{T}_0},D_{\os{\circ}{T}_0})/S(T)^{\nat}}(E))
\tag{7.6.3}\label{ali:okxd}\\
&\lo 
P_{k-2}R^qf_{(X_{\os{\circ}{T}_0},D_{\os{\circ}{T}_0})/S(T)^{\nat}*}
(\eps^*_{(X_{\os{\circ}{T}_0},D_{\os{\circ}{T}_0})/S(T)^{\nat}}(E))(-1;u)\quad (k\in  {\mab Z}). 
\end{align*} 
\end{coro} 

\section{Bifiltered base change theorem}\label{sec:bckf}
In this section we prove the bifiltered base change theorem 
of $(A_{\rm zar},P^D,P)$. 
\par
Let the notations be as in the previous section. 
Assume that $\os{\circ}{X}_{T_0}$ is quasi-compact. 
Let $f \col X_{\os{\circ}{T}_0} \lo S(T)^{\nat}$ 
be the structural morphism.   

\begin{prop}\label{prop:bdccd}  
Assume that $\os{\circ}{f} \col \os{\circ}{X}_{T_0}\lo \os{\circ}{T}$ 
is quasi-compact and quasi-separated. 
Then $Rf_*((A_{\rm zar}((X_{\os{\circ}{T}_0},D_{\os{\circ}{T}_0})/S(T)^{\nat},E),P^{D_{\os{\circ}{T}_0}},P))$ 
is isomorphic to a bounded bifiltered complex of 
${\mathcal O}_T$-modules. 
\end{prop}
\begin{proof}
By  \cite[7.6 Theorem]{bob},  
$Rf_{\os{\circ}{X}{}^{(l)}_{T_0}\cap \os{\circ}{D}{}^{(m)}_{T_0}/\os{\circ}{T}*}
(E_{\os{\circ}{X}{}^{(l)}_{T_0}\cap \os{\circ}{D}{}^{(m)}_{T_0}/\os{\circ}{T}})$ 
$(0\leq l,m\in {\mab N})$ is bounded. 
Hence 
$Rf_*((A_{\rm zar}((X_{\os{\circ}{T}_0},D_{\os{\circ}{T}_0})/S(T)^{\nat},E),P^{D_{\os{\circ}{T}_0}},P))$ 
is bounded by the spectral sequence (\ref{eqn:escssp}). 
\end{proof}

\begin{theo}[{\bf Log base change theorem of 
$(A_{\rm zar},P^D,P)$}]\label{theo:bccange} 
Let the assumptions be as in {\rm (\ref{prop:bdccd})}.  
Let $(T',{\mathcal J}',\del')$ be another log PD-enlargement over $S$. 
Assume that  ${\mathcal J}'$ is quasi-coherent. 
Set $T'_0:=\ul{\rm Spec}^{\log}_{T'}({\mathcal O}_{T'}/{\mathcal J}')$. 
Let $u\col (S(T')^{\nat},{\mathcal J}',\del') \lo (S(T)^{\nat},{\mathcal J},\del)$ be 
a morphism of fine log PD-schemes. 
Let 
$f' \col (X_{\os{\circ}{T}{}'_0},D_{\os{\circ}{T}{}'_0})=
(X,D)\times_{S}S_{\os{\circ}{T}{}'_0} \lo S(T')^{\nat}$ 
be the base change morphism of $f$  
by the morphism $S(T')^{\nat}\lo S(T)^{\nat}$.  
Let $q \col (X_{\os{\circ}{T}{}'_0},D_{\os{\circ}{T}{}'_0}) 
\lo (X_{\os{\circ}{T}_0},D_{\os{\circ}{T}_0})$ 
be the induced morphism by $u$. 
Then there exists 
the following canonical bifiltered isomorphism
\begin{align*}
& Lu^*R\os{\circ}{f}_*((A_{\rm zar}((X_{\os{\circ}{T}_0},D_{\os{\circ}{T}_0})
/S(T)^{\nat},E),P^{D_{\os{\circ}{T}_0}},P)) \os{\sim}{\lo} 
\tag{8.2.1}\label{eqn:blucpw}\\
&R\os{\circ}{f}{}'_*((A_{\rm zar}((X_{\os{\circ}{T}{}'_0},D_{\os{\circ}{T}{}'_0})/S(T')^{\nat},
\os{\circ}{q}{}^{*}_{\rm crys}(E)),P^{D_{\os{\circ}{T}{}'_0}},P))
\end{align*}
in ${\rm DF}(f'{}^{-1}({\mathcal O}_{T'}))$. 
\end{theo}
\begin{proof}  
Let the notations be as in \S\ref{sec:psc}. 
Set 
$\ol{\mathcal P}_{\bul,\ol{S(T')^{\nat}}}
:=\ol{\mathcal P}_{\bul}\times_{\ol{S(T)^{\nat}}}\ol{S(T')^{\nat}}$. 
Let   
$\ol{\mathfrak D}{}'_{\bul}$ 
be the log PD-envelope of the immersion 
$(X_{\os{\circ}{T}{}'_0,\bul},D_{\os{\circ}{T}{}'_0,\bul})\os{\sus}{\lo} \ol{\mathcal P}_{\bul,\ol{S(T')^{\nat}}}$ 
over $(\os{\circ}{T}{}',{\mathcal J}',\del')$. 
Then we have the natural morphisms 
$\ol{\mathcal P}_{\bul,\ol{S(T')^{\nat}}}\lo \ol{\mathcal P}_{\bul}$ 
and 
$\ol{\mathfrak D}{}'_{\bul}  \lo \ol{\mathfrak D}_{\bul}$.  
We also have the identity morphism 
${\rm id}\col \os{\circ}{q}{}^{*}_{\rm crys}(E)
\lo \os{\circ}{q}{}^{*}_{\rm crys}(E)$. 
Obviously the morphism 
$q\col X_{\os{\circ}{T}{}'_0}\lo X_{\os{\circ}{T}{}_0}$ satisfies 
the assumption in (\ref{theo:funas}).  
%and (5.1.2.6). 
Hence we have the following natural morphism  
\begin{equation*} 
(A_{\rm zar}(X_{\os{\circ}{T}_0}/S(T)^{\nat},E),P^{D_{\os{\circ}{T}_0}},P) 
\lo 
Rq_*((A_{\rm zar}(X_{\os{\circ}{T}{}'_0}/S(T')^{\nat},E),P^{D_{\os{\circ}{T}{}'_0}},P))
\tag{8.2.2}\label{eqn:bcxa}
\end{equation*} 
by (\ref{theo:funas}).
By applying $Rf_*$ to (\ref{eqn:bcxa}) and using 
the adjoint property of $L$ and $R$ (\cite[(5.2)]{nlf}), 
we have the natural morphism (\ref{eqn:blucpw}). 
Here we have used the boundedness in 
(\ref{prop:bdccd}) for the well-definedness of $Lu^*$. 
\par
The rest of the proof is the same as that of \cite[(1.6.2)]{nb}. 
\end{proof} 

\par 
Let $\os{\circ}{Y}$ be a smooth scheme over $\os{\circ}{T}$. 
Endow $\os{\circ}{Y}$ with the inverse image of $M_{S_{\os{\circ}{T}}}$ 
and let $Y$ be the resulting log scheme. 
Assume that $Y$ has a log smooth lift ${\mathcal Y}$ over $S(T)^{\nat}$. 
Let $D_{{\mathcal Y}/S(T)^{\nat}}(1)$ be the log PD-envelope of the immersion 
${\mathcal Y}\os{\sus}{\lo} {\mathcal Y}\times_{S(T)^{\nat}}{\mathcal Y}$ over $(S(T)^{\nat},{\mathcal J},\del)$. 
As in \cite[V]{bb} and \cite[\S7]{bob}, we have the following two corollaries 
(cf.~\cite[(2.10.5), (2.10.7)]{nh2}) by 
using (\ref{theo:bccange}) and 
a fact that 
$p_i \col \os{\circ}{D}_{{\mathcal Y}/S(T)}(1)\lo \os{\circ}{\mathcal Y}$ 
$(i=1,2)$ 
is flat (\cite[(6.5)]{klog1}): 

\begin{coro}\label{coro:connfil}
Let $g\col (X_{\os{\circ}{T}_0},D_{\os{\circ}{T}_0})
\lo Y$ be an SNCL scheme with a relative SNCD on $X_{\os{\circ}{T}_0}/Y$. 
Let $q$ be an integer.
Let $g\col (X_{\os{\circ}{T}_0},D_{\os{\circ}{T}_0})\lo {\mathcal Y}$ 
be also the structural morphism. 
Let $Q$ be $P^D$, $P$ or $P^D\cap P$. 
Then there exists  a quasi-nilpotent integrable connection 
\begin{align*}
&Q_kR^qg_{(X_{\os{\circ}{T}_0},D_{\os{\circ}{T}_0})/{\mathcal Y}*}
(\eps^*_{(X_{\os{\circ}{T}_0},D_{\os{\circ}{T}_0})/S(T)^{\nat}}(E))
\os{\nabla_k}{\lo} \tag{8.3.1}\\
& Q_kR^qg_{(X_{\os{\circ}{T}_0},D_{\os{\circ}{T}_0})/{\mathcal Y}*}
(\eps^*_{(X_{\os{\circ}{T}_0},D_{\os{\circ}{T}_0})/S(T)^{\nat}}(E))
{\otimes}_{{\mathcal O}_{\mathcal Y}}{\Om}_{{\mathcal Y}/S(T)^{\nat}}^1
\end{align*}
making the following diagram commutative 
for any two nonnegative integers $k\leq l:$
\begin{equation*}
\begin{CD}
Q_kR^qg_{(X_{\os{\circ}{T}_0},D_{\os{\circ}{T}_0})/{\mathcal Y}*}
(\eps^*_{(X_{\os{\circ}{T}_0},D_{\os{\circ}{T}_0})/S(T)^{\nat}}(E))
@>{\nabla_k}>> \\
@V{\bigcap}VV  \\
Q_lR^qg_{(X_{\os{\circ}{T}_0},D_{\os{\circ}{T}_0})/{\mathcal Y}*}
(\eps^*_{(X_{\os{\circ}{T}_0},D_{\os{\circ}{T}_0})/S(T)^{\nat}}(E))
@>{\nabla_l}>>
\end{CD}
\tag{8.3.2}
\end{equation*}
\begin{equation*}
\begin{CD}
Q_kR^qg_{(X_{\os{\circ}{T}_0},D_{\os{\circ}{T}_0})/{\mathcal Y}*}
(\eps^*_{(X_{\os{\circ}{T}_0},D_{\os{\circ}{T}_0})/S(T)^{\nat}}(E))
{\otimes}_{{\mathcal O}_{\mathcal Y}}{\Om}_{{\mathcal Y}/S(T)^{\nat}}^1\\
@V{\bigcap}VV \\ 
Q_lR^qg_{(X_{\os{\circ}{T}_0},D_{\os{\circ}{T}_0})/{\mathcal Y}*}
(\eps^*_{(X_{\os{\circ}{T}_0},D_{\os{\circ}{T}_0})/S(T)^{\nat}}(E)) 
{\otimes}_{{\mathcal O}_{\mathcal Y}}{\Om}_{{\mathcal Y}/S(T)^{\nat}}^1.
\end{CD} 
\end{equation*}
\end{coro}
\begin{proof} 
This follows from (\ref{theo:bccange}) as in \cite[\S7]{bob}.
\end{proof} 

\begin{coro}\label{coro:fctd}
Let the notations and the assumptions be as in $(\ref{prop:bdccd})$. 
Let $Q$ be $P^D$, $P$ or $P^D\cap P$. 
Then 
$$Rf_*
(A_{\rm zar}((X_{\os{\circ}{T}_0},D_{\os{\circ}{T}_0})/S(T)^{\nat},E))$$
and 
$$Rf_*
(Q_kA_{\rm zar}((X_{\os{\circ}{T}_0},D_{\os{\circ}{T}_0})/S(T)^{\nat},E))\quad 
(k \in{\mab N})$$
have finite tor-dimension. 
Moreover, if $\os{\circ}{T}$ is noetherian and 
if $\os{\circ}{f}$ is proper,
then $Rf_*
(A_{\rm zar}((X_{\os{\circ}{T}_0},D_{\os{\circ}{T}_0})/S(T)^{\nat},E))$ and 
$Rf_*(Q_kA_{\rm zar}((X_{\os{\circ}{T}_0},D_{\os{\circ}{T}_0})/S(T)^{\nat},E))$ 
are perfect complexes of ${\mathcal O}_T$-modules.
\end{coro}

%In this paper we generalize 
%the definition of the strictly perfectness in \cite[(2.10.8)]{nh2} as follows: 

%\begin{defi}\label{defi:stpf}
%Let $A$ be a noetherian commutative ring. 
%Let $(E^{\bul},\{E^{(1)\bul}_k\},\{E^{(2)\bul}_k\}) \allowbreak \in {\rm CF}^2(A)$ 
%be a bifiltered complex of $A$-modules.  We say that 
%$(E^{\bul},\{E^{(1)\bul}_k\},\{E^{(2)\bul}_k\})$ 
%is {\it bifilteredly strictly perfect} if $(E^{\bul},\{E^{(1)\bul}_k\},\{E^{(2)\bul}_k\})$ 
%is a bounded bifiltered complex of $A$-modules 
%and if all $E^q$'s and all $E_{k_1k_2}^{(i_1i_2)q}:=E_{k_1}^{(i_1)q}\cap 
%E_{k_2}^{(i_2)q}$'s $(1\leq i_1,i_2\leq 2)$ 
%are finitely generated projective $A$-modules. 
%\end{defi}

%\begin{defi}
%\label{defi:dsak} 
%Let $A$ be a commutative ring with unit element. 
%For a bifiltered $A$-module $(E, \{E^{(1)}_k\},\{E^{(2)}_k\})$ 
%whose filtrations are finite and for  
%a family $\{T_{l_1l_2}^{(i_1i_2)}\}_{1\leq i_1,i_2\leq 2,l_1l_2 \in {\mab Z}}$ 
%of $A$-modules,   
%we say that $(E, \{E^{(1)}_k\},\{E^{(2)}_k\})$ is the 
%{\it direct sum} of $\{T_{l_1l_2}^{(i_1i_2)}\}_{1\leq i_1,i_2\leq 2,l_1,l_2 \in {\mab Z}}$ if 
%$E^{(i_1i_2)}_{k_1k_2} =\bigoplus_{l_1 \leq k_1,l_2\leq k_2} T_{l_1l_2}^{(i_1i_2)}$ 
%$(\forall k_1,\forall k_2 \in {\mab Z})$.  
%\end{defi}

Using (\ref{theo:bccange}), 
%K.~Kato's log base change theorem of log crystalline cohomologies (\cite[(6.10)]{klog1}) 
the base change theorem of classical  crystalline cohomologies
and \cite[(2.10.10)]{nh2}, 
we have the following corollary 
(cf.~\cite[(2.10.11)]{nh2}): 

\begin{coro}\label{coro:filpcerf}
Let the notations 
and the assumptions be as in $(\ref{coro:fctd})$.
Then the filtered complexes  
$Rf_*((A_{\rm zar}((X_{\os{\circ}{T}_0},D_{\os{\circ}{T}_0})/S(T)^{\nat},E),P^{D_{\os{\circ}{T}_0}}))$,  
$Rf_*((A_{\rm zar}((X_{\os{\circ}{T}_0},D_{\os{\circ}{T}_0})/S(T)^{\nat},E),P))$ 
and 
$Rf_*((A_{\rm zar}((X_{\os{\circ}{T}_0},D_{\os{\circ}{T}_0})/S(T)^{\nat},E),
P^{D_{\os{\circ}{T}_0}}\cap P))$ 
are filtered perfect complexes of ${\mathcal O}_T$-modules, that is, 
locally on $T_{\rm zar}$, filteredly quasi-isomorphic to 
a filtered strictly perfect complex {\rm (\cite[(2.10.8)]{nh2})}.
\end{coro}
\begin{proof}
(\ref{coro:filpcerf}) 
immediately follows 
from (\ref{coro:fctd}) and \cite[(2.10.10)]{nh2}.
\end{proof}

\section{Infinitesimal deformation invariance}\label{sec:infhdi}  
In this section we prove the infinitesimal deformation invariance 
of the pull-back of a morphism of SNCL schemes with relative SNCD's in characteristic $p>0$ on 
zariskian $p$-adic bifiltered El-Zein-Steenbrink-Zucker complexes.  
As in \cite{boi}, to prove the invariance, we use Dwork's trick for enlarging 
the radius of convergence of log $F$-isocrystals by the use of the relative Frobenius.  
Precisely speaking, in our case, 
we use the base change by the iteration of the abrelative
Frobenius morphism (not usual relative Frobenius morphism) of the base scheme as in 
\cite{nb} and \cite{nhir}. 
%The notion of the truncated simplicial base change of SNCL schemes 
%and admissible immersions defined in \S\ref{sec:bcsncl} 
%gives us an appropriate framework. 
\par 
For a bifiltered complex $(K^{\bul},P^{(1)}, P^{(2)})$ in 
the derived category of bifiltered complexes 
(\cite{nlf}), denote $(K^{\bul},P^{(1)}, P^{(2)})\otimes_{\mab Z}^L{\mab Q}$ 
by  $(K^{\bul},P^{(1)}, P^{(2)})_{\mab Q}$ 
for simplicity of notation. 
We omit the proofs of the results in this section because the proofs are 
the same as those in \cite[(6.1)]{nhir}.

\par 
The following is a main result in this section.   

\begin{theo}[{\bf Infinitesimal deformation invariance of the pull-back of a morphism 
on $p$-adic bifiltered El Zein-Steenbrink-Zucker complexes}]
\label{theo:definv}
Let $\star$ be nothing  or $\prime$. 
Let $n$ be a positive integer.  
Let $S^{\star}$ be a family of log points. 
Assume that $S^{\star}$ is of characteristic $p>0$. 
Let $F_{S^{\star}}\col S^{\star}\lo S^{\star}$ be the absolute Frobenius endomorphism. 
Set $S^{{\star}[p^n]}:=
S^{\star}\times_{\os{\circ}{S}{}^{\star},\os{\circ}{F}{}^n_{S^{\star}}}\os{\circ}{S}{}^{\star}$. 
Let $(T^{\star},{\mathcal J}^{\star},\del^{\star})$ be 
a log $p$-adic formal PD-thickening of $S^{\star}$. 
Set $T^{\star}_0:=T^{\star}~{\rm mod}~{\mathcal J}^{\star}$. 
Let $f^{\star} \col (X^{\star},D^{\star}) \lo S^{\star}$ be an SNCL scheme 
with a relative SNCD over $S^{\star}$. 
Assume that $\os{\circ}{X}{}^{\star}_{T^{\star}_0}$ is quasi-compact. 
Let $\iota^{\star} \col T^{\star}_0(0) \os{\subset}{\lo} T^{\star}_0$  
be an exact closed nilpotent immersion.  
Set $S^{\star}_{\os{\circ}{T}{}^{\star}_{0}}
:=S^{\star}\times_{\os{\circ}{S}{}^{\star}}\os{\circ}{T}{}^{\star}_0$ 
and 
$S^{\star}_{\os{\circ}{T}{}^{\star}_0(0)}:=S^{\star}
\times_{\os{\circ}{S}{}^{\star}}\os{\circ}{T}{}^{\star}_0(0)$ 
and $(X^{\star}_{\os{\circ}{T}_0},D^{\star}_{\os{\circ}{T}_0}):=
(X^{\star},D^{\star})\times_{S^{\star}}S^{\star}_{\os{\circ}{T}{}^{\star}_0}$,  
$(X^{\star}_{\os{\circ}{T}{}^{\star}_0}(0),D^{\star}_{\os{\circ}{T}{}^{\star}_0}(0))
:=(X^{\star},D^{\star})\times_{S^{\star}}S^{\star}_{\os{\circ}{T}{}^{\star}_0(0)}$. 
Set $(X^{\star}{}^{[p^n]},D^{\star}{}^{[p^n]}):=
(X^{\star},D^{\star})\times_{S^{\star}}S^{{\star}[p^n]}$ 
and 
$(X^{\star}{}^{[p^n]}_{\! \! \!\os{\circ}{T}{}^{\star}_0},
D^{\star}{}^{[p^n]}_{\! \! \!\os{\circ}{T}{}^{\star}_0}):=
(X^{\star}{}^{[p^n]},D^{\star}{}^{[p^n]})\times_{S^{{\star}[p^n]}}
S^{{\star}[p^n]}_{\os{\circ}{T}{}^{\star}_0}$. 
Note that the underlying scheme of 
$X^{\star}{}^{[p^n]}_{\! \! \!\os{\circ}{T}{}^{\star}_0}$ is equal to that of 
$X^{\star}_{\os{\circ}{T}{}^{\star}_0}$ and that we have the log scheme 
$S^{{\star}[p^n]}_{\os{\circ}{T}{}^{\star}_0}$ by using the composite morphism 
$T^{\star}_0\lo S^{\star}\lo S^{\star [p^n]}$, 
where 
the morphism $S^{\star}\lo S^{\star[p^n]}$ is the composite  morphism of 
the abrelative Frobenius morphisms of $S^{[p^m]\star}$ $(0\leq m\leq n-1)$.   
%Let $X^{\star}{}^{[p^n]}_{\! \! \!\os{\circ}{T}_0}$ 
%be the \v{C}ech diagram of the 
%disjoint union of the members of 
%an affine open covering of $X^{\star}{}^{[p^n]}_{\! \! \!\os{\circ}{T}_0}$ 
%obtained by that of $X^{\star}_{\os{\circ}{T}_0}$. 
Let $n$ be a positive integer such that the pull-back morphism 
$F^{n*}_{T^{\star}_0}\col 
{\mathcal O}_{T^{\star}_0}\lo {\mathcal O}_{T^{\star}_0}$ 
kills ${\rm Ker}({\mathcal O}_{T^{\star}_0}\lo {\mathcal O}_{T^{\star}_0(0)})$.
Let 
\begin{align*} 
g_0 \col (X_{\os{\circ}{T}_0}(0),D_{\os{\circ}{T}_0}(0)) \lo 
(X'_{\os{\circ}{T}{}'_0}(0),D'_{\os{\circ}{T}_0}(0))
\tag{9.1.1}\label{eqn:ldehtvn}
\end{align*}   
be a morphism of log schemes over $S(T)^{\nat}\lo S'(T')^{\nat}$ 
satisfying the conditions {\rm (6.5.2)} and  {\rm (6.5.3)}
%and {\rm (5.1.2.6)} 
for 
$(X_{\os{\circ}{T}_0}(0),D_{\os{\circ}{T}_0}(0))$ and 
$(X'_{\os{\circ}{T}{}'_0}(0),D'_{\os{\circ}{T}{}'_0}(0))$. 
%Set $X^{\star}_{T^{\star}_0}
%:=X^{\star}\times_{S^{\star}}T^{\star}_0$. 
Then the following hold$:$ 
\par 
$(1)$ There exist a canonical bifiltered morphism
\begin{align*}
g^*_{0} &: 
(A_{\rm zar}((X'_{\os{\circ}{T}{}'_0},D'_{\os{\circ}{T}{}'_0})/S'(T')^{\nat}),P^{D_{\os{\circ}{T}{}'_0}},P)_{\mab Q} 
\lo 
Rg_{0*}((A_{\rm zar}((X_{\os{\circ}{T}_0},D_{\os{\circ}{T}_0})/S(T)^{\nat}),P^{D_{\os{\circ}{T}_0}},P)_{\mab Q}) 
\tag{9.1.2}\label{eqn:ldefinvn}
\end{align*}  
fitting into the following commutative diagram
\begin{equation*} 
\begin{CD}
A_{\rm zar}((X'_{\os{\circ}{T}{}'_0},D'_{\os{\circ}{T}{}'_0})/S'(T')^{\nat})_{\mab Q} 
@>{g^*_{0}}>>
Rg_{0*}((A_{\rm zar}((X_{\os{\circ}{T}_0},D_{\os{\circ}{T}_0})/S(T)^{\nat}))_{\mab Q})\\
@V{\simeq}VV @VV{\simeq}V \\
Ru_{(X'_{\os{\circ}{T}{}'_0},D'_{\os{\circ}{T}{}'_0})/S'(T')^{\nat}*}
({\mathcal O}_{(X'_{\os{\circ}{T}{}'_0},D'_{\os{\circ}{T}{}'_0})/S'(T')^{\nat}})_{\mab Q} 
@>{g^*_{0}}>>
Rg_{0*}Ru_{(X_{\os{\circ}{T}_0},D_{\os{\circ}{T}_0})/S(T)^{\nat}*}
({\mathcal O}_{(X_{\os{\circ}{T}_0},D_{\os{\circ}{T}_0})/S(T)^{\nat}})_{\mab Q}. 
\end{CD}
\tag{9.1.3}\label{cd:ldtvn}
\end{equation*} 
Here the last horizontal morphism $g_0^*$ is the morphism constructed in 
{\rm \cite[(5.3.1)]{nb}}. 
\par 
$(2)$ Let $S''$, $(T'',{\mathcal J}'',\del'')$ and $\iota'' \col T''_0(0)\os{\subset}{\lo} T''_0$ 
be analogous objects to 
$S'$, $(T',{\mathcal J}',\del')$ and $\iota' \col T'_0(0)\os{\subset}{\lo} T'_0$, respectively.  
Let $g'_{0}\col (X'_{\os{\circ}{T}{}'_0}(0),D'_{\os{\circ}{T}{}'_0}(0))\lo 
(X''_{\os{\circ}{T}{}''_0}(0),D''_{\os{\circ}{T}{}''_0}(0))$ be 
a similar morphism to $g_{0}$. 
%Let $n'$ be a similar positive integer to $n$. 
Then 
\begin{align*} 
(g'_{0}\circ g_{0})^*
& =Rg'_{0*}(g^*_{0})\circ g'{}^*_{\! \!0}\col 
Ru_{(X''_{\os{\circ}{T}{}''_0},D''_{\os{\circ}{T}{}''_0})/S''(T'')^{\nat}*}
({\mathcal O}_{(X''_{\os{\circ}{T}{}''_0},D''_{\os{\circ}{T}{}''_0})/S''(T'')^{\nat}})_{\mab Q} 
\tag{9.1.4}\label{eqn:ldfilnvn}\\
& \lo 
R(g'_{0}\circ g_{0})_*Ru_{(X_{\os{\circ}{T}{}_0},D_{\os{\circ}{T}{}_0})/S(T)^{\nat}*}
({\mathcal O}_{(X_{\os{\circ}{T}{}_0},D_{\os{\circ}{T}{}_0})/S(T)^{\nat}})_{\mab Q}. 
\end{align*}
\par 
$(3)$ 
\begin{align*}
{\rm id}^*_{(X_{T^{\star}_0}(0),D_{T^{\star}_0}(0))}=
{\rm id}_{(A_{\rm zar}((X_{\os{\circ}{T}_0},
D_{\os{\circ}{T}{}'_0})/S(T)^{\nat})_{\mab Q},P^{D_{\os{\circ}{T}_0}},P)}.
\tag{9.1.5}\label{eqn:ldeoxnvn}
\end{align*}
\par 
$(4)$ If $g_{0}$ has a lift $g_{} \col (X_{\os{\circ}{T}_0},D_{\os{\circ}{T}_0})
\lo (X'_{\os{\circ}{T}{}'_0},D'_{\os{\circ}{T}{}'_0})$ over 
$S_{\os{\circ}{T}_0}\lo S'_{\os{\circ}{T}{}'_0}$ satisfying the conditions {\rm (6.5.2)} and  {\rm (6.5.3)},  
%and $(5.1.2.6)$, 
then $g^*_{0}$ in $(\ref{eqn:ldefinvn})$ is equal to the induced morphism by $g^*_{}$ 
%in $(\ref{ali:ccm})$ 
for $E^{}={\mathcal O}_{(X_{\os{\circ}{T}{}_0},D_{\os{\circ}{T}{}_0})/S(T)^{\nat}}$
and 
$F^{}={\mathcal O}_{(Y_{\os{\circ}{T}{}'_0},C_{\os{\circ}{T}{}'_0})/S'(T')^{\nat}}$.
\end{theo}

\begin{coro}[{\bf Infinitesimal deformation invariance of 
$p$-adic bifiltered El Zein-Steenbrink-Zucker complexes}]\label{coro:finvcae}
If $S'=S$, $T'=T$ and $(X_{\os{\circ}{T}_0}(0),D_{\os{\circ}{T}_0}(0))=
(X'_{\os{\circ}{T}_0}(0),D'_{\os{\circ}{T}_0}(0))$, then 
\begin{align*} 
(A_{\rm zar}((X'_{\os{\circ}{T}_0},D'_{\os{\circ}{T}_0})/S(T)^{\nat}),
P^{D_{\os{\circ}{T}_0}},P)_{\mab Q}=
(A_{\rm zar}((X_{\os{\circ}{T}_0},D_{\os{\circ}{T}_0})/S(T)^{\nat}),
P^{D_{\os{\circ}{T}_0}}, P)_{\mab Q}. 
\tag{9.2.1}\label{ali:xdnz}
\end{align*} 
\end{coro}

\begin{coro}[{\bf Infinitesimal deformation invariance of 
bifiltered log isocrystalline cohomologies}]\label{coro:finvliae}
Let the notations be as in {\rm (\ref{coro:finvcae})}. 
Let $P^{D'_{\os{\circ}{T}_0}}$, $P$ and $P^{D_{\os{\circ}{T}_0}}$, $P$  be the induced filtrations on 
$$R^qf_{(X'_{\os{\circ}{T}_0},D'_{\os{\circ}{T}_0})/S(T)^{\nat}*}
({\mathcal O}_{(X'_{\os{\circ}{T}_0},D'_{\os{\circ}{T}_0})/S(T)^{\nat}})_{\mab Q}~{\rm and}~ 
R^qf_{(X_{\os{\circ}{T}_0},D_{\os{\circ}{T}_0})/S(T)^{\nat}*}
({\mathcal O}_{(X_{\os{\circ}{T}_0},D_{\os{\circ}{T}_0})/S(T)^{\nat}})_{\mab Q}$$ 
by 
$$(A_{\rm zar}((X'_{\os{\circ}{T}_0},D'_{\os{\circ}{T}_0})/S'(T')^{\nat}),
P^{D'_{\os{\circ}{T}_0}},P)\quad 
{\rm and}  \quad 
(A_{\rm zar}((X_{\os{\circ}{T}_0},D_{\os{\circ}{T}_0})/S(T)^{\nat}), P^{D_{\os{\circ}{T}_0}},P),$$ 
respectively.  
Then 
\begin{align*} 
&(R^qf_{(X'_{\os{\circ}{T}_0},D'_{\os{\circ}{T}_0})/S(T)^{\nat}*}
({\mathcal O}_{(X'_{\os{\circ}{T}_0},D'_{\os{\circ}{T}_0})/S(T)^{\nat}})_{\mab Q},P^{D'_{\os{\circ}{T}_0}},P)=\tag{9.3.1}\label{ali:wfoa}\\
&(R^qf_{(X_{\os{\circ}{T}_0},D_{\os{\circ}{T}_0})/S(T)^{\nat}*}
({\mathcal O}_{(X_{\os{\circ}{T}_0},D_{\os{\circ}{T}_0})/S(T)^{\nat}})_{\mab Q},P^{D_{\os{\circ}{T}_0}},P).
\\
\end{align*}  
\end{coro}

\section{The $E_2$-degeneration of the $p$-adic weight spectral sequence}\label{sec:filbo}
Let ${\mathcal V}$ be a complete discrete valuation ring with perfect residue field $\kap$ 
of mixed characteristics. 
In this section we assume that the underlying formal scheme $\os{\circ}{S}$ of 
the family of log points $S$ is a $p$-adic formal ${\mathcal V}$-scheme in the sense 
of \cite{of}. Let $X/S$ be a proper SNCL scheme.   Let 
$(T,z)$ be a flat log $p$-adic enlargement (see e.g.,~\cite{oc} (or \cite{nb}) for this notion); 
$z$ is a morphism $T_1\lo S$, 
where $T_1$ is an exact log subscheme of $T$ defined by $p{\mathcal O}_T$. 
Endow $p{\mathcal O}_T$ with the canonical PD-structure.  
In this section we prove the $E_2$-degeneration of 
the $p$-adic weight spectral sequence of 
$(X_{\os{\circ}{T}_1},D_{\os{\circ}{T}_1})/S(T)^{\nat}$ modulo torsion obtained by 
the filtered complex $(A_{\rm zar}((X_{\os{\circ}{T}_1},D_{\os{\circ}{T}_1})/S(T)^{\nat}),P)$ by using 
the infinitesimal deformation invariance of isocrystalline cohomologies 
with weight filtrations in the previous section ((\ref{ali:wfoa})).

\begin{theo}[{\bf $E_2$-degeneration I}]\label{theo:e2dam}  
Let $s$ be the log point of a perfect field of characteristic $p>0$. 
The spectral sequence $(\ref{eqn:espwfsp})$ for the case $S=s$ 
and $E={\mathcal O}_{\os{\circ}{X}_{T_0}/\os{\circ}{T}}$ 
degenerates at $E_2$.  
\end{theo} 
\begin{proof} 
As in the proof of \cite[(5.4.1)]{nb}, we may assume that 
$T={\mathcal W}(s)$.  In this case we have the absolute Frobenius endomorphism 
$F_{{\mathcal W}(s)}\col {\mathcal W}(s)\lo {\mathcal W}(s)$. 
If $\os{\circ}{s}$ is the spectrum of a finite field, 
then the $E_1$-term 
$E^{-k,q+k}_1$of (\ref{eqn:espwfsp}) is of pure weight of $q+k$ 
by \cite[Corollary 1. 2)]{kme}, \cite[(1.2)]{clt} and \cite[(2.2) (4)]{ndw}.  
However see  \cite[(6.11)]{ny} for the gap of the proof 
the weak-Lefschetz conjecture 
for a hypersurface of a large degree  in \cite{bwl}; 
I have filled the gap in \cite[(6.10)]{ny}. 
\par 
The rest of the proof is the same as that of \cite[(5.4.1)]{nb}. 
\end{proof}

\begin{theo}[{\bf $E_2$-degeneration II}]\label{theo:e2dgfam} 
Let $T$ be a log $p$-adic enlargement of $S/{\mathcal V}$ 
with structural morphism $T_1\lo S$. 
The spectral sequence $(\ref{eqn:espwfsp})$ modulo torsion 
for the case $E={\mathcal O}_{\os{\circ}{X}_{T_0}/\os{\circ}{T}}$ 
degenerates at $E_2$.  
\end{theo}
\begin{proof} 
Because the problem is local on $\os{\circ}{T}$, we may assume that 
$\os{\circ}{T}$ is quasi-compact.  
By virtue of (\ref{theo:e2dam}) and (\ref{coro:finvcae}), 
the proof of this theorem is the same as that of \cite[(5.4.3)]{nb}.
\end{proof}

\section{Log convergence of the weight filtrations}\label{sec:e2}
In this section we prove the log convergence of 
the weight filtration on the log isocrystalline cohomological sheaf 
induced by the spectral sequence (\ref{eqn:espwfsp}) if $\os{\circ}{X}/\os{\circ}{S}$ is proper. 
\par 
Roughly speaking, we can obtain all the results in this section 
by using the log base change theorem of 
$(A_{\rm zar},P^D,P)$ ((\ref{theo:bccange})) and by 
replacing $(A_{{\rm zar},{\mab Q}},P^D,P)$ defined in this paper 
with $(A_{{\rm zar},{\mab Q}},P)$ in \cite{nb}. 
%Here we have only to note that the base change morphism 
%satisfies the conditions {\rm (6.5.2)} and  {\rm (6.5.3)}
%as in the proof of (\ref{theo:bccange}). 
For this reason, we omit or sketch the proofs of almost all the results in this section. 
We use fundamental notions and results in \cite[(5.2)]{nb}. 
\par 
Let ${\mathcal V}$ be a complete discrete valuation ring of mixed characteristics 
$(0,p)$ with perfect residue field. Let $K$ be the fraction field of ${\mathcal V}$. 
Set $B=({\rm Spf}({\mathcal V}),{\mathcal V}^*)$.   
Let $S$ be a $p$-adic formal family of log points over $B$ 
such that $\os{\circ}{S}$ is a ${\mathcal V}/p$-scheme.  
%in the sense of \cite[\S1]{of} (e.~g., ${\mathcal V}/p$-scheme).   
Let $(X,D)/S$ be a proper SNCL scheme with a relative SNCD.  
\par  
Let $n$ be a nonnegative integer.  
Let $T$ be an object of the category 
${\mathcal E}^{\sq}_{p,n}:={\rm Enl}^{\sq}_p(S^{[p^n]}/{\mathcal V})$ 
of (solid) $p$-adic enlargements of $S^{[p^n]}/{\mathcal V}$ 
($\sq=$sld or nothing) (\cite[(5.1.3)]{nb}).   
Then the hollowing out $S^{[p^n]}(T)^{\nat}$ of $S^{[p^n]}(T)$ 
is a formal family of log points. 
Let $z_i\col T_i \lo S^{[p^n]}$ $(i=0,1)$ be the structural morphism. 
Here $T_1:=\ul{\rm Spec}^{\log}_T({\mathcal O}_T/p)$ and $T_0:=(T_1)_{\rm red}$. 
Set $(X^{[p^n]}_{\os{\circ}{T}_i},D^{[p^n]}_{\os{\circ}{T}_i})
:=(X,D)\times_SS^{[p^n]}_{\os{\circ}{T}_i} 
=(X^{[p^n]},D^{[p^n]})\times_{\os{\circ}{S}}\os{\circ}{T}_i$, 
where $(X^{[p^n]},D^{[p^n]}):=(X,D)\times_SS^{[p^n]}$. 
Let 
$f^{[p^n]}_{\os{\circ}{T}} \col (X^{[p^n]}_{\os{\circ}{T}_i},D^{[p^n]}_{\os{\circ}{T}_i})
\lo S^{[p^n]}(T)^{\nat}$ 
be the structural morphism. 
(Note that this notation is different from the notation 
$f$ in \S\ref{sec:psc} since we add the symbol $\os{\circ}{T}$ to $f^{[p^n]}$ as a subscript.) 
Because $S^{[p^n]}(T)^{\nat}$  is a $p$-adic formal family of log points,  
$(X^{[p^n]}_{\os{\circ}{T}_1},D^{[p^n]}_{\os{\circ}{T}_1})/S^{[p^n]}_{\os{\circ}{T}_1}$ 
is a proper SNCL scheme 
with an exact PD-closed immersion 
$S^{[p^n]}_{\os{\circ}{T}_1} \os{\sus}{\lo} S^{[p^n]}(T)^{\nat}$. 

\par 
%Consider the Case I. 
%Set ${\mathcal E}^{\sq}_{p,n}:={\rm Enl}^{\sq}_p(S^{[p^n]}/{\mathcal V})$. 
Assume that we are given a flat coherent crystal 
$E_n=\{E_n(\os{\circ}{T})\}_{T\in {\mathcal E}^{\sq}_{p,n}}$ of 
${\mathcal O}_{\{\os{\circ}{X}{}^{[p^n]}_{T_1}
/\os{\circ}{T}\}_{T\in {\mathcal E}^{\sq}_{p,n}}}$-modules.  
Let $T$ be an object of ${\mathcal E}^{\sq}_{p,n}$.
%Let $z\col T_1\lo S^{[p^n]}$ be the structural morphism. 
Then we obtain the following $p$-adic iso-zariskian bifiltered complex
\begin{align*} 
(A_{\rm zar}((X^{[p^n]}_{\os{\circ}{T}_1},D^{[p^n]}_{\os{\circ}{T}_1})/
S^{[p^n]}(T)^{\nat},E^{}(\os{\circ}{T})),P^{D^{[p^n]}_{\os{\circ}{T}_1}},P)_{\mab Q}
%=(A_{\rm zar}(X_{T_1}/T,E^{}(\os{\circ}{T})),P)
%\otimes_{\mab Z}{\mab Q}
\in {\rm D}^+{\rm F}^2(
%f^{-1}_{\os{\circ}{T}}({\mathcal K}_T))
f^{-1}_{\os{\circ}{T}}({\mathcal K}_T))
\tag{11.0.1}\label{eqn:auxtds}
\end{align*} 
for each $T\in {\mathcal E}^{\sq}_{p,n}$.

\begin{prop}[{\bf cf.~\cite[(5.2.2)]{nb}}]\label{prop:tptt} 
Let ${\mathfrak g}\col T'\lo T$ be a morphism in ${\mathcal E}^{\sq}_{p,n}$.  
Then the induced morphism 
$g_{} \col X^{[p^n]}_{\os{\circ}{T}{}'_1}\lo X^{[p^n]}_{\os{\circ}{T}_1}$ 
by ${\mathfrak g}$
gives us the following natural morphism 
\begin{align*} 
g^*_{}\col &
(A_{\rm zar}((X^{[p^n]}_{\os{\circ}{T}_1},D^{[p^n]}_{\os{\circ}{T}_1})/
S^{[p^n]}(T)^{\nat},E^{}(\os{\circ}{T})),P^{D^{[p^n]}_{\os{\circ}{T}_1}},P) 
\tag{11.1.1}\label{eqn:auqeds}\\
&\lo
Rg_{*}
((A_{\rm zar}((X^{[p^n]}_{\os{\circ}{T}{}'_1},D^{[p^n]}_{\os{\circ}{T}{}'_1})
/S^{[p^n]}(T')^{\nat},E^{}(\os{\circ}{T}{}')),P^{D^{[p^n]}_{\os{\circ}{T}{}'_1}},P))
\end{align*}
of bifiltered complexes in 
${\rm D}^+{\rm F}^2(f_{\os{\circ}{T}}^{-1}({\mathcal K}_T))$ 
fitting into the following commutative diagram$:$ 
\begin{equation*} 
\begin{CD} 
A_{\rm zar}((X^{[p^n]}_{\os{\circ}{T}_1},D^{[p^n]}_{\os{\circ}{T}_1})/
S^{[p^n]}(T)^{\nat},E^{}(\os{\circ}{T}))
@>{g^*_{}}>> \\
@A{\simeq}AA \\
Ru_{(X^{[p^n]}_{\os{\circ}{T}_1},D^{[p^n]}_{\os{\circ}{T}_1})/
S^{[p^n]}(T)^{\nat}*}(\eps_{(X^{[p^n]}_{\os{\circ}{T}_1},D^{[p^n]}_{\os{\circ}{T}_1}))/
S^{[p^n]}(T)^{\nat}}(E^{}(\os{\circ}{T})))
@>{g^*_{}}>>
\end{CD} 
\end{equation*} 
\begin{equation*} 
\begin{CD} 
Rg_{*}
(A_{\rm zar}((X^{[p^n]}_{\os{\circ}{T}{}'_1},D^{[p^n]}_{\os{\circ}{T}{}'_1})/S^{[p^n]}(T')^{\nat},
E^{}(\os{\circ}{T}{}')))\\ 
@AA{\simeq}A \\
Rg_{*}(Ru_{(X^{[p^n]}_{\os{\circ}{T}{}'_1},D^{[p^n]}_{\os{\circ}{T}{}'_1})/
S^{[p^n]}(T')^{\nat}*}(\eps^*_{(X^{[p^n]}_{\os{\circ}{T}{}'_1},D^{[p^n]}_{\os{\circ}{T}{}'_1})/
S^{[p^n]}(T')^{\nat}}(E^{}(\os{\circ}{T}{}')))).
\end{CD} 
\end{equation*} 
For a similar morphism ${\mathfrak h}\col T''\lo T'$ to ${\mathfrak g}$ and 
a similar morphism 
$h_{}\col 
(X_{\os{\circ}{T}{}''_1},D_{\os{\circ}{T}{}''_1})\lo (X_{\os{\circ}{T}{}'_1},D_{\os{\circ}{T}{}'_1})$ 
to $g_{}$, the following relation 
\begin{align*} 
(h_{}\circ g_{})^*=
Rh_{*}(g^*_{})
\circ h^*_{} \col &
(A_{\rm zar}((X^{[p^n]}_{\os{\circ}{T}_1},D^{[p^n]}_{\os{\circ}{T}_1})/
S^{[p^n]}(T)^{\nat},E^{}(\os{\circ}{T})),P^{D^{[p^n]}_{\os{\circ}{T}_1}},P)\\
&\lo 
R(h\circ g)_{*}
((A_{\rm zar}((X^{[p^n]}_{\os{\circ}{T}{}''_1},D^{[p^n]}_{\os{\circ}{T}{}''_1})/S^{[p^n]}(T'')^{\nat},
E^{}(\os{\circ}{T}{}'')),P^{D^{[p^n]}_{\os{\circ}{T}{}''_1}},P))
\end{align*} 
holds.  
\begin{equation*} 
{\rm id}_{(X^{[p^n]}_{\os{\circ}{T}_1},D^{[p^n]}_{\os{\circ}{T}_1})}^*={\rm id} 
_{(A_{\rm zar}((X^{[p^n]}_{\os{\circ}{T}_1},D^{[p^n]}_{\os{\circ}{T}_1})/
S(T)^{\nat},E^{}(\os{\circ}{T})),P^{D^{[p^n]}_{\os{\circ}{T}_1}},P)}.  
\end{equation*} 
\end{prop} 
\begin{proof} 
This immediately follows from the contravariant functoriality (\ref{theo:funas}). 
\end{proof} 

\parno 
The morphism (\ref{eqn:auqeds}) induces the following morphism 
\begin{align*} 
{\mathfrak g}^*\col 
&
Rf^{[p^n]}_{\os{\circ}{T}*}((A_{\rm zar}((X^{[p^n]}_{\os{\circ}{T}_1},D^{[p^n]}_{\os{\circ}{T}_1})
/S^{[p^n]}(T)^{\nat},E^{}(\os{\circ}{T})),P^{D^{[p^n]}_{\os{\circ}{T}_1}},P))
\tag{11.1.2}\label{eqn:auaeds}\\
\lo 
&R{\mathfrak g}_*R f^{[p^n]}_{\os{\circ}{T}{}'*}
((A_{\rm zar}((X^{[p^n]}_{\os{\circ}{T}{}'_1},D^{[p^n]}_{\os{\circ}{T}{}'_1})/S^{[p^n]}(T')^{\nat},
E^{}(\os{\circ}{T}{}')),P^{D^{[p^n]}_{\os{\circ}{T}{}'_1}},P))
\end{align*}
$\! \! \!$of bifiltered complexes in ${\rm D}^+{\rm F}^2(f^{-1}({\mathcal K}_T))$. 

\par
Let ${\rm IsocF}^{\sq}_p(S/{\mathcal V})$ (resp.~${\rm IsocF}^{\sq}(S/{\mathcal V})$) 
be the category of filtered (solid) log $p$-adically convergent isocrystals on 
${\rm Enl}^{\sq}_p(S/{\mathcal V})$ 
(resp.~the category of filtered (solid) log convergent isocrystals on 
${\rm Enl}^{\sq}(S/{\mathcal V})$) defined in \cite{nb}. 
Let $F{\textrm -}{\rm IsocF}^{\sq}_p(S/{\mathcal V})$ 
(resp.~$F{\textrm -}{\rm IsocF}^{\sq}(S/{\mathcal V})$) 
be the category of filtered (solid) log $p$-adically convergent $F$-isocrystals on 
${\rm Enl}^{\sq}_p(S/{\mathcal V})$ 
(resp.~the category of filtered (solid) log convergent $F$-isocrystals on 
${\rm Enl}^{\sq}(S/{\mathcal V})$). 
Let $\star$ be nothing or $p$. 
Let ${\rm IsocF}^{2\sq}_{\star}(S/{\mathcal V})$ be the obvious bifiltered version of 
${\rm IsocF}^{\sq}_{\star}(S/{\mathcal V})$.  
Let 
$F^{\infty}{\textrm -}{\rm IsosF}^{\sq}_{\star}(S/{\mathcal V})$ be the category 
of filtered $F^{\infty}$-isospans on ${\rm Enl}^{\sq}_{\star}(S/{\mathcal V})$ 
defined in [loc.~cit.]
($F^{\infty}{\textrm -}{\rm IsocF}^{\sq}_{\star}(S/{\mathcal V})$ 
is a full subcategory of $F^{\infty}{\textrm -}{\rm IsosF}^{\sq}_{\star}(S/{\mathcal V})$.)   
Let 
$F^{\infty}{\textrm -}{\rm IsosF}^{2\sq}_{\star}(S/{\mathcal V})$ be the obvious  
bifiltered  version of $F^{\infty}{\textrm -}{\rm IsosF}^{\sq}_{\star}(S/{\mathcal V})$.

\par  
The following is a key lemma for (\ref{theo:pwfec}) below.

\begin{lemm}[{\bf cf.~\cite[(5.2.3)]{nb}}]\label{lemm:pnlcfi}
Let ${\rm IsocF}^{2\sq}_p(S/{\mathcal V})$ be the category of bifiltered log $p$-adically convergent isocrystals on 
${\rm Enl}^{\sq}_p(S/{\mathcal V})$. 
Assume that $M_S$ is split. Let $k$ be a nonnegative integer or $\infty$. 
Then there exists an object
\begin{align*} 
(R^q f^{[p^n]}_*(A_{\rm zar}((X^{[p^n]},D^{[p^n]})/K,E^{})),P^D,P)
\tag{11.2.1}\label{ali:pkalkf} 
\end{align*} 
of ${\rm IsocF}^{2\sq}_p(\os{\circ}{S}/{\mathcal V})$ 
such that 
{\footnotesize{\begin{align*} 
&(R^qf^{[p^n]}_*
(A_{\rm zar}((X^{[p^n]},D^{[p^n]})/K,E^{})),P^D,P)_{\os{\circ}{T}}=
 (R^qf^{[p^n]}_{\os{\circ}{T}*}
(A_{\rm zar}((X^{[p^n]}_{\os{\circ}{T}_1},
D^{[p^n]}_{\os{\circ}{T}_1})/S^{[p^n]}({\os{\circ}{T})},E^{}(\os{\circ}{T})))_{\mab Q},P^{D_{\os{\circ}{T}_1}},P)
\tag{11.2.2}\label{ali:pkakf}
\end{align*}}} 
for any object $\os{\circ}{T}$ in ${\rm Enl}^{\sq}_p(\os{\circ}{S}/{\mathcal V})$. 
In particular, there exists an object
\begin{align*} 
(R^qf^{[p^n]}_{*}
(\eps^*_{(X^{[p^n]},D^{[p^n]})/K}(E_K)),P^D,P)
\tag{11.2.3}\label{ali:pklkf} 
\end{align*} 
of ${\rm IsocF}^2_p(\os{\circ}{S}/{\mathcal V})$ 
such that 
\begin{align*} 
&(R^q
f^{[p^n]}_{*}
(\eps^*_{(X^{[p^n]},D^{[p^n]})/K}(E^{}_K)),P^{D^{[p^n]}},P)_{\os{\circ}{T}} 
\tag{11.2.4}\label{ali:pkkf}\\
&=
(R^qf^{[p^n]}_{(X^{[p^n]}_{\os{\circ}{T}_1},D^{[p^n]}_{\os{\circ}{T}_1})/S^{[p^n]}(\os{\circ}{T})*}
(\eps^*_{(X^{[p^n]}_{\os{\circ}{T}_1},D^{[p^n]}_{\os{\circ}{T}_1})/S^{[p^n]}(\os{\circ}{T})}
(E^{}(\os{\circ}{T})))_{\mab Q},P^{D^{[p^n]}_{\os{\circ}{T}_1}},P)
\end{align*} 
for any object $\os{\circ}{T}$ in ${\rm Enl}_p(\os{\circ}{S}/{\mathcal V})$. 
%locally on $\os{\circ}{S}$. 
\end{lemm}

%\par 
%Let $A$ be a commutative ring with unit element. 
%Recall that we have said that a bifiltered $A$-module $(M,P,Q)$ is bifilteredly flat 
%if $M$ and $M/\sum_{j=1}^NP_kM$ $(N\in {\mab Z}_{\geq 1},1\leq \forall i^j_1\leq \forall i^j_n \leq n, 
%\forall k^j_1, \forall k^j_n \in {\mab Z})$ are flat $A$-modules 
%(\cite[(4.4)]{nlf}). As a corollary of (\ref{lemm:pnlcfi}), we obtain the following: 

%\begin{coro}[{\bf cf.~\cite[(5.2.4)]{nb}}]\label{coro:flft}
%For a hollow log $p$-adic enlargement $T$ of $S^{[p^n]}/{\mathcal V}$,    
%the bifiltered sheaf
%\begin{align*}  
%(R^qf^{[p^n]}_{\os{\circ}{T}*}
%(A_{\rm zar}((X^{[p^n]}_{\os{\circ}{T}_1},D^{[p^n]}_{\os{\circ}{T}_1})/S^{[p^n]}(T)^{\nat},
%E^{}(\os{\circ}{T}))),P^{D^{[p^n]}_{\os{\circ}{T}_1}},P) 
%\tag{11.3.1}\label{ali:pce}
%\end{align*} 
%is a bifilteredly flat ${\cal K}_T$-modules. 
%In particular, the bifiltered sheaf 
%\begin{align*}  
%(R^qf^{[p^n]}_{\os{\circ}{T}*}
%(\eps^*_{(X^{[p^n]}_{\os{\circ}{T}_1},D^{[p^n]}_{\os{\circ}{T}_1})
%/S^{[p^n]}(T)^{\nat}}E^{}(\os{\circ}{T})),P^{D^{[p^n]}_{\os{\circ}{T}_1}}.P)
%\tag{11.3.2}\label{ali:pcte}
%\end{align*} 
%is a bifilteredly flat ${\cal K}_T$-module. 
%\end{coro} 

\begin{theo}[{\bf Log $p$-adic convergence of the weight filtration (cf.~\cite[(5.2.3)]{nb})}]
\label{theo:pwfec} 
Let $k$ be a nonnegative integer or $\infty$. 
Then there exists a unique object
\begin{align*} 
(R^qf^{[p^n]}_{*}
(A_{\rm zar}((X^{[p^n]},D^{[p^n]})/K,E^{}))^{\sq},P^{D^{[p^n]}},P)
\tag{11.3.1}\label{ali:pcpkxee}
\end{align*} 
of ${\rm IsocF}^{2\sq}_p(S^{[p^n]}/{\mathcal V})$ 
such that  
\begin{align*} 
&(R^qf^{[p^n]}_{*}(A_{\rm zar}((X^{[p^n]},D^{[p^n]})/K,E^{}))^{\sq},P^{D^{[p^n]}},P)_{T}=
\tag{11.3.2}\label{ali:pkaskf} \\
&(R^qf^{[p^n]}_{\os{\circ}{T}*}
(A_{\rm zar}((X^{[p^n]}_{\os{\circ}{T}_1},D^{[p^n]}_{\os{\circ}{T}_1})
/S^{[p^n]}(T)^{\nat},E^{}(\os{\circ}{T})))_{\mab Q},P^{D^{[p^n]}_{\os{\circ}{T}_1}},P)
\end{align*} 
for any object $T$ of ${\rm Enl}^{\sq}_p(S^{[p^n]}/{\mathcal V})$.  
In particular, there exists a unique object
\begin{align*} 
(R^qf^{[p^n]}_{*}
(\eps^*_{(X^{[p^n]},D^{[p^n]})/K}(E^{}_K))^{\nat,\sq},P^{D^{[p^n]}},P)
\tag{11.3.3}\label{ali:pcee}
\end{align*}  
of ${\rm IsocF}^{2\sq}_p(S^{[p^n]}/{\mathcal V})$ 
such that  
\begin{align*} 
&(R^q f^{[p^n]}_{*}
(\eps^*_{(X^{[p^n]},D^{[p^n]})/K}(E^{}_K))^{\nat,\sq},P^{D^{[p^n]}},P)_T
=\tag{11.3.4}\label{ali:apce}\\
&(R^qf^{[p^n]}_{(X^{[p^n]}_{\os{\circ}{T}_1},D^{[p^n]}_{\os{\circ}{T}_1})/S^{[p^n]}(T)^{\nat}*}
(\eps^*_{(X^{[p^n]}_{\os{\circ}{T}_1},D^{[p^n]}_{\os{\circ}{T}_1})/S^{[p^n]}(T)^{\nat}}
(E^{}(\os{\circ}{T})))_{\mab Q},P^{D^{[p^n]}_{\os{\circ}{T}_1}},P)
\end{align*} 
for any object $T$ of ${\rm Enl}^{\sq}_p(S^{[p^n]}/{\mathcal V})$. 
\end{theo} 
\begin{proof} 
By using (\ref{lemm:pnlcfi}), the proof is the same as that of \cite[(5.2.3)]{nb}. 
\end{proof}

%\begin{coro}\label{coro:flft} 
%For any object $T$ of ${\rm Enl}_p(S^{[p^n]}/{\mathcal V})$, the filtered sheaf
%\begin{align*} 
%(R^qf^{[p^n]}_{\os{\circ}{T}*}
%(P_kA_{\rm zar}(X^{[p^n]}_{\os{\circ}{T}_1}/S^{[p^n]}(T)^{\nat},
%E^{}(\os{\circ}{T}))),P)\tag{11.3.1}\label{ali:kxee}
%\end{align*} 
%is a filteredly flat ${\cal K}_T$-modules. 
%In particular, the filtered sheaf 
%\begin{align*}  
%(R^qf^{[p^n]}_{\os{\circ}{T}*}
%(\eps^*_{X^{[p^n]}_{\os{\circ}{T}_1}
%/S^{[p^n]}(T)^{\nat}}E^{}(\os{\circ}{T})),P)
%\tag{11.3.2}\label{ali:pcte}
%\end{align*} 
%is a filteredly flat ${\cal K}_T$-module. 
%\end{coro} 
%\begin{proof} 
%Because the question is local on $T$, we have only to prove that 
%$$P_{q+k'}
%R^q f^{[p^n]}_{\os{\circ}{T}*}
%(P_kA_{\rm zar}
%(X^{[p^n]}_{\os{\circ}{T}_1}/S^{[p^n]}(T)^{\nat},
%E^{}(\os{\circ}{T})))$$ 
%is a flat ${\cal K}_T$-module. 
%This follows from the proof of (\ref{theo:pwfec}) and (\ref{coro:flft}).  
%\end{proof} 

\parno 
The following is a bifiltered version of \cite[(3.5)]{of}. 
%To consider the category ${\rm IsocF}^{\rm sld}_p(S^{[p^n]}/{\mathcal V})$ 
%(not ${\rm IsocF}_p(S^{[p^n]}/{\mathcal V})$) is important. 

\begin{lemm}[{\bf cf.~\cite[(5.2.6)]{nb}}]\label{lemm:nmr}
Let ${\mathcal V}'$ be a finite extension of complete discrete valuation ring of ${\mathcal V}$. 
Let $h\col S'\lo S$ be a morphism of $p$-adic formal families of log points over 
${\rm Spec}({\mathcal V}')\lo {\rm Spec}({\mathcal V})$.  
Let $g\col Y\lo X$ be a morphism of SNCL schemes over $S'\lo S$. 
%Assume that the admissible immersion {\rm (\ref{eqn:audtds})} for $Y_{}/S'$ 
%and for any object of ${\rm Enl}^{\rm sld}_p((S')^{[p^n]}/{\mathcal V}')$ exists. 
Let $f'\col Y\lo S'$ be the structural morphism. 
Then there exists a natural morphism 
\begin{align*} 
g^*\col &
g^*((R^qf^{[p^n]}_{*}
(A_{\rm zar}((X^{[p^n]},D^{[p^n]})/K,E))^{\rm sld},P^{D^{[p^n]}/K},P))\\
&\lo (R^qf'{}^{[p^n]}_{*}
(A_{\rm zar}
((Y^{[p^n]},C^{[p^n]})/K',\os{\circ}{g}{}^*(E^{})))^{\rm sld},P^{C^{[p^n]}/K},P). 
\end{align*}
in ${\rm IsocF}^{2{\rm sld}}_p((S')^{[p^n]}/{\mathcal V}')$.  
If $(Y,C)=(X,D)\times_SS'
%(=X_{}\times_{\os{\circ}{S}}{\os{\circ}{S}{}'})
$, 
then this morphism is an isomorphism. 
\end{lemm}

\begin{rema}[{\bf cf.~\cite[(5.2.7)]{nb}}]\label{rema:ud}
As in \cite[(3.6)]{of}, the bifiltered cohomological sheaf 
$(R^qf^{[p^n]}_{*}
(A_{\rm zar}((X^{[p^n]},D^{[p^n]})/K,E^{})),P^{D^{[p^n]}},P)$ 
in ${\rm IsocF}^{2\sq}_p(S^{[p^n]}/{\mathcal V})$ descends to 
the object 
$(R^qf^{[p^n]}_{*}
(A_{\rm zar}((X^{[p^n]},D^{[p^n]})/K_0,E^{})),P^{D^{[p^n]}},P)$
of ${\rm IsocF}^{2\sq}_p(S^{[p^n]}/{\mathcal W})$.
\end{rema}

\begin{theo}[{\bf Log convergence of the weight filtration (cf.~\cite[(5.2.8)]{nb})}]\label{theo:pwfaec}
Let the notations and the assumptions be as above.  
Consider the morphism $(X^{[p^m]},D^{[p^m]})\lo S^{[p^m]}$ 
%and $X^{[p]}_{1}\lo S^{[p]}_1$ 
over ${\rm Spf}({\mathcal V})$ $(m=0,1)$ 
as a morphism $(X^{[p^m]},D^{[p^m]})\lo S^{[p^m]}$ 
over ${\rm Spf}({\mathcal W})$. 
Set ${\mathcal E}^{\sq}_{\star,{\mathcal W}}:={\rm Enl}^{\square}_{\star}(S/{\mathcal W})$.  
%$(n=0,1)$. 
%Assume that there exists an admissible immersion {\rm (\ref{eqn:audgtds})}. 
Let 
$F^{\rm ar}_{X/S/{\mathcal W}} \col (X,D)\lo (X^{[p]},D^{[p]})$ 
be the abrelative Frobenius morphism 
over the morphism $S\lo S^{[p]}$ over ${\rm Spf}({\mathcal W})$.  
Let $W_{(X,D)/S^{[p]}/{\mathcal W}}
\col (X^{[p]},D^{[p]})\lo (X,D)$ 
be the projection. 
Let $E:=\{E_n\}_{n=0}^{\inf}$ be a sequence of flat coherent 
$\{{\mathcal O}_{\os{\circ}{X}{}_{\os{\circ}{T}_1}
/\os{\circ}{T}}\}_{T\in {\mathcal E}^{\sq}_{p,{\mathcal W}}}$-modules 
with a morphism 
\begin{align*} 
\Phi_n\col F_{\os{\circ}{X}{}}^*  (E_{n+1})\lo E_n
\tag{11.6.1}\label{ali:pchkee} 
\end{align*} 
of $\{{\mathcal O}_{\os{\circ}{X}{}_{\os{\circ}{T}_1}
/\os{\circ}{T}}\}_{T\in {\mathcal E}^{\sq}_{p,{\mathcal W}}}$-modules.  
Let $\os{\circ}{W}{}^{(l)}_{T}
\col (\os{\circ}{X}{}^{[p]}_{T_1})^{(l)}\lo \os{\circ}{X}{}^{(l)}_{T_1}$ 
$(l\in {\mab N})$ 
be also the projection over $\os{\circ}{T}$. 
Let 
$F^{\inf}{\textrm -}{\rm IsosF}^{2\sq}(S/{\mathcal V})$ be the category of 
$F^{\inf}$-isospans on $S/{\mathcal V}$ {\rm (\cite[(5.1.14)]{nb})}. 
Assume that for any $l,m\geq 0$ and $n\geq 0$, 
the morphism 
\begin{align*} 
& R^qf_{((\os{\circ}{X}{}^{[p]}_{T_1})^{(l)}\cap (\os{\circ}{D}{}^{[p]}_{T_1})^{(m))}/\os{\circ}{T}*}
(\os{\circ}{W}{}^{(l)*}_{T,{\rm crys}}
(E_{n+1}(\os{\circ}{T})_{(\os{\circ}{X}{}^{[p]}_{T_1})^{(l)}/\os{\circ}{T}}))_{\mab Q} \\
&\lo  
R^qf_{\os{\circ}{X}{}^{(l)}_{T_1}\cap \os{\circ}{D}{}^{(m)}_{T_1}/\os{\circ}{T}*}
(E_n(\os{\circ}{T})_{\os{\circ}{X}{}^{(l)}_{T_1}\cap \os{\circ}{D}{}^{(m)}_{T_1}/\os{\circ}{T}})_{\mab Q}
\end{align*}
induced by $\Phi_n$ is an isomorphism for any object $T$ of ${\mathcal E}^{\sq}_{p,{\mathcal W}}$.  
Then there exists an object
\begin{align*} 
\{((R^q f_{(X,D)/K*}(\eps_{(X,D)/K}^*(E_n))^{\nat,\sq},P^{D/K},P),\Phi_n)\}_{n=0}^{\inf}
\tag{11.6.2}\label{ali:pcxpkee} 
\end{align*}  
of $F^{\inf}{\textrm -}{\rm IsosF}^{2\sq}(S/{\mathcal V})$ such that 
\begin{align*} 
&(R^q f_{(X,D)/K*}(\eps_{(X,D)/K}^*(E_n))^{\nat,\sq},P^{D/K},P)_T=
(R^q f_*(A_{\rm zar}((X_{\os{\circ}{T}_1},D_{\os{\circ}{T}_1})/S(T)^{\nat},
E_n(\os{\circ}{T})))_{\mab Q},P^{D_{\os{\circ}{T}_1}},P)\tag{11.6.3}\label{ali:pcpnle}\\
&=(R^q f_{((X_{\os{\circ}{T}_1},D_{\os{\circ}{T}_1})/S(T)^{\nat}*}
(\eps^*_{(X_{\os{\circ}{T}_1},D_{\os{\circ}{T}_1})/S(T)^{\nat}}(E_n(\os{\circ}{T})))
_{\mab Q},P^{D_{\os{\circ}{T}_1}},P)
\end{align*} 
for any object $T$ of ${\rm Enl}^{\sq}_p(S/{\mathcal V})$. 
%In particular, there exists an object
%\begin{align*} 
%\{(R^qf_{*}(\eps^*_{X/K}(E_{n,K}))^{\nat,\sq},P),\Phi_n\}_{n=0}^{\inf}
%\tag{11.6.4}\label{ali:pcxakee} 
%\end{align*}  
%of $F^{\infty}{\textrm -}{\rm IsosF}^{\sq}(S/{\mathcal V})$ such that 
%\begin{align*} 
%\{(R^qf_*(\eps^*_{X/K}(E_{n,K}))^{\nat,\sq}_T,P)\}_{n=0}^{\inf}=
%(R^qf_{X_{\os{\circ}{T}_1}/S(T)^{\nat}*}
%(\eps^*_{X_{\os{\circ}{T}_1}/S(T)^{\nat}}(E_n(\os{\circ}{T})))_{\mab Q},P) 
%\tag{11.6.5}\label{ali:pcxnle}
%\end{align*} 
%for any object $T$ of ${\rm Enl}^{\sq}_p(S/{\mathcal V})$. 
\end{theo}
%\begin{proof} 
%By using (\ref{theo:pwfec}), the proof is the same as that of \cite[(5.2.8)]{nb}. 
%\end{proof}

\begin{coro}\label{coro:fenlt}
Assume that $E_n=E_0$ for any $n\in {\mab N}$. Set $E:=E_0$. 
For any object $T$ of ${\rm Enl}^{\sq}(S/{\mathcal V})$,    
\begin{align*}  
(R^q f_{((X_{\os{\circ}{T}_1},D_{\os{\circ}{T}_1})/S(T)^{\nat}*}
(\eps^*_{(X_{\os{\circ}{T}_1},D_{\os{\circ}{T}_1})/S(T)^{\nat}}(E(\os{\circ}{T}))_{\mab Q},
P^{D_{\os{\circ}{T}_1}},P)
\tag{11.7.1}\label{ali:bpce}
\end{align*} 
is a bifilteredly flat ${\mathcal K}_T$-module. 
In this case we denote 
$$\{((R^q f_{(X,D)/K*}(\eps_{(X,D)/K}^*(E_n))^{\nat,\sq},P^{D/K},P),\Phi_n)\}_{n=0}^{\inf}$$
by 
$$(R^q f_{(X,D)/K*}(\eps_{(X,D)/K}^*(E))^{\nat,\sq},P^{D/K},P),\Phi).$$ 
%In particular, the bifiltered sheaf 
%\begin{align*}  
%(R^qf_{T*}(\eps^*_{(X_{\os{\circ}{T}_1},D_{\os{\circ}{T}_1})
%/S(T)^{\nat}}(E(\os{\circ}{T})))_{\mab Q},P^{D_{\os{\circ}{T}_1}},P) 
%\tag{11.7.1}\label{ali:bpce}
%\end{align*} 
%is a bifilteredly flat ${\mathcal K}_T$-module. 
\end{coro}

\begin{exem}\label{exam:ofl} 
{\rm Let the notations be as before (\ref{theo:pwfaec}) (1). 
Then there exists an object 
\begin{align*} 
(R^qf_*({\mathcal O}_{(X,D)/K})^{\nat,\sq},P^{D/K},P)
\tag{11.8.1}\label{ali:pcxnee} 
\end{align*}  
of $F{\textrm -}{\rm IsocF}^{2\sq}(S/{\mathcal V})$ such that 
\begin{align*} 
(R^qf_*({\mathcal O}_{(X,D)/K})^{{\nat},{\sq}},P^{D/K},P)_T=
(R^qf_{(X_{\os{\circ}{T}_1},D_{\os{\circ}{T}_1})/S(T)^{\nat}*}
({\mathcal O}_{(X_{\os{\circ}{T}_1},D_{\os{\circ}{T}_1})/S(T)^{\nat}})_{\mab Q},P^{D_{\os{\circ}{T}_1}},P) 
\tag{11.8.2}\label{ali:pctoe}
\end{align*} 
for any object $T$ of ${\rm Enl}_p^{\sq}(S/{\mathcal V})$. 
(cf.~the proof of \cite[(3.7)]{of}).}
%there exists an object
%$$(R^qf_*(A_{\rm zar}((X,D)/K))^{\sq},P^D,P)$$ 
%of $F{\textrm -}{\rm IsocF}^{\sq}(S/{\mathcal V})$
%such that 
%\begin{align*} 
%(R^qf_*(A_{\rm zar}((X,D)/K))^{\sq},P^D,P)_T
%=(R^qf_*(A_{\rm zar}((X_{\os{\circ}{T}_1},D_{\os{\circ}{T}_1})/S(T)^{\nat})_{\mab Q},
%P^{D_{\os{\circ}{T}_1}},P)
%\tag{11.8.1}\label{ali:pcoe}
%\end{align*} 
%for any object $T$ of ${\rm Enl}^{\sq}_p(S/{\mathcal V})$.   
%Indeed, the assumption in (\ref{theo:pwfaec})  
%is satisfied by the base change of \cite[(1.3)]{boi} 
%(cf.~the proof of \cite[(3.7)]{of}).

%In particular,  there exists an object 
%\begin{align*} 
%(R^qf_*({\cal O}_{(X,D)/K})^{\nat,\sq},P^D,P)
%\tag{11.8.2}\label{ali:pcxnee} 
%\end{align*}  
%of $F{\textrm -}{\rm IsocF}^{\sq}(S/{\mathcal V})$ such that 
%\begin{align*} 
%(R^qf_*({\cal O}_{(X,D)/K})^{{\nat},{\sq}},P^D,P)_T=
%(R^qf_{(X_{\os{\circ}{T}_1},D_{\os{\circ}{T}_1})/S(T)^{\nat}*}
%({\cal O}_{(X_{\os{\circ}{T}_1},D_{\os{\circ}{T}_1})/S(T)^{\nat}})_{\mab Q},P^{D_{\os{\circ}{T}_1}},P) 
%\tag{11.8.3}\label{ali:pctoe}
%\end{align*} 
%for any object $T$ of ${\rm Enl}_p^{\sq}(S/{\mathcal V})$. }
\end{exem}

\bigskip
\parno
{\bf (2) Contravariant functoriality}
\bigskip
\begin{prop}\label{prop:fcuccv}
$(1)$ 
Let $g$ be as in {\rm (\ref{theo:funas})}. 
Let the notations and the assumption be as in 
{\rm (\ref{coro:fenlt})}.  
Then the log $p$-adically convergent isocrystals  
$$P^{D/K}_kR^qf_{(X,D)/K*}(\eps^*_{(X,D)/K}(E_{K}))^{\nat,\sq}\quad 
{\textrm a}{\textrm n}{\textrm d} 
\quad P_kR^qf_{(X,D)/K*}(\eps^*_{(X,D)/K}(E_{K}))^{\nat,\sq}$$
are contravariantly functorial.  
\end{prop}

%\begin{rema} 
%We leave the reader to the formulations of the log convergent 
%$F$-isocrystal versions of (\ref{prop:fcuccv}) (cf.~(\ref{prop:spcnvuc}) below). 
%\end{rema} 

\begin{prop}\label{prop:spcnvuc}
Let the notations and the assumption be as in {\rm (\ref{coro:fenlt})}. 
%Assume that $E^{\bul \leq N}_n=
%E^{\bul \leq N}_0$ for any $n\in {\mab N}$. 
%Set $E^{\bul \leq N}:=E^{\bul \leq N}_0$. 
Let ${\mathcal V}'/{\mathcal V}$ be a finite extension.   
Let $S'\lo S$ be a morphism of log $p$-adic formal families of log points 
over ${\rm Spf}({\mathcal V}')\lo {\rm Spf}({\mathcal V})$. 
Set $K':={\rm Frac}({\mathcal V}')$. 
Let $T'$ and $T$ be log $(p$-adic$)$ enlargements of $S'$ and $S$, respectively. 
Let $T'\lo T$ be a morphism of log $(p$-adic$)$ enlargements over 
$S'\lo S$. Let $u\col S'(T')^{\nat}\lo S(T)^{\nat}$ be the induced morphism.  
Let $Y$ be a log scheme over $S'$ 
which is similar to $X$ over $S$. 
Let $F$ be a similar $F$-isocrystal of 
$\{{\mathcal O}_{\os{\circ}{Y}{}_{\os{\circ}{T}{}'_1}
/\os{\circ}{T}{}'}\}_{T'\in {\mathcal E}^{\sq}_{p,{\mathcal V}}}$-modules to $E$. 
%Assume that $(17.6.1)$ and $(17.6.2)$ hold.  
%Let $k$ be a nonnegative integer or $\infty$. 
Let 
$$
\bigoplus_{k'\leq k}\bigoplus_{j\geq 
\max \{-k',0\}}
R^qf_{\os{\circ}{X}{}^{(2j+k')}\cap \os{\circ}{D}{}^{(k-k')}/K*} 
(E_{\os{\circ}{X}{}^{(2j+k')}\cap \os{\circ}{D}{}^{(k-k')}/K}
\otimes_{\mab Z}
\vp^{(2j+k'),(k-k')}((\os{\circ}{X},\os{\circ}{D})/K))$$ 
be an object of 
$F{\textrm -}{\rm Isoc}^{2\sq}_{\star}(S/{\mathcal V})$  
such that
{\footnotesize
{\begin{align*} 
&
\bigoplus_{k'\leq k}\bigoplus_{j\geq 
\max \{-k',0\}}
R^qf_{\os{\circ}{X}{}^{(2j+k')}\cap \os{\circ}{D}{}^{(k-k')}/K*} 
(E_{\os{\circ}{X}{}^{(2j+k')}\cap \os{\circ}{D}{}^{(k-k')}/K}
\otimes_{\mab Z}
\vp^{(2j+k'),(k-k')}((\os{\circ}{X},\os{\circ}{D})/K))_T\\
&= 
\bigoplus_{k'\leq k}\bigoplus_{j\geq 
\max \{-k',0\}}R^{q}f_{\os{\circ}{X}{}^{(2j+k')}_{T_1}\cap 
\os{\circ}{D}{}^{(k-k')}_{T_1}/\os{\circ}{T}*}
(E_{\os{\circ}{X}{}^{(2j+k')}_{T_1}\cap \os{\circ}{D}{}^{(k-k')}_{T_1}
/\os{\circ}{T}}\otimes_{\mab Z}  \vp^{(2j+k'),(k-k')}_{\rm crys}
((\os{\circ}{X}_{T_1},\os{\circ}{D}_{T_1})/\os{\circ}{T}))_{\mab Q}
\end{align*}}}
$\! \! \!$
for any object $T$ of ${\rm Enl}^{\sq}_p(S/{\mathcal V})$. 
Then 
%\par 
%$(1)$ 
there exists the following spectral sequence 
in $F{\textrm -}{\rm Isoc}^{2\sq}(S/{\mathcal V}):$
\begin{align*} 
{} & 
E_1^{-k,q+k}=
\bigoplus_{k'\leq k}\bigoplus_{j\geq 
\max \{-k',0\}}R^{q-2j-k}f_{\os{\circ}{X}{}^{(2j+k')}\cap \os{\circ}{D}{}^{(k-k')}/K}
(E_{\os{\circ}{X}{}^{(2j+k')}\cap \os{\circ}{D}{}^{(k-k')}
/K}\otimes_{\mab Z}  \tag{11.10.1}\label{eqn:getcpsp} \\
& \vp^{(2j+k'),(k-k')}_{\rm crys}
((\os{\circ}{X}_{T_1},\os{\circ}{D}_{T_1})/\os{\circ}{T}))(-j-k,u)\\
& \Lo 
R^qf_{(X,D)/K*}(\eps^*_{(X,D)/K}(E_K))^{\nat,\sq}.  
\quad (q\in {\mab Z}). 
\end{align*}  
\end{prop}

%\otimes_{\mab Z}\varpi_{\ul{\mu}}(\os{\circ}{X}/K))(-(k+m-j)) 

\begin{defi}
We call (\ref{eqn:getcpsp}) 
the {\it Poincar\'{e} spectral sequence} of 
$$R^qf_{(X,D)/K*}(\eps^*_{(X,D)/K}(E_K))^{\nat,\sq}$$ 
in ${\rm Isoc}_p^{\sq}(S/{\mathcal V})$ and $F{\textrm -}{\rm Isoc}^{\sq}(S/{\mathcal V})$, 
respectively.
\end{defi}

%\bigskip
%\parno
%{\bf (3) Boundary morphisms}
%\bigskip
%\parno
%\begin{prop}\label{prop:gsbc} 
%Set 
%\begin{align*}
%E_1^{-k,q+k}(X/K):= & \bigoplus_{m\geq 0}
%\bigoplus_{j\geq \max \{-(k+m),0\}} 
%R^{q-2j-k-2m}f_{\os{\circ}{X}{}^{(2j+k+m)}_m/K}
%(E_{\os{\circ}{X}{}^{(2j+k+m)}_m/K}
%\otimes_{\mab Z} 
%\tag{4.1.11.1}\label{eqn:espcvwcsp} \\
%& \vp^{(2j+k+m)}(\os{\circ}{X}_m/K))(-j-k-m).  
%\end{align*} 
%Then the boundary morphism 
%$$d^{-k,q+k}_1 \col E_1^{-k,q+k}(X/K) \lo 
%E_1^{-k+1,q+k}(X/K)$$ 
%induced by that of the spectral sequence 
%{\rm (\ref{eqn:espcvwcsp})}  
%is a morphism of log convergent $F$-isocrystals on $S/{\mathcal W}(s)$. 
%Consequently we have the following spectral sequence of 
%$F$-isocrystals on $S/{\mathcal W}(s):$   
%\begin{equation*}
%E_1^{-k,q+k}(X/K) \Lo 
%R^qf_{X/K*}
%(\eps^*_{X/K}(E^{\bul \leq N}_K)). 
%\tag{4.1.11.2}\label{eqn:consfi}
%\end{equation*} 
%\end{prop} 

\bigskip
\parno
{\bf (3) Monodromy}
\bigskip
\parno
Let the notations and the assumptions be as in {\rm (\ref{coro:fenlt})}. 
%Assume that $E^{\bul \leq N}_n=E^{\bul \leq N}_0$ $(n\in {\mab N})$. 
%Set $E^{\bul \leq N}:=E^{\bul \leq N}_0$.

\begin{prop}\label{prop:convmon} 
There exists the monodromy operator  
{\footnotesize{\begin{align*} 
N_{\rm zar} \col & (R^qf_{(X,D)/K*}(\eps^*_{(X,D)/K}(E_K)),P^{D/K},P)
\lo (R^qf_{(X,D)/K*}(\eps^*_{(X,D)/K}(E_K)),P^{D/K},P\langle -2 \rangle)(-1) 
\end{align*}}} 
in $F{\textrm -}{\rm IsocF}^{2\sq}(S/{\mathcal V})$.
\end{prop}

\begin{theo}[{\bf Bifiltered log Berthelot-Ogus isomorphism}]\label{theo:bofis}
Let the notations and the assumptions be as in {\rm (\ref{theo:pwfaec})}.  
Let $T$ be an object of ${\rm Enl}^{\sq}(S/{\mathcal V})$. 
Let $T_0\lo S$ be the structural morphism.  
Let $f_0 \col (X_{\os{\circ}{T}_0},D_{\os{\circ}{T}_0})\lo S_{\os{\circ}{T}_0}$ 
be the structural morphism.  
If there exists an SNCL lift with a relative SNCD lift
$f_1 \col (X_1,D_1) \lo S_{\os{\circ}{T}_1}$ of $f_0$, 
then there exists the following canonical 
bifiltered isomorphism
\begin{equation*} 
(R^qf_*({\mathcal O}_{(X,D)/K})^{\nat},P^{D/K},P)_T
\os{\sim}{\lo}
(R^qf_{(X_{\os{\circ}{T}_1},D_{\os{\circ}{T}_1})/S(T)^{\nat}*} 
({\mathcal O}_{(X_{\os{\circ}{T}_1},D_{\os{\circ}{T}_1})/S(T)^{\nat}})_{\mab Q},P^{D_{\os{\circ}{T}_1}},P).
\tag{11.13.1}\label{eqn:xntp}
\end{equation*} 
\end{theo}
\begin{proof} 
This follows from the proof of \cite[(3.8)]{of} and that of (\ref{theo:pwfaec}). 
\end{proof}

\section{Strict compatibility}\label{sec:st}
In this section we prove the strict compatibility of the pull-back of a morphism of 
proper SNCL schemes 
with respect to the weight filtration. 
Because the proofs of the results are the same as those of \cite[(5.4.6)]{nb} and 
\cite[(5.4.7)]{nb}, we omit the proofs.  

\begin{theo}[{\bf Strict compatibility I}]\label{theo:stpfgb}
Let the notations be as in {\rm (\ref{theo:e2dam})} and the proof of it. 
Let $(Y,C)/s'$ be an analogous object to $(X,D)/s$. 
Let $f' \col (Y,C)\lo s'$ be the structural morphism. 
Let $h\col s\lo s'$ be a morphism of log schemes. 
Let $g\col (X_{\os{\circ}{T}_0},D_{\os{\circ}{T}_0})\lo (Y_{\os{\circ}{T}{}'_0},C_{\os{\circ}{T}{}'_0})$ 
be the morphism in {\rm (\ref{eqn:xdxduss})} for the case $S=s$, $S'=s'$, 
$T={\mathcal W}(s)$ and $T'={\mathcal W}(s')$  
satisfying the conditions {\rm (6.5.2)} and  {\rm (6.5.3)}. 
%fitting into the commutative diagram {\rm (\ref{cd:xdjqbxy})}. 
Let ${\mathcal W}'$ be the Witt ring of $\Gam(s',{\mathcal O}_{s'})$ 
and set $K'_0:={\rm Frac}({\mathcal W}')$. 
Assume that $\os{\circ}{s}\lo \os{\circ}{s}{}'$ is finite. 
Let $q$ be a nonnegative integer. 
%Let us endow 
%$H^q_{{\rm crys}}(Y/{\mathcal W}(s'))\otimes_{{\mathcal W}'}K'_0$  
%with the induced filtration by $P$ on 
%$H^q_{{\rm crys}}(Y/{\mathcal W}(s'))$.   
Then the induced morphism 
\begin{equation*}
g^* \col H^q_{{\rm crys}}((Y,C)/{\mathcal W}(s'))\otimes_{{\mathcal W}'}K'_0
\lo 
H^q_{{\rm crys}}((X,D)/{\mathcal W}(s))\otimes_{\mathcal W}K_0
\tag{12.1.1}\label{eqn:gbwstn}
\end{equation*} 
is strictly compatible with the weight filtration. 
\end{theo}

\begin{theo}[{\bf Strict compatibility II}]\label{theo:stpbgb}
%$(1)$  
Let the notations be as in the beginning of {\rm \S\ref{sec:filbo}} 
%{\rm (\ref{theo:e2dgfam})} 
and {\rm (\ref{eqn:xdxduss})}. 
Let $(Y,C)/S'$ and $T'$ be analogous objects to $(X,D)/S$ 
and $T$, respectively. 
%Let $f \col X \lo S$ and $f' \col Y\lo S'$ 
%be proper $N$-truncated simplicial SNCL schemes over $S$ and over $S'$, 
%respectively.
Let $g$ be the morphism in {\rm (\ref{eqn:xdxduss})} satisfying the conditions 
{\rm (6.5.2)} and  {\rm (6.5.3)}. 
Let $q$ be a nonnegative integer. 
Then the induced morphism 
\begin{equation*}
g^* \col v^*(R^qf'_{(Y_{\os{\circ}{T}{}'_1},C_{\os{\circ}{T}{}'_1})/S'(T')^{\nat}*}
({\mathcal O}_{(Y_{\os{\circ}{T}{}'_1},C_{\os{\circ}{T}{}'_1})/S'(T')^{\nat}}))_{\mab Q}
\lo R^qf_{(X_{\os{\circ}{T}_1},D_{\os{\circ}{T}_1})/S(T)^{\nat}*}
({\mathcal O}_{(X_{\os{\circ}{T}{}_1},D_{\os{\circ}{T}_1})/S(T)^{\nat}})_{\mab Q}
\tag{12.2.1}\label{eqn:gvbstn}
\end{equation*}  
is strictly compatible with the weight filtrations $P$'s.
Consequently the induced morphism 
\begin{equation*}
g^* \col v^*(R^qf'_*({\mathcal O}_{(Y,C)/K'})^{\nat,\sq})
\lo R^qf_*({\mathcal O}_{(X,D)/K})^{\nat,\sq}
\tag{12.2.2}\label{eqn:gbstn}
\end{equation*} 
in $F{\textrm -}{\rm Isoc}^{\sq}(S/{\mathcal V})$ 
is strictly compatible with the weight filtrations $P$'s.
\end{theo}

\section{Log $p$-adic relative monodromy-weight conjecture}\label{sec:mn} 
Let the notations be as in the beginning of \S\ref{sec:filbo}. 
Following the relative monodromy filtration 
in characteristic $p>0$ in \cite[(1.8.5)]{dw2} and 
the relative monodromy filtration 
over the complex number field in 
\cite{stz} and \cite{ezth}, 
we propose the following conjecture which we call 
the {\it log $p$--adic relative monodromy-weight conjecture}:

\begin{conj}[{\bf log $p$-adic relative monodromy-weight conjecture}]
\label{conj:remc} 
Let $q$ be a nonnegative integer.  
Assume that $\os{\circ}{X}$ is projective over $\os{\circ}{S}$. 
Then the filtration $P$ on 
$R^qf_{(X_{\os{\circ}{T}_1},D_{\os{\circ}{T}_1})/S(T)^{\nat}*}
({\mathcal O}_{(X_{\os{\circ}{T}_1},D_{\os{\circ}{T}_1})/S(T)^{\nat}})_{\mab Q}$ 
is the monodromy filtration of $N$ relative to 
$P^{D_{\os{\circ}{T}_1}}$, that is,  the well-defined induced morphism 
\begin{align*} 
&N^e \col {\rm gr}^P_{q+k+e}{\rm gr}^{P^{D_{\os{\circ}{T}_1}}}_{k}
R^qf_{(X_{\os{\circ}{T}_1},D_{\os{\circ}{T}_1})/S(T)^{\nat}*}
({\mathcal O}_{(X_{\os{\circ}{T}_1},D_{\os{\circ}{T}_1})/S(T)^{\nat}})_{\mab Q}
\lo  \tag{13.1.1}\label{eqn:clrm}\\
&{\rm gr}^P_{q+k-e}{\rm gr}^{P^{D_{\os{\circ}{T}_1}}}_{k}
R^qf_{(X_{\os{\circ}{T}_1},D_{\os{\circ}{T}_1})/S(T)^{\nat}*}
({\mathcal O}_{(X_{\os{\circ}{T}_1},D_{\os{\circ}{T}_1})/S(T)^{\nat}})_{\mab Q}(-e)  
\end{align*} 
for $e,k\in {\mab N}$ {\rm ((\ref{coro:mn}))}
is an isomorphism.  
\end{conj}

We also recall the following conjecture which is 
the $p$-adic generalized version of the conjecture by K.~Kato 
(\cite{kln}, \cite[(2.0.9;$l$)]{ndeg}): 

\begin{conj}[{\bf log $p$--adic monodromy-weight conjecture}] 
\label{conj:log} 
Let $q$ be a nonnegative integer.  
Assume that $\os{\circ}{X}$ is projective over $\os{\circ}{S}$. 
Then the filtration $P$ on 
$R^qf_{X_{\os{\circ}{T}_1}/S(T)^{\nat}*}
({\mathcal O}_{X_{\os{\circ}{T}_1}/S(T)^{\nat}})_{\mab Q}$ 
is the monodromy filtration of $N$, that is,  
the induced morphism 
\begin{equation*} 
N^e \col {\rm gr}^P_{q+e}
R^qf_{X_{\os{\circ}{T}_1}/S(T)^{\nat}*}({\mathcal O}_{X_{\os{\circ}{T}_1}/S(T)^{\nat}})_{\mab Q}
\lo {\rm gr}^P_{q-e}
R^qf_{X_{\os{\circ}{T}_1}/S(T)^{\nat}*}({\mathcal O}_{X_{\os{\circ}{T}_1}/S(T)^{\nat}})_{\mab Q}(-e)  
\quad (e\in {\mab N}) 
\tag{13.2.1}\label{eqn:ulmw} 
\end{equation*} 
is an isomorphism.  
\end{conj}

\begin{rema}\label{rema:lmirm} 
In \cite{kln} Kato kindly suggested to me that 
the weight filtration and the monodromy filtration on 
the first log $l$-adic cohomology of the
degeneration of a Hopf surface 
(this is a proper SNCL surface) are different. 
However the proof in [loc.~cit.] is not complete. 
A generalization of his suggestion 
and the totally different and complete proof for the generalization 
have been given in \cite[(6.5)]{ndeg}. 
\end{rema}

It is evident that (\ref{conj:remc}) is a generalization of 
(\ref{conj:log}).  Conversely (\ref{conj:log}) implies (\ref{conj:remc}): 

\begin{theo}\label{theo:mimr} 
Assume that $\os{\circ}{X}$ is projective over $\os{\circ}{S}$.
If {\rm (\ref{conj:log})} for $D^{(k)}/S$ for any $k\in {\mab N}$ is true, 
then {\rm (\ref{conj:remc})} is true. 
%then there exists a monodromy filtration 
%$M$ on $R^qf_{(X_{\os{\circ}{T}_1},D_{\os{\circ}{T}_1})/S(T)^{\nat}*}
%({\cal O}_{(X_{\os{\circ}{T}_1},D_{\os{\circ}{T}_1})/S(T)^{\nat}})_{\mab Q}$ 
%relative to $P^{D_{\os{\circ}{T}_1}}$ and the relative monodromy filtration 
%$M$ is equal to $P$. 
\end{theo}
\begin{proof} 
By (\ref{coro:mn}) we indeed have the morphism (\ref{eqn:clrm}). 
%By (\ref{prop:cmzoqm}) and (\ref{prop:nst}), 
%\begin{align*} 
%&N(P_kR^qf_{(X_{\os{\circ}{T}_1},D_{\os{\circ}{T}_1})/S(T)^{\nat}*}
%({\cal O}_{(X_{\os{\circ}{T}_1},D_{\os{\circ}{T}_1})/S(T)^{\nat}})_{\mab Q})
%\subset \\
%& P_{k-2}R^qf_{(X_{\os{\circ}{T}_1},D_{\os{\circ}{T}_1})/S(T)^{\nat}*}
%({\cal O}_{(X_{\os{\circ}{T}_1},D_{\os{\circ}{T}_1})/S(T)^{\nat}})_{\mab Q}))_{\mab Q} \quad 
%(k\in {\mab Z}).
%\end{align*}   
%It suffices to prove that the morphism 
%(\ref{eqn:ulmw}) is an isomorphism. 
By (\ref{theo:pwfec}) we may assume that $S$ is the log point 
$({\rm Spec}({\kap}),{\mab N}\oplus {\kap}^*\lo \kap)$. 
Consider the $E_1$--term of (\ref{ali:dks})$\otimes_{\mab Z}{\mab Q}$ for 
the trivial coefficient. 
Then, by the assumption, 
$N$ induces an isomorphism 
\begin{equation*} 
N^e \col {\rm gr}^P_{q+k+e}E_1^{-k,q+k} \os{\sim}{\lo} 
{\rm gr}^P_{q+k-e}E_1^{-k,q+k}(-e) 
\end{equation*} 
%with a suitable shift of the filtration $P$ on each 
%direct factor of 
for $E_1^{-k,q+k}=
R^qf_{D^{(k)}_{\os{\circ}{T}_0}/S(T)^{\nat}*}({\mathcal O}_{D^{(k)}_{\os{\circ}{T}_0}/S(T)^{\nat}})
\otimes_{\mab Z}{\mab Q}$.  
%(If $\os{\circ}{D}_{\lam_1} \cap \cdots \cap \os{\circ}{D}_{\lam_k}$ 
%is a subscheme of an irreducible component of $\os{\circ}{X}$ 
%with dimension $d$, then the shift is $d-k$.). 
By the proof of (\ref{theo:e2dgfam}) 
the edge morphism $d^{-k,q+k}_r$ 
$(r\geq 1)$ is strictly compatible with $P$. 
Now, by the easy lemma below and by induction on $r$, 
we see that the morphism 
\begin{equation*} 
N^e \col {\rm gr}^P_{q+k+e}E_r^{-k,q+k} \lo 
{\rm gr}^P_{q+k-e}E_r^{-k,q+k}(-e) 
\end{equation*} 
is an isomorphism for any $r\geq 1$. 
\end{proof} 
%whether (\ref{conj:log}) implies (\ref{conj:remc}). 
%(\ref{prop:stc}) may be useful as in \cite[p.~527]{stz} 
%and ??? below. 

In \cite[(13.10)]{nlf} we have proved the following: 

\begin{lemm}\label{lemm:easy} 
Let 
\begin{equation*} 
\begin{CD}
(U,P^U) @>{f}>> (V,P^V) @>{g}>> (W,P^W)\\ 
@V{N}VV @V{N}VV @V{N}VV \\
(U,P^U\langle -2\rangle) @>{f}>> (V,P^V\langle -2\rangle) @>{g}>> 
(W,P^W\langle -2\rangle)
\end{CD} 
\end{equation*} 
be a commutative diagram of filtered objects of 
an abelian category such that $g\circ f=0$. 
%and 
%such that the vertical morphisms $N$'s are nilpotent. 
%Let $M$'s be the monodromy filtrations  on 
%$U$, $V$ and $W$. 
Set $\Lam=U, V,W$.  
Assume that $P^U$ is finite and 
that $f$ and $g$ are strict with respect to $P^{\Lam}$'s. 
Assume that the induced morphisms 
$N^e \col {\rm gr}^{P^{\Lam}}_e\Lam\os{\sim}{\lo} {\rm gr}^{P^{\Lam}}_{-e}\Lam$  
are isomorphisms. 
Then the  
induced morphism 
\begin{equation*} 
N^e \col {\rm gr}^{P^V}_e({\rm Ker}(g)/{\rm Im}(f)) 
{\lo} 
{\rm gr}^{P^V}_{-e}({\rm Ker}(g)/{\rm Im}(f)) 
\tag{13.5.1}\label{eqn:meg} 
\end{equation*} 
is an isomorphism.  
\end{lemm}

\begin{prop}\label{coro:nm}
Let ${\mathcal V}$ be a complete discrete valuation ring of mixed characteristics 
$(0,p)$ with perfect residue field. 
%Let $K$ be the fraction field of ${\mathcal V}$. 
Set $B=({\rm Spf}({\mathcal V}),{\mathcal V}^*)$.   
Let $\os{\circ}{T}$ be a $p$-adic formal family of log points over $B$.  
If ${\rm dim}(\os{\circ}{X}/\os{\circ}{S})\leq 2$, then {\rm (\ref{conj:log})} is true.  
\end{prop} 
\begin{proof} 
It suffices to prove that 
the conjecture (\ref{conj:log}) holds for $D^{(k)}/S$ $(k\in {\mab N})$. 
We may assume that $S$ is the log point of a perfect field of characteristic $p>0$.  
By \cite[(5.4.1)]{ndw} the $p$-adic weight spectral sequence for $D^{(k)}/S$ degenerates at $E_2$. 
The conjecture for the case $q=1$ has been proved by 
Kajiwara-Achinger  \cite[Theorem 3.6]{ash} (cf.~\cite{kaj}). 
Hence the conjecture for the case $q=3$  holds by the classical Poincar\'{e} duality and 
the explicit description of the edge morphisms of the $E_1$-terms of 
the weight spectral sequence. 
The conjecture for the case $q=2$ is true by the proof of \cite[(6.2.1)]{msemi}. 
\end{proof} 

\begin{coro}
Let the notations be as in {\rm (\ref{coro:nm})}. 
If ${\rm dim}(\os{\circ}{X}/\os{\circ}{S})\leq 2$, then {\rm (\ref{conj:remc})} is true.  
\end{coro}

\begin{theo}\label{theo:sprme} 
If, for each connected component $S'_{\os{\circ}{T}{}_0}$ of $S_{\os{\circ}{T}_0}$, 
there exists an exact closed point $t\in S_{\os{\circ}{T}{}'_0}$ 
such that the fiber $D^{(k)}_{\os{\circ}{T}{}'_0,t}/t$ of 
$D^{(k)}_{\os{\circ}{T}{}'_0}/S_{\os{\circ}{T}{}'_0}$ at $t$ 
is the log special fiber of a proper strict semistable family over a complete discrete valuation ring 
of equal characteristic, then {\rm (\ref{conj:remc})} is true. 
\end{theo} 
\begin{proof} 
By the proof of (\ref{theo:mimr}) it suffices to prove that 
(\ref{conj:log}) for $D^{(k)}/S$ $(k\in {\mab N})$ is true. 
This follows from \cite[(5.5.3)]{nb}. 
\end{proof}

\begin{prob}\label{prob:indlp}
%We consider the case $D=\emptyset$.  
Let ${\mathcal V}$ be a complete discrete valuation ring of mixed characteristics 
$(0,p)$ with perfect residue field. Let $K$ be the fraction field of ${\mathcal V}$. 
Set $B=({\rm Spf}({\mathcal V}),{\mathcal V}^*)$.   
Let $S$ be a $p$-adic formal family of log points over $B$ 
such that $\os{\circ}{S}$ is a ${\mathcal V}/p$-scheme.  
%in the sense of \cite[\S1]{of} (e.~g., ${\mathcal V}/p$-scheme).   
Let $(X,D)/S$ be a proper SNCL scheme with a relative SNCD over $S$. 
Assume that $\os{\circ}{S}$ is connected. 
Let $q$ be a nonnegative integer. 
Are the ranks of 
$P^D_kR^qf_{(X,D)/K*}({\mathcal O}_{X/K})$ and 
$P_kR^qf_{(X,D)/K*}({\mathcal O}_{X/K})$
$(k\in {\mab N})$ are equal to the ranks of 
$P^D_kR^qf_*({\mab Q}_l)$ and $P_kR^qf_*({\mab Q}_l)$ for $l\not=p$, 
respectively?
\end{prob}

\bigskip
\bigskip
\parno
\begin{center}
{{\rm \Large{\bf Appendix}}}
\end{center}

\section{Edge morphisms between 
the $E_1$-terms of $p$-adic weight spectral sequences}
\label{sec:lbddmlif} 
In this section we prove that 
the edge morphism $d_1^{\bul \bul}$ 
between the $E_1$-terms of 
(\ref{eqn:escssp}) is described by 
\v{C}ech-Gysin morphisms of two types and 
the induced \v{C}ech morphism of pull-back morphisms 
by closed immersions. 
We also prove that 
the edge morphism $d_1^{\bul \bul}$ 
between the $E_1$-terms of 
(\ref{ali:dkfs}) is described by a \v{C}ech-Gysin morphism.

\par
For the time being we consider the case 
$\os{\circ}{D}=\emptyset$ and we recall results in \cite{ndw} and \cite{nb}. 
\par 
Let $\os{\circ}{X}_{T_0}:=\bigcup_{\lam\in \Lam}\os{\circ}{X}_{\lam}$ 
be the union of smooth components of $\os{\circ}{X}_{T_0}$ over $\os{\circ}{T}_0$, 
where $\Lam$ is a set of indexes.  
We fix a total order on $\Lam$.  
Let $k\geq 2$ be an integer. 
Set 
$\Lam_k
:=\{(\lam_0, \ldots, \lam_{k-1})\in \Lam^k~\vert~
\lam_m< \lam_{m'}~(0\leq m<m' \leq k-1)\}$ 
and $\ul{\lam}:=(\lam_0, \ldots, \lam_{k-1})\in \Lam_k$. 
For an integer $0 \leq m \leq k-1$, 
set $\ul{\lam}_m:=(\lam_0, \ldots, \hat{\lam}_m, \ldots, \lam_{k-1})$. 
Set 
$\os{\circ}{X}_{\ul{\lam}}
:=\os{\circ}{X}_{\lam_{0}} \cap 
\cdots 
\cap 
\os{\circ}{X}_{\lam_{k-1}}$
and 
$\os{\circ}{X}_{\ul{\lam}_m}:=
\os{\circ}{X}_{\lam_{0}} \cap \cdots 
\cap \widehat{\os{\circ}{X}}_{{\lam}_{m}} \cap 
\cdots 
\cap \os{\circ}{X}_{\lam_{m}}$. 
Here $~~\widehat{}~~$ means the elimination. 
%Set $\os{\circ}{X}_{\ul{\lam}_m}:=\os{\circ}{X}_{\ul{\lam}_m}
%\times_{\os{\circ}{S}(0)}\os{\circ}{T}_0$
%and 
%$\os{\circ}{X}_{\ul{\lam}_{m\bet}}:=\os{\circ}{X}_{\ul{\lam}_{m\bet}}
%\times_{\os{\circ}{S}(0)}\os{\circ}{T}_0$.  
Then $\os{\circ}{X}_{\ul{\lam}}$ 
is a smooth divisor on 
$\os{\circ}{X}_{\ul{\lam}_{m}}/\os{\circ}{T}_0$.
%Let $E_{\os{\circ}{X}_{\ul{\lam}_{m}}}$ and 
%$E_{\os{\circ}{X}_{\ul{\lam}_{m\bet}}}$ 
%be the pull-backs of $\os{\circ}{p}{}^*_{T{\rm crys}}(E)$ 
%by $a_{\ul{\lam}_{m}}\col 
%\os{\circ}{X}_{\ul{\lam}_{m}}\lo \os{\circ}{X}_{T_0}$ 
%and $a_{\ul{\lam}_{m\bet}}\col 
%\os{\circ}{X}_{\ul{\lam}_{m\bet}} 
%\lo \os{\circ}{X}_{T_0}$, 
%respectively.  
For a nonnegative integer $k$ and an integer $l$, 
let
\begin{align*}
(-1)^mG^{\ul{\lam}_m}_{\ul{\lam}}
\col &
R^kf_{\os{\circ}{X}_{\ul{\lam}}/\os{\circ}{T}*}(E_{\os{\circ}{X}_{\ul{\lam}}/T}
\otimes_{\mab Z} 
\vp_{\rm crys}(\os{\circ}{X}_{\ul{\lam}}/\os{\circ}{T}))(-l) 
\tag{14.0.1}\label{eqn:egs}\\
& \lo 
R^{k+2}f_{\os{\circ}{X}_{\ul{\lam}_m}/\os{\circ}{T}*}(
E_{\os{\circ}{X}_{\ul{\lam}_m}/T}
\otimes_{\mab Z} \vp_{\rm crys}(\os{\circ}{X}_{\ul{\lam}_m}/\os{\circ}{T}))(-(l-1))
\end{align*}
be the obvious sheafied version of the Gysin morphism 
defined in \cite[(2.8.4.5)]{nh2}.
Here 
$$\vp_{\ul{\lam}{\rm crys}}
(\os{\circ}{X}_{T_0}/\os{\circ}{T}) 
\quad \text{and} \quad  
\vp_{\ul{\lam}_{m}{\rm crys}}(\os{\circ}{X}_{T_0}/\os{\circ}{T})$$  
are the crystalline orientation sheaves of
$\os{\circ}{X}_{\ul{\lam}}$ and 
$\os{\circ}{X}_{\ul{\lam}_m}$ 
in 
$(\os{\circ}{X}_{\ul{\lam}}
/\os{\circ}{T})_{\rm crys}$ 
and $(\os{\circ}{X}_{\ul{\lam}_{m}}
/\os{\circ}{T})_{\rm crys}$, 
respectively,  
defined similarly in \cite[p.~81, (2.8)]{nh2}.
Set
\begin{align*}
& G:=\bigoplus_{j\geq \max \{-k,0\}}
\sum_{\ul{\lam}=\{\lam_{0},\ldots, \lam_{2j+k}\}\in \Lam_{2j+k}}
\sum_{m=0}^{2j+k}(-1)^m
G_{\ul{\lam}}^{\ul{\lam}_{m}} \col 
\tag{14.0.2}\label{eqn:togsn}\\
&\bigoplus_{j\geq \max \{-k,0\}} 
R^{q-2j-k}
f_{\os{\circ}{X}{}^{(2j+k)}_{T_0}
/\os{\circ}{T}*}
(E_{\os{\circ}{X}{}^{(2j+k)}_{T_0}/\os{\circ}{T}}
\otimes_{\mab Z} 
\vp^{(2j+k)}_{\rm crys}(
\os{\circ}{X}_{T_0}/\os{\circ}{T}))\\ 
&(-j-k,u)\lo \\
&\bigoplus_{j\geq \max \{-k+1,0\}} 
R^{q-2j-k+2}
f_{\os{\circ}{X}{}^{(2j+k-1)}_{T_0}
/\os{\circ}{T}*}
(E_{\os{\circ}{X}{}^{(2j+k-1)}_{T_0}
/\os{\circ}{T}}
\otimes_{\mab Z} 
\vp^{(2j+k-1)}_{\rm crys}(
\os{\circ}{X}_{T_0}/\os{\circ}{T}))\\
&(-j-k+1,u). 
\end{align*}
\par 
Let $\iota_{\ul{\lam}}^{\ul{\lam}_m} \col 
\os{\circ}{X}_{\ul{\lam}} 
\os{\sus}{\lo} 
\os{\circ}{X}_{\ul{\lam}_{m}}$ be 
the natural immersion.  
The morphism $\iota_{\ul{\lam}}^{\ul{\lam}_m}$ 
induces the morphism  
\begin{equation}
(-1)^m\iota_{\ul{\lam}{\rm crys}}^{\ul{\lam}_{m}*}
\col 
\iota_{\ul{\lam}{\rm crys}}^{\ul{\lam}_{m}*}
(E_{\os{\circ}{X}_{\ul{\lam}_{m}}/\os{\circ}{T}}
\otimes_{\mab Z}\vp_{\ul{\lam}_{m}{\rm crys}}
(\os{\circ}{X}_{T_0}/\os{\circ}{T})) 
\lo 
E_{\os{\circ}{X}_{\ul{\lam}}/\os{\circ}{T}}
\otimes_{\mab Z}
\vp_{\ul{\lam}{\rm crys}}
(\os{\circ}{X}_{T_0}/\os{\circ}{T}) 
\tag{14.0.3}\label{eqn:defcbd}
\end{equation}
as in \cite[(2.11.1.2)]{nh2}.
Set 
\begin{align*}
& \rho_m:=\bigoplus_{j\geq \max \{-k,0\}}
\sum_{\ul{\lam}:=\{\lam_{0},\ldots,\lam_{2j+k}\}\in \Lam_{2j+k}}
\sum_{m=0}^{2j+k}(-1)^m
\iota_{\ul{\lam}{\rm crys}}^{\ul{\lam}_{m*}} 
\col \tag{14.0.4}\label{eqn:rhogsn}\\
&\bigoplus_{j\geq \max \{-k,0\}} 
R^{q-2j-k}f_{\os{\circ}{X}{}^{(2j+k)}_{T_0}
/\os{\circ}{T}*}(E_{\os{\circ}{X}{}^{(2j+k)}_{T_0}
/\os{\circ}{T}}
\otimes_{\mab Z}\vp^{(2j+k)}_{\rm crys}
(\os{\circ}{X}_{T_0}/\os{\circ}{T}))(-j-k,u)\lo \\ 
& \bigoplus_{j\geq \max \{-k,0\}} 
R^{q-2j-k}
f_{\os{\circ}{X}{}^{(2j+k+1)}_{T_0}/\os{\circ}{T}}
(E_{\os{\circ}{X}{}^{(2j+k+1)}_{T_0}
/\os{\circ}{T}}
\otimes_{\mab Z}\vp^{(2j+k+1)}_{\rm crys}
(\os{\circ}{X}_{T_0}/\os{\circ}{T}))(-j-k,u).
\end{align*}

\begin{prop}\label{prop:deccbd} 
The edge morphism between the $E_1$-terms of 
the spectral sequence {\rm (\ref{eqn:escssp})} 
for the case $D=\emptyset$ 
is given by the following diagram$:$  
\begin{equation*} 
\tag{14.1.1}\label{cd:gsmsd}
\end{equation*} 
\begin{equation*} 
\begin{split} 
{} & \bigoplus_{j\geq \max \{-k,0\}} 
R^{q-2j-k}
f_{\os{\circ}{X}{}^{(2j+k)}_{T_0}
/\os{\circ}{T}*}
(E_{\os{\circ}{X}{}^{(2j+k)}_{T_0}
/\os{\circ}{T}} 
%{} & \phantom{R^{q-2m-k}
%f_{(\os{\circ}{X}^{(k)}, 
%Z\vert_{\os{\circ}{X}^{(k)}_m})/S*}
%(E\quad \quad}
\otimes_{\mab Z}\vp^{(2j+k)}_{\rm crys}
(\os{\circ}{X}_{T_0}/\os{\circ}{T}))(-j-k,u)
\end{split}  
\end{equation*}  
$$
G+\rho~\downarrow$$ 
\begin{equation*} 
\begin{split} 
{} & 
\bigoplus_{j\geq \max \{-k+1,0\}} 
R^{q-2j-k+2}
f_{\os{\circ}{X}{}^{(2j+k-1)}_{T_0}
/\os{\circ}{T}*}
(E_{\os{\circ}{X}{}^{(2j+k-1)}_{T_0}
/\os{\circ}{T}} \\
{} & \phantom{R^{q-2m-k}
f_{(\os{\circ}{X}^{(k)}_{m}, 
Z\vert_{\os{\circ}{X}^{(k)}_m})/S*}
({\mathcal O}\quad \quad} 
\otimes_{\mab Z}\vp^{(2j+k-1)}_{\rm crys}
(\os{\circ}{X}_{T_0}/\os{\circ}{T}))(-j-k+1,u). 
\end{split}  
\end{equation*}  
\begin{equation*} 
\end{equation*} 
\end{prop} 
\begin{proof} 
By using (\ref{lemm:ti}) for the case $D=\emptyset$, 
this proposition follows from the proof of \cite[(10.1)]{ndw}. 
\end{proof}

\par 
Next consider the general case where the horizontal SNCD 
$\os{\circ}{D}$ is not necessarily empty.  
Let $\os{\circ}{D}_{T_0}=\bigcup_{\mu\in M}\os{\circ}{D}_{\mu}$ 
be a union of smooth divisors, 
where $M$ is a set of indexes.   
Fix a total order on $M$. 
\par 
Let $k$ be a positive integer. 
Set 
$M_k
:=\{(\mu_0, \ldots, \mu_{k-1})\in M^k~\vert~\mu_m< \mu_{m'}~(0\leq m<m' \leq k-1)\}$ 
and $\ul{\mu}:=(\mu_0, \ldots, \mu_{k-1})$. 
For an integer $0 \leq m \leq k-1$, 
set $\ul{\mu}_m:=(\mu_0, \ldots, \hat{\mu}_m, \ldots, \mu_{k-1})$. 
Set 
$\os{\circ}{D}_{\ul{\mu}}:=\os{\circ}{D}_{\mu_0} \cap \cdots \cap \os{\circ}{D}_{\mu_{k-1}}$ 
and $\os{\circ}{D}_{\ul{\mu}_m}:=\os{\circ}{D}_{\mu_0} \cap \cdots \cap 
\hat{\os{\circ}{D}}_{\mu_m} \cap \cdots \cap \os{\circ}{D}_{\mu_{k-1}}$.  
\par 
For a nonnegative integer $e$, 
let $\iota_{\ul{\mu}}^{\ul{\mu}_m}\vert_{\os{\circ}{X}{}^{(e)}} 
\col \os{\circ}{X}{}^{(e)}_{T_0}  \cap \os{\circ}{D}_{\ul{\mu}}
\os{\sus}{\lo}  \os{\circ}{X}{}^{(e)}_{T_0}\cap \os{\circ}{D}_{\ul{\mu}_m}$ 
be the restriction of $\iota_{\ul{\mu}}^{\ul{\mu}_m}$ to 
$\os{\circ}{X}_{\ul{\mu}}\cap \os{\circ}{D}{}^{(e)}$  and let 
$G_{\ul{\mu}}^{\ul{\mu}_m} \col 
({\mab Z}/l^n(-1))_{ \os{\circ}{X}{}^{(e)}_{T_0} \cap \os{\circ}{D}_{\ul{\mu}}}\{-1\} \lo 
({\mab Z}/l^n)_{\os{\circ}{X}{}^{(e)}_{T_0}\cap \os{\circ}{D}_{\ul{\mu}_m}}[1]$ 
be the Gysin morphism  of the closed immersion 
$\iota_{\ul{\mu}}^{\ul{\mu}_m}\vert_{\os{\circ}{X}{}^{(e)}_{T_0}\cap \os{\circ}{D}_{\ul{\mu}}}$. 
For a nonnegative integer $e$, 
let $\iota_{\ul{\mu}}^{\ul{\mu}_m}(\os{\circ}{D})\vert_{\os{\circ}{X}{}^{(e)}_{T_0}}  
\col \os{\circ}{X}{}^{(e)}_{T_0}\cap \os{\circ}{D}_{\ul{\mu}}
\os{\sus}{\lo} \os{\circ}{X}{}^{(e)}_{T_0}\cap \os{\circ}{D}_{\ul{\mu}_m}$
be the natural closed immersion and let 
\begin{align*}
G_{\ul{\mu}}^{\ul{\mu}_m}(\os{\circ}{D}) \col 
R^kf_{\os{\circ}{X}{}^{(e)}_{T_0}\cap {\os{\circ}{D}_{\ul{\mu}}}/\os{\circ}{T}*}
(E_{\os{\circ}{X}{}^{(e)}_{T_0}\cap {\os{\circ}{D}_{\ul{\mu}}}/T}
\otimes_{\mab Z} 
\vp_{\rm crys}(\os{\circ}{D}_{\ul{\mu}}/\os{\circ}{T}))(-l;g,\Del,\Del')  
\tag{14.1.2}\label{eqn:gdfs}\\
\lo 
R^{k+2}f_{\os{\circ}{X}{}^{(e)}_{T_0}\cap \os{\circ}{D}_{\ul{\mu}_m}/\os{\circ}{T}*}(
E_{\os{\circ}{X}{}^{(e)}_{T_0}\cap \os{\circ}{D}_{\ul{\mu}_m}/T}
\otimes_{\mab Z} \vp_{\rm crys}(\os{\circ}{D}_{\ul{\mu}_m}/\os{\circ}{T}))(-(l-1),\Del,\Del')
\end{align*} 
be the Gysin morphism of 
$\iota_{\ul{\mu}}^{\ul{\mu}_m}(\os{\circ}{D})\vert_{\os{\circ}{X}{}^{(e)}_{T_0}}$. 
\par 
Let $k$ and $q$ be integers.   
Consider the following three morphisms for 
$k' \leq k$ and $j\geq \max\{-k',0\}$ appearing 
in the edge morphism 
$E^{-k,q+k}_1 \lo E^{-(k-1),q+k}_1$ of 
the spectral sequence  {\rm (\ref{ali:dwks})}:  
{\footnotesize{\begin{align*} 
&G^{(k'),(k-k')}  :=\sum_{\ul{\lam} \in \Lam_{2j+k'}}
\sum_{m=0}^{2j+k'}(-1)^mG_{\ul{\lam}}^{\ul{\lam}_m} 
\col R^{q-2j-k}f_{\os{\circ}{X}{}^{(2j+k')}_{T_0}\cap 
\os{\circ}{D}{}^{(k-k')}_{T_0}/\os{\circ}{T}*}(
E_{\os{\circ}{X}{}^{(2j+k')}_{T_0}\cap 
\os{\circ}{D}{}^{(k-k')}_{T_0}/\os{\circ}{T}}
\otimes_{\mab Z}  \\ 
{} & (\varpi^{(2j+k')}_{\rm crys}(\os{\circ}{X}_{T_0}/\os{\circ}{T}))
\vert_{\os{\circ}{X}{}^{(2j+k')}_{T_0}\cap 
\os{\circ}{D}{}^{(k-k')}_{T_0}}\otimes_{\mab Z}
\varpi^{(k-k')}_{\rm crys}(\os{\circ}{D}_{T_0}/\os{\circ}{T})
\vert_{\os{\circ}{X}{}^{(2j+k')}_{T_0}\cap 
\os{\circ}{D}{}^{(k-k')}_{T_0}})(-j-k) 
\lo \\ 
& R^{q+1-2j-(k-1)}f_{\os{\circ}{X}{}^{(2j+k'-1)}_{T_0}\cap 
\os{\circ}{D}{}^{(k-k')}_{T_0}/\os{\circ}{T}*}(
E_{\os{\circ}{X}{}^{(2j+k'-1)}_{T_0}\cap 
\os{\circ}{D}{}^{(k-k')}_{T_0}/\os{\circ}{T}}\otimes_{\mab Z}\\
& (\varpi^{(2j+k-1')}_{\rm crys}(\os{\circ}{X}_{T_0}/\os{\circ}{T}))
\vert_{\os{\circ}{X}{}^{(2j+k-1')}_{T_0}\cap 
\os{\circ}{D}{}^{(k-k')}_{T_0}}\otimes_{\mab Z}
\varpi^{(k-k')}_{\rm crys}(\os{\circ}{D}_{T_0}/\os{\circ}{T})
\vert_{\os{\circ}{X}{}^{(2j+k'-1)}_{T_0}\cap 
\os{\circ}{D}{}^{(k-k')}_{T_0}})(-j-(k-1)), 
\end{align*} 
\begin{align*} 
\iota^{(k'),(k-k')*}
& :=\sum_{\ul{\lam} \in \Lam_{2j+k'}}
\sum_{m=0}^{2j+k'}(-1)^m\iota_{\ul{\lam}}^{\ul{\lam}_m*}  
\col R^{q-2j-k}f_{\os{\circ}{X}{}^{(2j+k')}_{T_0}\cap 
\os{\circ}{D}{}^{(k-k')}_{T_0}/\os{\circ}{T}*}(
E_{\os{\circ}{X}{}^{(2j+k')}_{T_0}\cap 
\os{\circ}{D}{}^{(k-k')}_{T_0}/\os{\circ}{T}}
\otimes_{\mab Z}  \\ 
{} & (\varpi^{(2j+k')}_{\rm crys}(\os{\circ}{X}_{T_0}/\os{\circ}{T}))
\vert_{\os{\circ}{X}{}^{(2j+k')}_{T_0}\cap 
\os{\circ}{D}{}^{(k-k')}_{T_0}}\otimes_{\mab Z}
\varpi^{(k-k')}_{\rm crys}(\os{\circ}{D}_{T_0}/\os{\circ}{T})
\vert_{\os{\circ}{X}{}^{(2j+k')}_{T_0}\cap 
\os{\circ}{D}{}^{(k-k')}_{T_0}})(-j-k) 
\lo \\ 
{} & R^{q+1-2(j+1)-(k-1)}f_{\os{\circ}{X}{}^{(2j+k'+1)}\cap 
\os{\circ}{D}{}^{(k-k')}/\os{\circ}{T}*}(E_{\os{\circ}{X}{}^{(2j+k'+1)}_{T_0}\cap 
\os{\circ}{D}{}^{(k-k')}_{T_0}/\os{\circ}{T}}
\otimes_{\mab Z}  \\ 
{} & (\varpi^{(2j+k'+1)}_{\rm crys}(\os{\circ}{X}_{T_0}/\os{\circ}{T}))
\vert_{\os{\circ}{X}{}^{(2j+k'+1)}_{T_0}\cap 
\os{\circ}{D}{}^{(k-k')}_{T_0}}\otimes_{\mab Z}
\varpi^{(k-k')}_{\rm et}(\os{\circ}{D}_{T_0}/\os{\circ}{T})
\vert_{\os{\circ}{X}{}^{(2j+k'+1)}_{T_0}\cap 
\os{\circ}{D}{}^{(k-k')}_{T_0}})(-j-k) 
\end{align*}}} 
and 
{\footnotesize{\begin{align*}
&G^{(k'),(k-k')}(\os{\circ}{D})
 :=\sum_{\ul{\mu} \in M_{k-k'}}
\sum_{m=0}^{k-k'-1}(-1)^mG_{\ul{\mu}}^{\ul{\mu}_m}(\os{\circ}{D})   
\col R^{q-2j-k}f_{\os{\circ}{X}{}^{(2j+k')}_{T_0}\cap 
\os{\circ}{D}{}^{(k-k')}_{T_0}/\os{\circ}{T}*}(
E_{\os{\circ}{X}{}^{(2j+k')}_{T_0}\cap 
\os{\circ}{D}{}^{(k-k')}_{T_0}/\os{\circ}{T}}
\otimes_{\mab Z}  \\ 
{} & (\varpi^{(2j+k')}_{\rm crys}(\os{\circ}{X}_{T_0}/\os{\circ}{T}))
\vert_{\os{\circ}{X}{}^{(2j+k')}_{T_0}\cap 
\os{\circ}{D}{}^{(k-k')}_{T_0}}\otimes_{\mab Z}
\varpi^{(k-k')}_{\rm crys}(\os{\circ}{D}_{T_0}/\os{\circ}{T})
\vert_{\os{\circ}{X}{}^{(2j+k')}_{T_0}\cap 
\os{\circ}{D}{}^{(k-k')}_{T_0}})(-j-k')(-(k-k');g,\Del,\Del')  \\ 
& 
\lo R^{q+1-2j-(k-1)}f_{\os{\circ}{X}{}^{(2j+k')}_{T_0}\cap 
\os{\circ}{D}{}^{((k-1)-k')}_{T_0}/\os{\circ}{T}*}
(E_{\os{\circ}{X}{}^{(2j+k')}_{T_0}\cap 
\os{\circ}{D}{}^{(k-k'-1)}_{T_0}/\os{\circ}{T}}
\otimes_{\mab Z}  \\ 
{} & (\varpi^{(2j+k')}_{\rm crys}(\os{\circ}{X}_{T_0}/\os{\circ}{T}))
\vert_{\os{\circ}{X}{}^{(2j+k')}_{T_0}\cap 
\os{\circ}{D}{}^{(k-k'-1)}_{T_0}}\otimes_{\mab Z}
\varpi^{((k-1)-k')}_{\rm crys}(\os{\circ}{D}_{T_0}/\os{\circ}{T})
\vert_{\os{\circ}{X}{}^{(2j+k')}_{T_0}\cap \os{\circ}{D}{}^{(k-k'-1)}_{T_0}})\\
&(-j-(k-1))(-(k-1-k');g,\Del,\Del').  
\end{align*}}}

\begin{coro}\label{coro:lbdl}
The edge morphism 
$d^{-k, q+k}_1 \col E_{1,l}^{-k, q+k} 
\lo E_{1,l}^{-k+1, q+k}$ 
of the spectral sequence  {\rm (\ref{eqn:espwfsp})} is 
identified with the following morphism$:$
\begin{equation*}
\sum_{k'\leq k}\sum_{j\geq {\rm max}\{-k', 0\}}
\{G^{(k'),(k-k')}+\iota^{(k'),(k-k')*}+
(-1)^{2j+k'+1}G^{(k'),(k-k')}(\os{\circ}{D})\}. 
\tag{14.2.1}\label{eqn:glbd}
\end{equation*}
\end{coro}
\begin{proof} 
Let the notations be as in the proof of (\ref{coro:pwspp}). 
(\ref{coro:lbdl}) follows from the following exact sequence, the following isomorphism 
obtained by the simplicial version of the isomorphism (\ref{eqn:rspys}) and 
the simplicial version of the commutative diagram (\ref{eqn:xdxgras}):
{\footnotesize{\begin{align*} 
0\lo {\rm gr}_{k-1}^P
A_{\rm zar}({\mathcal P}^{\rm ex}_{\bul}/S(T)^{\nat},{\mathcal E}^{\bul})
\lo (P_k/P_{k-2})A_{\rm zar}({\mathcal P}^{\rm ex}_{\bul}/S(T)^{\nat},{\mathcal E}^{\bul})
\lo {\rm gr}_k^P
A_{\rm zar}({\mathcal P}^{\rm ex}_{\bul}/S(T)^{\nat},{\mathcal E}^{\bul})\lo 0
\end{align*}}} 
and 
\begin{align*} 
&
{\rm gr}_k^P
A_{\rm zar}({\mathcal P}^{\rm ex}_{\bul}/S(T)^{\nat},{\mathcal E}^{\bul})
\os{\sim}{\lo} \\
&\bigoplus_{k'\leq k}
\bigoplus_{j\geq \max\{-k',0\}}
({\mathcal E}^{\bul}\otimes_{{\mathcal O}_{{\mathcal X}_{\bul}}} 
\Om^{\bul}_{\os{\circ}{\mathcal X}{}^{(2j+k')}_{\bul}
\cap \os{\circ}{\mathcal D}{}^{(k-k')}_{\bul}/\os{\circ}{T}} 
\otimes_{\mab Z}
\vp^{(2j+k'),(k-k')}_{\rm zar}
((\os{\circ}{X}_{T_0\bul},\os{\circ}{D}_{T_0\bul})/\os{\circ}{T}))[-2j-k]. 
%\tag{5.6.4}\label{ali:rzrwp}\\
\end{align*}
\end{proof} 

%\begin{prop}\label{prop:gy}
%Let 
%\begin{equation*}
%G^{(k)}(D)_{\ul{i}}^{\ul{\mu}_m} \col  
%({\mab Z}/l^n(-1))_{D_{\ul{\mu}}}\{-1\} \lo 
%({\mab Z}/l^n)_{D_{\ul{\mu}_m}}[1]
%\end{equation*} 
%be the Gysin morphism 
% of the closed immersion 
%$\iota_{\ul{\mu}}^{\ul{\mu}_m}(D)\col D_{\ul{\mu}}\os{\sus}{\lo} D_{\ul{\mu}_m}$ 
%obtained by $\iota_{\ul{i}}^{\ul{\mu}_m}(\os{\circ}{D})$. 
%Then the edge morphism $d_1^{-k,q+k} \col E_1^{-k,q+k}\lo  E_1^{-k+1,q+k}$ of 
%{\rm (\ref{ali:dks})} is expressed as 
%$\sum_{\ul{\mu} \in M_{k}}\sum(-1)^mG^{(k)}(D)_{\ul{\mu}}^{\ul{\mu}_m}$. 
%\end{prop} 
%\begin{proof} 
%This is obvious. 
%\end{proof} 

Let ${\mathfrak D}_{\bul \ul{\mu}_m}$ and ${\mathfrak D}_{\bul \ul{\mu}}$ be 
the log PD-envelopes of the exact closed immersions 
$D_{\os{\circ}{T}_0\bul \ul{\mu}_m}\os{\sus}{\lo} {\mathcal D}_{\bul \ul{\mu}_m}$ and 
$D_{\os{\circ}{T}_0\bul \ul{\mu}}\os{\sus}{\lo} {\mathcal D}_{\bul \ul{\mu}}$ 
over $S(T)^{\nat}$, respectively. 
Consider the following exact sequence 
\begin{align*} 
&0\lo 
{\mathcal O}_{{\mathfrak D}_{\bul \ul{\mu}_m}}
\otimes_{{\mathcal O}_{{\mathcal D}_{\bul \ul{\mu}_m}}}
\Om^{\bul}_{{\mathcal D}_{\bul \ul{\mu}_m}/\os{\circ}{T}}
\otimes_{\mab Z}\vp_{\ul{\mu}_m{\rm zar}}
(\os{\circ}{\mathcal D}_{\bul}/\os{\circ}{T})\\
&\lo 
{\mathcal O}_{{\mathfrak D}_{\bul \ul{\mu}_m}}
\otimes_{{\mathcal O}_{{\mathcal D}_{\bul \ul{\mu}_m}}}
\Om^{\bul}_{{\mathcal D}_{\bul \ul{\mu}_m}/\os{\circ}{T}}
(\log {\mathcal D}_{\bul \ul{\mu}})
\otimes_{\mab Z}\vp_{\ul{\mu}_m{\rm zar}}
(\os{\circ}{\mathcal D}_{\bul}/\os{\circ}{T})\\
&\os{{\rm Res}_{\bul \mu_m}}{\lo} 
\iota_{\ul{\mu}}^{\ul{\mu}_m}(\os{\circ}{D}_{T_0\bul})_*
({\mathcal O}_{{\mathfrak D}_{\bul \ul{\mu}}}\otimes_{{\mathcal O}_{{\mathcal D}_{\bul \ul{\mu}}}}
\Om^{\bul}_{{\mathcal D}_{\bul \ul{\mu}}/\os{\circ}{T}}
\otimes_{\mab Z}\vp_{\ul{\mu}{\rm zar}}
(\os{\circ}{\mathcal D}_{\bul}/\os{\circ}{T}))[-1]\lo 0,
\end{align*} 
where ${\rm Res}_{\bul \mu_m}$ is 
the Poincar\'{e} residue morphism with respect to 
${\mathcal D}_{\bul \mu_m}$. 
Let 
\begin{align*} 
G^{\bul\ul{\mu}_m}_{\bul\ul{\mu}}\col &
\iota_{\ul{\mu}}^{\ul{\mu}_m}(\os{\circ}{D}_{T_0\bul})_*
({\mathcal O}_{{\mathfrak D}_{\bul \ul{\mu}}}\otimes_{{\mathcal O}_{{\mathcal D}_{\bul \ul{\mu}}}}
\Om^{\bul}_{{\mathcal D}_{\bul \ul{\mu}}/\os{\circ}{T}}
\otimes_{\mab Z}\vp_{\ul{\mu}{\rm zar}}
(\os{\circ}{\mathcal D}_{\bul}/\os{\circ}{T}))[-1]\\
& \lo 
{\mathcal O}_{{\mathfrak D}_{\bul \ul{\mu}_m}}
\otimes_{{\mathcal O}_{{\mathcal D}_{\bul \ul{\mu}_m}}}
\Om^{\bul}_{{\mathcal D}_{\bul \ul{\mu}_m}/\os{\circ}{T}}
\otimes_{\mab Z}\vp_{\ul{\mu}_m{\rm zar}}
(\os{\circ}{\mathcal D}_{\bul}/\os{\circ}{T})[1]
\end{align*}
be the boundary morphism of the exact sequence above. 
This morphism induces the following morphism 
\begin{align*} 
G^{\ul{\mu}_m}_{\ul{\mu}} \col & 
R^qf_{D_{\os{\circ}{T}_0\ul{\mu}_m}/S(T)^{\nat}*}(
\eps^*_{D_{\os{\circ}{T}_0\ul{\mu}_m}/S(T)^{\nat}}
(E_{\os{\circ}{D}_{\os{\circ}{T}_0\ul{\mu}_m}/\os{\circ}{T}})\otimes_{\mab Z}
\eps^{-1}_{D_{\os{\circ}{T}_0\ul{\mu}_m}/S(T)^{\nat}}\vp_{\ul{\mu}_m{\rm crys}}
((\os{\circ}{D}_{\os{\circ}{T}_0}/\os{\circ}{T}_0)))\tag{14.2.2}
\label{ali:dkfis}\\
&\lo 
R^{q+2}f_{D_{\os{\circ}{T}_0\ul{\mu}}/S(T)^{\nat}*}(
\eps^*_{D_{\os{\circ}{T}_0\ul{\mu}}/S(T)^{\nat}}
(E_{\os{\circ}{D}_{\os{\circ}{T}_0\ul{\mu}}/\os{\circ}{T}})\otimes_{\mab Z}
\eps^{-1}_{D_{\os{\circ}{T}_0\ul{\mu}}/S(T)^{\nat}}\vp_{\ul{\mu}{\rm crys}}
((\os{\circ}{D}_{\os{\circ}{T}_0}/\os{\circ}{T}_0))). 
\end{align*} 
Set 
{\footnotesize{\begin{align*} 
&G^{(k)}(D):=\sum_{\ul{\mu}\in M_k}\sum_{m=0}^{k-1}\col 
R^{q-k}f_{D^{(k)}_{\os{\circ}{T}_0}/S(T)^{\nat}*}
(\eps^*_{D^{(k)}_{\os{\circ}{T}_0}/S(T)^{\nat}}
(E_{\os{\circ}{D}{}^{(k)}/\os{\circ}{T}})\otimes_{\mab Z}
\eps^{-1}_{D^{(k)}_{\os{\circ}{T}_0}/S(T)^{\nat}}\vp^{(k)}_{\rm crys}
((\os{\circ}{D}_{T_0}/\os{\circ}{T}_0)))(-k;g,\Del,\Del') \\
&\lo 
R^{q+2-k}f_{D^{(k-1)}_{\os{\circ}{T}_0}/S(T)^{\nat}*}
(\eps^*_{D^{(k-1)}_{\os{\circ}{T}_0}/S(T)^{\nat}}
(E_{\os{\circ}{D}{}^{(k-1)}/\os{\circ}{T}})\otimes_{\mab Z}
\eps^{-1}_{D^{(k-1)}_{\os{\circ}{T}_0}/S(T)^{\nat}}\vp^{(k-1)}_{\rm crys}
((\os{\circ}{D}_{T_0}/\os{\circ}{T}_0)))(-(k-1);g,\Del,\Del'). 
\tag{14.2.3}\label{ali:dke1s}
\end{align*}}} 

\begin{prop}\label{prop:dem}
The edge morphism 
$d^{-k, q+k}_1 \col E_{1,l}^{-k, q+k} 
\lo E_{1,l}^{-k+1, q+k}$ 
of the spectral sequence  {\rm (\ref{ali:dkfs})} is 
identified with the morphism {\rm (\ref{ali:dke1s})}.
\end{prop}
\begin{proof}
We leave the proof to the reader. 
\end{proof}

\bigskip
\parno
Yukiyoshi Nakkajima 
\parno
Department of Mathematics,
Tokyo Denki University,
5 Asahi-cho Senju Adachi-ku,
Tokyo 120--8551, Japan.

\end{document}